\documentclass[11pt]{amsart}
\usepackage{epsfig}
\usepackage{psfrag}
\usepackage{amssymb}
\usepackage[usenames, dvipsnames]{color}
\usepackage{txfonts}
\usepackage{hyperref}
\usepackage{verbatim}
\usepackage[pagewise]{lineno}
\usepackage{marginnote}
\usepackage{todonotes}
\usepackage[normalem]{ulem}
\usepackage{cancel}
\usepackage{soul}
\usepackage{enumerate}
\usepackage{mathrsfs}

\theoremstyle{plain}
\newtheorem{theorem}{Theorem}[section]
\newtheorem{lemma}[theorem]{Lemma}
\newtheorem{corollary}[theorem]{Corollary}
\newtheorem{proposition}[theorem]{Proposition}

\newtheorem{definition}[theorem]{Definition}

\newtheorem{example}[theorem]{Example}

\theoremstyle{remark}
\newtheorem{remark}[theorem]{Remark}

\newcommand{\bR}{{\mathbb R}}

\def\m{\mathbb}

\def\lam{\lambda}

\def\t{\tilde}
\def\wt{\widetilde}

\def\th{\theta}

\def\G{\Gamma}

\def\Dl{\Delta}

\def\i{\infty}
\def\e{\epsilon}

\def\p{\partial}
\def\f{\frac}
\def\na{\nabla}
\def\a{\alpha}

\def\O{\Omega}
\def\o{\omega}

\def\be{\begin{equation}}
\def\ee{\end{equation}}
\def\bes{\begin{equation*}}
\def\ees{\end{equation*}}
\def\bali{\begin{aligned}}
\def\eali{\end{aligned}}
\def\al{\begin{aligned}}
\def\eal{\end{aligned}}
\def\erf{\eqref}
\def\lab{\label}

\def\2O{\underline{\O}}

\def\ol{\overline}

\numberwithin{equation}{section}

\makeatletter
\def\dashint{\operatorname%
{\,\,\text{\bf--}\kern-.98em\DOTSI\intop\ilimits@\!\!}}
\makeatother

\begin{document}


\title[bounded strong solution passing a cone]{Finite speed axially symmetric Navier-Stokes flows passing a cone}

\thanks{}

\author[]{Zijin Li, Xinghong Pan, Xin Yang, Chulan Zeng, Qi S. Zhang and Na Zhao}

\address[]{(Z.  Li) School of Mathematics and Statistics, Nanjing University of Information Science and Technology, Nanjing 210044, China}

\email{zijinli@nuist.edu.cn}

\address[]{(X.  Pan) College of Mathematics and Key Laboratory of MIIT, Nanjing University of Aeronautics and Astronautics, Nanjing 211106, China}

\email{xinghong{\_}87@nuaa.edu.cn}

\address[]{(X.  Yang) Department of mathematics, University of California, Riverside, CA 92521, USA}

\email{xiny@ucr.edu}

\address[]{(C.  Zeng) Department of mathematics, University of California, Riverside, CA 92521, USA}

\email{czeng011@ucr.edu}

\address[]{(Q.  S.  Zhang) Department of mathematics, University of California, Riverside, CA 92521, USA}

\email{qizhang@math.ucr.edu}

\address[]{(N.  Zhao) School of Mathematics,  Shanghai University of Finance and Economics,  Shanghai 200433, China}

\email{zhaona@shufe.edu.cn}



\begin{abstract}
Let $D$ be the exterior of a cone inside a ball, with its altitude angle at most $\pi/6$ in $\mathbb{R}^3$,  which touches the $x_3$ axis at the origin.  For any initial value $v_0 = v_{0,r}e_{r} + v_{0,\theta} e_{\theta} + v_{0,3} e_{3}$ in a  $C^2(\overline{D})$ class, which has the usual even-odd-odd symmetry in the $x_3$ variable and has the partial smallness only in the swirl direction:
$ | r v_{0, \theta} | \leq \frac{1}{100}$,
the axially symmetric Navier-Stokes equations (ASNS) with Navier-Hodge-Lions slip boundary condition has a finite-energy solution that stays bounded for all time. In particular, no finite-time blowup of the fluid velocity occurs. Compared with standard smallness assumptions on the initial velocity,  no size restriction is made on the components $v_{0,r}$ and $v_{0,3}$. In a broad sense, this result appears to solve $2/3$ of the regularity problem of ASNS in such domains in the class of solutions with the above symmetry. Equivalently, this result is connected to the general open question which asks that if an absolute smallness of one component of the initial velocity implies the global smoothness, see e.g. page 873 in \cite{CZZ17}. Our result seems to give a positive answer in a special setting.

As a byproduct, we also construct an unbounded solution of the forced Navier Stokes equation in a special cusp domain that has finite energy. The forcing term, with the scaling factor of $-1$, is in the standard regularity class. This result confirms the intuition that if the channel of a fluid is very thin, arbitrarily high speed in the classical sense can be attained under a mildly singular force which is physically reasonable in view that Newtonian gravity and Coulomb force have scaling factor $-2$.
\end{abstract}
\maketitle

\tableofcontents

\section{Introduction}

The goal of the paper is to construct a class of global bounded solutions to the axially symmetric Navier-Stokes equations, abbreviated as ASNS henceforth.
\be
\begin{aligned}
\label{eqasns}
\begin{cases}
   \big (\Delta-\frac{1}{r^2} \big )
v_r-(v_r \p_r + v_3 \p_{x_3})v_r+\frac{(v_{\theta})^2}{r}-\partial_r
P-\p_t  v_r=0,  \\
   \big   (\Delta-\frac{1}{r^2}  \big
)v_{\theta}-(v_r \p_r + v_3 \p_{x_3} )v_{\theta}-\frac{v_{\theta} v_r}{r}-
\partial_t v_{\theta}=0,\\
 \Delta v_3-(v_r \p_r + v_3 \p_{x_3})v_3-\p_{x_3} P-\p_t v_3=0,\\
 \frac{1}{r} \p_r (rv_r) +\p_{x_3}
v_3=0.
\end{cases}
\end{aligned}
\ee
Here,  $v = v_{r}e_{r} + v_{\th}e_{\th} + v_{3}e_{3}$ is the velocity in the cylindrical system with the standard basis $\{e_{r}, e_{\th}, e_{3}\}$, where for any $ x=(x_1,x_2,x_3)\in\bR^3 $, $ r=\sqrt{x_1^2+x_2^2}$ and
\be\label{ert3}
e_r=(x_1/r, x_2/r, 0),\quad e_\th=(-x_2/r, x_1/r, 0),\quad e_3=(0,0,1).\ee
The components $v_{r}$, $v_{\th}$ and $v_{3}$ are independent of the azimuthal angle $\th$. Although ASNS is a special case of the full 3D Navier-Stokes equations,
\be
\label{nse}
\Delta v -  (v\cdot \nabla) v - \nabla P -\partial_t v =0, \quad \text{div} \, v=0,
\ee
the regularity problem of the former is as wide open as the latter. In the last several decades, there has been an outburst of research on ASNS,  see e.g. \cite{La, UY, CSTY1, CSTY2, KNSS, HLL, CFZ, LZ17, Weid, Zha22} and the references therein.  Especially after it was realized in \cite{LZ17} that ASNS is essentially a critical system, there is some expectation that the regularity problem is becoming accessible one way or the other.

A little of the expectation is achieved in \cite{Zha22} where the regularity problem is solved for a cusp domain under the Navier-slip boundary condition.  This is the first time that the regularity problem of ASNS is settled when the essential difficulty is beyond that in 2D.  Actually, the regularity problem of the 3D Navier-Stokes equations is also solved in \cite{MTL90} under the helical symmetry assumption of the solution. It is such an assumption that makes the classical 2D Ladyzhenskaya's inequality available in 3D.  With that being said, the fundamental obstacle of the 3D regularity problem is absent in this situation.

One may feel that the cusp domain in \cite{Zha22} is somewhat special. In the current paper,  we consider the ASNS in some wider domains, those outside a cone (see Figure \ref{Fig,domain-cyl}),  which seems to be the next most feasible case.
The problem we are studying can be used to model water flows in a circular lake passing a cone shaped reef. Although we are not able to fully solve the regularity problem in our main result, Theorem \ref{Thm, cyl}, since there is a size assumption on the initial velocity, this assumption is only applied in the swirl direction and no size assumption is made on the other components of the initial velocity.

Since there are many well-established results of global smoothness for the Navier Stokes equations involving size assumptions for the initial value, we hereby explain the main new feature of this paper. The standard global smoothness result for ASNS in the literature can be summarized as follows. There exists a function $ \lam=\lam(s) $, whose value goes to $+\infty$ as $ s\to +\infty $, such that for any small $ \epsilon>0 $, the solution to the ASNS is globally smooth if the initial condition $v_0$ satisfies
\[
\| v_{0, \theta} \|_X < \epsilon, \quad \| v_{0, r} \|_Y +\| v_{0, 3} \|_Y<  \lambda(\epsilon^{-1}).
\]
Here $X$ is a scaling-invariant suitable space of various choices, and $ \| \cdot \|_{Y} $ is a quantity which may involve both velocity and vorticity. Notice that the non-swirl components $ v_{0, r}$ and $v_{0, 3}$ of the initial velocity are also restricted in size, unless the swirl component $v_{0, \theta}=0$.
In contrast, these restrictions are removed in our Theorem \ref{Thm, cyl} below. This result is also connected to the general open question, which asks that if an absolute smallness of one component of the initial velocity implies the global smoothness, see e.g. page 873 in \cite{CZZ17}  in which the space $ X = \dot{H}^{\frac12} $. Our result seems to give a positive answer in the special setting stated in Theorem \ref{Thm, cyl}.


Now we make more precise description of the domains in this paper which are the exterior of certain cones inside a ball that touches the $x_3$ axis at the origin. We remark that similar regions were also introduced before to study other fluid problems, such as the singular formation for Euler flows \cite{EJ19}, but these regions are bounded away from the $ x_3 $-axis.

\begin{definition}\label{Def, domain}
Let $\a\in\big(0,\frac{\pi}{2}\big)$ be any fixed angle.  The domain $D$ with boundary surfaces $R_1$, $R_2$ and $A$ is defined in the cylindrical coordinates as follows  (also see Figure \ref{Fig,domain-cyl}):
\be\label{domain-cyl}
D=\big\{(r,\theta,x_3): 0<r^2+x_3^2<1, \, -r\tan\a <x_3< r\tan\a, \, \theta\in [0,2\pi) \big\}.\ee
Moreover, for convenience of notation, we denote
\[\p^{R}D = R_{1} \cup R_2,\quad \p^{A}D = A,\]
where the superscripts $R$ and $A$ stand for the radial boundary and the annular boundary respectively.

\begin{figure}[!ht]

\begin{tikzpicture}[scale=0.65]
\draw [->] (0,0)--(6,0) node [anchor=north west]{$r$};
\draw [->] (0,-2.8)--(0,2.8) node [anchor=south west] {$x_3$};
\draw (0,0) node [left] {$O$};
\draw [red, thick, domain=0:4.5] plot({\x},{\x/1.732});
\draw [red, thick, domain=0:4.5] plot({\x},{-\x/1.732});
\draw [red, thick, domain=60:120] plot({5.196*sin(\x)},{5.196*cos(\x)});
\draw [green, thick, domain=60:90] plot({0.5*sin(\x)},{0.5*cos(\x)});
\draw [green, thick, domain=90:120] plot({0.6*sin(\x)},{0.6*cos(\x)});
\draw (0.5,0.2) node[right][green] {\large $\a$};
\draw (0.6,-0.2) node[right][green] {\large $\a$};
\draw (2.8,0.3) node[right][purple] {\large $D$};
\draw (2.2,2.2) node[right][blue] {\large $R_1$};
\draw (2.2,-2.2) node[right][blue] {\large $R_2$};
\draw (5.2,0.3) node[right][blue] {\large $A$};
\draw (8,1) node[right][blue] {\large $A: r^2+x_3^2=1$};
\draw (8,0) node[right][blue] {\large $R_1: x_3=r\tan\a$};
\draw (8,-1) node[right][blue] {\large $R_2: x_3=-r\tan\a$};
\end{tikzpicture}
\hspace{0.1in}
\includegraphics[scale=0.2]{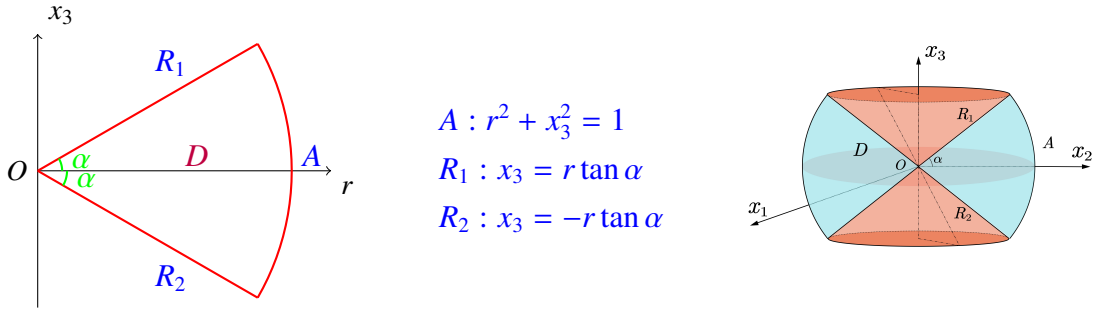}
\caption{Domain $D$ in cylindrical coordinates}
\label{Fig,domain-cyl}
\end{figure}
\end{definition}

The associated boundary condition is
\be\label{NHL slip bdry}
v\cdot n=0, \quad \o\times n=0, \quad \text{on}\quad \p D,\ee
where $n$ is the unit outward normal on the smooth part of $\p D$ and $\o$ is the vorticity defined as
\be\label{curl}
\o =\text{curl $v$} =\nabla\times v. \ee
Condition (\ref{NHL slip bdry}) is a special case in a family of boundary conditions proposed by Navier \cite{Nav}. This condition has been studied extensively in the literature and was attributed to different authors.  For example, it was studied in \cite{TZ96}. Later, it was called the Navier-Hodge boundary condition in \cite{MM09DIE}, and the Navier-Lions boundary condition in \cite{Kel06}.  For this reason, we will name it as Navier-Hodge-Lions boundary condition in this paper, which is abbreviated as the NHL boundary condition thereafter.  For more details on the history of this boundary condition and other types of Navier boundary conditions, see also (\cite{Kel06, XX, CQ10, MR, PR17, PR19, BdVY20}).

Due to Leray \cite{Le2}, if $D=\bR^3$, $v_0 \in L^2(\bR^3)$, the Cauchy problem (\ref{nse}) has a solution in the energy space (c.f. \erf{enorm} below). By finite energy, we mean the solutions are in the energy space $\mathbf{E}=L^2_t H^{1}_x \cap L^\infty_t L^2_x$.  Here and throughout, the norm in $\mathbf{E}$ for a function $v$ on $D \times [0, T]$ is taken as
\be
\label{enorm}
\Vert v \Vert^2_\mathbf{E} = \int^T_0 \int_D |\nabla v|^2 dxdt + \sup_{t \in [0, T]} \int_D |v(x, t)|^2 dx.
\ee
Here, $ T>0 $ and the function $v$ can be  vector-valued or scalar-valued, depending on the context. The solutions with finite energy are also called Leray-Hopf solutions. In general,
it is not known if Leray-Hopf solutions stay bounded or regular for all $t>0$. Recently, by allowing a super-critical forcing term in (\ref{nse}), it was shown in \cite{ABC22} that even with zero initial value and identical forcing term, Leray-Hopf solutions may not be unique after some finite time, and thus some singularity occurs.

In this paper, we will focus on a special case of \eqref{nse}, namely when $v$ and $P$ are independent of the azimuthal angle $ \th $ in the cylindrical coordinate system $(r,\,\th,\,x_3)$. Although ASNS seems more complicated than the full 3-dimensional equation, a simplification happens in the 2nd equation where the pressure term disappears. For a succinct derivation of the ASNS (\ref{eqasns}) using the tensor notations, we refer the readers to \cite{Zha22}. If the swirl $v_\theta=0$, then it is well-known that finite energy solutions to the Cauchy problem of (\ref{eqasns}) in $\mathbb{R}^3$ are smooth for all time $t>0$, see e.g. \cite{La, UY, LMNP}. In the presence of swirl, it is still not known in general if finite energy solutions blow up in finite time.

By the partial regularity result in \cite{CKN}, possible singularity for suitable weak solutions of ASNS can only appear at the $x_3$ axis. See also \cite{Linf} for a simplified proof and \cite{BuZh} for  the same statement but without the "suitable" requirement. Moreover, in \cite{CSTY1, CSTY2, KNSS, SS}, it was shown that if
\be
\lab{v<1/r}
|v(x, t)| \le
\frac{C}{r},
\ee
then finite energy solutions to the Cauchy problem of ASNS are smooth for all time. Here, $C$ is any positive constant. Later, there are some logarithmic improvements on the order of the criterion (\ref{v<1/r}), see e.g. \cite{Panx, Ser22a, Ser22b, CTZarxiv}. Also see \cite{Tao21} for a similar improvement in full 3D Navier-Stokes equations. In contrast, the energy bound scales as
$-1/2$. So even with axial symmetry, there is a finite scaling gap which makes the ASNS supercritical, just like the full equations. Promisingly in \cite{LZ17},  the authors revealed the following property.

{\it The vortex stretching term of the ASNS is critical after a  suitable change of dependent variables.}

Thus, the aforementioned scaling gap is zero, which makes the regularity problem of ASNS appears less formidable. Nevertheless, all major open problems are still open.

The main result in \cite{LZ17} includes the following statement.
Let $\delta_0 \in (0, \frac{1}{2})$ and $C_{*} > 1$.
If
\begin{equation}\label{CD}
	\sup_{0 \leq t < T}|r v_\theta (r, x_3, t)| \leq C_{*} |\ln r|^{- 2},\ \ r \leq \delta_0,
\end{equation}
then the above $v$ is regular globally in time.
Note that a priori we have $|r v_\theta (r, x_3, t)| \leq C$ by the maximal principle applied on equation
\erf{eqvth} of $ \Gamma $:
\be
\lab{eqvth}
\Delta \Gamma - b \cdot \nabla \Gamma- \frac{2}{r} \p_r
\Gamma-\p_t \Gamma=0,
\ee
where $\Gamma= r v_\theta$ and $b=v_r e_r + v_3 e_3$.  So there is still a gap of logarithmic nature from regularity. Later, the power index $ -2 $ in (\ref{CD}) was improved to $ -\frac32 $ in  \cite{Weid}.

Now we specify the meaning of solutions to ASNS (\ref{eqasns}) associated with the NHL boundary condition (\ref{NHL slip bdry}). In the rest of this paper, functions and vector fields are always assumed to be axially symmetric with respect to the $ x_3 $-axis unless stated otherwise. Fix any $ T>0 $ and any $ v_0\in H^2 (D) $ that is divergence free in $ D $ and satisfies the NHL boundary condition (\ref{NHL slip bdry}). Consider
\be\label{NS1} \left\{\, \begin{aligned}
	\Delta v - (v\cdot \nabla) v - \nabla P - \p_{t} v = 0 \quad\text{in}\quad & D\times (0,T], \\
	\nabla \cdot v = 0  \quad \text{in} \quad & D\times (0,T], \\
	v\cdot n = 0,\quad \o\times n = 0 \quad\text{on} \quad & \p D\times (0,T],\\
	v(\cdot, 0) = v_0(\cdot) \quad\text{in} \quad & D,
\end{aligned} \right.\ee
where $ \o=\nabla\times v $. Define the space of testing vector fields to be
\be\label{ts}\begin{split}
\mathscr{S}_{T} := &  \Big\{ f\in  H_t^1 L_x^2 \cap L_t^2 H_x^2 \big( D\times[0,T] \big):  \nabla\cdot f = 0 \text{\ in \ } D\times[0, T], \\
& \quad  f\cdot n = 0 \text{\ on \ } \p D\times [0, T]  \Big\}.
\end{split}\ee
If there exist $ v\in\mathscr{S}_T $ and $ P\in L_t^2 H_x^1(D\times[0,T]) $ such that $ (v,P) $ satisfies (\ref{NS1}), then we test (\ref{NS1}) by any function $ f\in \mathscr{S}_T $ to obtain (see Section \ref{Subsec, weak-form-soln} for detailed computations)
	\be\label{test by sf}\begin{split}
		& \int_{D} v(x,T) \cdot f(x,T) \,dx + \int_{0}^{T} \int_{D} (\nabla\times v) \cdot (\nabla \times f) \,dx\,dt \\
		= & \int_{D} v_0(x) \cdot f(x,0) \,dx - \int_{0}^{T}\int_{D} [(v\cdot \nabla) v] \cdot f \,dx\,dt + \int_{0}^{T}\int_{D} v\cdot (\p_t f) \,dx\,dt.
	\end{split}\ee
	If we replace $ f $ by $ v $, then (\ref{test by sf}) yields the following identity:
	\be\label{eif}
	\int_{D} |v(x,T)|^2 \,dx + 2\int_{0}^{T} \int_{D} | \nabla\times v(x,t) |^2 \,dx\,dt = \int_{D} |v_0(x)|^2 \,dx.
	\ee
	We point out that the left-hand side of (\ref{eif}) is not the energy norm (\ref{enorm}) of $ v $. Actually, without further assumptions, it is not clear if (\ref{eif}) implies the uniform (in time) finite-energy of the solution since the $ L^2 $ norm of $ \na v $ may not be controlled by the $ L^2 $ norm of $ \na\times v $, see the discussion in Section \ref{Subsec, energy-ineq}.
	

%
%
%

In this paper, we are looking for strong solutions of (\ref{NS1}) which are defined as below.
\begin{definition}\label{Def, ss}
	If there exist $ v\in L_t^2 H_x^2 \cap H_t^1 L_x^2 \big(D\times[0,T]\big) $ and $ P\in L_t^2 H_x^1  \big(D\times[0,T]\big) $ such that $ (v,P) $ satisfies (\ref{NS1}) in $ L_{tx}^2 $ sense, then $ v $ or $ (v,P) $ is called a strong solution of (\ref{NS1}) on $ D\times [0,T] $.
\end{definition}

Note that if $ (v,P) $ is a strong solution, then $ (v,P) $ satisfies (\ref{NS1}) almost everywhere. In addition, both the integration identities (\ref{test by sf}) and (\ref{eif}) are valid for $ v $. For the bounded domain $ D $ in (\ref{domain-cyl}) with $ \a\in\big(0, \frac{\pi}{6}\big] $ and under the NHL boundary condition (\ref{NHL slip bdry}), we manage to obtain a strong solution to ASNS (\ref{eqasns}) under the assumptions (i) and (\ref{small gamma-cyl}) in the main result of this paper, Theorem \ref{Thm, cyl}, which removes the logarithmic term in (\ref{CD}). We emphasize that the assumption (\ref{small gamma-cyl}) is only made on the initial swirl $ v_{0,\th} $ and no smallness restriction is imposed on the other components $ v_{0,r} $ and $ v_{0,3} $. Assumption (i) is a symmetry condition which we describe now.

\begin{definition}\label{Def, eoo sym}
	Let $ v = v_{r}e_{r} + v_{\th}e_{\th} + v_{3}e_{3} $ be a vector field in $ \mathbb{R}^3 $. We say $ v $ has the even-odd-odd symmetry if $ v_{r} $ is even, and $ v_{\th} $ and $ v_{3} $ are odd symmetric in $ x_3 $.
\end{definition}

This symmetry condition will be not only used to find a strong solution, but also utilized to establish the uniform (in time) energy inequality (\ref{v-energy-est}) in Theorem \ref{Thm, cyl}. Next, we introduce the admissible class $ \mathscr{A} $ of the initial vector fields that we consider in this paper. Since the original domain $ D $ touches the $ x_3 $ axis with an angle, the singularity of the velocity might have more chance to occur. Moreover, the solution may not be expected to have higher regularity than $ L_t^2 H_x^1$. In order to acquire more regularity and to prove the boundedness of the velocity, we first cut the corner of $ D $ and then study the problem in approximating domains $ D_m $ $ (m\geq 2) $, which are defined as
\be\label{app domain-cyl}
D_m = \bigg\{(r,\theta,x_3): \frac{1}{m^2} < r^2+x_3^2 < 1, \, -r\tan\a <x_3< r\tan\a, \, \theta\in [0,2\pi) \bigg\}.\ee
The NHL boundary condition (\ref{NHL slip bdry}) associated with $D_m$ is:
\be\label{NHL slip bdry for Dm}
v\cdot n=0, \quad \o\times n=0, \quad \text{on}\quad \p D_m.
\ee	
Due to the above strategy, it is natural to choose the elements in $ \mathscr{A} $ to be the  limits of vector fields on $ D_m $.

\begin{definition}[Admissible classes $ \mathscr{A}_m $ and $ \mathscr{A} $]
\label{Def, admissible sets}
	Fix any angle $ \a\in\big(0, \frac{\pi}{2}\big) $.
	\begin{itemize}
		\item[(1)] For any integer $ m\geq 2 $, we define the admissible class $ \mathscr{A}_{m} $ on $ D_m $ to be the space of vector fields $ v_{0}^{(m)} $ in $ C^2(\ol{D_m}) $ that are divergence free in $ D_m $ and satisfy the NHL condition (\ref{NHL slip bdry for Dm}) on $ \p D_m$.
		
		\item[(2)] For the domain $ D $, we define the admissible class $ \mathscr{A} $ on it to be the space of vector fields $ v_0 $ in $ C^2(\ol{D}) $ such that there exist vector fields $ \big\{v_0^{(m)}\big\}_{m\geq 2} $ such that $ v_0^{(m)}\in\mathscr{A}_m $ and
		\[ \lim_{m\to\infty} \big\| v_0 - v_0^{(m)} \big\|_{C^2(\ol{D_m})} = 0. \]
	\end{itemize}
\end{definition}

Now we are ready to state the main result of this paper.

\begin{theorem}\label{Thm, cyl}
	Let the domain $D$ be as defined in (\ref{domain-cyl}) with the angle $\a\in\big(0,\frac{\pi}{6}\big]$. Suppose the initial velocity $v_0$ lies in the admissible class $\mathscr{A}$ with the following two properties:
	\begin{itemize}
		\item[(i)] $ v_0 $ has the even-odd-odd symmetry as in Definition \ref{Def, eoo sym};
		\item[(ii)] the swirl component of the initial velocity satisfies
		\be\label{small gamma-cyl}
		\sup_{D} r|v_{0,\th}|\leq \f{1}{100}.\ee
	\end{itemize}
	Then for any $ T>0 $, equation (\ref{eqasns}) with the initial data $ v_0 $ and the NHL boundary condition (\ref{NHL slip bdry}) has a strong solution $ (v,P) $ on $ D\times[0,T] $ such that $ v $ is bounded uniformly in time and possesses the even-odd-odd symmetry. More precisely,
	\be\label{ss-ub} \|v\|_{L_{tx}^\infty(D\times[0,T])} + \|v\|_{H_t^1 L_x^2(D\times[0,T])} + \|v\|_{L_t^2 H_x^2(D\times[0,T])} + \| P \|_{L_t^2 H_x^1(D\times[0,T])} \leq C,
	\ee
	where $ C $ is a constant that only depends on $ \a $ and $ \| v_0\|_{C^2(\ol{D})} $.
	In addition, the following energy inequality holds:
	\be\label{v-energy-est}
	\int_{D} |v(x, T)|^2 \,dx + \frac23 \int^T_0 \int_{D} |\na v(x, t)|^2 \,dx\,dt \le \int_{D} |v_0(x)|^2 \,dx.
	\ee
	On the other hand, if $ (\t{v}, \t{P}) $ is another strong solution on $ D\times[0,T] $ with the even-odd-odd symmetry, then $ \t{v} $ coincides with the above strong solution $ v $.
\end{theorem}

\begin{remark}
We remark that without the symmetry assumption (i) in the above theorem, there is indication that the energy inequality (\ref{v-energy-est}) may fail. One example is the stationary solution
\[ v= \frac{1}{r} e_\th, \]
which satisfies equation (\ref{eqasns}) with $ P = -\frac{1}{2r^2} $, and the NHL boundary condition (\ref{NHL slip bdry}), in the domain $ D_m $ for any $ m\geq 2 $. See also Section 7.
\end{remark}

We also want to mention that there do exist vector fields $ v_0 = v_{0,\rho}e_\rho + v_{0,\phi}e_\phi + v_{0,\th}e_\th $ in $ \mathscr{A} $ which satisfy the assumptions (i) and (ii) in Theorem \ref{Thm, cyl}, and for which the size of $ v_{0,\rho} $ and $ v_{0,\phi} $ can be chosen arbitrarily large. We will provide such an initial vector field $ v_0 $ in Example \ref{Ex, iv}.

Let us describe the organization of the paper. After some preparations in Section \ref{Sec, Pre}, we will prove, in Section \ref{Sec, exist on Dm},  the existence and uniqueness of strong solutions in approximating domains $D_m$. The core of the paper is contained in Section \ref{Sec, inf-ub} and Section \ref{Sec, hreg-ub} where we will prove the required uniform a priori bounds on the solutions found in Section \ref{Sec, exist on Dm}. After these two sections, the proof of the main result, Theorem \ref{Thm, cyl}, will be completed in Section \ref{Sec, exist-uniq}. Finally, as a byproduct of studying the NHL boundary condition, we will construct a special class of blowup solutions of \eqref{eqasns} on some cusp domains in Section \ref{Sec, blowupsoln}.

Here are some key ideas in the proofs. The first step is to rewrite the ASNS and the vorticity equations in the spherical coordinate system.  It is well known that the vorticity equations contain the supercritical vortex stretching terms which block the path to the standard energy estimates, without size restrictions on all components. Our new input is the discovery of two new quantities $ K $ and $ F $ (see (\ref{K-F-O})) involving the vorticity for which the vortex stretching terms become critical. In addition, the boundary behaviors of these quantities are manageable so that an energy estimate can be achieved under only the partial smallness condition (ii) in Theorem \ref{Thm, cyl}. One may wonder, if the well known quantities $\Omega=\omega_\theta/r$ (see \cite{UY}) and $J=\omega_r/r$ (see \cite{CFZ})  in the cylindrical system are still useful in our situation. It turns out that $ \O $ is still necessary but we are not able to control the boundary terms coming out from the equation of $J$. The next step is to derive an energy estimate for the system of equations for $K$, $F$ and $ \O $ (see (\ref{eq of K-F-O})).  Since there are a large number of terms in the system, which need to be handled separately, and which may satisfy various boundary conditions, the calculation will be relatively long.  Although the modified vortex equations for $F, K, \Omega$ are essentially critical now, some of the bad terms still appear bigger than the good viscosity term in size.  This is also why we need the extra restrictions on the angle of the domain and the even-odd-odd symmetry of the data. We hope to remove these restrictions in the future. With the energy estimate in hand, we can prove the boundedness of the velocity $ v $ by using a modified version of the Biot-Savart law and the Moser's iteration.

We finish the introduction with a list of some notations and conventions to be used throughout this paper.
\begin{itemize}	
	\item Functions or vector fields in this paper are always assumed to be axially symmetric unless stated otherwise.
	
	\item The velocity field is usually called $v$ and the vorticity $\nabla \times v$ is denoted as $\o$. We use subscripts to denote their components in either the cylindrical or spherical coordinate systems (see Section \ref{Sec, Pre}). For instance, $v_{\rho} = v \cdot e_{\rho}$, $\o_\th = \o \cdot e_\th$, $\o_\phi= \o \cdot e_\phi$. Here, $ \th $ refers to the azimuthal (longitude) angle and $ \phi $ is the angle between the radius vector and the positive $ x_3 $-axis. In addition, we write $b=v_r e_r + v_3 e_3 = v_\rho e_\rho + v_\phi e_\phi$.
	
	\item $L^p(\O)$, $p \ge 1$, denotes the usual Lebesgue space on a domain $\O$ which may be a spatial, temporal or space-time domain. Let $X$ be a Banach space defined for functions on $\O \subset \bR^3$. $L^p(0, T; X)$ represents the Bochner-Banach space of functions $f$ on the space time domain $D \times [0, T]$ with the norm $\left(\int^T_0 \Vert f(\cdot, t) \Vert^p_X dt\right)^{1/p}$. We also use $L^p_xL^q_t$ or $L^q_tL^p_x$ to denote the mixed $p, q$ norm in space time.
	
	\item  Let $\O \subset \bR^3$ be an open domain, then $H^1(\O)=W^{1, 2}(\O)=\{f: \, f, \, \nabla f \in L^2(\O) \}$ and $H^2(\O)=W^{2, 2}(\O)=\{f : \,  f,\, \nabla f, \,  \nabla^2 f\in L^2(\O) \}$, denote the standard Sobolev spaces on $ \O $. Meanwhile, for any time interval $ I\subset\m{R} $,  the notation $ H^1(I) $ means the Sobolev space $ W^{1,2}(I) $.
	
	\item  Interchangeable notations $\text{div}\, v = \nabla \cdot v$, $ \text{curl}\, v = \nabla\times v $ will be used.
	
	
	\item $B(x, r)$ denotes the ball of radius $r$ centered at $x$ in a Euclidean space; and $B_X(f, r)$ denotes the open ball in a normed space $X$, centered at $f \in X$ with radius $r$.
	
	\item We use $C$ or $ C_i\, (i\geq 1) $ with or without index to denote generic constants which may change from line to line. Sometimes, we will make the dependence of constants on parameters explicitly. For example, the notation $C=C(a,b\dots)$ or $ C=C_{a,b,\dots} $ means that the constant C only depends on $a,b\dots$.
	
\end{itemize}

\section{Preliminaries }
\label{Sec, Pre}

Although the Navier-Stokes equations under the spherical coordinates are well-known, various notations exist in literatures. In this section, we will first fix the notations and derive the basic equations for the key quantities $ K $, $ F $ and $ \O $  in Section \ref{Subsec, notation}. We point out that the equation (\ref{asns-sph}) for the velocity $ v $ and the equation (\ref{asns vor-sph}) for the vorticity $ \o $ may look slightly differently from other literatures since we have rewritten some terms based on the divergence free condition. Then we will introduce some inequalities of Poincar\'e's or Hardy's type which will be used in latter sections. Furthermore, we will establish the a priori $ L^\infty $ bound for another crucial quantity $ \Gamma $.

\subsection{Reformulation of equations in spherical system with unknowns $K=\o_\rho/\rho$, $F=\o_\phi/\rho$ and $\O=\o_\th/(\rho \sin \phi)$}
\label{Subsec, notation}
\quad

Due to the geometry of the domain $D$ and the boundary condition (\ref{NHL slip bdry}),  it may be more beneficial to adopt the spherical coordinates $(\rho,\phi,\theta)$,  where $\rho$ is the radial distance and $\phi$ is the angle between the radius vector and the positive $x_3$ axis.  The relation between the cylindrical coordinates and the spherical coordinates  is
\be\label{coord-sph to cyl}
\begin{pmatrix}
r \\ \th \\ x_3
\end{pmatrix}
=
\begin{pmatrix}
\rho\sin\phi \\ \th \\ \rho\cos\phi
\end{pmatrix}.\ee
For any axially symmetric vector field $v$, we denote
\[v=v_{\rho}(\rho,\phi, t)e_{\rho}+v_{\phi}(\rho,\phi, t)e_{\phi}+v_{\th}(\rho,\phi, t) e_{\th},\]
where
\be\label{erpt}
e_\rho=
\begin{pmatrix}
\sin\phi\cos\th \\ \sin\phi\sin\th \\ \cos\phi
\end{pmatrix}
,\quad
e_{\phi}=
\begin{pmatrix}
\cos\phi\cos\th \\
\cos\phi\sin\th\\
-\sin\phi
\end{pmatrix}
,\quad e_\th=
\begin{pmatrix}
-\sin\th \\ \cos\th \\ 0
\end{pmatrix}.
\ee
Then
\be\label{cyl-to-sph}
\left\{\begin{array}{l}
e_\rho=\sin\phi\,e_{r}+\cos\phi\,e_3,\\
e_\phi=\cos\phi\,e_{r}-\sin\phi\,e_3,
\end{array}\right.
\qquad
\left\{\begin{array}{l}
v_\rho=\sin\phi\,v_{r}+\cos\phi\,v_3,\\
v_\phi=\cos\phi\,v_{r}-\sin\phi\,v_3.
\end{array}\right.\ee

Under the spherical coordinates, the domain $D$ in (\ref{domain-cyl}) is equivalent to the following (also see Figure \ref{Fig,domain-sph})
\be\label{domain-sph}
D=\big\{(\rho,\phi,\th): 0<\rho <1, \, \frac{\pi}{2}-\a < \phi < \frac{\pi}{2}+\a, \, \theta\in [0,2\pi) \big\}.\ee


\begin{figure}[!ht]
	\centering
	\begin{tikzpicture}[scale=2]
		\draw [->] (0,0)--(3,0) node [anchor=north west]{$\rho$};
		\draw [->] (0,0)--(0,2.5) node [anchor=south west] {$\phi$};
		\draw (-0.1,0.05) node [below] {$O$};
		
		\draw [red, thick, domain=0:2] plot({\x},{2*pi/3});
		\draw [red, thick, domain=0:2] plot({\x},{pi/3});
		\draw [red, thick, domain=pi/3:2*pi/3] plot({2},{\x});
		\draw [red, thick, domain=pi/3:2*pi/3] plot({0},{\x});
		\draw [blue, dashed, domain=0:pi/3] plot({2},{\x});
		
		\draw (2,0) node[below][blue] {$\rho = 1$};
		\draw (0,pi/3) node[left][blue] {$\phi_1 = \frac{\pi}{2}-\a$};
		\draw (0, 2*pi/3) node[left][blue] {$\phi_2 = \frac{\pi}{2}+\a$};
		
		\draw (0.85,1.57) node[right][purple] {\large $D$};
		\draw (2.05,1.57) node[right][blue] {\large $A$};
		\draw (1.05,2.1) node[above][blue] {\large $R_{2}$};
		\draw (1.05,1) node[below][blue] {\large $R_{1}$};
		
		\draw (3.5,1.9) node[right][blue] {\large $A: \rho=1$};
		\draw (3.5,1.4) node[right][blue] {\large $R_{1}: \phi=\frac{\pi}{2}-\a$};
		\draw (3.5,0.9) node[right][blue] {\large $R_{2}: \phi=\frac{\pi}{2}+\a$};
		
	\end{tikzpicture}
	\caption{Domain $D$ in spherical coordinates}
	\label{Fig,domain-sph}
\end{figure}
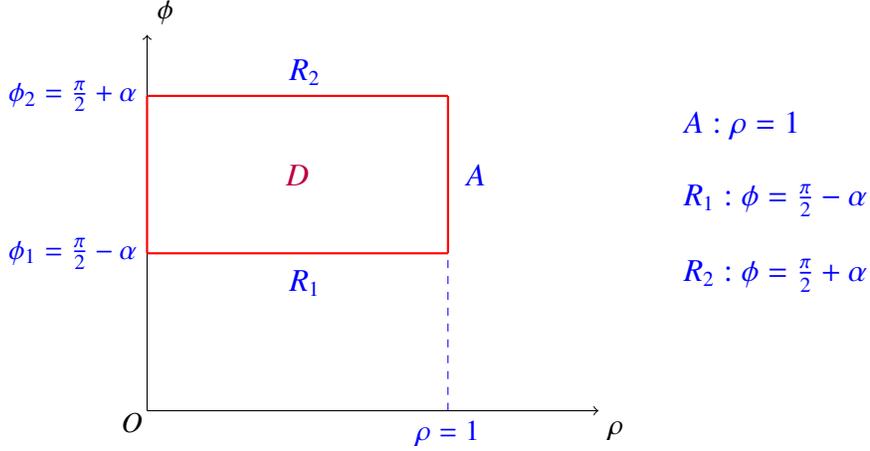

The boundary condition (\ref{NHL slip bdry}) becomes
\be\label{NHL slip bdry-sph}\left\{\begin{aligned}
 v_{\phi} = \o_\rho = \o_\th = 0,  &\quad \text{on} \quad \p^{R}D;\\
 v_\rho = \o_\phi = \o_\th = 0,  &\quad \text{on} \quad \p^{A}D,
\end{aligned}\right.\ee
and the initial vector field $v_0$ can be written as
\be\label{initial v-sph}
v_0 =v_{0,\rho}(\rho,\phi) e_\rho+v_{0,\phi}(\rho,\phi) e_\phi+v_{0,\theta}(\rho,\phi) e_{\th}.\ee

We can convert (\ref{eqasns}) from the cylindrical coordinates to the spherical coordinates.
For simplicity in notation, we denote $b=v_{r} e_{r}+v_{3}e_{3}$, or equivalently, in the spherical coordinates,
\[b=v_{\rho}e_{\rho}+v_{\phi}e_{\phi}.\]
Then (\ref{eqasns}) can be rewritten as the following well-known system for which we give a short derivation
 in Appendix \ref{Subsec, deri of eq of v}.

\be\label{asns-sph}
\left\{\begin{array}{l}
\big(\Delta + \frac{2}{\rho}\,\p_\rho + \frac{2}{\rho^2} \big)v_{\rho}-b\cdot\nabla v_{\rho}+\frac{1}{\rho}(v_{\phi}^2+v_{\th}^2)-\p_{\rho}P -\p_{t}v_{\rho}=0,  \vspace{0.08in}\\
\big(\Delta-\frac{1}{\rho^2\sin^2\phi}\big)v_{\phi}-b\cdot\nabla v_{\phi}+\frac{2}{\rho^2}\p_{\phi}v_{\rho} -\frac{1}{\rho}v_{\rho}v_{\phi}+\frac{\cot\phi}{\rho}v_{\th}^2-\frac{1}{\rho}\p_{\phi}P -\p_{t}v_{\phi}=0, \vspace{0.08in}\\
\big(\Delta-\frac{1}{\rho^2\sin^2\phi}\big)v_{\th}-b\cdot\nabla v_{\th}-\frac{1}{\rho}(v_{\rho}+\cot\phi\, v_{\phi})v_{\th} -\p_{t}v_{\th}=0,  \vspace{0.08in}\\
\frac{1}{\rho^2}\p_{\rho}(\rho^2 v_{\rho})+\frac{1}{\rho\sin\phi}\p_{\phi}(\sin\phi\,v_{\phi})=0.
\end{array}\right.\ee

%

We remark that under the spherical coordinates, the assumption (i) in Theorem \ref{Thm, cyl} means that $v_{0,\rho}$ is even, and $v_{0,\phi}$ and $v_{0,\th}$ are odd symmetric with respect to the plane $\big\{\phi=\frac{\pi}{2}\big\}$, respectively. In other words,
\be\label{sym-sph}
v_{0,\rho}(\rho, \phi) = v_{0,\rho}(\rho,\pi-\phi),  \quad v_{0, \phi}(\rho,\phi) = - v_{0,\phi}(\rho,\pi-\phi), \quad v_{0,\theta}(\rho,\phi) = - v_{0,\theta}(\rho,\pi-\phi).\ee

The quantity $\Gamma:=rv_{\th}$, in the cylindrical coordinate case, can now be expressed in the spherical coordinates as
\be\label{Gamma-sph}
\Gamma=\rho\sin\phi\, v_{\th}. \ee
It then follows from (\ref{eqvth}) that $ \Gamma $ satisfies the equation below.
\be\label{eq of Gamma-sph}
\Delta\Gamma-b\cdot\nabla\Gamma-\frac{2}{\rho}\,\p_{\rho}\Gamma-\frac{2\cot\phi}{\rho^2}\,\p_{\phi}\Gamma-\p_{t}\Gamma=0.\ee
Moreover, the restriction (\ref{small gamma-cyl}) is converted to be
\be\label{small gamma-sph}
\sup_{D} \rho\sin\phi\, |v_{0,\th}|\leq \f{1}{100}.\ee

The vorticity $\o=\nabla \times v$ can be written as $\o=\o_{\rho} e_{\rho}+\o_{\phi} e_{\phi}+\o_{\th} e_{\th}$, where
\be\label{vor f-sph}
\left\{\begin{array}{l}
\o_{\rho}=\frac{1}{\rho}(\p_{\phi}+\cot\phi)v_{\th}=\frac{1}{\rho\sin\phi}\,\p_{\phi}(\sin\phi\,v_{\th}),\vspace{0.08in}\\
\o_{\phi}=-(\p_{\rho}+\frac{1}{\rho})v_{\th}=-\frac{1}{\rho}\,\p_{\rho}(\rho v_{\th}), \vspace{0.08in}\\
\o_{\th}=(\p_{\rho}+\frac{1}{\rho})v_{\phi}-\frac{1}{\rho}\p_{\phi}v_{\rho}=\frac{1}{\rho}\,\p_{\rho}(\rho v_{\phi})-\frac{1}{\rho}\,\p_{\phi}v_{\rho}.
\end{array}\right.\ee
Meanwhile, $\o$ satisfies the following well-known system for which we also give a short derivation in appendix (\ref{Subsec, deri of eq of omega}).
\be\label{asns vor-sph}
\left\{\begin{array}{l}
\big(\Delta + \frac{2}{\rho}\,\p_{\rho} +\frac{2}{\rho^2}\big)\o_{\rho}-b\cdot\nabla \o_{\rho} + \o\cdot\nabla v_{\rho} -\p_{t}\o_{\rho}=0,  \vspace{0.1in}\\
\big(\Delta-\frac{1}{\rho^2\sin^2\phi}\big)\o_{\phi}-b\cdot\nabla \o_{\phi}+\frac{2}{\rho^2}\p_{\phi}\o_{\rho}+\o\cdot\nabla v_{\phi} +\frac{1}{\rho}(v_{\rho}\o_{\phi}-\o_{\rho}v_{\phi})-\p_{t}\o_{\phi}=0,   \vspace{0.1in}\\
\big(\Delta-\frac{1}{\rho^2\sin^2\phi}\big)\o_{\th}-b\cdot\nabla \o_{\th}+\frac{1}{\rho}(v_{\rho}+\cot\phi\,v_{\phi})\o_{\th}-\frac{1}{\rho^2}\p_{\phi}(v_{\th}^2)+\frac{\cot\phi}{\rho}\p_{\rho}(v_{\th}^2)-\p_{t}\o_{\th}=0,  \vspace{0.1in}\\
\frac{1}{\rho^2}\p_{\rho}(\rho^2 \o_{\rho})+\frac{1}{\rho\sin\phi}\p_{\phi}(\sin\phi\,\o_{\phi})=0.
\end{array}\right.\ee
Due to the presence of some super-critical terms in the above vorticity equation (\ref{asns vor-sph}), it is actually more effective to consider modified quantities $K$, $F$ and $\O$ which are defined by
\be\label{K-F-O}
K=\frac{\o_{\rho}}{\rho},\quad F=\frac{\o_{\phi}}{\rho},\quad \O=\frac{\o_{\th}}{\rho\sin\phi}.\ee
It follows from (\ref{asns vor-sph}) that $K$, $F$ and $\O$ satisfy the system below:
\be\label{eq of K-F-O}
\left\{\begin{array}{ll}
	\big(\Delta +\frac{4}{\rho}\p_\rho +\frac{6}{\rho^2} \big) K - b\cdot\nabla K + \o\cdot\nabla \big(\frac{v_\rho}{\rho}\big)-\p_{t}K = 0,  \vspace{0.1in}\\
	\big(\Delta + \frac{2}{\rho}\p_\rho + \frac{1-\cot^2 \phi}{\rho^2}\big) F - b\cdot \nabla F +\frac{2}{\rho^2}\p_{\phi}K + \o \cdot \nabla \big(\frac{v_\phi}{\rho}\big)-\p_{t}F = 0, \vspace{0.1in}\\
	\big(\Delta +\frac{2}{\rho}\p_\rho +\frac{2\cot\phi}{\rho^2}\p_\phi \big) \O-b\cdot\nabla\O-\frac{2v_{\th}}{\rho\sin\phi}\,(K+\cot\phi\, F)-\p_{t}\O = 0, \vspace{0.1in}\\
	\frac{1}{\rho^2}\, \p_{\rho}(\rho^3 K) + \frac{1}{\sin\phi}\, \p_{\phi}(\sin\phi\, F) = 0.
\end{array}\right.\ee
The derivations of (\ref{vor f-sph}), (\ref{asns vor-sph}) and (\ref{eq of K-F-O}) can be found in  Appendix \ref{Subsec, deri of KFO}. Meanwhile, since
\[ K + \cot\phi\, F = \frac{\o_\rho}{\rho} + \cot\phi\, \frac{\o_\phi}{\rho} = \frac{1}{\rho^2}\, \p_\phi v_\th - \frac{\cot\phi}{\rho}\, \p_{\rho} v_\th, \]
the third equation for $ \O $ in (\ref{eq of K-F-O}) is equivalent to
\be\label{alt eq for O}
	\bigg(\Delta + \frac{2}{\rho}\p_\rho + \frac{2\cot\phi}{\rho^2}\p_\phi \bigg) \O - b\cdot\nabla\O - \p_{t}\O = \frac{1}{\rho^3\sin\phi}\,\p_\phi(v_\th^2) - \frac{\cot\phi}{\rho^2\sin\phi}\,\p_\rho(v_\th^2).
\ee

Noticing that the system (\ref{eq of K-F-O}) contains two vortex stretching terms $\o \cdot \nabla \big(\frac{v_\rho}{\rho}\big)$ and $\o \cdot \nabla \big(\frac{v_\phi}{\rho}\big)$,  we hope to find relations between $\frac{v_\rho}{\rho}$, $\frac{v_\phi}{\rho}$ and $K, F, \O$ so that we can close the energy estimate. Similar to the cylindrical case,  one is able to establish equations between $\frac{v_\rho}{\rho}$, $\frac{v_\phi}{\rho}$  and $\O$, see Section \ref{Subsec, BS law}.  In this manner, the vortex stretching terms become critical, which allows us to prove the main result.

\subsection{Boundary conditions in approximating domains $D_m$ in spherical coordinates}

\quad

Under the spherical coordinates, the domain $ D_m $ in (\ref{app domain-cyl}) is equivalent to the following (also see Figure \ref{Fig,app domain-sph}):
\be\label{app domain-sph}
D_m=\Big\{(\rho,\phi,\th): \frac{1}{m}<\rho <1, \, \frac{\pi}{2}-\a < \phi < \frac{\pi}{2}+\a, \, \theta\in [0,2\pi) \Big\}.\ee
In addition, for convenience of notation, we denote the four pieces of the boundary $\p D_m $ to be $ R_{1,m} $, $ R_{2,m} $, $ A_{1,m} $ and $ A_{2,m} $, and write $\p^{R}D_{m} = R_{1,m} \cup R_{2,m}$, $ \p^{A}D_{m} = A_{1,m} \cup A_{2,m}$.
\begin{figure}[!ht]
	\centering
	\begin{tikzpicture}[scale=2]
		\draw [->] (0,0)--(3,0) node [anchor=north west]{$\rho$};
		\draw [->] (0,0)--(0,2.5) node [anchor=south west] {$\phi$};
		\draw (-0.1,0.05) node [below] {$O$};
		
		\draw [red, thick, domain=0.4:2] plot({\x},{2*pi/3});
		\draw [red, thick, domain=0.4:2] plot({\x},{pi/3});
		\draw [red, thick, domain=pi/3:2*pi/3] plot({0.4},{\x});
		\draw [red, thick, domain=pi/3:2*pi/3] plot({2},{\x});
		
		\draw [blue, dashed, domain=0:pi/3] plot({0.4},{\x});
		\draw [blue, dashed, domain=0:pi/3] plot({2},{\x});
		\draw [blue, dashed, domain=0:0.4] plot({\x},{pi/3});
		\draw [blue, dashed, domain=0:0.4] plot({\x},{2*pi/3});
		
		\draw (0.5,0) node[below][blue] {$\rho_1 = \frac1m$};
		\draw (2,0) node[below][blue] {$\rho_2 = 1$};
		\draw (0,pi/3) node[left][blue] {$\phi_1 = \frac{\pi}{2}-\a$};
		\draw (0, 2*pi/3) node[left][blue] {$\phi_2 = \frac{\pi}{2}+\a$};
		
		\draw (1.1,1.57) node[right][purple] {\large $D_{m}$};
		\draw (0.4,1.57) node[right][blue] {\large $A_{1,m}$};
		\draw (2,1.57) node[right][blue] {\large $A_{2,m}$};
		\draw (1.3,2.1) node[above][blue] {\large $R_{2,m}$};
		\draw (1.3,1) node[below][blue] {\large $R_{1,m}$};
		
		
		
		\draw (3.5,2.4) node[right][blue] {\large $A_{1,m}: \rho=\frac{1}{m}$};
		\draw (3.5,1.9) node[right][blue] {\large $A_{2,m}: \rho=1$};
		\draw (3.5,1.4) node[right][blue] {\large $R_{1,m}: \phi=\frac{\pi}{2}-\a$};
		\draw (3.5,0.9) node[right][blue] {\large $R_{2,m}: \phi=\frac{\pi}{2}+\a$};
		
	\end{tikzpicture}
	\caption{Domain $D_m$ in spherical coordinates}
	\label{Fig,app domain-sph}
\end{figure}
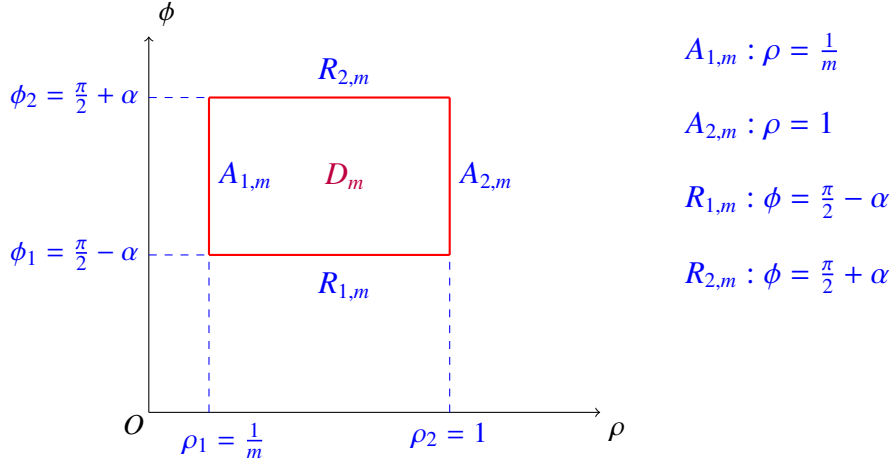

Then the NHL boundary condition (\ref{NHL slip bdry for Dm}) associated with $D_m$ becomes:
\be\label{NHL slip bdry for Dm-sph}\begin{cases}
 v_{\phi} = \o_\rho = \o_\th = 0,  &\quad \text{on} \quad \p^{R}D_{m};\\
 v_\rho = \o_\phi = \o_\th = 0,  &\quad \text{on} \quad \p^{A}D_{m}.
\end{cases}\ee
Making use of the vorticity formula (\ref{vor f-sph}), we see (\ref{NHL slip bdry for Dm-sph}) is equivalent to
\be\label{NHL-v}
\begin{cases}
	v_{\phi} = \p_\phi v_\rho = \p_\phi(\sin\phi\, v_\th) = 0,  &\quad \text{on} \quad \p^{R}D_{m};\\
	v_\rho = \p_\rho(\rho v_\phi) = \p_\rho(\rho v_\th) = 0,  &\quad \text{on} \quad \p^{A}D_{m}.
\end{cases}
\ee
Based on (\ref{NHL slip bdry for Dm-sph}) and (\ref{NHL-v}), we can also obtain the boundary conditions for $ \Gamma $, $ K $, $ F $ and $ \O $ by direct computation. We collect all these results in the lemma below.
\begin{lemma}\label{Lemma, bdry cond}
Let $D_{m}$ be the domain as in (\ref{app domain-sph}).  Then the following boundary conditions hold.
\be\label{bdry-phi const}\begin{cases}
\p_{\phi}v_\rho=0,\quad v_\phi=0, \quad \p_\phi v_\th = -\cot\phi\, v_\th, \quad \p_\phi\Gamma=0,\\
\o_\rho=\o_\theta=K=\O=0, \quad \p_\phi \o_\phi = -\cot\phi\, \o_\phi,   \quad \p_\phi F = -\cot\phi\, F.
\end{cases}
\quad\text{on\, $\p^{R}D_{m}$}. \ee
and
\be\label{bdry-rho const}\begin{cases}
v_\rho=0,\quad \p_\rho v_\phi = - v_\phi / \rho, \quad \p_\rho v_\th = - v_\th / \rho,\quad \p_\rho\Gamma=0,\\
\o_\phi=\o_\theta=F=\O=0, \quad \p_\rho \o_\rho = - 2 \o_\rho / \rho  \quad \p_\rho K = - 3 K / \rho.
\end{cases}
\quad\text{on\, $\p^{A}D_{m}$}. \ee
\end{lemma}

Before ending this subsection, we construct an element $ v_0 = v_{0,\rho}e_\rho + v_{0,\phi}e_\phi + v_{0,\th}e_\th $ in the admissible set $ \mathscr{A} $ (see Definition \ref{Def, admissible sets})
such that $ v_{0,\rho} $ and $ v_{0,\phi} $ can be chosen arbitrarily large while $ v_{0,\th} $ can be chosen arbitrarily small. In addition, $ v_0 $ enjoys the even-odd-odd symmetry as in (\ref{sym-sph}).

\begin{example}\label{Ex, iv}
	Let $ \a\in\big(0,\frac{\pi}{2}\big) $. We first choose
	\[ f(\rho) = \rho^7(\rho-1)^3, \quad g(\phi) = \sin^3\bigg( \frac{\pi}{\a} \Big(\phi-\frac{\pi}{2} \Big) \bigg), \quad h(s) = s^3(s-1).\]
	Then for any real numbers $ \lam_1 $ and $ \lam_2 $, we define $ v_0 = v_{0,\rho}e_\rho + v_{0,\phi}e_\phi + v_{0,\th}e_\th $, where
\begin{align*}
	& v_{0,\rho} = \frac{\lam_1}{\rho^2\sin\phi} \,f(\rho) g'(\phi), \quad v_{0,\phi}  = -\frac{\lam_1}{\rho\sin\phi} \,f'(\rho) g(\phi), \\
	& v_{0,\th}  = \frac{\lam_2}{\rho\sin\phi} \bigg( \int_{0}^{\rho} h(s)\,ds \bigg)\, \sin\bigg( \frac{\pi}{2\a} \Big(\phi-\frac{\pi}{2} \Big) \bigg).
\end{align*}
We claim that $ v $ belongs to $ \mathscr{A} $ and has the even-odd-odd symmetry as in (\ref{sym-sph}). Moreover, by taking $ \lam_1 $ sufficiently large and $ \lam_2 $ sufficiently small, $ v_{0,\rho} $ and $ v_{0,\phi} $ can be chosen arbitrarily large while $ v_{0,\th} $ can be chosen arbitrarily small.
\end{example}

In order to show $ v\in\mathscr{A} $, for any $ m\geq 2 $, we first choose $ g(\phi) $ to be the same function as the above example, and choose
\[ f_{m}(\rho) = \rho^4\Big(\rho-\frac{1}{m}\Big)^3 (\rho-1)^3, \quad h_{m}(s) = s^2\Big(s-\frac{1}{m}\Big)(s-1). \]
Then we define $ v_0^{(m)} = v^{(m)}_{0,\rho}e_\rho + v^{(m)}_{0,\phi}e_\phi + v^{(m)}_{0,\th}e_\th $, where
\begin{align*}
	& v^{(m)}_{0,\rho} = \frac{\lam_1}{\rho^2\sin\phi} \,f_m(\rho) g'(\phi), \quad v^{(m)}_{0,\phi}  = -\frac{\lam_1}{\rho\sin\phi} \,f_m'(\rho) g(\phi), \\
	& v^{(m)}_{0,\th}  = \frac{\lam_2}{\rho\sin\phi} \bigg( \int_{0}^{\rho} h_m(s)\,ds \bigg)\, \sin\bigg( \frac{\pi}{2\a} \Big(\phi-\frac{\pi}{2} \Big) \bigg).
\end{align*}
Then for each $ m\geq 2 $, one can directly check that $v^{(m)}_{0}$ satisfies the NHL boundary condition (\ref{NHL-v}). In addition,
\[\begin{split}
	\text{div}\, v_0^{(m)} & = \frac{1}{\rho^2}\p_{\rho} \big(\rho^2 v^{(m)}_{0,\rho}\big) + \frac{1}{\rho\sin\phi}\p_{\phi} \big(\sin\phi\,v^{(m)}_{0,\phi} \big) \\
	& = \frac{1}{\rho^2\sin\phi}\Big[ \p_\rho \Big( \rho^2\sin\phi\, v^{(m)}_{0,\rho}\Big) + \p_\phi \Big(  \rho\sin\phi\,v^{(m)}_{0,\phi} \Big) \Big]    = 0.
\end{split}\]
Thus, $ v_0^{(m)}\in \mathscr{A}_m $. Meanwhile, it is obvious that
\[ \lim_{m\to\infty} \big\| v_0 - v_0^{(m)} \big\|_{C^2(\ol{D_m})} = 0. \]
Therefore, $ v\in\mathscr{A} $.


\subsection{Two weighted Poincar\'e inequalities on $ \bR $}

\quad

In this subsection, we will introduce some weighted Poincar\'e inequalities, in the spirit of \cite{PW60}, which are needed in the sequel. Given $a,b\in\bR$ with $a<b$, let $p\in C^{\infty}\big([a,b]\big)$ and assume
\[\min_{y\in[a,b]} p(y)>0.\]
Denote the numbers $ p_A $ and $ p_B $ by
\be\label{P const}
\begin{cases}
p_{A}=\max\limits_{[a,b]}\Big( \frac12 \frac{p''}{p}-\frac34 \frac{(p')^2}{p^2} \Big),  \vspace{0.05in}\\
p_{B}=\max\limits_{[a,b]}\frac12 \Big( \frac{(p')^2}{p^2}-\frac{p''}{p} \Big).
\end{cases}\ee

\begin{lemma}\label{Lemma, P ave}
Define a functional $\Phi$ as
\[\Phi(u)=\frac{\int_{a}^{b}p(y)\big(u'(y)\big)^2\,dy}{\int_{a}^{b}p(y)u^2(y)\,dy}, \quad\forall u\in\mathcal{A}, \]
where $\mathcal{A}=\big\{ u\in H^1(a,b)\setminus\{0\}: \int_{a}^{b}p(y)u(y)\,dy=0 \big\}$.
Then
\[\inf_{u\in\mathcal{A}}\Phi(u) \geq \frac{\pi^2}{(b-a)^2}-p_{A},\]
where $p_{A}$ is defined in (\ref{P const}).
\end{lemma}

The proof of this lemma follows directly from the proof of the lemma on page 3 in Section 2 in \cite{PW60}.
By choosing $p(y)=\sin y$ on an interval $\big[\frac{\pi}{2}-\a, \,\frac{\pi}{2}+\a \big]$ for $\a\in\big(0,\frac{\pi}{2} \big)$, we immediately obtain the following corollary.

\begin{corollary}\label{Cor, P-sine-ave}
Let $0<\a<\frac{\pi}{2}$, $a=\frac{\pi}{2}-\a$, $b=\frac{\pi}{2}+\a$. Then for any $u\in H^{1}(a,b)\setminus\{0\}$ with $\int_{a}^{b}\sin y \, u(y)\,dy=0$, we have
\[\int_{a}^{b}\sin y \, u^2(y)\,dy \leq C_{\a, A} \int_{a}^{b}\sin y \, \big(u'(y)\big)^2\,dy, \]
where
\be\label{P-sine-ave-const}
C_{\a,A}=\frac{(b-a)^2}{\pi^2+2\a^2}=\frac{4\a^2}{\pi^2+2\a^2}.\ee
\end{corollary}

So far, the Poincar\'e inequalities cover functions whose weighted integral on $[a,b]$ is equal to 0.  In the next two results, we will consider the situation when the functions are equal to 0 on the boundary of the interval.

\begin{lemma}\label{Lemma, P bdry}
Define a functional $\Psi$ as
\[\Psi(u)=\frac{\int_{a}^{b}p(y)\big(u'(y)\big)^2\,dy}{\int_{a}^{b}p(y)u^2(y)\,dy}, \quad\forall u\in\mathcal{B}, \]
where $\mathcal{B}=H_{0}^{1}(a,b)\setminus\{0\}$.  Then
\[\inf_{u\in\mathcal{B}}\Psi(u)  \geq \frac{\pi^2}{(b-a)^2}-p_{B},\]
where $p_{B}$ is defined in (\ref{P const}).
\end{lemma}

In Lemma \ref{Lemma, P bdry},  by choosing $p(y)=\sin y$ on an interval $\big[\frac{\pi}{2}-\a, \,\frac{\pi}{2}+\a \big]$ for $\a\in\big(0,\frac{\pi}{4}\big]$,   we conclude the following result right away.

\begin{corollary}\label{Cor, P-sine-bdry}
	Let $0<\a\leq \frac{\pi}{4}$, $a=\frac{\pi}{2}-\a$, $b=\frac{\pi}{2}+\a$. Then for any $u\in H_{0}^{1}(a,b)\setminus\{0\}$, we have
	\[\int_{a}^{b}\sin y \, u^2(y)\,dy \leq C_{\a, B} \int_{a}^{b}\sin y \, \big(u'(y)\big)^2\,dy, \]
	where
	\be\label{P-sine-bdry-const}
	C_{\a,B}=\frac{(b-a)^2}{\pi^2-\frac{2\a^2}{\cos^{2}\a}}=\frac{4\a^2}{\pi^2-\frac{2\a^2}{\cos^{2}\a}}.\ee
\end{corollary}

Note that when $\a\in \big(0,\frac{\pi}{4}\big]$,  both $C_{\a,A}$ and $C_{\a,B}$ are  increasing functions in $\a$. In particular,
\be\label{two P-consts}\begin{cases}
	C_{\pi/6, A}=\frac{2}{19}, \quad C_{\pi/4, A}=\frac{2}{9},  \vspace{0.05in}\\
	C_{\pi/6, B}=\frac{3}{25}, \quad C_{\pi/4, B}=\frac{1}{3}.
\end{cases}\ee

\begin{proof}[\bf Proof of Lemma \ref{Lemma, P bdry}]
	The idea of this proof is similar to that of the lemma on page 3 in Section 2 in \cite{PW60}.  Define
\[\mathcal{B}_1 = \Big\{ u\in H^{1}_{0}(a,b): \int_{a}^{b} p(y)u^2(y)\,dy = 1 \Big\}.\]
Then $\mathcal{B}_{1} \subset \mathcal{B}$. By standard argument, there exists some $u_*\in \mathcal{B}_{1}$ such that the operator $\Psi$ attains its infimum over $\mathcal{B}$ at $u_*$.  Denote $\lam = \Psi(u_*)$.  Then
\[\inf_{u\in\mathcal{B}} \Psi(u) = \Psi(u_*)=\lam > 0.\]
Now for any $h\in H_{0}^{1}(a,b)$, $u_* + t h$ is still in $\mathcal{B}$ for any sufficiently small $t$. Define
\[g(t) = \Psi(u_* + t h).  \]
Then $g'(0)=0$.  This implies that
\be\label{int eq for u star}
\int_{a}^{b} p u_{*}' h' \,dy  - \lam \int_{a}^{b} p u_* h \,dy = 0, \quad \forall\, h\in H_{0}^{1}(a,b). \ee
So $u_*$ is a weak solution of
\be\label{weak soln u-star}
\big(p u_{*}'\big)' + \lam p u_* = 0, \quad \text{in} \quad (a,b).\ee
Since $p$ is smooth and bounded from below by a positive constant,  it follows from classical regularity theory that $u_* \in C^{\infty} \big( [a,b] \big)$.  So $u_*$ is a classical solution to
the following equation with Dirichlet boundary condition.
\be\label{u-star ode}\begin{cases}
u_{*}'' + \frac{p'}{p}\, u_{*}' + \lam u_{*} = 0,  \quad \text{in} \quad (a,b), \\
u_{*}(a) = u_{*}(b) = 0.
\end{cases}\ee
Testing (\ref{u-star ode}) by $u_*$ and using integration by parts,
\begin{align*}
\int_{a}^{b} \big( u_{*}' \big)^2 \,dy &= \frac12 \int_{a}^{b} \frac{p'}{p}\, \big( u_{*}^{2} \big)'\,dy + \lam \int_{a}^{b} u_{*}^{2} \,dy \\
&= \int_{a}^{b} \bigg[ - \frac12 \bigg( \f{p'}{p} \bigg)' + \lam \bigg]\, u_{*}^{2} \,dy\\
&\leq (\lam + p_{B}) \int_{a}^{b} u_{*}^{2} \,dy,
\end{align*}
where $p_{B}$ is as defined in (\ref{P const}).  Hence,
\[\lam + p_{B} \geq \frac{\int_{a}^{b} \big( u_{*}' \big)^2 \,dy}{\int_{a}^{b} u_{*}^{2} \,dy}.\]
Since $u_{*}(a) = u_{*}(b) = 0$,  it is well-known that the quotient on the right-hand side of the above inequality is bounded from below by $\pi^2/(b-a)^2$. Thus,
\[\lam\geq \frac{\pi^2}{(b-a)^2} - p_{B}.\]
\end{proof}

\subsection{A Hardy's type inequality in $ D_m $}
\quad

Let the region $ D_m $ be as defined in (\ref{domain-sph}) with $ m\geq 2 $ and the angle $ \a\in \big(0, \frac{\pi}{2}\big) $. If a scalar-valued function $ f\in H^1(D_m) $ with 0 boundary value, that is $ f\in H^1_0(D_m) $, then it follows from the classical Hardy's inequality that $ \big\| \frac{f}{\rho} \big\|_{L^2(D_m)} = \big\| \frac{f}{|x|} \big\|_{L^2(D_m)} \leq 2 \|\nabla f\|_{L^2(D_m)} $. But if a function does not vanish on the boundary, then the norm of the gradient $ \nabla f $ alone does not suffice to control the norm of $ f/\rho $. The next result says that in the special domains $ D_m $, after adding the norm of a lower-order term, only the norm of partial gradient, $ \p_\rho f $, is needed to control the norm of $ f/\rho $ with constants independent of $ m $. Such an estimate may be known, but we could not find the specific form in the literature when the domain is a finite cone.

\begin{lemma}\label{Lemma, sH}
	Let the region $ D_m $ be as defined in (\ref{app domain-sph}) with $ m\geq 2 $ and the angle $ \a\in \big(0, \frac{\pi}{2}\big) $. Then for any scalar-valued function $ f\in H^1(D_m) $ and for any $ \epsilon>0 $,
	\be\label{sH}
	\int_{D_m} \frac{f^2}{\rho^2}\,dx \leq (4+\e)\int_{D_m} |\p_\rho f|^2 \,dx + \bigg(40+\frac{16}{\e}\bigg)\int_{D_m} f^2 \,dx.
	\ee
\end{lemma}

\begin{proof}
By converting the integral into spherical coordinates, we have
\be\label{fi}
\int_{D_m} \frac{f^2}{\rho^2} \,dx = 2\pi \int_{\frac{\pi}{2}-\a}^{\frac{\pi}{2}+\a} \sin\phi\, \bigg( \int_{\frac1m}^{1} f^2(\rho,\phi)\,d\rho \bigg) \,d\phi.
\ee	
Using integration by parts,
\[\begin{split}
\int_{\frac1m}^{1} f^2(\rho,\phi)\,d\rho \leq f^2(1,\phi)  - 2 \int_{\frac1m}^1 \rho f \p_{\rho}f \,d\rho
\end{split}\]
Plugging this estimate into (\ref{fi}) yields
\be\label{fie1}
\int_{D_m} \frac{f^2}{\rho^2} \,dx \leq I_1 - I_2,
\ee
where
\[ I_1 = 2\pi \int_{\frac{\pi}{2}-\a}^{\frac{\pi}{2}+\a} \sin\phi\,f^2(1,\phi)\,d\phi, \quad I_2 = 4\pi \int_{\frac{\pi}{2}-\a}^{\frac{\pi}{2}+\a} \int_{\frac1m}^1 \rho\sin\phi\, f \p_{\rho} f \,d\rho\,d\phi.  \]
For $ I_2 $, by changing back to the Euclidean coordinates and using Cauchy-Schwarz inequality, we find
\be\label{sH-I2}
|I_2| \leq \frac12 \int_{D_m}\frac{f^2}{\rho^2}\,dx + 2\int_{D_m} |\p_\rho f|^2 \,dx.
\ee
In order to estimate $ I_1 $, we fix a cutoff function $ \eta\in C^{\infty}(\m{R}) $ such that $ 0\leq \eta \leq 1 $,
\[ \eta(t) = \left\{\begin{array}{lll}
0 & \text{ if }  & t\leq \frac34, \\
1 & \text{ if } & t\geq 1,
\end{array}\right. \]
and $ \sup\limits_{t\in\m{R}} |\eta'(t)|\leq 10 $. Then
\[\begin{split}
	I_1 & = 2\pi  \int_{\frac{\pi}{2}-\a}^{\frac{\pi}{2}+\a} \sin\phi\,\big[ f^2(1,\phi)\eta(1) - f^2(3/4, \phi)\eta(3/4) \big]\,d\phi \\
	& = 2\pi  \int_{\frac{\pi}{2}-\a}^{\frac{\pi}{2}+\a} \sin\phi \int_{\frac34}^1 \p_{\rho}\Big[f^2(\rho,\phi)\eta(\rho)\Big]\,d\rho  \,d\phi.
\end{split}\]
Since $ |\eta'|\leq 10 $, it follows from the above expression that
\[ I_1 \leq 20\pi \int_{\frac{\pi}{2}-\a}^{\frac{\pi}{2}+\a} \int_{\frac34}^1 \sin\phi\, f^2(\rho,\phi)\,d\rho\,d\phi + 4\pi \int_{\frac{\pi}{2}-\a}^{\frac{\pi}{2}+\a} \int_{\frac34}^1 \sin\phi\, |f \p_{\rho} f| \,d\rho\,d\phi.  \]
Since $ \rho $ has the lower bound $ \frac34 $ in the above integral, we further deduce that
\[ I_1 \leq 40\pi \int_{\frac{\pi}{2}-\a}^{\frac{\pi}{2}+\a} \int_{\frac34}^1 \rho^2\sin\phi\, f^2(\rho,\phi)\,d\rho\,d\phi + 8\pi \int_{\frac{\pi}{2}-\a}^{\frac{\pi}{2}+\a} \int_{\frac34}^1 \rho^2\sin\phi\, |f \p_{\rho} f| \,d\rho\,d\phi.  \]
Changing back to the Euclidean coordinates and applying Cauchy-Schwarz inequality, we find
\be\label{sH-I1}
	I_1 \leq 20 \int_{D_m} f^2\,dx + 4 \int_{D_m} |f \p_\rho f| \,dx \leq \bigg(20+\frac{8}{\e}\bigg) \int_{D_m} f^2\,dx + \frac{\e}{2} \int_{D_m} |\p_\rho f|^2 \,dx.
\ee
Putting (\ref{sH-I1}) and (\ref{sH-I2}) into (\ref{fie1}) leads to (\ref{sH}).	
\end{proof}

Let $ v $ be a vector field on $ D_m $. It has two decompositions under the Euclidean coordinates and the spherical coordinates respectively:
\[ v = v_1 e_1 + v_2 e_2 + v_3 e_3 = v_\rho e_\rho + v_\phi e_\phi + v_\th e_\th. \]
Then it is well-known that $ |\nabla v|^2 = \sum\limits_{i=1}^3 |\nabla v_i|^2 $. But according to formula (\ref{nabla vec-sph}), the relation $ |\nabla v|^2 = |\nabla v_\rho|^2 +  |\nabla v_\phi|^2 +  |\nabla v_\th|^2$ may not hold. Nonetheless, we can take advantage of Lemma \ref{Lemma, sH} to show the equivalence between the $ H^1(D_m) $ norm of $ v $ and the sum of $ H^1(D_m) $ norms of its components $ v_\rho$, $v_\phi $ and  $ v_\th $.

\begin{corollary}\label{Cor, eg}
	Let the region $ D_m $ be as defined in (\ref{app domain-sph}) with $ m\geq 2 $ and the angle $ \a\in \big(0, \frac{\pi}{2}\big) $. Let $ v=v_\rho e_\rho + v_\phi e_\phi + v_\th e_\th $ be a vector field on $ D_m $. Then $ v $ belongs to $ H^1(D_m) $ if and only if all its components $ v_\rho$, $ v_\phi $ and $ v_\th $ belong to $ H^1(D_m) $. In addition, there exists some constant $ C>1 $, which only depends on $ \a $, such that
	\be\label{eg}	
	\frac{1}{C} \| v \|_{H^1(D_m)} \leq \|v_\rho \|_{H^1(D_m)} + \| v_\phi \|_{H^1(D_m)} + \| v_\th \|_{H^1(D_m)} \leq C \| v \|_{H^1(D_m)}.
	\ee
\end{corollary}
\begin{proof}
Firstly, since $ \{e_\rho, e_\phi, e_\th\} $ forms an orthogonal basis in $ \m{R}^3 $, $ |v|^2 = |v_\rho|^2 + |v_\phi|^2  + |v_\th|^2 $, so
\[ \|v\|^2_{L^2(D_m)} = \|v_\rho \|^2_{L^2(D_m)} + \| v_\phi \|^2_{L^2(D_m)} + \| v_\th \|^2_{L^2(D_m)}.\]
On the other hand, according to formula (\ref{nabla vec-sph}), under the basis (\ref{m-basis in sph}), the gradient $ \nabla v$ can be represented as
\be\label{gov}
\nabla v =
\begin{pmatrix} \p_{\rho}v_{\rho} & \frac{1}{\rho}(\p_{\phi}v_{\rho}-v_{\phi}) & -\frac{1}{\rho}\,v_{\th} \vspace{0.05in}\\
	\p_{\rho}v_{\phi} & \frac{1}{\rho}(\p_{\phi}v_{\phi}+v_{\rho}) & -\frac{\cot \phi}{\rho}\, v_{\th} \vspace{0.05in} \\
	\p_{\rho}v_{\th} & \frac{1}{\rho}\p_{\phi}v_{\th}  & \frac{1}{\rho}(v_{\rho}+\cot\phi\,v_{\phi}) \end{pmatrix}.
\ee
Meanwhile,
\[ |\nabla v_\rho|^2 = |\p_\rho v_\rho|^2 + \Big|\frac{1}{\rho}\p_\phi v_\rho\Big|^2, \quad  |\nabla v_\phi|^2 = |\p_\rho v_\phi|^2 + \Big|\frac{1}{\rho}\p_\phi v_\phi\Big|^2, \quad  |\nabla v_\th|^2 = |\p_\rho v_\th|^2 + \Big|\frac{1}{\rho}\p_\phi v_\th\Big|^2.\]
Noticing that in the domain $ D_m $, $ \rho $ and $ \cot\phi $ are bounded:
\[ \frac1m < \rho < 1, \quad 0\leq \cot\phi < \cot\a,\]
so it is straightforward to check that $ v $ belongs to $ H^1(D_m) $ if and only if all its components $ v_\rho$, $ v_\phi $ and $ v_\th $ belong to $ H^1(D_m) $. Moreover, one can apply Lemma \ref{Lemma, sH} and Cauchy-Schwarz inequality to (\ref{gov}) to establish (\ref{eg}).

\end{proof}

\subsection{A priori $L^\infty$ bound for $\Gamma = \rho \sin \phi\, v_\th$ in  $D_m$.}

\quad

In this section, we study the quantity $ \Gamma $, defined as in (\ref{Gamma-sph}), in the approximating space-time domain $ D_{m}\times[0,T] $, where $ 0<T<\infty $. Define the energy space $ E_{m,T} $ as
\be\label{es in approx st}
E_{m,T} = L_t^\infty L_x^2\cap L_t^2 H_x^{1}\big(D_m\times[0,T]\big) \ee
which is equipped with the following norm:
\be\label{en on approx st}
\| v \|_{E_{m,T}}^2 = \int_0^T \int_{D_m} | \nabla v(x,t) |^2 \,dx\,dt + \sup_{t\in[0,T]} \int_{D_m} |v(x,t)|^2 \,dx.\ee
The function $ v $ can be either vector-valued or scalar-valued, depending on the context. We denote by $ E_{m,T}^{\sigma} $ the subspace of $ E_{m,T} $ which consists of vectors which are divergence free and whose normal component vanishes on the boundary of $ D_{m} $.
\be\label{en div-free, normal-0}
E_{m,T}^{\sigma} = \big\{v\in E_{m,T}: \text{ $ \nabla\cdot v=0 $ in $ D_{m}$ and $ v\cdot n=0 $ on $ \p D_m $ for a.e. $ t\in[0,T] $} \big\}.\ee
If a function $ v $ is independent of time, we may also say it belongs to $ E_{m,T} $ or $ E^{\sigma}_{m,T} $ by regarding it as a stationary function.


Based on the equation (\ref{eq of Gamma-sph}) and the boundary conditions in Lemma \ref{Lemma, bdry cond}, $ \Gamma $ is determined by the following problem:
\begin{equation}\label{Gamma-Dm}
	\left\{\begin{array}{l}
		\Delta \Gamma - b\cdot \nabla \Gamma - \frac{2}{\rho}\p_\rho \Gamma - \frac{2 \cot\phi}{\rho^2}\p_\phi \Gamma-\partial_{t} \Gamma=0, \quad \text { in } \quad D_{m} \times(0, T]; \\
		\partial_{n} \Gamma=0, \quad \text { on } \quad \p D_{m} \times(0, T]; \\
		\Gamma(x, 0)=\Gamma_{0}(x), \quad x \in D_{m},
	\end{array}\right.
\end{equation}
where $ b=v_\rho e_\rho + v_\phi e_\phi $, $ \Gamma_0 $ is the initial value defined as $ \Gamma_0 =  \rho \sin\phi\, v_{0, \theta}(x)$, and $\partial_{n} \Gamma$ means the directional derivative of $ \Gamma $ along the exterior normal direction of $\partial D_{m}$, except at the corners. In this section, we will study the solvability of (\ref{Gamma-Dm}) and the regularity of its solution. As a preparation, we first introduce an embedding result.

In general, for any 3D domain $ \O $ and for any function $ v $ that lies in the energy space $ L^{\infty}_{t}L^{2}_{x}\cap L_{t}^{2}H_{x}^{1}\big(\O\times[0,T]\big) $ automatically belongs to $ L_{tx}^{10/3}\big(\O\times[0,T] \big) $ by standard interpolation. But if the function $ v $ is axially symmetric and the domain $ \O $, say $ \O=D_m $, is bounded and has a positive distance to the $ x_3 $ axis, then we can regard $ v $ as a function on a 2D domain $ \O' $ in the $ \rho$-$\phi $ space. Thus, the 2D Ladyzhenskaya's inequality (or more precisely, the Gagliardo-Nirenberg inequality) is applicable and we are able to improve the regularity of $ v $ from $ L_{tx}^{10/3} $ to $ L_{tx}^{4} $. We point out that the range of $\alpha$ in the following Lemma \ref{Lemma, asLe} and \ref{Lemma, Gamma-Dm}  is larger than the one in the main theorem.

\begin{lemma}\label{Lemma, asLe}
	Let the region $ D_m $ be as defined in (\ref{app domain-sph}) with $ m\geq 2 $ and the angle $ \a\in \big(0, \frac{\pi}{2}\big) $.  Then for any $ T>0$,  the energy space $ E_{m,T} $ is embedded in $ L^{4}_{tx}\big( D_m\times[0,T]\big) $. In addition, there exists a constant $ C=C(\a,m) $ such that
	\be\label{asLe}
	\| f \|_{L^4_{tx}( D_m\times[0,T] )} \leq C \big( T^{1/4}+1 \big)\| f \|_{E_{m,T}}, \quad\forall\, f\in E_{m,T}.
	\ee
\end{lemma}
\begin{proof}
	Since the volume element $\rho^2\sin\phi\, d\rho d\phi d\theta$ on $ D_m $ is equivalent to the two-dimensional volume element $d\rho d\phi$ on the $\rho$-$\phi$ plane, we can apply the 2D Gagliardo-Nirenberg inequality to $ f $ in $ D_m $ to conclude that
	\[ \| f(\cdot,t) \|_{L^4(D_m)} \leq C\Big( \|f(\cdot,t)\|_{L^2(D_m)}^{\frac12}\| \na f(\cdot,t)\|_{L_2(D_m)}^{\frac12} + \|f(\cdot,t)\|_{L^2(D_m)} \Big), \quad \text{for a.e. $t\in[0,T]$}, \]
	where $ C $ is some constant that only depends on $ \a $ and $ m $. As a result, we deduce that
	\[  \| f(\cdot,t) \|_{L^4(D_m)}^4 \leq C\Big( \|f(\cdot,t)\|_{L^2(D_m)}^{2}\| \na f(\cdot,t)\|_{L_2(D_m)}^{2} + \|f(\cdot,t)\|_{L^2(D_m)}^4 \Big), \quad \text{for a.e. $t\in[0,T]$}. \]
	Then (\ref{asLe}) follows from integrating the above estimate in $ t $ on $ [0,T] $.
\end{proof}

Now we are ready to present the main result of this subsection.

\begin{lemma}\label{Lemma, Gamma-Dm}
	Let the region $ D_m $ be as defined in (\ref{app domain-sph}) with $ m\geq 2 $ and the angle $ \a\in \big(0, \frac{\pi}{2}\big) $.  Let $ T>0 $ and $ b=v_\rho e_\rho + v_\phi e_\phi\in E^{\sigma}_{m,T} $. Assume the initial velocity $ v_0\in H^{2}(D_m) $ is divergence free and satisfies the NHL boundary condition (\ref{NHL slip bdry for Dm-sph}). Then the problem (\ref{Gamma-Dm}) possesses a unique bounded weak solution $ \Gamma $ in the energy space $ E_{m,T} $ which satisfies
	\be\label{Gamma-max}
	\| \Gamma \|_{ L^{\infty}(D_m\times[0,T]) } \leq \| \Gamma_0\|_{L^{\infty}(D_m) },
	\ee
	and
	\be\label{Gamma-energy-Dm}
	\| \Gamma \|_{E_{m,T}} \leq C e^{CT} \| \Gamma_0\|_{L^2(D_m)},
	\ee
	where $ \Gamma_0 =  \rho \sin\phi\, v_{0, \theta} $ and $ C $ is a positive constant which only depends on $ \a $ and $ m $.
\end{lemma}
\begin{proof}
	We use the dimension reduction method and Lemma \ref{Lemma, asLe} to justify the conclusions. Firstly, we view $\Gamma$ as a function of variables $\rho$, $\phi$ and $t$,  and regard $ D_m $ as a 2D domain $ D'_{m} $ on the $\rho$-$\phi$ plane, which is defined as below (see Figure \ref{Fig,app domain-2D-sph}).
	\be\label{app domain-2D-sph}
	D'_{m}=\Big\{(\rho,\phi): \frac{1}{m}<\rho <1, \, \frac{\pi}{2}-\a < \phi < \frac{\pi}{2}+\a \Big\}.
	\ee
	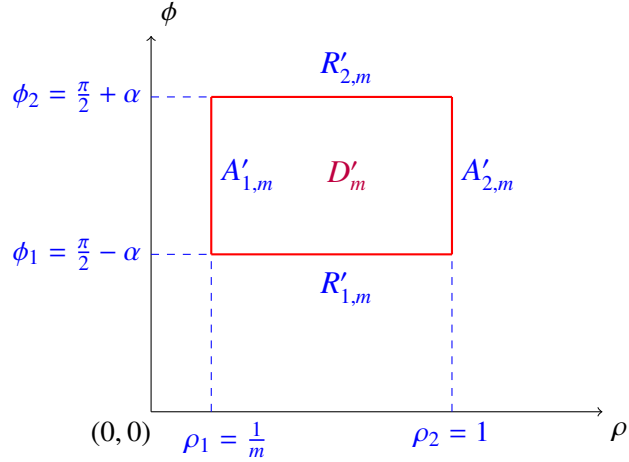
\begin{figure}[!ht]
		\centering
		\begin{tikzpicture}[scale=2]
			\draw [->] (0,0)--(3,0) node [anchor=north west]{$\rho$};
			\draw [->] (0,0)--(0,2.5) node [anchor=south west] {$\phi$};
			\draw (-0.2,0) node [below] {$(0,0)$};
			
			\draw [red, thick, domain=0.4:2] plot({\x},{2*pi/3});
			\draw [red, thick, domain=0.4:2] plot({\x},{pi/3});
			\draw [red, thick, domain=pi/3:2*pi/3] plot({0.4},{\x});
			\draw [red, thick, domain=pi/3:2*pi/3] plot({2},{\x});
			
			\draw [blue, dashed, domain=0:pi/3] plot({0.4},{\x});
			\draw [blue, dashed, domain=0:pi/3] plot({2},{\x});
			\draw [blue, dashed, domain=0:0.4] plot({\x},{pi/3});
			\draw [blue, dashed, domain=0:0.4] plot({\x},{2*pi/3});
			
			\draw (0.5,0) node[below][blue] {$\rho_1 = \frac1m$};
			\draw (2,0) node[below][blue] {$\rho_2 = 1$};
			\draw (0,pi/3) node[left][blue] {$\phi_1 = \frac{\pi}{2}-\a$};
			\draw (0, 2*pi/3) node[left][blue] {$\phi_2 = \frac{\pi}{2}+\a$};
			
			\draw (1.1,1.57) node[right][purple] {\large $D'_{m}$};
			\draw (0.4,1.57) node[right][blue] {\large $A'_{1,m}$};
			\draw (2,1.57) node[right][blue] {\large $A'_{2,m}$};
			\draw (1.3,2.1) node[above][blue] {\large $R'_{2,m}$};
			\draw (1.3,1) node[below][blue] {\large $R'_{1,m}$};
			
			
			
		\end{tikzpicture}
		\caption{Domain $D'_{m}$ in $\rho$-$\phi$ coordinates}
		\label{Fig,app domain-2D-sph}
	\end{figure}
	Then equation (\ref{Gamma-Dm}) can be rewritten as below:
	\begin{equation}\label{Gamma-sph-2D}
		\left\{\begin{array}{l}
			\p_{\rho}^2 \Gamma + \frac{1}{\rho^2} \p_{\phi}^2 \Gamma - v_{\rho} \p_{\rho}\Gamma - \Big( \frac{1}{\rho}v_{\phi} + \frac{\cot\phi}{\rho^2} \Big)\, \p_{\phi}\Gamma - \p_{t}\Gamma = 0, \quad \text { in } \quad D'_{m} \times(0, T]; \\
			\partial_{n} \Gamma=0, \quad \text { on } \quad \p D'_{m} \times(0, T]; \\
			\Gamma(x, 0)=\Gamma_{0}(x), \quad x \in D'_{m}.
		\end{array}\right.
	\end{equation}
	Thanks to Lemma \ref{Lemma, asLe}, both $ v_\rho $ and $ v_\phi $ belong to $ L^{4}_{tx}(D'_m\times(0,T]) $, and any function $ \Gamma $ in the energy space $ E_{m,T} $ also belongs to $ L^{4}_{tx}(D'_m\times(0,T]) $ which is the critical space for (\ref{Gamma-sph-2D}) in 2D space. Since the distance of $ D_m' $ to the $ x_3 $ axis is at least $ \frac1m $, the existence of a weak solution $ \Gamma $ of (\ref{Gamma-Dm}) in $ E_{m,T} $ follows from the classical theory. Meanwhile, since $ D'_m $ is a 2D domain and both $ v_\rho $ and $ v_\phi $ belong to $ L^{4}_{tx} $, the weak maximum principle is applicable for (\ref{Gamma-Dm}), see e.g. Theorem 2.1 in \cite{KR20}. As a result, the uniqueness of the solution and the estimate
	(\ref{Gamma-max}) are justified.

Since the above solution $ \Gamma $ lies in $ L^{\infty}\cap E_{m,T} $, it can be served as a test function to (\ref{Gamma-Dm}). Meanwhile, since $ b\in E^{\sigma}_{m,T} $, which implies $\nabla\cdot b = 0$ in $D_{m}\times(0,T]$ and $b\cdot n = 0$ on  $\p D_m\times (0,T]$, then it holds that $ \int_{D_m} (b\cdot\nabla\Gamma) \Gamma \,dx = 0 $. As a result, (\ref{Gamma-energy-Dm}) follows from the standard energy estimate.
\end{proof}

\section{Existence of strong solutions in $ D_m $}
\label{Sec, exist on Dm}

In this section, we study the existence of solutions in the approximating space-time domains $ D_{m}\times[0,T] $, where $ m\geq 2 $ and $ 0<T<\infty $. We point out that the local existence of the solution in the energy space $ E_{m,T} $ has already been proven in literature, see e.g. \cite{MM09DIE} which even covers more general Lipschitz domains. In the current situation, the local existence can be extended to the global one since $ D_m $ is away from $ x_3 $ axis.
Our main goal here is to prove the existence of the solution with higher regularity. Actually, we will establish the existence of the bounded strong solution $ v $ on $ D_m\times[0,T] $ for any $ T>0 $. The proof of higher regularity of solutions, although somewhat unsurprising, requires some detailed analysis because the domains are not smooth. Those, who would like to have a quick view of the key idea in the proof for the main Theorem \ref{Thm, cyl}, can skip this section for now and jump to Section \ref{Sec, inf-ub}. We also remark that if the NHL boundary condition is replaced by the Dirichlet boundary condition, then the local existence of strong solutions on general bounded Lipschitz domains has been established in \cite{DvW95} using the semi-group theory. But it may take much effort to adapt that method to treat the more complicated NHL boundary condition.

Besides the existence of the strong solution, we will also show that if the initial data enjoys the even-odd-odd symmetry, defined as in Definition (\ref{Def, eoo sym}), then this symmetry will be preserved in time for the strong solution. For convenience of notations, we define
\be\label{en div-free, normal-0, sym}
E_{m,T}^{\sigma,s} = \big\{v\in E_{m,T}^\sigma: \text{ $ v $ has the even-odd-odd symmetry} \big\},\ee
where $ s $ stands for symmetry. In the following, we will first construct a local solution in Proposition \ref{Prop, local soln in ad} and then extend it to be a global one in Corollary \ref{Cor, globle soln in ad}.

\begin{proposition}\label{Prop, local soln in ad}
	Let	$\a\in\big(0,\frac{\pi}{2}\big)$ and $m\geq 2$. Assume the initial velocity $ v_0\in H^{2}(D_m) $ is divergence free in $ D_m $ and satisfies the NHL condition (\ref{NHL slip bdry for Dm-sph}) on $ \p D_m $. Then there exists some time $ T>0 $ and a
	strong solution $ (v,P) $ of (\ref{asns-sph}) on $ D_m\times[0,T] $ with the initial data $ v_0 $ and the NHL condition (\ref{NHL slip bdry for Dm-sph}) such that
	\be\label{ex-reg-ss}
	v\in E_{m,T}^{\sigma}\cap H_t^1 L_x^2 \cap L_t^2 H_x^2\cap L_{tx}^{\infty}\big(D_m\times[0,T]\big), \quad P\in L_t^2 H_x^1(D_m\times[0,T]). \ee
	Moreover, if $ (\hat{v}, \hat{P}) $ is another strong solution, then $ \hat{v}$ coincides with $ v $ on $ D_m\times[0,T] $. As a result, if $ v_0 $ possesses the even-odd-odd symmetry, i.e. $ v_0\in E_{m,T}^{\sigma,s} $, then so does $ v $.
\end{proposition}

\begin{remark}
	In Euclidean coordinates, the strong solution $ v $ of (\ref{asns-sph}) in Proposition \ref{Prop, local soln in ad} is understood in the same sense as that in Definition \ref{Def, ss} with $ D $ being replaced by $ D_m $.
\end{remark}

\begin{proof}[Proof of Proposition \ref{Prop, local soln in ad}]
Firstly, we decompose the given initial data $ v_0 $ and the initial vorticity $ \o_0:=\text{curl}\ v_0 $ as
\[ v_0 = v_{0,\rho}e_\rho + v_{0,\phi}e_\phi + v_{0,\th} e_\th, \quad \o_0 = \o_{0,\rho}e_\rho + \o_{0,\phi}e_\phi + \o_{0,\th} e_\th.\]
Meanwhile, we denote
\be\label{fixed const}
A_0 := 1 + \| v_{0,\th} \|_{L^\infty(D_m)} + \| \o_{0,\th}\|_{L^6(D_m)}.
\ee
In the following proof, $ C $ denotes a generic constant which may depend on $ \a $ and $ m $. The values of $ C $ may be different from line to line. If a constant $ C $ also depends on other quantities, we will state it explicitly. Now we give an outline of the proof:
\begin{enumerate}[(i)]
\item For any $ T>0 $ and for any scalar functions $v_{\rho}$ and $v_{\phi}$ such that the vector field $b := v_{\rho} e_{\rho}+v_{\phi} e_{\phi}$ belongs to $ E^\sigma_{m,T} $, we use $b$ as a given data in the equation for $v_{\theta}$ (see (\ref{asns-sph})). This linearized equation, with suitable boundary condition and $ v_{0,\th}$ as the initial value, determines a vector field $ v_\th e_{\theta} \in E^{\sigma}_{m,T} \cap L_{tx}^{\infty}(D_m\times(0,T])$.

\item Use the above $b$ and $v_{\theta}$ as given data in the equation for $\widetilde{\Omega}$ (see (\ref{alt eq for O})) with 0 boundary value and $\omega_{0, \theta} /(\rho \sin \phi)$ as initial value, one finds $ \widetilde{\O} $ in $ E_{m,T}\cap L^\infty_{t}L^q_{x}\big(D_m\times[0,T]\big) $ for any $ q\geq 1 $. Then we define $ \widetilde{\omega}_{\theta} = \rho \sin \phi\, \widetilde{\Omega}$ and treat it as the angular vorticity.

\item Based on the $ \wt{\omega}_{\th} $ constructed above and the Biot-Savart law with a suitable boundary condition, we determine a vector
\[\tilde{b}=\tilde{v}_{\rho} e_{\rho}+\tilde{v}_{\phi} e_{\phi} \in E^\sigma_{m,T} \cap L_{t}^{2}H_{x}^{2}\cap L_{tx}^{\infty}(D_m\times [0,T]).\]
Thus, the correspondence between $b$ and $\tilde{b}$ determines a map $ \m{L} $:
\be\label{cm}
\mathbb{L}b = \t{b},
\ee
from the space $ E^\sigma_{m,T}\cap  \operatorname{span} \{e_{\rho}, e_{\phi} \} $ to itself.
As a summary of steps so far, a diagram of the process is given below:
\[ \mbox{ Diagram: $\quad b \Rightarrow v_{\th} \Rightarrow \wt{\Omega} \Rightarrow \wt{\omega}_{\th}\equiv \rho \sin \phi\, \wt{\Omega} \Rightarrow \tilde{b}$. } \]

\item  Next, we will find a suitably large number $ M $ such that  $ \mathbb{L} $ is a contraction mapping on the space $\overline{B_{E^\sigma_{m,T}}(0, M)} \bigcap \operatorname{span} \{e_{\rho}, e_{\phi} \}$ as long as $ T $ is sufficiently small. Thus, we obtain a fixed point $ b $ of $ \mathbb{L} $ thanks to the contraction mapping theorem.

\item Based on the fixed point $ b $ of $ \m{L} $ in the above step, we define $v \equiv b+v_{\th} e_{\th}$ and $ \o=\na\times v $, where $ v_\th $ is the function constructed in step (i). Then we show that $ \o_\th $ coincides with the previously constructed $ \wt{\o}_\th $. Based on this, we manage to prove $ v\in E^\sigma_{m,T} \cap H_t^1 L_x^2\cap L_{t}^{2} H_{x}^{2}\cap L_{tx}^{\infty}(D_m\times[0,T]) $ and find a pressure term $ P $ in $ L_t^2 H_x^1(D_m\times[0,T]) $ such that $ (v,P) $ is a strong solution of (\ref{asns-sph}) on $ D_m\times [0,T] $ subject to the initial data $ v_0 $ and the NHL boundary condition (\ref{NHL slip bdry for Dm-sph}).

\item Finally, the uniqueness of the strong solution $ v $ will be addressed. As a byproduct, we will justify the preservation of the even-odd-odd symmetry of the initial data.
\end{enumerate}

In the following argument, details of the above steps will be carried out.
\\

Step 1. Construction of $v_{\theta}$.

Fix any $b := v_{\rho} e_{\rho}+v_{\phi} e_{\phi} \in E^\sigma_{m,T} $. Based on the $ v_\th $ equation in (\ref{asns-sph}), we determine $v_{\theta}$ by the following initial boundary value problem:
\begin{equation}\label{vth}
	\quad\left\{\begin{array}{l}
		\Big(\Delta-\frac{1}{\rho^2 \sin^2\phi}\Big)v_\theta-b\cdot\nabla v_\theta-\frac{1}{\rho}\big(v_\rho + \cot\phi\, v_\phi\big) v_\theta - \p_t v_\theta = 0, \quad \text{in}\quad D_m\times(0,T];\\
 		\p_\phi v_\theta = - \cot\phi\, v_\th,  \quad \text{on}\quad \p^{R} D_m\times (0,T],  \quad \p_\rho v_\theta = -\frac{1}{\rho}\, v_\th, \quad \text{on} \quad \p^{A}D_m \times (0,T];\\
 		v_\theta(x,0) = v_{0,\theta} (x), \quad x\in D_m,
 	\end{array}\right.
\end{equation}
where $ v_{0,\th} $ is the $ \th $-component of the given initial data $ v_0 $. Since the boundary condition for $v_{\theta}$ is of  Robin type which is more complicated than the Neumann condition, we instead consider the equation for $ \Gamma $, defined as
$$
\Gamma = \rho\sin\phi\, v_{\theta},
$$
which satisfies the homogeneous Neumann boundary condition. More precisely, $ \Gamma $ is determined by the following problem based on (\ref{vth}).
\begin{equation}\label{Gamma-Dm-exist}
	\left\{\begin{array}{l}
		\Delta \Gamma - b\cdot \nabla \Gamma - \frac{2}{\rho}\p_\rho \Gamma - \frac{2 \cot\phi}{\rho^2}\p_\phi \Gamma-\partial_{t} \Gamma=0, \quad \text { in } \quad D_{m} \times(0, T]; \\
		\partial_{n} \Gamma=0, \quad \text { on } \quad \p D_{m} \times(0, T]; \\
		\Gamma(x, 0)=\Gamma_{0}(x), \quad x \in D_{m},
	\end{array}\right.
\end{equation}
where $ \Gamma_0 := \rho\sin\phi\, v_{0,\th} $. According to Lemma \ref{Lemma, Gamma-Dm}, (\ref{Gamma-Dm-exist}) possesses a unique bounded weak solution $ \Gamma $ in $ E_{m,T} $ which satisfies the estimates (\ref{Gamma-max}) and (\ref{Gamma-energy-Dm}). Since $ \rho $ is bounded from above and below, then (\ref{Gamma-max}) and (\ref{Gamma-energy-Dm}) imply that
\be\label{est for vth-Dm}
\begin{aligned}
	&\left\|v_{\theta}\right\|_{L^\infty(D_m\times[0,T])} \leq C \| v_{0,\theta}\|_{L^{\infty}(D_m)}, \\
	&\left\|v_{\theta}\right\|_{E_{m,T}} \leq C e^{CT} \|v_{0,\theta}\|_{L^{2}(D_m)}.
\end{aligned}
\ee

Step 2. Constructing an intermediate angular vorticity $\wt{\omega}_\theta$.

With the vector field $b$ and the corresponding $v_{\theta}$ from Step 1, we will introduce a function
\be\label{om}
\wt{\omega}_{\theta} := \rho\sin\phi\, \wt{\O},
\ee
where $\wt{\O}$ is determined by the following problem (also see (\ref{alt eq for O})):
 		\be\label{ome}
 		\left\{\begin{array}{l}
 			\left(\Delta+\frac{2}{\rho} \partial_{\rho}+\frac{2\cot\phi}{\rho^2}\p_\phi\right) \wt{\Omega}-b\cdot \nabla \wt{\Omega}-\partial_{t} \wt{\Omega} = \frac{1}{\rho^2\sin\phi}\Big( \frac{1}{\rho}\p_\phi(v_\th^2) - \cot\phi\, \p_\rho(v_\th^2) \Big), \text { in } \ D_{m} \times(0, T]; \\
 			\wt{\Omega}=0, \text { on }\ \partial D_{m} \times(0, T]; \\
 			\wt{\Omega}(x, 0) = \omega_{0, \theta}(x) / (\rho\sin\phi),\ x \in D_{m} .
 		\end{array}\right.
 		\ee
Here, $ \o_{0,\th} $ is the $ \th $-component of $\omega_{0} :=\operatorname{curl} v_{0}$. The reason that we study the equation (\ref{ome}) of $ \wt{\O} $ instead of the equation of $ \wt{\o}_\th $ (see (\ref{ome-th eq})) is to avoid the term $ \frac{1}{\rho}(v_{\rho}+\cot\phi\,v_{\phi})\wt{\o}_{\th} $ in (\ref{ome-th eq}).

{\it Claim A: The problem (\ref{ome}) has a unique weak solution $ \wt{\O} $ in the energy space $ E_{m,T} $. In addition, the energy of $ \wt{\O} $ has the following upper bound:
\be\label{eot}
\|\wt{\O}\|_{E_{m,T}} \leq C A_0^2 e^{CT},
\ee
where $ C=C(\a,m)$ and $ A_0 $ is as defined in (\ref{fixed const}).
}

{\it Proof of Claim A}: Firstly, we denote the function on the right-hand side of (\ref{ome}) to be $ R_1 $, that is
\[
R_1 := \frac{1}{\rho^2\sin\phi}\Big( \frac{1}{\rho}\p_\phi(v_\th^2) - \cot\phi\, \p_\rho(v_\th^2) \Big).
\]
Thanks to the estimates (\ref{est for vth-Dm}), we know $ v_\th\in E_{m,T}\cap L_{tx}^{\infty}(D_m\times[0,T]) $, which implies $ R_1\in L_{tx}^{2}(D_m\times[0,T]) $ and
\be\label{R_1}
\| R_1 \|_{L^{2}(D_m\times[0,T])} \leq C e^{CT}\| v_{0,\th}\|_{L^{\infty}(D_m)}\| v_{0,\th}\|_{L^{2}(D_m)} \leq C e^{CT} A_0^2.
\ee
Next, similar to the proof of Lemma \ref{Lemma, Gamma-Dm}, we regard the problem (\ref{ome}) as a 2D problem on the domain $ D_{m}' $ which is defined as in (\ref{app domain-2D-sph}). Then the energy space $ E_{m,T} $ is embedded into $ L^{4}_{tx}(D_{m}\times[0,T]) $ due to Lemma \ref{Lemma, asLe}. So the vector field $ b $ in the drift term $ b\cdot \nabla\wt{\O} $ is in the critical class. As a result, the existence part in Claim A follows from standard parabolic theory. To address the uniqueness part, we assume there are two weak solutions $ \wt{\O}_{1} $ and $ \wt{\O}_2 $ in the energy space $ E_{m,T} $ and then consider the equation for their difference $ \wt{\O}_1 - \wt{\O}_2 $. Then it follows from the standard energy estimate that $ \wt{\O}_1-\wt{\O}_2\equiv 0 $ on $ D_m\times[0,T] $. Finally, the estimate (\ref{eot}) can be established by testing (\ref{ome}) with $ \wt{\O} $ and taking advantage of the estimate (\ref{R_1}). Hence, Claim A is verified.

For the solution $ \wt{\O} $ in the above claim, we can actually obtain higher integrability of $ \wt{\O} $ which will be used later. Since $ v_0\in H^{2}(D_m) $ and satisfies the NHL boundary condition (\ref{NHL slip bdry for Dm-sph}), then $ \wt{\O}(\cdot, 0) = \omega_{0, \theta}(x) / (\rho\sin\phi)\in H^{1}(D_m) $ with 0 boundary value. So by regarding it as a function on the 2D domain $ D_m' $, we find $ \wt{\O}(\cdot, 0)\in L^{q}(D_m) $ for any $ q\geq 1 $ due to the 2D Sobolev inequality. Then using the standard energy estimate for $ \wt{\O}^{q} $ and the fact that the drift terms are integrated out, we have
\[
\| \wt{\O}\|_{L_t^\infty L_x^q(D_m\times[0,T])} \leq \exp\Big( Cq\| v_\th\|^4_{L^\infty(D_m\times[0,T])}  T\Big)\Big( \|\wt{\O}(\cdot, 0)\|_{L^q(D_m)} + 1 \Big).
\]
Since $ \| v_\th\|_{L^\infty(D_m\times[0,T])} \leq C \| v_{0,\th}\|_{L^\infty(D_m)}\leq C A_0 $, we deduce
\be\label{olqe1}
\| \wt{\O}\|_{L_t^\infty L_x^q(D_m\times[0,T])} \leq e^{Cq A_0^4 T}  \Big( \|\wt{\O}(\cdot, 0)\|_{L^q(D_m)} + 1 \Big).
\ee
In particular, if $ q $ is restricted in the interval $ [1,6] $, then
\be\label{olqe2}
\| \wt{\O}\|_{L_t^\infty L_x^q(D_m\times[0,T])} \leq C e^{C A_0^4 T} A_0, \quad\forall\, 1\leq q\leq 6.
\ee

After the construction of $ \wt{\O} $, we define
\[ \wt{\o}_{\th} = \rho\sin\phi\, \wt{\O}. \]
Then it is the unique weak solution of the following problem (\ref{ome-th eq}) in the energy space $ E_{m,T} $.
\be\label{ome-th eq}
\left\{\begin{array}{l}
 			\big(\Delta-\frac{1}{\rho^2\sin^2\phi}\big)\wt{\omega}_{\theta} - b\cdot\nabla \wt{\omega}_{\theta} + \frac{1}{\rho}(v_{\rho}+\cot\phi\,v_{\phi})\wt{\omega}_{\theta} - \p_{t}\wt{\omega}_{\theta} \\
 			\hspace{2in}= \frac{1}{\rho^2}\p_{\phi}(v_{\th}^2) - \frac{\cot\phi}{\rho}\p_{\rho}(v_{\th}^2), \text { in } \  D_{m} \times(0, T]; \\
 			\wt{\omega}_{\theta}=0, \text { on } \ \partial D_{m} \times(0, T]; \\
 			\wt{\omega}_{\theta}(x, 0) = \omega_{0,\th}(x), \ x \in D_{m}.
 		\end{array}\right.
\ee
Note that $\wt{\omega}_{\theta}$ may not be equal to curl $b$ yet. Next, we will use $\wt{\Omega}$ to construct a vector field $\tilde{b}$ according to the Biot-Savart law $ \Delta \t{b} = -\na\times \wt{\o}_\th $. Eventually, the map that assigns $b$ to $\tilde{b}$ will be shown to have a fixed point. For such a fixed point $ b $, we will prove in Step 5 that $ \operatorname{curl} b = \wt{\o}_\th $. \\

Step 3. Introducing a map $ \mathbb{L} $ from $E^\sigma_{m,T} \bigcap \operatorname{span} \{e_{\rho}, e_{\phi} \}$ into itself.

	Using the function $\wt{\Omega}$  in Step 2 and the Biot-Savart law in the spherical system (see (\ref{Biot-Savart-sph}) and (\ref{v-rho, v-phi, O}) in Section \ref{Subsec, BS law}), we construct two functions $\tilde{v}_\rho, \tilde{v}_\phi \in E_{m,T}$ by solving the elliptic problems (\ref{v-rho-t}) and (\ref{v-phi-t}) respectively in $ H^{1}(D_m) $ for a.e. $ t\in[0,T] $.
 	\be\label{v-rho-t}
 	\left\{\begin{array}{l}
 		\Big(\Delta + \frac{2}{\rho}\,\p_\rho + \frac{2}{\rho^2} \Big)\tilde{v}_\rho = -\frac{1}{\sin\phi}\,\p_{\phi}(\sin^2\phi\,\wt{\O}),\ \text { in } \ D_{m}; \\
		\p_\phi \tilde{v}_\rho = 0\ \text{ on }\ \p^R D_m, \quad \tilde{v}_\rho = 0\ \text{ on }\ \p^A D_m.
 	\end{array}\right.
 	\ee
 		
 	\be\label{v-phi-t}
 	\left\{\begin{array}{l}
 		\Big( \Delta+\frac{2}{\rho}\p_\rho + \frac{1-\cot^2\phi}{\rho^2} \Big)\tilde{v}_\phi = \frac{1}{\rho^3}\,\p_\rho(\rho^4\sin\phi\,\wt{\O}),\ \text { in } \ D_{m}; \\
 		\tilde{v}_\phi = 0\ \text{ on }\ \p^R D_m, \quad \p_\rho \tilde{v}_\phi = - \frac{1}{\rho} \tilde{v}_\phi \ \text{ on }\ \p^A D_m.
 	\end{array}\right.
 	\ee
 	In particular, when $ t=0 $, recalling that $ \wt{\O}(x, 0) = \o_{0,\th}(x)/(\rho\sin\phi) $ in (\ref{ome}), then by defining
 	\be\label{initial-bt}
 	\t{v}_\rho(x,0) = v_{0,\rho}(x) \quad\text{and}\quad \t{v}_\phi(x,0) = v_{0,\phi}(x) \quad \text{on} \quad \ol{D_m},
 	\ee
 	one can verify that $ \t{v}_\rho(x,0) $ and $ \t{v}_\phi(x,0) $ satisfy (\ref{v-rho-t}) and (\ref{v-phi-t}) respectively when $ t=0 $.
 		
 	{\it Claim B: For a.e. $ t\in[0,T] $, (\ref{v-rho-t}) (resp. (\ref{v-phi-t})) has a unique solution $ \t{v}_\rho(\cdot, t) $ (resp. $ \t{v}_\phi(\cdot, t) $) in the space $ H^{1}(D_m) $. Moreover, both $ \t{v}_\rho(\cdot, t) $ and $ \t{v}_\phi(\cdot, t) $ belong to $ H^{2}(D_m) $ and satisfy the following estimates:
 		\begin{align}
 			\| \t{v}_\rho(\cdot, t)\|_{H^1(D_m)} + \| \t{v}_\phi(\cdot, t)\|_{H^1(D_m)} &\leq C \|\wt{\O}(\cdot, t)\|_{L^2(D_m)}, \label{e1-v-t}\\
 			\| \t{v}_\rho(\cdot, t)\|_{H^{2}(D_m)} + \| \t{v}_\phi(\cdot, t)\|_{H^{2}(D_m)} &\leq C \|\wt{\O}(\cdot, t)\|_{H^{1}(D_m)}, \label{e2-v-t}
 		\end{align}
 	where $ C=C(\a,m) $.}

 	{\it Proof of Claim B:} Firstly, since $ \wt{\O}\in E_{m,T} $, we can find a set $ S_{T}\subseteq [0,T] $ such that $ [0,T]\setminus S_{T} $ has measure 0 and for any $ t\in S_{T} $, $ \wt{\O}(\cdot,t)\in H^{1}(D_m) $. Fix any $ t\in S_{T} $, the functions on the right-hand side of (\ref{f-t}) and (\ref{g-t}) are in $ L^2(D_m) $. Noting the signs of the potential terms in (\ref{v-rho-t}) and (\ref{v-phi-t}) are not helpful when proving the existence and uniqueness of the solutions, so we introduce
 	\[ \t{f}(\cdot) := \rho \t{v}_{\rho}(\cdot,t) \quad\text{and}\quad \t{g}(\cdot) := \rho \t{v}_{\phi}(\cdot,t) \]
 	which are determined by the following problems:
 	\be\label{f-t}
 	\left\{\begin{array}{l}
 		\Delta \tilde{f} = -\frac{\rho}{\sin\phi}\,\p_{\phi}(\sin^2\phi\,\wt{\O}),\ \text { in } \ D_{m}; \\
 		\p_\phi \tilde{f} = 0\ \text{ on }\ \p^R D_m, \quad \t{f} = 0\ \text{ on }\ \p^A D_m.
 	\end{array}\right.
 	\ee
 		
 	\be\label{g-t}
 	\left\{\begin{array}{l}
 		\Big( \Delta - \frac{1+\cot^2\phi}{\rho^2} \Big)\tilde{g} = \frac{1}{\rho^2}\,\p_\rho(\rho^4\sin\phi\,\wt{\O}),\ \text { in } \ D_{m}; \\
 		\tilde{g} = 0\ \text{ on }\ \p^R D_m, \quad \p_\rho \tilde{g} = 0 \ \text{ on }\ \p^A D_m.
 	\end{array}\right.
 	\ee
 	Now the potential term in (\ref{v-rho-t}) disappears and the potential term in (\ref{g-t}) has the good sign, so the existence and uniqueness of the solutions of (\ref{f-t}) and (\ref{g-t}) in the space $ H^{1}(D_m) $ can be established using classical methods, e.g. the Lax-Milgram theory. Next, we will show both $ \t{f} $ and $ \t{g} $ belong to the stronger space $ H^{2}(D_m) $. Analogous to the proof of Lemma \ref{Lemma, Gamma-Dm}, we view (\ref{f-t}) and (\ref{g-t}) as 2D elliptic problems on the rectangular domain $ D_m' $ in the $ \rho $-$ \phi $ plane, see Figure \ref{Fig,app domain-2D-sph}. Then the problems become
 	\be\label{f-t-2D}
 	\left\{\begin{array}{l}
 		\Big(\p_\rho^2 + \frac{1}{\rho^2}\p_\phi^2 + \frac{2}{\rho}\p_\rho + \frac{\cot\phi}{\rho^2}\p_\phi \Big) \tilde{f} = -\frac{\rho}{\sin\phi}\,\p_{\phi}(\sin^2\phi\,\wt{\O}),\ \text { in } \ D_{m}'; \\
 		\p_\phi \tilde{f} = 0\ \text{ on }\ \p^R D_m', \quad \t{f} = 0\ \text{ on }\ \p^A D_m'.
 	\end{array}\right.
 	\ee
 	\be\label{g-t-2D}
 	\left\{\begin{array}{l}
 		\Big(\p_\rho^2 + \frac{1}{\rho^2}\p_\phi^2 + \frac{2}{\rho}\p_\rho + \frac{\cot\phi}{\rho^2}\p_\phi - \frac{1+\cot^2\phi}{\rho^2} \Big)\tilde{g} = \frac{1}{\rho^2}\,\p_\rho(\rho^4\sin\phi\,\wt{\O}),\ \text { in } \ D_{m}'; \\
 		\tilde{g} = 0\ \text{ on }\ \p^R D_m', \quad \p_\rho \tilde{g} = 0 \ \text{ on }\ \p^A D_m'.
 	\end{array}\right.
 	\ee
 	Since $ \wt{\O}(\cdot, t)\in L^2(D_m') $ and $ \frac{1}{m}\leq \rho \leq 1 $, we can use the standard interior regularity theory to estimate the $H^{1}(D_m') $ (resp. $ H^{2}(D_m') $) norms of $ \t{f} $ and $ \t{g} $ in terms of the $ L^2(D_m') $ (resp. $ H^{1}(D_m') $) norms of $ \wt{\O}(\cdot, t) $. In addition, since $ D_m' $ is a rectangle in $ \rho $-$ \phi $ plane and the boundary conditions of $ \t{f} $ and $ \t{g} $ are of mixed Dirichlet-Neumann type, we can apply appropriate reflection near the boundary of $ D_m' $ (two reflections are needed near any corner) to reduce the boundary regularity estimates into interior regularity estimates. Thus, we know both $ \t{f} $ and $ \t{g} $ belong to $ H^{2}(D_m) $ and
 	\begin{align}
 		\| \t{f}\|_{H^{1}(D_m)} + \| \t{g}\|_{H^{1}(D_m)} &\leq C \|\wt{\O}(\cdot, t)\|_{L^2(D_m)}, \label{e1-f-t}\\
 		\| \t{f}\|_{H^{2}(D_m)} + \| \t{g}\|_{H^{2}(D_m)} &\leq C \|\wt{\O}(\cdot, t)\|_{H^{1}(D_m)}.
 	\end{align}
 	Now changing back to $ \t{v}_{\rho}(\cdot, t) $ and $ \t{v}_{\phi}(\cdot,t) $ from $ \t{f} $ and $ \t{g} $, we conclude that both $ \t{v}_{\rho}(\cdot, t) $ and $ \t{v}_{\phi}(\cdot,t) $ belong to $ H^{2}(D_m) $ and they satisfy the estimates (\ref{e1-v-t}) and (\ref{e2-v-t}). Hence, Claim B is justified.
 		
 	For  any $ t\in S_{T} $ and for the function $ \t{f} $ defined in the above proof, if we apply the Moser iteration on (\ref{f-t}), then one can find
 	\be\label{f-t-Moser}
 	\|\t{f}\|_{L^\infty(D_m)} \leq C\big( 1+\|\wt{\O}(\cdot,t)\|_{L^6(D_m)} \big)\big( 1+ \|\t{f}\|_{L^6(D_m)}\big),
 	\ee
 	By Sobolev inequality, $ \|\t{f}\|_{L^6(D_m)} \leq C \|\t{f}\|_{H^{1}(D_m)} $. Then we combine the estimates (\ref{f-t-Moser}) with (\ref{e1-f-t}) to obtain
 	\begin{align*}
 		\|\t{f}\|_{L^\infty(D_m)} &\leq C\big( 1+\|\wt{\O}(\cdot,t)\|_{L^6(D_m)} \big)\big( 1+\|\wt{\O}(\cdot,t)\|_{L^2(D_m)} \big) \\
 		&\leq C\big( 1+\|\wt{\O}(\cdot,t)\|_{L^6(D_m)} \big)^2.
 	\end{align*}
 	By similar argument, the above inequality also holds if the function $ \t{f} $ is replaced by $ \t{g} $. Therefore,
 	\be\label{bdd-v-t}
 	\| \t{v}_\rho(\cdot, t)\|_{L^\infty(D_m)} + \| \t{v}_\phi(\cdot, t)\|_{L^\infty(D_m)} \leq C\big( 1+\|\wt{\O}(\cdot,t)\|_{L^6(D_m)} \big)^2.
 	\ee
 		
 	Recalling the estimates (\ref{eot}) and (\ref{olqe2}) for $ \wt{\O} $, we know $ \wt{\O}\in E_{m,T}\cap L_t^\infty L_x^6(D_m\times[0,T]) $ and
 	\[ \|\wt{\O}\|_{E_{m,T}} + \|\wt{\O}\|_{L_t^\infty L_x^6(D_m\times[0,T])} \leq C A_0^2 e^{C A_0^4 T}.  \]
 	Consequently, we deduce from (\ref{e1-v-t}), (\ref{e2-v-t}) and (\ref{bdd-v-t}) that
 	\be\label{s-reg-bt}
 		\t{v}_\rho, \t{v}_\phi\in L_t^\infty H_x^{1}\cap L_t^2 H_x^{2}\cap L_{tx}^{\infty}\big(D_m\times[0,T]\big)
 	\ee
 	and
 	\begin{align}
 		\| \t{v}_\rho\|_{L_t^\infty H_x^{1}(D_m\times[0,T])} + \| \t{v}_\rho\|_{L_t^2 H_x^{2}(D_m\times[0,T])} + \| \t{v}_\rho\|_{L_{tx}^\infty (D_m\times[0,T])} & \leq C A_0^4 e^{C A_0^4 T},  \label{re-v-rho-t}\\
 		\| \t{v}_\phi\|_{L_t^\infty H_x^{1}(D_m\times[0,T])} + \| \t{v}_\phi\|_{L_t^2 H_x^{2}(D_m\times[0,T])} + \| \t{v}_\phi\|_{L_{tx}^\infty (D_m\times[0,T])} & \leq C A_0^4 e^{C A_0^4 T}.  \label{re-v-phi-t}
 	\end{align}

 	Define
 	\be\label{b-tilde}
 	\t{b} = \t{v}_\rho e_\rho + \t{v}_\phi e_\phi.
 	\ee
 	Then the above steps determine the map $ \mathbb{L} $  ( \ref{cm}) from $ b $ to $ \t{b} $. Next, we will prove $ \t{b}\in E^\sigma_{m,T} $. Due to the regularity property (\ref{s-reg-bt}) and the boundary conditions in (\ref{v-rho-t}) and (\ref{v-phi-t}) for $ \t{v}_\rho $ and $ \t{v}_\phi $, it remains to show $ \text{div}\, \t{b} = 0 $ for a.e. $ t\in[0,T] $. Instead of showing $ \text{div}\, \t{b} = 0$ directly, we will take advantage of the fact that $\rho^2 \text{div}\, \t{b}$ satisfies a simple equation (\ref{divb1}) with a good boundary condition, which allows us to conclude $\rho^2 \text{div}\, \t{b}=0$ for a.e. $ t\in[0,T] $.
 		
 	In fact, we fix any $ t\in S_{T} $, where $ S_{T} $ is the set defined in the proof of Claim B, and then define
 	\[h(\cdot) = \rho^2 \text{div}\, \tilde{b}(\cdot,t), \quad \text{on}\quad D_m.\]
	By direct calculation, it follows from the equations (\ref{v-rho-t}) and (\ref{v-phi-t}) for $ \t{v}_\rho $ and $ \t{v}_\phi $ that
 	\be\label{divb1}
 	\left\{\begin{array}{l}
 		\Delta h = 0,\ \text { in } \ D_{m}, \\
 		\p_n h = 0,\ \text { on } \ \p D_{m}.
 	\end{array}\right.
 	\ee
 	Testing \eqref{divb1} with $h$, we have
 	$$\|\nabla h\|_{L^2(D_m)}=0,$$
 	which implies $h \equiv C$ is a constant on $ D_m $. Next, we will prove this constant $ C $ must be 0. Based on the divergence formula (\ref{div free-sph}) in spherical coordinates,
 	\be\label{exp for h}
 	h(\cdot) = \rho^2 \text{div}\, \tilde{b}(\cdot,t) = \p_{\rho}\big(\rho^2\t{v}_\rho (\cdot, t) \big) + \frac{1}{\sin\phi}\,\p_\phi \big( \rho\sin\phi\, \t{v}_\phi(\cdot, t) \big).
 	\ee 		
 	Denote $ \phi_1 = \frac{\pi}{2}-\a $ and $ \phi_2 = \frac{\pi}{2} + \a $. For any $ \rho\in\big[\frac{1}{m}, 1\big] $, we multiply (\ref{exp for h}) by $ \sin\phi $ and then integrate both sides with respect to $ \phi $ from $ \phi_1 $ to $ \phi_2 $. Then due to the fact that $ \t{v}_\phi = 0 $ on $ \p^{R}D_m $, we know the second term on the right-hand side disappears. Thus, we obtain
 	\[ \int_{\phi_1}^{\phi_2} h(\rho,\phi)\sin\phi \,d\phi = \p_{\rho}\bigg( \rho \int_{\phi_1}^{\phi_2} \rho \t{v}_\rho (\rho,\phi,t)\sin\phi \,d\phi \bigg).\]
 	Define
 	\[ H(\rho) = \rho \int_{\phi_1}^{\phi_2} \t{v}_\rho(\rho,\phi,t) \sin\phi \,d\phi, \quad\forall\, \rho\in [1/m, 1]. \]
 	In order to show $ h $ is identically 0, it suffices to prove
 	\be\label{mean 0 for v-rho}
 	H(\rho) = 0, \quad\forall\, \rho\in [1/m, 1].
 	\ee
 	Denote $ \t{f}(\cdot) = \rho\t{v}_\rho(\cdot, t) $ on $ D_m $ as we did in the proof of Claim B, then $ H(\rho) = \int_{\phi_1}^{\phi_2} \t{f}(\rho,\phi)\sin\phi\,d\phi $. Meanwhile, it follows from (\ref{f-t-2D}) that $ \t{f} $ can be regarded as a solution of the following equation on the 2D domain $ D_m' $.
 	\be\label{f-t-2D-2}
 	\left\{\begin{array}{l}
 		\Big(\p_\rho^2 + \frac{1}{\rho^2}\p_\phi^2 + \frac{2}{\rho}\p_\rho + \frac{\cot\phi}{\rho^2}\p_\phi \Big) \tilde{f} = -\frac{\rho}{\sin\phi}\,\p_{\phi}(\sin^2\phi\,\wt{\O}),\ \text { in } \ D_{m}'; \\
 		\p_\phi \tilde{f} = 0\ \text{ on }\ \p^R D_m', \quad \t{f} = 0\ \text{ on }\ \p^A D_m'.
 	\end{array}\right.
 	\ee
 	For any $ \rho\in \big(\frac1m, 1\big) $, by taking advantage of the boundary conditions $ \p_\phi\t{f} = \wt{\O} = 0 $ on $ \p^{R}D_m' $ and the relation
 	\[ (\p_\phi^2\t{f})\sin\phi + \cos\phi\, \p_\phi \t{f} = \p_\phi\big( \sin\phi\, \p_\phi\t{f}\big), \]
 	we can multiply (\ref{f-t-2D-2}) by $ \sin\phi $ and then integrate both sides with respect to $ \phi $ from $ \phi_1 $ to $ \phi_2 $ to obtain
 	\be\label{ODE for H}
 	H''(\rho) + \frac{2}{\rho} H'(\rho) = 0, \quad \forall\, \rho\in (1/m, 1).
 	\ee
 	In addition, we have $ H(\rho_1) = H(\rho_2) =0 $ since $ \t{v}_\rho = 0 $ on $ \p^{A}D_m' $. By solving (\ref{ODE for H}) with the Dirichlet boundary condition, we conclude $ H\equiv 0 $ on $ [\rho_1,\rho_2] $.
 	As a result, $\nabla\cdot \tilde{b} = h/\rho^2 = 0$, completing this step. Meanwhile, thanks to (\ref{mean 0 for v-rho}), we also obtain the following byproduct:
 	\be\label{0int-vrho}
 	\int_{\phi_1}^{\phi_2} \t{v}_\rho(\rho,\phi,t) \sin\phi \,d\phi = 0, \quad\forall\, \rho\in [1/m, 1], \quad t>0.
 	\ee
 		
 	Step 4. We prove $\mathbb{L}$ is a contraction map from  $\overline{B_{E^\sigma_{m,T}}(0, M)} \bigcap \operatorname{span} \{e_{\rho}, e_{\phi} \}$ into itself for some large $ M $ and small $ T $.
 	
 	For any $ b\in E^\sigma_{m,T} \cap \operatorname{span} \{e_{\rho}, e_{\phi} \} $, denote $ \t{b}=\mathbb{L}b $. We point out that although the initial value of $ b $ is not required to be $ v_{0,\rho}e_\rho + v_{0,\phi}e_\phi $, where $ v_0 $ is the given initial velocity in Proposition \ref{Prop, local soln in ad}, the initial value of $ \t{b} $ is guaranteed to be $ v_{0,\rho}e_\rho + v_{0,\phi}e_\phi $ according to the construction of $ \mathbb{L} $ (see (\ref{initial-bt})). In addition, based on (\ref{vth}), the constructed $ v_\th $ is also ensured to have the initial value $ v_{0,\th} $, where $ v_0 $ is again the given initial velocity. As a result, when $ T\leq 1 $, it follows from the estimates (\ref{re-v-rho-t}) and (\ref{re-v-phi-t}) that
 	\[ \|\t{b}\|_{E^\sigma_{m,T}} \leq C A_0^4 e^{C A_0^4},\]
 	where $ A_0 $ is as defined in (\ref{fixed const}).
 	Now we denote $ M $ to be the above upper bound:
 	\be\label{bdd M}
 	M:= C A_0^4 e^{C A_0^4}.
 	\ee
 	Then $ \mathbb{L} $ maps $\overline{B_{E^\sigma_{m,T}}(0, M)} \bigcap \operatorname{span} \{e_{\rho}, e_{\phi} \}$ into itself. We fix such an $ M $ and then we will prove $ \mathbb{L} $ is a contraction map if $ T $ is sufficiently small.
 	
 	For $ i=1,2 $, let $b^{(i)}=v_\rho^{(i)}e_\rho+v_\phi^{(i)}e_\phi\in E^\sigma_{m,T}$, and denote $ \Gamma^{(i)} $, $ v_\th^{(i)} $, $ \wt{\O}^{(i)} $, $ \t{v}_\rho^{(i)} $, $ \t{v}_\phi^{(i)} $ and $ \t{b}^{(i)} $ to be the functions constructed as in the previous steps  1-3. According to the equations (\ref{v-rho-t}) and (\ref{v-phi-t}) with $ \t{v}_\rho $, $ \t{v}_\phi $ and $ \wt{\O} $ being replaced by $ \t{v}_\rho^{(i)} $, $ \t{v}_\phi^{(i)} $ and $ \wt{\O}^{(i)} $ for $ i=1,2 $ respectively, we have
 		\[
 	\left\{\begin{array}{l}
 		\Big(\Delta + \frac{2}{\rho}\,\p_\rho + \frac{2}{\rho^2} \Big) \big(\tilde{v}_\rho^{(2)} - \tilde{v}_\rho^{(1)} \big) = -\frac{1}{\sin\phi}\,\p_{\phi}\big[\sin^2\phi\,\big(\wt{\O}^{(2)} - \wt{\O}^{(1)}\big)\big],\ \text { in } \ D_{m}; \\
 		\p_\phi \big(\tilde{v}_\rho^{(2)} - \tilde{v}_\rho^{(1)} \big) = 0\ \text{ on }\ \p^R D_m, \quad \big(\tilde{v}_\rho^{(2)} - \tilde{v}_\rho^{(1)} \big) = 0\ \text{ on }\ \p^A D_m.
 	\end{array}\right.
 	\]
 	
 	\[
 	\left\{\begin{array}{l}
 		\Big( \Delta+\frac{2}{\rho}\p_\rho + \frac{1-\cot^2\phi}{\rho^2} \Big)\big(\tilde{v}_\phi^{(2)} - \tilde{v}_\phi^{(1)} \big) = \frac{1}{\rho^3}\,\p_\rho \big[ \rho^4\sin\phi\, \big(\wt{\O}^{(2)} - \wt{\O}^{(1)}\big) \big],\ \text { in } \ D_{m}; \\
 		\big(\tilde{v}_\phi^{(2)} - \tilde{v}_\phi^{(1)} \big) = 0\ \text{ on }\ \p^R D_m, \quad \p_\rho \big(\tilde{v}_\phi^{(2)} - \tilde{v}_\phi^{(1)} \big) = - \frac{1}{\rho} \big(\tilde{v}_\phi^{(2)} - \tilde{v}_\phi^{(1)} \big) \ \text{ on }\ \p^A D_m.
 	\end{array}\right.
 	\]
 	Then similar to the derivation of (\ref{e1-v-t}), we find
 	\be\label{est-diff-bt}
 	 \| \t{b}^{(2)} - \t{b}^{(1)} \|_{E_{m,T}} \leq C \| \wt{\O}^{(2)} - \wt{\O}^{(1)} \|_{L^\infty_t L^2_x(D_m\times[0,T])}.
 	 \ee

 	Denote $ f = \wt{\O}^{(2)} - \wt{\O}^{(1)} $. Then based on the equations in (\ref{ome}) with $ \wt{\O} $ and $ b $ being replaced by $ \wt{\O}^{(i)} $ and $ b^{(i)} $ for $ i=1,2 $, we know $ f $ is a weak solution to the following problem:
 	\be\label{diff-ome}
 	\left\{\begin{array}{l}
 		\left(\Delta+\frac{2}{\rho} \partial_{\rho}+\frac{2\cot\phi}{\rho^2}\p_\phi\right) f -
 		b^{(2)}\cdot \nabla f - (b^{(2)}-b^{(1)})\cdot \nabla\wt{\O}^{(1)} - \p_{t} f\\
 		\quad =  \frac{1}{\rho^2\sin\phi} \Big( \frac{1}{\rho}\p_\phi \Big[ \big(v_\th^{(2)}\big)^2 - \big(v_\th^{(1)}\big)^2 \Big] - \cot\phi\, \p_\rho \Big[ \big(v_\th^{(2)}\big)^2 - \big(v_\th^{(1)}\big)^2 \Big] \Big),\quad \text {in} \quad D_{m} \times(0, T]; \\
 		f = 0, \quad \text {on}\quad \partial D_{m} \times(0, T]; \\
 		f(x,0) = 0,\quad x \in D_{m} .
 	\end{array}\right.
 	\ee
 	Testing (\ref{diff-ome}) with $ f $ and using integration by parts, we have
 	\be\label{ome4}
 	\begin{aligned}
 		\int_0^T\int_{D_{m}}|\nabla f|^{2} \,dx\,dt + \frac{1}{2} \int_{D_{m}} |f(x, T)|^{2} \,dx = I_1 + I_2,
 	\end{aligned}
 	\ee
 	where
 	\[
 	\begin{split}
 		I_1 &= -\int_0^T\int_{D_{m}} \big[(b^{(2)}-b^{(1)})\cdot \nabla\wt{\O}^{(1)}\big] f \,dx\,dt,\\
 		I_2 & = \int_{0}^{T}\int_{D_m} \frac{1}{\rho^2\sin\phi}\, \Big[ \big(v_\th^{(2)}\big)^2 - \big(v_\th^{(1)}\big)^2 \Big] \bigg( \frac{1}{\rho}\,\p_\phi f - \cot\phi \,\p_\rho f \bigg) \,dx\,dt.
 	\end{split}
 	\]
 	Applying the integration by parts and the H\"older's inequality, we know
 	\begin{align*}
 		I_1 & \leq \| \wt{\O}^{(1)}\|_{L^5_{tx}(D_m\times[0,T])} \| b^{(2)} - b^{(1)} \|_{L^{10/3}_{tx}(D_m\times[0,T])} \| \nabla f\|_{L^2_{tx}(D_m\times[0,T])} \\
 		& \leq T^{1/5} \| \wt{\O}^{(1)}\|_{L^\infty_{t}L^5_{x}(D_m\times[0,T])} \| b^{(2)} - b^{(1)} \|_{E_{m,T}}  \| \nabla f\|_{L^2_{tx}(D_m\times[0,T])}.
 	\end{align*}
 	Then it follows from the Cauchy-Schwarz inequality and the estimate (\ref{olqe2}) with $ q=5 $ that
 	\be\label{est for I}
 	I_1 \leq \frac14 \| \nabla f\|_{L^2_{tx}(D_m\times[0,T])}^2 + C A_0^2 e^{C A_0^4 T} T^{2/5} \| b^{(2)} - b^{(1)} \|_{E_{m,T}}^2,
 	\ee
 	Next, we estimate $ I_2 $. By H\"older's inequality, we find  	
 	\[
 		I_2 \leq C \| \nabla f\|_{L^2_{tx}(D_m\times[0,T])} \big\| v_\th^{(2)} + v_\th^{(1)} \big\|_{L^\infty_{tx}(D_m\times[0,T])} \big\| v_\th^{(2)} - v_\th^{(1)} \big\|_{L^2_{tx}(D_m\times[0,T])}.
 	\]
 	By Cauchy-Schwarz inequality and the bound (\ref{est for vth-Dm}), we have
 	\be\label{est for II}
 	I_2 \leq \frac14 \|\nabla f\|_{L^2_{tx}(D_m\times[0,T])}^2 + C A_0^2 \big\| v_\th^{(2)} - v_\th^{(1)} \big\|_{L^2_{tx}(D_m\times[0,T])}^2,  \ee
 	Plugging (\ref{est for I}) and (\ref{est for II}) into (\ref{ome4}) leads to
 	\be\label{est-diff-ot}
 	\|f\|_{E_{m,T}}^2 \leq C A_0^2 e^{C A_0^4 T} T^{2/5} \| b^{(2)} - b^{(1)} \|_{E_{m,T}}^2 + CA_0^2 \big\| v_\th^{(2)} - v_\th^{(1)} \big\|_{L^2_{tx}(D_m\times[0,T])}^2.
 	\ee

 	Combining (\ref{est-diff-bt}) with (\ref{est-diff-ot}) yields
 	\be\label{est-diff-bt2}
 	 \| \t{b}^{(2)} - \t{b}^{(1)} \|_{E_{m,T}} \leq C A_0^2 e^{C A_0^4 T} T^{2/5} \| b^{(2)} - b^{(1)} \|_{E_{m,T}}^2 + CA_0^2 \big\| v_\th^{(2)} - v_\th^{(1)} \big\|_{L^2_{tx}(D_m\times[0,T])}^2.
 	\ee
 	So it remains to estimate $ \| v_\th^{(2)} - v_\th^{(1)} \|_{L^2_{tx}(D_m\times[0,T])} $ or equivalently $ \| \Gamma^{(2)} - \Gamma^{(1)} \|_{L^2_{tx}(D_m\times[0,T])} $. Denote $ g=\Gamma^{(2)} - \Gamma^{(1)} $. Then according to (\ref{Gamma-Dm-exist}), it holds that
 	\begin{equation}\label{diff-Gamma}
 		\left\{\begin{array}{l}
 			\Delta g - b^{(2)}\cdot \nabla g - \big(b^{(2)}-b^{(1)}\big)\cdot \nabla \Gamma^{(1)} - \frac{2}{\rho} \p_\rho g - \frac{2 \cot\phi}{\rho^2}\p_\phi g - \partial_{t} g = 0, \quad \text {in} \quad D_{m} \times(0, T];
 			\\ \partial_{n} g = 0, \quad \text{on}\quad \partial D_{m} \times(0, T]; \\
 			g(x,0)=0, \quad x \in D_{m}.
 		\end{array}\right.
 	\end{equation}
 	Testing (\ref{diff-Gamma}) by $ g $, then we have
	 \be\label{G1}
 	 \begin{aligned}
 	 	\int_0^T\int_{D_{m}}|\nabla g|^{2} \,dx\,dt + \frac{1}{2} \int_{D_{m}} |g(x, T)|^{2} \,dx = J_1 + J_2,
 	 \end{aligned}
 	 \ee
	 where
 	 \[
 		\begin{split}
 			J_1 &= -\int_0^T\int_{D_{m}} \big[(b^{(2)}-b^{(1)})\cdot \nabla\Gamma^{(1)}\big] g \,dx\,dt,\\
 			J_2 & = - \int_{0}^{T}\int_{D_m} \bigg( \frac{2}{\rho}\,\p_\rho g + \frac{2\cot\phi}{\rho^2} \,\p_\phi g \bigg)g \,dx\,dt.
	 	\end{split}
 	 \]
 	We first estimate $ J_1 $. Applying the integration by parts and the H\"older's inequality yields
 	\begin{align*}
 		J_1 & \leq \| \Gamma^{(1)}\|_{L^\infty_{tx}(D_m\times[0,T])} \| b^{(2)} - b^{(1)} \|_{L^{2}_{tx}(D_m\times[0,T])} \| \nabla g\|_{L^2_{tx}(D_m\times[0,T])} \\
 		& \leq T^{\frac12} \| b^{(2)} - b^{(1)} \|_{E_{m,T}}  \| \Gamma^{(1)}\|_{L^\infty_{tx}(D_m\times[0,T])} \| \nabla g\|_{L^2_{tx}(D_m\times[0,T])}.
 	\end{align*}
	It then follows from the estimate (\ref{est for vth-Dm}) and the Cauchy-Schwarz inequality that
	\be\label{est for I2}
	J_1 \leq \frac14 \| \nabla g\|_{L^2_{tx}(D_m\times[0,T])}^2 + C A_0^2 T \| b^{(2)} - b^{(1)} \|_{E_{m,T}}^2.
	\ee
	Next, we estimate $ J_2 $ by Cauchy-Schwarz inequality to get
	\be\label{est for II2}
	J_2 \leq \frac14 \| \nabla g\|_{L^2_{tx}(D_m\times[0,T])}^2  + C \| g\|_{L^2_{tx}(D_m\times[0,T])}^2.
	\ee
	Plugging (\ref{est for I2}) and (\ref{est for II2}) into (\ref{G1}) leads to
	\begin{align*}
	\frac12 \int_0^T\int_{D_{m}}|\nabla g|^{2} \,dx\,dt + \frac{1}{2} \int_{D_{m}} |g(x, T)|^{2} \,dx
	\leq C A_0^2 T \| b^{(2)} - b^{(1)} \|_{E_{m,T}}^2 + C \| g\|_{L^2_{tx}(D_m\times[0,T])}^2.
	\end{align*}
	Now by Gronwall's inequality, we obtain
	\[
	\| g \|_{E_{m,T}}^2 \leq C A_0^2 T e^{CT} \| b^{(2)} - b^{(1)} \|_{E_{m,T}}^2,
	\]
	which implies
	\be\label{est-diff-tht}
	\| v_\th^{(2)} - v_\th^{(1)} \|_{L^2_{tx}(D_m\times[0,T])}^2 \leq T \| v_\th^{(2)} - v_\th^{(1)} \|_{E_{m,T}}^2 \leq C A_0^2 T^2 e^{CT} \| b^{(2)} - b^{(1)} \|_{E_{m,T}}^2.
	\ee
	
	Finally, substituting (\ref{est-diff-tht}) into (\ref{est-diff-bt2}) leads to
	\[
	\| \t{b}^{(2)} - \t{b}^{(1)} \|_{E_{m,T}} \leq C \big(A_0^2 e^{C A_0^4 T} T^{2/5} + A_0^4 e^{CT}T^2 \big) \| b^{(2)} - b^{(1)} \|_{E_{m,T}}^2.
	\]
	Now by choosing
	\be\label{met}
	T \leq  e^{-C A_0^4},
	\ee
	where $ C $ is some large constant that only depends on $ \a $ and $ m $, we obtain
	\[ \| \t{b}^{(2)} - \t{b}^{(1)} \|_{E_{m,T}} \leq \frac12 \| b^{(2)} - b^{(1)} \|_{E_{m,T}}^2.\]
	Hence, for  any $ M $ and $ T $ that satisfies (\ref{bdd M}) and (\ref{met}), $\mathbb{L}$ is a contraction map. Thanks to the contraction mapping theorem, $ \mathbb{L} $ has a fixed point $ b $ that lies in $\overline{B_{E^\sigma_{m,T}}(0, M)} \bigcap \operatorname{span} \{e_{\rho}, e_{\phi} \}$. In addition, by taking advantage of the fact that $ b=\t{b} $ and (\ref{s-reg-bt}),
	\[ b\in L_t^\infty H_x^{1}\cap L_t^2 H_x^{2}\cap L_{tx}^{\infty}\big(D_m\times[0,T]\big).  \]

	Step 5. Existence of a strong solution $ v $ such that
	\[v\in E^\sigma_{m,T} \cap H_t^1 L_x^2\cap L_t^2 H_x^{2} \cap L^\infty_{tx} \big( D_m\times[0,T] \big), \quad  (\na\times v)_{\th} \in L_t^\infty L_x^6\big( D_m\times[0,T] \big).\]
	
	Based on the fixed point $ b $ defined in the previous step, we define $ v =  b + v_\th e_\th $, where $ v_\th $ is the function constructed in Step 1 based on $ b $. We will first show $v_\theta\in L_t^2 H_x^{2}\big( D_m\times[0,T] \big)$. Recall the equation for $ \Gamma $ (\ref{Gamma-Dm-exist}):
	\[
	\left\{\begin{array}{l}
	\Delta \Gamma - b\cdot \nabla \Gamma - \frac{2}{\rho}\p_\rho \Gamma - \frac{2 \cot\phi}{\rho^2}\p_\phi \Gamma-\partial_{t} \Gamma=0, \quad \text { in } \quad D_{m} \times(0, T]; \\
	\partial_{n} \Gamma=0, \quad \text { on } \quad \p D_{m} \times(0, T]; \\
	\Gamma(x, 0)=\Gamma_{0}(x), \quad x \in D_{m}.
	\end{array}\right.
	\]
	Now the function $ b $ is in $ L^{\infty}(D_m\times[0,T]) $, so it follows from the standard theory that $ \Gamma $ is $ L^{2}_t H^{2}_{x} $ in $ D_{\text{int}}\times[0,T] $, where $ D_{\text{int}} $ is any interior domain of $ D_m $, i.e. $ \ol{D_{\text{int}}} \subset D_m $. Moreover, by the reflection argument as that in Step 4.1, we can show $ \Gamma $ is $ L^{2}_t H^{2}_{x} $ on the whole region $ D_m\times[0,T] $. As a result, $ v\in E^\sigma_{m,T} \cap L_t^2 H_x^{2} \cap L^\infty_{tx} \big( D_m\times[0,T] \big)$.
	
	Define $\o=\nabla\times v  $ and write $ \o=\o_\rho e_\rho + \o_\phi e_\phi + \o_\th e_\th $. Then $ \o\in L_t^2 H_x^{1}(D_m\times[0,T]) $. Let $\wt{\o}_\th $ be given by (\ref{ome-th eq}). We remark that although $ \t{b} $ is constructed from $ \wt{\o}_\th $ according to the Biot-Savart law $ \Delta \t{b} = -\na\times \wt{\o}_\th $ (also see (\ref{v-rho-t}) and (\ref{v-phi-t})), it is not obvious that $ \na\times \t{b} = \wt{\o}_\th $. As a result, although $ b=\t{b} $, it is not readily seen that $ \na \times b = \wt{\o}_\th $. Next, we will carry out a detailed argument to show that $ \o_\th $ indeed coincides with $ \wt{\o}_\th $ so that $ \o_\th $ also satisfies (\ref{ome-th eq}).
	Firstly, since $ \o=\nabla\times v $, then it follows from (\ref{vor f-sph}) that $ \o_\th e_\th = \nabla\times b $, where $ b = v_\rho e_\rho + v_\phi e_\phi $. Thus,
	\be\label{beot}
	\Delta b = -\nabla\times (\o_\th e_\th).
	\ee
	On the other hand, since $ b $ is divergence free, we can use formula (\ref{lapla of vec-div free-sph}) to find
	\be\label{laplace b}
	\Delta b = \bigg( \Delta +\frac{2}{\rho}\,\p_\rho + \frac{2}{\rho^2} \bigg) v_\rho e_\rho + \bigg[\bigg(\Delta-\frac{1}{\rho^2\sin^2\phi}\bigg)v_{\phi}+\frac{2}{\rho^2}\p_{\phi}v_{\rho}\bigg] e_{\phi}.
	\ee	
	Recall that $ b $ is the fixed point of the mapping $ \mathbb{L} $, $ v_\rho $ and $ v_\phi $ are given by (\ref{v-rho-t}) and (\ref{v-phi-t}) respectively. As a result,
	\be\label{v-rho-v-phi-ott}\left\{\begin{split}
	&\bigg(\Delta + \frac{2}{\rho}\,\p_\rho + \frac{2}{\rho^2} \bigg) v_\rho = -\frac{1}{\rho\sin\phi}\,\p_{\phi}(\sin\phi\,\wt{\o}_\th), \\
	&\Delta v_\phi  = - \frac{2}{\rho}\p_\rho v_\phi - \frac{1-\cot^2\phi}{\rho^2} v_\phi   +   \frac{1}{\rho^3}\,\p_\rho(\rho^3 \wt{\o}_\th).
	\end{split}\right.\ee
	Meanwhile, it follows from (\ref{vor f-sph}) that $	\o_{\th}=\frac{1}{\rho}\,\p_{\rho}(\rho v_{\phi}) - \frac{1}{\rho}\,\p_{\phi}v_{\rho}$, which implies
	\be\label{ppvr}
	\p_{\phi}v_{\rho} = \p_{\rho}(\rho v_{\phi}) -\rho \o_\th.
	\ee	
	Putting (\ref{v-rho-v-phi-ott}) and (\ref{ppvr}) into (\ref{laplace b}) yields
	\[ \Delta b = -\frac{1}{\rho\sin\phi}\,\p_{\phi}(\sin\phi\,\wt{\o}_\th) e_\rho + \bigg( \frac{1}{\rho^3}\,\p_\rho(\rho^3 \wt{\o}_\th) - \frac{2}{\rho} \o_\th \bigg) e_\phi. \]
	Applying formula (\ref{vor f-sph}) again (replacing $ v $ by $ \wt{\o}_\th e_\th $), we find
	\[ \nabla\times (\wt{\o}_\th e_\th ) = \frac{1}{\rho\sin\phi}\,\p_{\phi}(\sin\phi\,\wt{\o}_\th) e_\rho - \frac{1}{\rho}\p_\rho(\rho \wt{\o}_\th) e_\phi. \]
	Combining the above two relations, we know
	\[ \Delta b = - \nabla\times (\wt{\o}_\th e_\th) + \frac{2}{\rho} (\wt{\o}_\th - \o_\th) e_\phi. \]
	Since we have already derived in (\ref{beot})  that $ \Delta b = -\nabla\times (\o_\th e_\th) $, the above equation implies
	\be\label{curl uet} \nabla \times (u e_\th) - \frac{2}{\rho} u e_\phi = 0, \ee
	where $ u:= \wt{\o}_\th - \o_\th $. By computing $ \nabla \times (u e_\th) $ based on formula (\ref{vor f-sph}) (replacing $ v $ by $ u e_\th $), it follows from (\ref{curl uet}) that
	\[
	\frac{1}{\rho\sin\phi}\,\p_{\phi}(\sin\phi\,u) e_\rho - \frac{1}{\rho^3}\p_\rho(\rho^3 u) e_\phi = 0.
	\]
	So $ \p_{\phi}(\sin\phi\,u) = \p_\rho(\rho^3 u) = 0 $. Define $\t{u} = \rho^3\sin\phi\, u$. Then
	\be\label{utgv}
	\p_\phi \t{u} = \p_\rho \t{u} =0 \quad \text{in} \quad D_m\times (0,T].
	\ee
	On the boundary $ \p D_m $, $ \wt{\o}_\th = 0 $ by the construction (\ref{ome-th eq}). Meanwhile, since $\o_{\th}=\frac{1}{\rho}\,\p_{\rho}(\rho v_{\phi}) - \frac{1}{\rho}\,\p_{\phi}v_{\rho}$, 	it follows from the constructions of $ v_\rho $ and $ v_\phi $ in (\ref{v-rho-t}) and (\ref{v-phi-t}) that $ \o_\th = 0 $ on $ \p D_m $. Hence,
	\be\label{utbv}
	\t{u} = 0 \quad \text{on} \quad \p D_m\times (0,T].
	\ee	
	Since $ \t{u}\in L_t^2 H_x^{1}(D_m\times [0,T]) $, we deduce from (\ref{utgv}) and (\ref{utbv}) that $ \t{u} = 0 $ in $ D_m $ for a.e. $ t\in (0,T] $. This implies that $ \wt{\o}_\th = \o_\th $ in $ D_m $ for a.e. $ t\in (0,T] $. Now the interior regularity of $ \wt{\o}_\th $ and $ \o_\th $ indicates that $\wt{\o}_\th = \o_\th $ in $ D_m \times (0,T]$. For the initial data, it again follows from the constructions of $ \wt{\o}_\th $, $ v_\rho $ and $ v_\phi $ that $ \wt{\o}_\th(\cdot, 0) = \o_{0,\th}(\cdot) = \o_\th (\cdot, 0)  $. Thus,
	\[ \wt{\o}_\th = \o_\th \quad \text{in} \quad D_m \times [0,T]. \]
	In particular, $ \o_\th $ also satisfies (\ref{ome-th eq}):
	\be\label{omega-th-eq}
	\left\{\begin{array}{l}
		\big(\Delta-\frac{1}{\rho^2\sin^2\phi}\big)\o_{\th} - b\cdot\nabla \o_{\th} + \frac{1}{\rho}(v_{\rho} + \cot\phi\,v_{\phi})\o_{\th} - \p_{t}\o_{\th} \\
		\hspace{2in}= \frac{1}{\rho^2}\p_{\phi}(v_{\th}^2)-\frac{\cot\phi}{\rho}\p_{\rho}(v_{\th}^2), \quad \text {in} \quad  D_{m} \times(0, T]; \\
		\o_{\th} = 0, \quad \text {on} \quad \partial D_{m} \times(0, T]; \\
		\o_{\th}(x, 0) = \omega_{0,\th}(x), \quad  \text{in} \quad  D_{m}.
	\end{array}\right.
	\ee
	Meanwhile, it follows from (\ref{olqe2}) that $ \o_\th\in L_t^\infty L_x^6(D_m\times[0,T]) $ and
	\be\label{oth-bdd-inf6}
	\|\o_\th\|_{L_t^\infty L_x^6(D_m\times[0,T])} \leq C e^{C A_0^4 T}A_0.
	\ee
	
	Finally, we will take advantage of (\ref{omega-th-eq}) to find a pressure term $ P $ such that $ (v,P) $ satisfies (\ref{asns-sph}) and the NHL boundary condition (\ref{NHL slip bdry for Dm-sph}) pointwisely so that $ (v,P) $ is a strong solution. First, we recall a vector calculus identity (see equation (2.45) on page 429 in \cite{Zha22}) in the cylindrical coordinates:	
	\be\label{vec-id-cyl}\begin{split}
	&\quad \nabla\times \bigg( \Delta b - (b\cdot\nabla)b + \frac{v_\th^2}{r}e_r - \p_t b \bigg) \\
	&= \bigg[ \bigg(\Delta - \frac{1}{r^2} \bigg)\o_\th - b\cdot \nabla \o_\th + 2\frac{v_\th}{r}\p_{x_3}v_\th + \frac{v_r}{r} \o_\th - \p_t \o_\th \bigg] e_\th.
	\end{split}\ee
	Next, we will convert this identity in the form of spherical coordinates. Noticing
	\[ r=\rho\sin\phi, \quad e_{r}= \sin\phi\, e_\rho + \cos\phi\, e_\phi, \quad v_{r} = \sin\phi\, v_\rho + \cos\phi\, v_\phi\]
	and
	\[ 2v_\th \p_{x_3}v_\th = \p_{x_3}(v_\th^2) = \Big(\cos\phi\, \p_\rho - \frac{\sin\phi}{\rho}\p_\phi\Big)(v_\th^2), \]
	so the identity (\ref{vec-id-cyl}) can be equivalently written as
	\be\label{vec-id-sph}\begin{split}
		&\quad \nabla\times \bigg( \Delta b - (b\cdot\nabla)b + \frac{1}{\rho}v_\th^2\, e_\rho +  \frac{\cot\phi}{\rho}v_\th^2\, e_\phi - \p_t b \bigg) \\
		&= \bigg[ \bigg(\Delta - \frac{1}{\rho^2\sin^2\phi} \bigg)\o_\th - b\cdot \nabla \o_\th + \frac{\cot\phi}{\rho}\p_{\rho}(v_\th^2) - \frac{1}{\rho^2}\p_\phi(v_\th^2) + \frac{v_\rho + \cot\phi\, v_\phi}{\rho} \o_\th - \p_t \o_\th \bigg] e_\th.
	\end{split}\ee
	Define
	\[ B = \Delta b - (b\cdot\nabla)b + \frac{1}{\rho}v_\th^2\, e_\rho +  \frac{\cot\phi}{\rho}v_\th^2\, e_\phi - \p_t b. \]
	Then it follows from (\ref{vec-id-sph}), (\ref{omega-th-eq}) and the interior regularity of $ v $ and $ \o_\th $ that
	\be\label{B-curl-free}
	\nabla\times B = 0,  \quad \text{pointwise in}\quad  D_m\times (0,T].
	\ee
	By direct computation, $ B $ can be written as $ B = B_\rho e_\rho + B_\phi e_\phi $, where
	\be\label{B comp}\left\{\begin{split}
	B_\rho &= \bigg(\Delta + \frac{2}{\rho}\,\p_\rho + \frac{2}{\rho^2} \bigg)v_{\rho}-b\cdot\nabla v_{\rho}+\frac{1}{\rho}(v_{\phi}^2+v_{\th}^2) - \p_{t}v_{\rho}, \\
	B_\phi &= \bigg(\Delta-\frac{1}{\rho^2\sin^2\phi}\bigg)v_{\phi}-b\cdot\nabla v_{\phi}+\frac{2}{\rho^2}\p_{\phi}v_{\rho} -\frac{1}{\rho}v_{\rho}v_{\phi}+\frac{\cot\phi}{\rho}v_{\th}^2 - \p_{t}v_{\phi}.
	\end{split}\right.\ee
	Next, we discuss the regularity of $ B $. Firstly, since $ v\in L_t^2 H_x^{2} \cap L^\infty_{tx} \big( D_m\times[0,T] \big)$ and $ \o_\th=\wt{\o}_\th\in L_t^2 H_x^{1} \big( D_m\times[0,T] \big) $, it then follows from (\ref{vth}) and (\ref{ome-th eq}) that $ \p_t v_\th\in L_{tx}^2 \big( D_m\times[0,T] \big) $ and $ \p_t \o_\th \in L_t^2 H_{x}^{-1} \big( D_m\times[0,T] \big) $. Now we take advantage of (\ref{v-rho-t}) and (\ref{v-phi-t}) to find that both $ \p_t v_\rho $ and $ \p_t v_\phi $ belong to $ L_{tx}^2 \big( D_m\times[0,T] \big) $. As a consequence, $ v\in H_t^1 L_x^2\big(D_m\times[0,T] \big) $ and $ B\in  L_{tx}^2 \big( D_m\times[0,T] \big)$.
	
	Based on formula (\ref{vor f-sph}) and equation (\ref{B-curl-free}), we have
	\[ \p_\rho (\rho B_\phi) - \p_\phi B_\rho = \rho (\nabla\times B)_{\th} = 0 \quad \text{pointwise in}\quad  D_m\times (0,T].  \]
	Since the domain $ D_m $ can be regarded as a simply connected 2D domain $ D_m' $, defined in (\ref{app domain-2D-sph}), on the $ \rho $-$ \phi $ plane, by viewing both $ \rho B_\phi$ and  $ B_\rho $ as functions in $ \rho $ and $ \phi $ in the domain $ D_m' $, we can apply Green's theorem to find a scalar function $ P\in L_t^2 H_x^{1} \big( D_m\times[0,T]\big)$ such that
	\[ \p_\phi P = \rho B_\phi, \quad \p_\rho P = B_\rho, \quad \text{pointwise in}\quad  D_m'\times (0,T]. \]
	This implies that
	\be\label{grad-press eq}
	B_\rho = \p_\rho P,\quad B_\phi = \frac{1}{\rho}\p_\phi P, \quad \text{pointwise in}\quad  D_m\times (0,T].
	\ee
	Meanwhile, without loss of generality, we can assume the average of $ P $ in the space variable on $ D_m $ is 0 for any fixed time $ t $, that is $ \int_{D_m} P(x,t)\,dx=0 $ for any $ t $. Then it follows from Poincar\'e inequality that $ P\in L_t^2 H_x^{1} \big( D_m\times[0,T]\big) $. Substituting (\ref{grad-press eq}) into (\ref{B comp}) and combining with equation (\ref{vth}) for $ v_\th $, we conclude that $ (v,P) $ satisfies the NS system (\ref{asns-sph}) in $ L^2_{tx} $ sense on the space-time domain $ D_m\times (0,T] $. In addition, from the construction (\ref{vth}) for $ v_\th $, and (\ref{v-rho-t}) and (\ref{v-phi-t}) for $ v_\rho $ and $ v_\phi $, the initial condition and the NHL boundary condition (\ref{NHL-v}) are also satisfied. Hence, $ (v,P) $ is a strong solution such that
	\[v\in E^\sigma_{m,T} \cap H_t^1 L_x^2\cap L_t^2 H_x^{2} \cap L^\infty_{tx} \big( D_m\times[0,T] \big), \quad  P\in L_t^2 H_x^1\big( D_m\times[0,T] \big).\]
	

	Step 6. Uniqueness of the strong solution and preservation of the even-odd-odd symmetry.
	
	Suppose that $ (\hat{v}, \hat{P}) $ is another strong solution of (\ref{asns-sph}) with the initial data $ v_0 $ and the NHL boundary condition (\ref{NHL-v}).
	Define
	\[ \hat{b} = \hat{v}_{\rho} e_\rho + \hat{v}_{\phi} e_\phi.\]
	Then $ \hat{b}\in E_{m,T}^{\sigma}\bigcap \text{span}\{e_\rho,e_\phi\} $ and $ \hat{b} $ is also a fixed point of the map $ \mathbb{L}$ defined in Step 3.  As a result,
	\be\label{L-2fp} \mathbb{L}(b) - \mathbb{L}(\hat{b}) = b-\hat{b}. \ee
	On the other hand, due to the choice (\ref{met}) of the time $ T $ in Step 4, the map $ \mathbb{L} $ is contractive so that
	\be\label{L-c2fp} \|\mathbb{L}(b) - \mathbb{L}(\hat{b})\|_{E_{m,T}} \leq \frac12 \| b-\hat{b} \|_{E_{m,T}}. \ee
	The combination of (\ref{L-2fp}) and (\ref{L-c2fp}) leads to $ b=\hat{b} $ in $ D_m\times[0,T] $. This further implies that $ \hat{v}_\th = v_\th $ since both of them satisfy the equation (\ref{vth}) whose solution in the energy space $ E_{m,T} $ is unique. Hence, $ \hat{v}=v $ in $ D_m\times[0,T] $ and the uniqueness is verified.
	
	Now we assume the initial data $ v_0 $ enjoys the even-odd-odd symmetry as in Definition \ref{Def, eoo sym}. Then we will prove the unique strong solution $ v $ as constructed above also has this property. Firstly, by the characterization (\ref{sym-sph}), we know
	\[
	v_{0,\rho}(\rho, \phi) = v_{0,\rho}(\rho,\pi-\phi),  \quad v_{0, \phi}(\rho,\phi) = - v_{0,\phi}(\rho,\pi-\phi), \quad v_{0,\theta}(\rho,\phi) = - v_{0,\theta}(\rho,\pi-\phi).
	\]
	Then we define a new vector field $ \hat{v} = \hat{v}_\rho e_\rho + \hat{v}_\phi e_\phi + \hat{v}_\th e_\th $ and another pressure $ \hat{P} $ as
	\be\label{def vhat}
	\left\{\begin{split}
		&\hat{v}_\rho(\rho,\phi,t) = v_\rho(\rho, \pi-\phi,t), \quad \hat{v}_\phi(\rho,\phi,t) = -v_\phi(\rho, \pi-\phi,t), \\
		& \hat{v}_\th(\rho,\phi,t) = -v_\th(\rho, \pi-\phi,t), \quad \hat{P}(\rho,\phi,t) =  P(\rho,\pi-\phi,t).
	\end{split}\right.
	\ee
	According to this definition, one can directly check that
	\begin{enumerate}[(1)]
		\item The initial value of $ \hat{v} $ matches $ v_0 $;
		
		\item $ (\hat{v},\hat{P}) $ satisfies the equations (\ref{asns-sph}).
		
		\item $ \hat{v} $ satisfies the NHL boundary condition (\ref{NHL slip bdry for Dm-sph}).
	\end{enumerate}
	So $ \hat{v} $ is also a strong solution, which implies $ \hat{v}=v $ on $ D_m\times[0,T] $ due to the uniqueness of the strong solution that we just established. Based on the definition (\ref{def vhat}), we deduce from the fact $ \hat{v}=v $ that $ v $ has the even-odd-odd symmetry on $ D_m\times[0,T] $.
	
\end{proof}

Next, we aim to extend the local solution in Proposition \ref{Prop, local soln in ad} with a small lifespan $ T $ to be a solution with arbitrarily large lifespan. In fact, according to the proof in Step 4, the existence time $ T $ in (\ref{met}) only depends on $ \a$, $m$ and $ A_0 $. Noticing $ A_0 $ in (\ref{fixed const}) is determined by $ \| v_\th(\cdot, 0) \|_{L^\infty_x(D_m)} $ and $ \| \o_{\th}(\cdot, 0) \|_{L^6_x(D_m)} $, and we have uniform (in time) bounds (\ref{est for vth-Dm}) and (\ref{oth-bdd-inf6}) on $ \| v_\th(\cdot, t) \|_{L^\infty_x(D_m)} $ and $ \| \o_{\th}(\cdot, t) \|_{L^6_x(D_m)} $. As a result, the solution $ v $ constructed in Step 4 on a small time interval $ [0,T] $ can be extended to arbitrary finite time. Thus, we obtain the following result.

\begin{corollary}\label{Cor, globle soln in ad}
	Let $ \a $, $ m $ and $ v_0 $ be the same as in Proposition \ref{Prop, local soln in ad}. Then for any time $ T>0 $, the problem (\ref{asns-sph}) on $ D_m\times[0,T] $ with the initial data $ v_0 $ and the NHL boundary condition (\ref{NHL slip bdry for Dm-sph}) has a strong solution $ (v,P) $ such that
	\[v\in E_{m,T}^{\sigma}\cap H_t^1 L_x^2 \cap L_t^2 H_x^2\cap L_{tx}^{\infty}\big(D_m\times[0,T]\big), \quad P\in L_t^2 H_x^1(D_m\times[0,T]). \]
	Moreover, if $ (\hat{v}, \hat{P}) $ is another strong solution, then $ \hat{v}$ coincides with $ v $ on $ D_m\times[0,T] $. As a result, if $ v_0 $ belongs to $ E^{\sigma,s}_{m,T} $, then so does $ v $.
\end{corollary}

\begin{remark}\label{Re, not uniform bdd}
Although a bounded strong solution is obtained in the above corollary for any finite time $ T $ and any fixed $ m $, the $L^\infty_{tx}$ bound on the velocity $v$ is neither uniform in $ T $ nor uniform in $m$.  In the next section, after introducing some new quantities involving the vorticity \big(see (\ref{K-F-O})\big), we will prove that the $L^\infty_{tx}$ norm of $v$ on $ D_m\times[0,T] $ is uniformly bounded in $ T $ and this uniform bound only depends on $ m $ through $ \|v_0\|_{C^2(D_m)} $, as long as some mild restrictions on the angle $ \a $ and the size of $ \Gamma_0 $ are imposed.
\end{remark}

\section{ Uniform bounds for $ \|v\|_{L_{tx}^\infty} $ on $ D_m\times[0,T] $}
\label{Sec, inf-ub}

In this section, for any fixed $ m\geq 2 $ and $ T>0 $, we consider the initial data $ v_0$ which lies in the admissible class $ \mathscr{A}_{m} $ with the even-odd-odd symmetry (See Definitions  \ref{Def, eoo sym} and \ref{Def, admissible sets}). For such initial data, we denote by $ v $ the solution in Corollary \ref{Cor, globle soln in ad} so that $ v\in E^{\sigma,s}_{m,T} \cap H_t^1 L_x^2\cap L_t^2 H_x^{2} \cap L_{tx}^\infty(D_m\times[0,T]) $.
Moreover, by restricting the range of $ \a $ within $ \big(0, \frac{\pi}{6}\big] $ and by requiring  $ \|\Gamma(\cdot,0)\|_{L^\infty(D_m)}\leq \frac{1}{95} $, we will deduce a uniform bound, which is independent of $ T $ and dependent on $ m $ only through $ \|v_0\|_{C^{2}(\ol{D_m})} $, for $ \|v\|_{L^\infty_{tx}(D_m\times[0,T])} $.  The plan of this section is as follows:
\begin{itemize}
\item Step 1: We will derive an energy inequality about $ v $ in Section \ref{Subsec, energy-ineq}. This energy inequality provides a uniform bound on $ \|v\|_{E_{m,T}} $.
	
\item Step 2: In Sections \ref{Subsec, BS law}--\ref{Subsec, est on v_phi/rho}, we will take advantage of the Biot-Savart law and the condition $ \a\in\big(0,\frac{\pi}{6}\big] $ to control the $ L^2(D_m) $ norms of $\na (v_\rho/\rho)(\cdot, t)$ and $\na (v_\phi/\rho)(\cdot, t) $ by $ \|\O(\cdot, t)\|_{L^2(D_m)} $, and control the $ L^2(D_m) $ norms of $ \frac{1}{\rho} \na (v_\rho/\rho)(\cdot, t) $ and $ \frac{1}{\rho} \na (v_\phi/\rho)(\cdot, t) $ by $ \|\nabla \O(\cdot, t)\|_{L^2(D_m)} $.

\item Step 3: Thanks to the smallness condition $ \| \Gamma(\cdot,0)\|_{L^\infty(D_m)}\leq \frac{1}{95} $, the estimates in Step 1 will be used in Section \ref{Subsec, KFO-ub} to obtain an upper bound, which is uniform in $ m $ and $ T $, on $ \| (K,F,\O) \|_{L_t^\infty L_x^2(D_m\times[0,T])} $.

\item Step 4: According to the uniform bound on  $ \| (K,F,\O) \|_{L_t^\infty L_x^2(D_m\times[0,T])} $, we will derive in Section \ref{Subsec, LtwLx6} a uniform bound on $\| v / \rho \|_{L^\infty_t L^6_x(D_m\times[0,T])}$.

\item Step 5: Finally in Section \ref{Subsec, v-ub}, we will bound $ \| v \|_{L_{tx}^{\infty}(D_m\times[0,T])} $ in terms of $ \| v_0 \|_{C^2(\ol{D_m})} $, $ \|v\|_{E_{m,T}} $, $ \| (K,F,\O) \|_{L_t^\infty L_x^2(D_m\times[0,T])} $ and $\| v / \rho \|_{L^\infty_t L^6_x(D_m\times[0,T])}$. Due to the uniform estimates in Steps 1, 3 and 4, the bound on $ \| v \|_{L_{tx}^{\infty}(D_m\times[0,T])} $ will also be uniform in $ m $ and $ T $.
\end{itemize}

\subsection{An energy inequality}
\label{Subsec, energy-ineq}
\quad

In this section, we present a result on bounding the $ L^2 $ norm of $ \nabla v $ by the $ L^2 $ norm of its vorticity $ \nabla\times v $. This result is well-known for incompressible vector fields $ v $ with zero boundary value (see e.g. Lemma 2 in \cite{NP}), however, it may not be true if the boundary value is nonzero. For example, if $ v = \frac{1}{\rho\sin\phi}\, e_\th $, then $ \nabla \times v = 0 $ while $ \nabla v \neq 0  $. But we will show in Lemma \ref{Lemma, m-korn-ineq} that such an estimate still holds in $ D_m $ if the vector field satisfies the NHL boundary condition and possesses the even-odd-odd symmetry as defined in Definition \ref{Def, eoo sym}.

\begin{lemma}\label{Lemma, m-korn-ineq}
	Let the region $D_{m}$ be as defined in (\ref{app domain-sph}) with $m \geq 2$ and the angle $\alpha \in\big(0, \frac{\pi}{6}\big]$. Let $ u \in H^{2}(D_m)$ be an incompressible vector field. Assume further that $ u $ satisfies the NHL boundary condition  (\ref{NHL slip bdry for Dm-sph}) and possesses the even-odd-odd symmetry. Then
	\be\label{curl-bdd-grad}
	\|\na u\|_{L^2(D_m)}\leq \sqrt{3}\|\na\times u\|_{L^2(D_m)}.
	\ee
\end{lemma}
\begin{proof}
	Firstly, by similar computation as that in Section \ref{Subsec, weak-form-soln}, we know
	\[ \int_{D_m} u\Delta u\,dx = -\int_{D_m} |\nabla \times u|^2 \,dx. \]
	On the other hand, it directly follows from integration by parts that
	\[ \int_{D_m} u\Delta u\,dx = \int_{\p D_m} u \frac{\p u}{\p n}\,dS - \int_{D_m} |\nabla u|^2\,dx.  \]
	As a result,
	\be\label{ZJEST00}
	\int_{D_m} |\na u|^2 \,dx = \int_{D_m} |\na\times u|^2 \,dx + \underbrace{\f{1}{2}\int_{\p D_m} \f{\p|u|^2}{\p n}\,dS}_{T_1}.
	\ee
	
	Now we give a detailed computation of $T_1$ on $\p^R D_m$ and $\p^A D_m$ separately. For the convenience of notation, we denote $\rho_0 = \frac{1}{m}$. Noticing that the normal direction on $ \p^R D_m $ is parallel to the $ \phi $ direction, then we can take advantage of the boundary conditions in \eqref{bdry-phi const} to see that $ \frac{\p (u_\rho^2)}{\p n} = \frac{\p (u_\phi^2)}{\p n} = 0 $ on $ \p^R D_m $. Therefore, 	
	\[
	\begin{split}
		&\quad\, \f{1}{2}\int_{\p^R D_m}\f{\p |u|^2}{\p n} \,dS\\
		&=-\pi\int_{\rho_0}^1 \f{1}{\rho} \,\p_\phi (u_\th^2)\Big|_{\phi=\f{\pi}{2}-\alpha} \, \rho\sin\left(\f{\pi}{2}-\alpha\right) d\rho + \pi\int_{\rho_0}^1 \f{1}{\rho} \,\p_\phi (u_\th^2) \Big|_{\phi=\f{\pi}{2}+\alpha}\, \rho\sin\left(\f{\pi}{2}+\alpha\right) d\rho\\
		&=2\pi\int_{\rho_0}^1 u_\th^2\Big|_{\phi=\f{\pi}{2}-\alpha}\cos\left(\f{\pi}{2}-\alpha\right)d\rho-2\pi\int_{\rho_0}^1 u_\th^2\Big|_{\phi=\f{\pi}{2}+\alpha}\cos\left(\f{\pi}{2}+\alpha\right)d\rho.
	\end{split}
	\]
	Now using the fundamental theorem of Calculus, we find
	\be\label{ZJb1}\begin{split}
		\f{1}{2}\int_{\p^R D_m}\f{\p |u|^2}{\p n} \,dS & = -2\pi \int_{\rho_0}^1 \int_{\frac{\pi}{2}-\a}^{\frac{\pi}{2}+\a} \p_\phi \big[ u_\th^2(\rho,\phi)\cos \phi \big]\,d\phi\,d\rho\\
		& = -2\int_{D_m} \frac{1}{\rho^2}\,u_\th \,\p_\phi u_\th \cot\phi \,dx + \int_{D_m} \frac{u_\th^2}{\rho^2} \,dx.
	\end{split}\ee
	Similarly, by the boundary condition in \eqref{bdry-rho const}, one deduces
	\[
	\begin{split}
		\f{1}{2}\int_{\p^A D_m}\f{\p |u|^2}{\p n} \,dS & =\pi\int_{\f{\pi}{2}-\alpha}^{\f{\pi}{2}+\alpha} \big[\rho^2\p_\rho (u_\th^2 + u_\phi^2)\big]\Big|^{\rho=1}_{\rho=\rho_0}\sin\phi \,d\phi\\
		&=-2\pi\int_{\f{\pi}{2}-\alpha}^{\f{\pi}{2}+\alpha} \big[\rho\big(u_\th^2+u_\phi^2\big) \big]\Big|^{\rho=1}_{\rho=\rho_0}\sin\phi \,d\phi.
	\end{split}
	\]
	Then applying the fundamental theorem of Calculus,
	\be\label{ZJb2}\begin{split}
		\f{1}{2}\int_{\p^A D_m}\f{\p |u|^2}{\p n} \,dS & = -2\pi 	\int_{\f{\pi}{2}-\alpha}^{\f{\pi}{2}+\alpha} \int_{\rho_0}^1 \p_\rho\big[ \rho(u_\th^2+u_\phi^2) \big] \sin\phi \,d\rho\,d\phi \\
		& = - \int_{D_m} \frac{1}{\rho^2}\, (u_\th^2+u_\phi^2) \,dx - 2\int_{D_m} \frac{1}{\rho}\, (u_\th \p_\rho u_\th + u_\phi \p_\rho u_\phi) \,dx.
	\end{split}\ee
	Thus, by adding (\ref{ZJb1}) and (\ref{ZJb2}),
	\[ \begin{split}
		T_1 & = - \int_{D_m} \frac{u_\phi^2}{\rho^2} \,dx - 2\int_{D_m} \frac{1}{\rho^2}\,u_\th \,\p_\phi u_\th \cot\phi \,dx - 2\int_{D_m} \frac{1}{\rho}\, (u_\th \p_\rho u_\th + u_\phi \p_\rho u_\phi) \,dx.
	\end{split}\]

	Since $ \a\leq \frac{\pi}{6} $, this implies $ 0\leq \cot\phi \leq \frac{1}{\sqrt{3}} $ and
	\[ T_1 \leq \bigg( - \int_{D_m} \frac{u_\phi^2}{\rho^2} \,dx + 2\int_{D_m} \frac{1}{\rho}\, \big| u_\phi \p_\rho u_\phi \big|\,dx\bigg) + 2\int_{D_m} \frac{|u_\th|}{\rho}\, \bigg( \frac{1}{\sqrt{3}}\,\Big|\frac{1}{\rho}\p_\phi u_\th \Big| + |\p_\rho u_\th|  \bigg)\,dx. \]
	By Cauchy-Schwarz inequality, we know
	\[ T_1 \leq \bigg( \frac23\int_{D_m} | \p_\rho u_\phi |^2\,dx + \frac12 \int_{D_m} \frac{u_\phi^2}{\rho^2} \,dx\bigg) + \frac23\int_{D_m} \bigg( \frac13 \Big| \frac{1}{\rho}\,\p_\phi u_\th \Big|^2 + |\p_\rho u_\th|^2 \bigg) \,dx
	+ 3\int_{D_m} \frac{u_\th^2}{\rho^2}\,dx.
	\]
	Since $ u $ satisfies the even-odd-odd symmetry assumption, both $ u_\phi $ and $ u_\th $ are odd with respect to the plane $ \{\phi=\frac{\pi}{2}\} $. Hence, it follows from the Poincar\'e inequality in Corollary \ref{Cor, P-sine-ave} and the fact $ \a\leq \frac{\pi}{6} $ that
	\be\label{Pforu} \begin{split}
		\int_{D_m} \frac{u_\phi^2}{\rho^2} \,dx  \leq  \frac{2}{19} \int_{D_m} \Big| \frac{1}{\rho}\, \p_\phi u_\phi \Big|^2 \,dx, \\
		\int_{D_m} \frac{u_\th^2}{\rho^2} \,dx  \leq  \frac{2}{19} \int_{D_m} \Big| \frac{1}{\rho}\, \p_\phi u_\th \Big|^2 \,dx. 		
	\end{split}\ee
	As a result,
	\be\label{T_1 est}
	T_1 \leq \frac23 \int_{D_m} |\p_\rho u_\phi|^2 \,dx + \frac{1}{19} \int_{D_m} \Big|\frac{1}{\rho}\,\p_\phi u_\phi \Big|^2\,dx + \frac23 \int_{D_m} \bigg( |\p_\rho u_\th|^2 + \Big| \frac{1}{\rho}\,\p_\phi u_\th \Big|^2\bigg) \,dx.
	\ee
	Next, we claim
		\be\label{gultcgu}
		\int_{D_m} |\nabla u|^2 \,dx \geq \int_{D_m}  \bigg(|\p_\rho u_\phi|^2  +  |\p_\rho u_\th|^2 + \Big| \frac{1}{\rho}\,\p_\phi u_\th \Big|^2\bigg) \,dx + \frac14 \int_{D_m} \Big|\frac{1}{\rho}\,\p_\phi u_\phi \Big|^2\,dx.
		\ee
	
	Assuming this claim for a moment, then it follows from (\ref{T_1 est})  that
	\[ T_1 \leq \frac23 \int_{D_m} |\nabla u|^2\,dx.  \]
	Putting this estimate into (\ref{ZJEST00}) yields the desired conclusion (\ref{curl-bdd-grad}).
	
	Thus, it remains to verify (\ref{gultcgu}) in the above claim. According to formula (\ref{nabla vec-sph}),
	\[
	\nabla u =
	\begin{pmatrix} \p_{\rho}u_{\rho} & \frac{1}{\rho}(\p_{\phi}u_{\rho}-u_{\phi}) & -\frac{1}{\rho}\,u_{\th} \vspace{0.05in}\\
		\p_{\rho}u_{\phi} & \frac{1}{\rho}(\p_{\phi}u_{\phi}+u_{\rho}) & -\frac{\cot \phi}{\rho}\, u_{\th} \vspace{0.05in} \\
		\p_{\rho}u_{\th} & \frac{1}{\rho}\p_{\phi}u_{\th}  & \frac{1}{\rho}(u_{\rho}+\cot\phi\,u_{\phi}) \end{pmatrix}
	\]
	under the basis (\ref{m-basis in sph}), so in order to prove (\ref{gultcgu}), it suffices to justify the following estimate:
	\be\label{phi-deri-bdd}
	\int_{D_m} \bigg( \frac1\rho \p_\phi u_\phi + \frac1\rho u_\rho \bigg)^2 + \bigg( \frac1\rho u_\rho + \frac{\cot\phi}{\rho}u_\phi \bigg)^2 \,dx \geq  \frac14 \int_{D_m} \Big|\frac{1}{\rho}\,\p_\phi u_\phi \Big|^2\,dx.
	\ee
	Using the basic inequality that for any $ A,B,C $ in $ \m{R} $ and for any $ 0<\lam<1 $,
	\[ (A+B)^2 + (B+C)^2\geq \frac12(A-C)^2 \geq \frac12\Big( \lam A^2 - \frac{\lam}{1-\lam} C^2 \Big), \]
	we know
	\[  \bigg( \frac1\rho \p_\phi u_\phi + \frac1\rho u_\rho \bigg)^2 + \bigg( \frac1\rho u_\rho + \frac{\cot\phi}{\rho} u_\phi \bigg)^2 \geq \frac{\lam}{2} \bigg( \frac1\rho \p_\phi u_\phi \bigg)^2 - \frac{\lam}{2(1-\lam)}\bigg( \frac{\cot\phi}{\rho} u_\phi \bigg)^2.   \]
	By choosing $ \lam = \frac23 $ and using the fact that $0\leq \cot\phi \leq \frac{1}{\sqrt{3}}$, we find
	\[
	\bigg( \frac1\rho \p_\phi u_\phi + \frac1\rho u_\rho \bigg)^2 + \bigg( \frac1\rho u_\rho + \frac{\cot\phi}{\rho} u_\phi \bigg)^2 \geq \frac13 \bigg( \frac1\rho \p_\phi u_\phi \bigg)^2 - \frac13\bigg( \frac{1}{\rho} u_\phi \bigg)^2.
	\]
	Integrating both sides on $ D_m $ and taking advantage of (\ref{Pforu}) yields
	\[
	\int_{D_m} \bigg( \frac1\rho \p_\phi u_\phi + \frac1\rho u_\rho \bigg)^2 + \bigg( \frac1\rho u_\rho + \frac{\cot\phi}{\rho}u_\phi \bigg)^2 \,dx \geq \Big(\frac13 - \frac{2}{57} \Big)\int_{D_m} \Big|\frac{1}{\rho}\,\p_\phi u_\phi \Big|^2\,dx,
	\]
	which implies (\ref{phi-deri-bdd}).	
\end{proof}

\begin{remark}\label{Re, curl-grad-D}
	Let $ D $ be the original target region as defined in (\ref{domain-cyl}) or (\ref{domain-sph}). Let $ u\in H^2(D) $ be an incompressible vector field such that $ u $ satisfies the NHL boundary condition (\ref{NHL slip bdry}) and possesses the even-odd-odd symmetry. Then (\ref{curl-bdd-grad}) also holds when $ D_m $ is being replaced with $ D $. That is $\|\na u\|_{L^2(D)}\leq \sqrt{3}\|\na\times u\|_{L^2(D)}$.
	The proof is essentially the same as that for Lemma \ref{Lemma, m-korn-ineq}.
\end{remark}

For the Cauchy problem of (\ref{nse}) involving finite energy solutions $ v $, Leray discovered the classical energy inequality as follows.
\[
\int_{\m{R}^3} |v(x, T)|^2 dx + 2 \int^T_0 \int_{\m{R}^3} |\na v(x, t)|^2 dxdt \le \int_{\m{R}^3} |v(x, 0)|^2 dx.
\]
But under various boundary conditions, the above inequality may need to be modified. For example, under the NHL boundary condition (\ref{NHL slip bdry for Dm}), we obtain an energy inequality with a slightly different form in the following result.

\begin{proposition}\label{Prop, mod l-h energy}
	Let the region $D_{m}$ be as defined in (\ref{app domain-sph}) with $m \geq 2$ and the angle $\alpha \in\big(0, \frac{\pi}{6}\big]$. Let $ v $ be the solution in Corollary \ref{Cor, globle soln in ad} on $ D_m\times[0,T] $. Then
	\be\label{en1s}
	\int_{D_m} |v(x, T)|^2 \,dx + 2 \int^T_0 \int_{D_m} |\na\times v(x, t)|^2 \,dx\,dt = \int_{D_m} |v(x, 0)|^2 \,dx.
	\ee
	In addition,
	\be\label{qenest}
	\int_{D_m} |v(x, T)|^2 \,dx + \frac23 \int^T_0 \int_{D_m} |\na v(x, t)|^2 \,dx\,dt \le \int_{D_m} |v(x, 0)|^2 \,dx.
	\ee
\end{proposition}

\begin{proof}
	The proof of (\ref{en1s}) is essentially the same as that in Section \ref{Subsec, weak-form-soln} by replacing $ D $ with $ D_m $. After (\ref{en1s}) is established, one can combine it with Lemma \ref{Lemma, m-korn-ineq} to justify (\ref{qenest}).
%
\end{proof}

\subsection{Modified Biot-Savart law in spherical coordinates}
\label{Subsec, BS law}
\quad

We first derive the relations between $\frac{v_\rho}{\rho}$, $\frac{v_\phi}{\rho}$ and $\O$ by taking advantage of the Biot-Savart law: $\Delta v = -\nabla \times \o$.  On the one hand, since $\text{div}\, v=0$, it follows from (\ref{lapla of vec-div free-sph}) that
\[\Delta v = \Big( \Delta +\frac{2}{\rho}\,\p_\rho + \frac{2}{\rho^2} \Big)v_\rho e_\rho + \Big[\Big(\Delta-\frac{1}{\rho^2\sin^2\phi}\Big)v_{\phi}+\frac{2}{\rho^2}\p_{\phi}v_{\rho}\Big] e_{\phi} + \Big(\Delta-\frac{1}{\rho^2\sin^2\phi}\Big)v_{\th} e_{\th}. \]
On the other hand,  we know from (\ref{vor f-sph}) that
\be\label{curl v formula} \nabla\times v = \frac{1}{\rho\sin\phi}\,\p_{\phi}(\sin\phi\,v_{\th})\, e_\rho - \frac{1}{\rho}\,\p_{\rho}(\rho v_{\th})\, e_\phi + \bigg( \frac{1}{\rho}\,\p_{\rho}(\rho v_{\phi})-\frac{1}{\rho}\,\p_{\phi}v_{\rho} \bigg)\,e_\th. \ee
Applying the above formula (\ref{curl v formula}) to $\o$ gives
\[\nabla\times \o = \frac{1}{\rho\sin\phi}\,\p_{\phi}(\sin\phi\,\o_{\th})\, e_\rho - \frac{1}{\rho}\,\p_{\rho}(\rho \o_{\th})\, e_\phi + \bigg( \frac{1}{\rho}\,\p_{\rho}(\rho \o_{\phi})-\frac{1}{\rho}\,\p_{\phi}\o_{\rho} \bigg)\,e_\th. \]
Hence, the Biot-Savart law $\Delta v = -\nabla \times \o$ is equivalent to the following form.
\be\label{Biot-Savart-sph}
\begin{cases}
\big(\Delta + \frac{2}{\rho}\,\p_\rho + \frac{2}{\rho^2} \big)v_{\rho} = -\frac{1}{\rho\sin\phi}\,\p_{\phi}(\sin\phi\,\o_{\th}),  \vspace{0.05in}\\
\big(\Delta-\frac{1}{\rho^2\sin^2\phi}\big)v_{\phi}+\frac{2}{\rho^2}\p_{\phi}v_{\rho} = \frac{1}{\rho}\,\p_{\rho}(\rho \o_{\th}),   \vspace{0.05in}\\
\big(\Delta-\frac{1}{\rho^2\sin^2\phi}\big)v_{\th} = -\frac{1}{\rho}\,\p_{\rho}(\rho \o_{\phi}) + \frac{1}{\rho}\,\p_{\phi}\o_{\rho}.
\end{cases}\ee

Recalling from (\ref{vor f-sph}) that $\o_\th=\frac{1}{\rho}\,\p_{\rho}(\rho v_{\phi})-\frac{1}{\rho}\,\p_{\phi}v_{\rho}$,  so
\[\p_{\phi}v_{\rho} = \p_\rho(\rho v_\phi) - \rho \o_\th.\]
Therefore, the second equation in (\ref{Biot-Savart-sph}) can be rewritten as
\[\Big(\Delta-\frac{1}{\rho^2\sin^2\phi}\Big)v_{\phi} + \frac{2}{\rho^2}\p_\rho(\rho v_\phi)  = \frac{1}{\rho}\,\p_{\rho}(\rho \o_{\th}) + \frac{2}{\rho}\,\o_\th,\]
which is equivalently to
\[\Big( \Delta+\frac{2}{\rho}\p_\rho + \frac{1-\cot^2\phi}{\rho^2} \Big)v_\phi=\frac{1}{\rho^3}\,\p_\rho(\rho^3 \o_\th).\]
Combining with the first equation in (\ref{Biot-Savart-sph}) and recalling $\O=\o_\th/(\rho\sin\phi)$, we obtain
\be\label{v-rho, v-phi, O}\begin{cases}
\big(\Delta + \frac{2}{\rho}\,\p_\rho + \frac{2}{\rho^2} \big)v_{\rho} = -\frac{1}{\sin\phi}\,\p_{\phi}(\sin^2\phi\,\O),  \vspace{0.1in}\\
\big( \Delta+\frac{2}{\rho}\p_\rho + \frac{1-\cot^2\phi}{\rho^2} \big)v_\phi = \frac{1}{\rho^3}\,\p_\rho(\rho^4\sin\phi\,\O).
\end{cases}\ee
Consequently, one can get the following relations between $\frac{v_\rho}{\rho}$, $\frac{v_\phi}{\rho}$ and $\O$, which we call the modified Biot-Savart law.
\be\label{energy trans eq}
\begin{cases}
\big(\Delta + \frac{4}{\rho}\,\p_\rho + \frac{6}{\rho^2} \big)\big(\frac{v_{\rho}}{\rho}\big) = -\frac{1}{\rho\sin\phi}\,\p_{\phi}(\sin^2\phi\,\O),  \vspace{0.1in}\\
\big( \Delta+\frac{4}{\rho}\p_\rho + \frac{5-\cot^2\phi}{\rho^2} \big) \big(\frac{v_\phi}{\rho}\big) = \frac{1}{\rho^4}\,\p_\rho(\rho^4\sin\phi\,\O).
\end{cases}\ee

In the rest of this paper,  for simplicity of notation, when dealing with estimates in the domain $D_{m}$
(See Figure \ref{Fig,app domain-sph}), we denote $\rho_1=\frac{1}{m}$, $\rho_2=1$, $\phi_1=\frac{\pi}{2}-\a$ and $\phi_2=\frac{\pi}{2}+\a$.  In addition, the odd symmetry of $v_\th$ with respect to $\{\phi=\f{\pi}{2}\}$ plays an important role in the following estimates.

%

\subsection{ Control of $\Vert \na (v_\rho/\rho)(\cdot, t) \Vert_{L^2}$ and $\big\| \frac{1}{\rho} \na (v_\rho/\rho)(\cdot, t) \big\|_{L^2}$  via $ \O(\cdot, t) $.}
\label{Subsec, est on v_rho/rho}

\quad

Firstly, recalling (\ref{0int-vrho}) in the proof of Proposition \ref{Prop, local soln in ad}, we know for any $ t>0 $,
\begin{equation}\label{mean0NA}
	\int_{\phi_1}^{\phi_2} v_\rho (\rho,\phi, t) \sin \phi \,d\phi=0, \quad \forall\, \rho\in [\rho_1,\rho_2] .
\end{equation}
Next, we will take advantage of (\ref{mean0NA}) to estimate $\Vert \na (v_\rho/\rho)(\cdot, t) \Vert_{L^2(D_m)}$  and $\big\| \frac{1}{\rho} \na (v_\rho/\rho)(\cdot, t) \big\|_{L^2(D_m)}$ via $ \O(\cdot, t) $.

 \begin{lemma}\label{Lemma, e-trans-v_rho}
 	Let the region $D_{m}$ be as defined in (\ref{app domain-sph}) with $m \geq 2$ and the angle $\alpha \in\big(0, \frac{\pi}{6}\big]$. Then for any $ T>0 $ and for a.e. $ t\in[0,T] $,
 	\begin{align}
 		\left\|\nabla \Big(\frac{v_\rho}{\rho}(\cdot, t) \Big) \right\|_{L^2(D_m)} & \leq \sqrt{3} \left\|\Omega(\cdot,t)\right\|_{L^2(D_m)}, \label{es1NA}\\
 		\left\|\frac{1}{\rho}\, \nabla \Big( \frac{v_\rho}{\rho}(\cdot, t) \Big)\right\|_{L^2(D_m)} & \leq \sqrt{44} \left\| \nabla \Omega(\cdot,t)\right\|_{L^2(D_m)}. \label{es2NA}
 	\end{align}
 \end{lemma}
 \begin{proof}
 	Since $ v\in E^{\sigma,s}_{m,T} \cap L_t^2 H_x^{2} \cap L_{tx}^\infty(D_m\times[0,T]) $ and $ \rho $ has the lower bound $ \frac1m $ on $ D_m $, we know $ \O\in L_t^2 H_x^{1}(D_m\times[0,T])$. So there exists a set $ S_T\subset[0,T] $ such that $ [0,T]\setminus S_T $ has measure 0 and for any $ t\in S_T $, $ \O(\cdot,t) $ belongs to $ H^{1}(D_m) $. Fixing any $ t\in S_T $, it suffices to prove (\ref{es1NA}) and (\ref{es2NA}) for such $ t $. For ease of notation, we will drop all the temporal variables in the following argument.
 	
 	We first consider \eqref{es1NA} and denote $f_{1}=\frac{v_\rho}{\rho}$.  Then it follows from (\ref{energy trans eq}) that
 	\be\label{f_1}
 	\Big(\Delta + \frac{4}{\rho}\,\p_\rho + \frac{6}{\rho^2} \Big) f_1 = -\frac{1}{\rho\sin\phi}\,\p_{\phi}(\sin^2\phi\,\O).
 	\ee
Moreover,  we see from Lemma \ref{Lemma, bdry cond} that
 	\begin{equation}\label{bdryNA}
 		\begin{cases}
 			\partial_\phi f_{1}=0, \ \Omega=0 \  &\text{on} \quad \p^{R}D_{m}; \\
 			f_{1}=0,\ \Omega=0 \ & \text{on}\quad \p^{A}D_{m}.
 		\end{cases}
 	\end{equation}
 	In particular, the above relations imply that
 	\be\label{f_1, Neumann}
 	f_1 \p_{n}f_1 = 0 \text{\quad on \quad } \p^{R}D_{m} \cup \p^{A}D_{m}.
 	\ee
 	Applying $f_{1}$ as a test function to (\ref{f_1}),  we deduce
 	\begin{align}\label{eq2NA}
 		\int_{D_m} (\Delta f_{1}) f_{1} \,dx + \int_{D_m}\frac{4}{\rho}\, (\partial_\rho f_{1})  f_{1}\, dx + \int_{D_m}\frac{6}{\rho^2}\,  f_{1}^2 \,dx = -\int_{D_m}\frac{1}{\rho\sin\phi}\partial_\phi(\sin^2\phi\, \Omega) f_{1} \,dx.
 	\end{align}
 	 Now using integration by parts in \eqref{eq2NA} and taking advantage of \eqref{bdryNA} and (\ref{f_1, Neumann}),  we have
 	\begin{align*}
 		\int_{D_m} (\Delta f_{1}) f_{1}dx = \int_{\partial D_m}(\partial_n f_{1}) f_{1} \,dS - \int_{D_m}|\nabla f_{1}|^2 \,dx = - \int_{D_m}|\nabla f_{1}|^2 \,dx,
 	\end{align*}
 	\begin{align*}
 	\int_{D_m}\frac{4}{\rho}\, (\p_\rho f_{1}) f_{1} \,dx =  2\int_{D_m}\frac{1}{\rho}\p_\rho (f_{1}^2) \,dx & =  4\pi \int_{\phi_1}^{\phi_2}\int_{\rho_1}^{\rho_2} \p_\rho(f_{1}^2)\, \rho\sin\phi\, d\rho\, d\phi\\
 	& = -4\pi \int_{\phi_1}^{\phi_2}\int_{\rho_1}^{\rho_2} f_{1}^2 \sin\phi\, d\rho \,d\phi\\
 	& = -2 \int_{D_m}\frac{1}{\rho^2} f_{1}^2 \,dx,
 	\end{align*}
 	and
 	\begin{align*}
 		-\int_{D_m}\frac{1}{\rho\sin\phi} \,\p_\phi(\sin^2\phi\, \O) f_{1} \,dx & = - 2\pi\int_{\rho_1}^{\rho_2}\int_{\phi_1}^{\phi_2} \rho\, \p_\phi(\sin^2\phi \,\O) f_{1} \,d\phi \,d\rho \\
 		& = 2\pi \int_{\rho_1}^{\rho_2}\int_{\phi_1}^{\phi_2}\rho \sin^2\phi \,\O \,\p_\phi f_{1} \,d\phi \,d\rho\\
 		& = \int_{D_m} \sin\phi \,\O \, \frac{\p_\phi f_{1}}{\rho} \,dx.
 	\end{align*}
Putting the above relations into (\ref{eq2NA}) yields
\be\label{f_1 energy eq}
 \int_{D_m}|\nabla f_{1}|^2 \,dx =  4 \int_{D_m}\frac{1}{\rho^2} f_{1}^2 \,dx - \int_{D_m} \sin\phi \,\O \,\frac{\p_\phi f_{1}}{\rho} \,dx. \ee
 	
As a result,  it follows from Cauchy--Schwarz inequality that for any $\e>0$,
 	\begin{equation}\label{es-1NA}
 			\int_{D_m}|\nabla f_{1}|^2 \,dx \leq  4\int_{D_m} \frac{1}{\rho^2} f_{1}^2 \,dx + \epsilon\int_{D_m} \bigg( \frac{ \p_\phi f_{1}}{\rho} \bigg)^2 \,dx + \frac{1}{4\e} \int_{D_m} |\O |^2 \,dx.
 	\end{equation}
 	Note that $v_\rho$ satisfies \eqref{mean0NA},  so
 	\begin{equation*}
 		\int_{\phi_1}^{\phi_2} f_{1} \sin \phi \,d\phi = \frac{1}{\rho}\int_{\phi_1}^{\phi_2} v_\rho \sin \phi \,d\phi = 0.
 	\end{equation*}
 	Then it follows from Corollary \ref{Cor, P-sine-ave} that
 	\begin{align}\label{PoinNA}
 		\int_{\phi_1}^{\phi_2} f_{1}^2 \sin\phi \,d\phi \leq C_{\a,A}\int_{\phi_1}^{\phi_2} (\partial_\phi f_{1})^2 \sin\phi \,d\phi,
 	\end{align}
 	where $C_{\a,A}$ is defined as in (\ref{P-sine-ave-const}).  Hence,
 	\be\label{P-ave-energy-est}\begin{split}
 		\int_{D_m}\frac{1}{\rho^2} f_{1}^2 \,dx &= 2\pi\int_{\rho_1}^{\rho_2}\int_{\phi_1}^{\phi_2}f_{1}^2 \sin\phi \, d\phi \,d\rho\\
 		&\leq 2\pi\, C_{\a,A}\int_{\rho_1}^{\rho_2}\int_{\phi_1}^{\phi_2} (\partial_\phi f_{1})^2 \sin\phi \,d\phi \,d\rho\\
 		&= C_{\a,A} \int_{D_m} \bigg(\frac{\p_\phi f_{1}}{\rho}  \bigg)^2  \,dx.
 	\end{split}\ee
	Putting the above inequality into (\ref{es-1NA}) and noticing $\big| \frac{1}{\rho}\,\p_{\phi} f_1 \big|\leq |\nabla f_1|$, we obtain
 	\[ \Big( 1- 4C_{\a,A} - \e \Big) \int_{D_m}|\nabla f_{1}|^2 \,dx \leq \frac{1}{4\e} \int_{D_m} |\O |^2 \,dx. \] 	
	Since $C_{\a,A}$ is increasing in $\a$ which lies in $\big(0,\frac{\pi}{6}\big]$,  it follows from (\ref{two P-consts}) that $C_{\a,A}\leq C_{\pi/6,A}=\frac{2}{19}$. 	Thus,
 	\[ \Big(\frac{11}{19}-\e\Big) \int_{D_m}|\nabla f_{1}|^2 \,dx \leq \frac{1}{4\e} \int_{D_m} |\O |^2 \,dx. \]
 	Choosing $\e=\frac{11}{38}$ implies that
 	\begin{align*}
 		\int_{D_m} |\nabla f_{1}(x,t)|^2 \,dx\leq 3\int_{D_m} |\O(x,t)|^2 \,dx.
 	\end{align*}
    Thus, (\ref{es1NA}) is justified.
 	
 	Next we estimate $\|\frac{1}{\rho}\nabla f_{1}\|_{L^2(D_m)}$.  Applying $\frac{f_{1}}{\rho^2}$ as a test function to \eqref{f_1}, we deduce
 	\begin{align}\label{eq3NA}
 		\int_{D_m} (\Delta f_{1})\, \frac{f_{1}}{\rho^2} \,dx + 4\int_{D_m} \frac{f_1 \p_\rho f_1 }{\rho^3} \,dx + 6 \int_{D_m} \frac{f_{1}^2}{\rho^4} \,dx = -\int_{D_m}\frac{f_1}{\rho^3\sin\phi}\partial_\phi(\sin^2\phi\, \O) \,dx.
 	\end{align}
 	Since $f_1=0$ on $\p^{A}D_m$ and $\p_{n}f_1=0$ on $\p^{R}D_{m}$, it follows from integration by parts that
 	\begin{align*} \int_{D_m} (\Delta f_{1})\, \frac{f_{1}}{\rho^2} \,dx &= - \int_{D_m} \nabla f_1 \cdot \nabla\Big( \frac{f_1}{\rho^2} \Big)\,dx \\
 	&= -\int_{D_m} \frac{| \nabla f_1 |^2}{\rho^2}\,dx + 2 \int_{D_m} \frac{f_1 \p_\rho f_1 }{\rho^3} \,dx.
 	 \end{align*}
 	 Plugging the above equality into (\ref{eq3NA}) yields
 	 \be\label{eq3bNA}
 	 -\int_{D_m} \frac{| \nabla f_1 |^2}{\rho^2}\,dx + 6\int_{D_m} \frac{f_1 \p_\rho f_1 }{\rho^3} \,dx + 6 \int_{D_m} \frac{f_{1}^2}{\rho^4} \,dx = -\int_{D_m}\frac{f_1}{\rho^3\sin\phi}\partial_\phi(\sin^2\phi\, \O) \,dx.  \ee
 	Using integration by parts and noting $f_1=0$ on $\p^{A}D_{m}$,  we obtain
 	\begin{equation}\label{LH2NA}
 		\begin{split}
 			6\int_{D_m} \frac{f_1 \p_\rho f_1 }{\rho^3} \,dx &= 6\pi \int_{\phi_1}^{\phi_2}\int_{\rho_1}^{\rho_2}\frac{1}{\rho}\, \p_\rho (f_{1}^2) \sin\phi \,d\rho \,d\phi\\
 			&= -6\pi \int_{\phi_1}^{\phi_2}\int_{\rho_1}^{\rho_2} \p_\rho \Big(\frac{1}{\rho}\Big)\, f_{1}^2 \sin\phi \,d\rho \,d\phi\\
 			&= 3\int_{D_m}\frac{f_{1}^2}{\rho^4}  \,dx.
 		\end{split}
 	\end{equation}
 	Applying integration by parts again and recalling $\O=0$ on $\p^{R}D_{m}$, we get
 	\begin{equation}\label{RHNA}
 		\begin{split}
 			-\int_{D_m}\frac{f_1}{\rho^3\sin\phi}\, \p_\phi(\sin^2\phi\, \O) \,dx &= -2\pi\int_{\rho_1}^{\rho_2}\int_{\phi_1}^{\phi_2}\frac{f_{1}}{\rho}\, \p_\phi(\sin^2\phi \,\O) \,d\phi \,d\rho \\
 			&= 2\pi\int_{\rho_1}^{\rho_2}\int_{\phi_1}^{\phi_2}\frac{\partial_\phi f_{1}}{\rho}(\sin^2\phi\, \O)  \,d\phi \,d\rho\\
 			&=\int_{D_m}\frac{\sin\phi}{\rho^3}\, (\p_\phi f_{1})\,\O \,dx.
 		\end{split}
 	\end{equation}
 	Plugging (\ref{LH2NA}) and (\ref{RHNA}) into \eqref{eq3bNA}, we have
 	\begin{align*}
		\int_{D_m} \frac{| \nabla f_1 |^2}{\rho^2}\,dx = 9 \int_{D_m} \frac{f_{1}^2}{\rho^4} \,dx - \int_{D_m}\frac{\sin\phi}{\rho^3}\, (\p_\phi f_{1})\,\O \,dx.
 	\end{align*}
 	
 	It then follows from Cauchy--Schwarz's inequality that for any $\e>0$,
 	\be\label{es-prhoNA}
 		\int_{D_m} \frac{| \nabla f_1 |^2}{\rho^2}\,dx \leq 9 \int_{D_m} \frac{f_{1}^2}{\rho^4} \,dx + \e \int_{D_m} \frac{1}{\rho^2}\bigg( \frac{\p_\phi f_1}{\rho} \bigg)^2 \,dx + \frac{1}{4\e} \int_{D_m} \frac{\O^2}{\rho^2}\,dx.
 	\ee
 	Moreover,  it follows from Corollary \ref{Cor, P-sine-ave} that
 	\begin{align*}
 		 \int_{D_m} \frac{f_{1}^2}{\rho^4} \,dx  &= 2\pi \int_{\rho_1}^{\rho_2}\frac{1}{\rho^2}\int_{\phi_1}^{\phi_2}f_{1}^2 \sin\phi \,d\phi \,d\rho\\
 		&\leq 2\pi\,C_{\a,A} \int_{\rho_1}^{\rho_2}\frac{1}{\rho^2}\int_{\phi_1}^{\phi_2} (\p_\phi f_{1})^2 \sin\phi \,d\phi \,d\rho\\
 		&= C_{\a,A} \int_{D_m} \frac{1}{\rho^2} \bigg( \frac{\p_\phi f_1}{\rho} \bigg)^2 \,dx.
 	\end{align*}
Putting the above inequality into (\ref{es-prhoNA}) and noticing that $\big| \frac{1}{\rho} \,\p_\phi f_1 \big|\leq | \nabla f_1 |$, we attain
\be\label{nabla f_1-ee1}
\big(1-9C_{\a,A}-\e \big) \int_{D_m} \frac{| \nabla f_1 |^2}{\rho^2}\,dx \leq  \frac{1}{4\e} \int_{D_m} \frac{\O^2}{\rho^2}\,dx.\ee 	
Again, since $C_{\a,A}\leq \frac{2}{19}$, we choose $\e=\frac{1}{38}$ and conclude
\be\label{es-4NA}
\int_{D_m} \frac{| \nabla f_1 |^2}{\rho^2}\,dx \leq 19^2\int_{D_m} \frac{\O^2}{\rho^2}\,dx.\ee


Finally,  since $\O=0$ on $\p^{R}D_{m}$, it follows from Corollary \ref{Cor, P-sine-bdry} that
 	\begin{align}\label{es-omgNA}
 		\int_{D_m} \frac{\O^2}{\rho^2} \,dx &= 2\pi\int_{\rho_1}^{\rho_2}\int_{\phi_1}^{\phi_2} \O^2 \sin\phi \,d\phi \,d\rho \nonumber \\
 		&\leq 2\pi\, C_{\a,B}\int_{\rho_1}^{\rho_2}\int_{\phi_1}^{\phi_2} |\p_{\phi}\O|^2 \sin\phi \,d\phi \,d\rho = C_{\a,B} \int_{D_m} \bigg(\frac{\p_\phi\O}{\rho}\bigg)^2 \,dx.
 	\end{align}
 	Combining \eqref{es-4NA},  \eqref{es-omgNA} and the fact that $C_{\a,B}\leq C_{\pi/6,B}= \frac{3}{25}$, we get
 	\begin{align*}
 		\int_{D_m}\frac{1}{\rho^2}|\nabla f_{1}(x,t)|^2dx\leq 44\int_{D_m} \bigg(\frac{\p_\phi\O(x,t)}{\rho}\bigg)^2 \,dx \leq 44\int_{D_m} | \nabla \O(x,t) |^2 \,dx.
 	\end{align*}
 	This completes the proof of (\ref{es2NA}).
 \end{proof}

 \subsection{ Control of $\Vert \na (v_\phi/\rho)(\cdot, t) \Vert_{L^2}$ and $\big\| \frac{1}{\rho}  \na (v_\phi/\rho)(\cdot, t) \big\|_{L^2}$ via $ \O(\cdot, t) $.}
 \label{Subsec, est on v_phi/rho}

\quad

\begin{lemma}\label{Lemma, e-trans-v_phi}
Let the region $D_{m}$ be as defined in (\ref{app domain-sph}) with $m \geq 2$ and the angle $\alpha \in\big(0, \frac{\pi}{6}\big]$.  Then for any $ T>0 $ and for a.e. $ t\in[0,T] $,
\begin{align}
 		\left\|\nabla \Big(\frac{v_\phi}{\rho}(\cdot, t) \Big) \right\|_{L^2(D_m)} & \leq \sqrt{3} \left\|\Omega(\cdot,t)\right\|_{L^2(D_m)},\label{EE1ZJ}\\
 		\left\|\frac{1}{\rho}\, \nabla \Big( \frac{v_\phi}{\rho}(\cdot, t) \Big)\right\|_{L^2(D_m)} & \leq 20 \left\| \nabla \Omega(\cdot,t)\right\|_{L^2(D_m)}. \label{EE2ZJ}
 	\end{align}
\end{lemma}

\begin{proof}
Since $ v\in E^{\sigma,s}_{m,T} \cap L_t^2 H_x^{2} \cap L_{tx}^\infty(D_m\times[0,T]) $ and $ \rho $ has the lower bound $ \frac1m $ on $ D_m $, we know $ \O\in L_t^2 H_x^{1}(D_m\times[0,T])$. So there exists a set $ S_T\subset[0,T] $ such that $ [0,T]\setminus S_T $ has measure 0 and for any $ t\in S_T $, $ \O(\cdot,t) $ belongs to $ H^{1}(D_m) $. Fixing any $ t\in S_T $, it suffices to prove (\ref{es1NA}) and (\ref{es2NA}) for such $ t $. For simplicity of notation, we will drop all the temporal variables in the following proof.

We first focus on \eqref{EE1ZJ} and define $f_{2}=\frac{v_\phi}{\rho}$. Then it follows from the second equation in (\ref{energy trans eq}) that
\be\label{EFZJ}
\Big( \Delta+\frac{4}{\rho}\p_\rho + \frac{5-\cot^2\phi}{\rho^2} \Big) f_2 = \frac{1}{\rho^4}\,\p_\rho(\rho^4\sin\phi\,\O).
\ee
On the boundary portion $\p^{R}D_{m}$, owing to $v_\phi=\p_\rho v_\phi=0$, one concludes that
\be\label{BCPHIZJ}
f_{2}=\p_\rho f_{2}=0,\quad \text{on}\quad \p^{R}D_{m}.
\ee
Meanwhile, since $\p_\rho (\rho v_\phi) = 0$ on the boundary portion $\p^{A}D_{m}$, one deduces that
\be\label{BCRHOZJ}
\p_\rho f_{2} + \frac{2}{\rho}\,f_2 = 0,  \quad \text{on} \quad \p^{A}D_{m}.
\ee

Multiplying \eqref{EFZJ} by $f_{2}$ and integrating on domain $D_m$, one derives that
\be\label{E0ZJ}
\underbrace{\int_{D_m}f_{2}\, \Dl f_{2} \,dx}_{I_1} + 4\int_{D_m}\frac{1}{\rho}f_{2} \,\p_\rho f_{2} \,dx + \int_{D_m} (5 - \cot^2\phi) \frac{f_{2}^2}{\rho^2} \,dx = \underbrace{\int_{D_m}\frac{f_{2}}{\rho^4} \,\p_\rho(\rho^4\sin\phi\,\O) \,dx}_{I_2}.
\ee
Using integration by parts,
\be\label{EI1ZJ}
I_1=-\int_{D_m}|\na f_{2}|^2dx+\underbrace{\int_{\p D_m}f_{2}\, \p_{n}f_2 \,dS}_{I_{11}}.
\ee
By boundary conditions \eqref{BCPHIZJ} and \eqref{BCRHOZJ}, $I_{11}$ satisfies
\[I_{11} = -\int_{A_{1,m}} f_{2}\, \p_\rho f_{2} \,dS + \int_{A_{2,m}} f_{2}\, \p_\rho f_{2} \,dS = \int_{A_{1,m}}\frac{2f_{2}^2}{\rho} \,dS - \int_{A_{2,m}} \frac{2f_{2}^2}{\rho} \,dS,\]
where the meaning of $ A_{1,m} $ and $ A_{2,m} $ can be found in Figure \ref{Fig,app domain-sph}. By the fundamental theorem of calculus,  we further notice that
\[
\begin{split}
I_{11} &= 4\pi \int_{\phi_1}^{\phi_2} \rho_1 \sin\phi \, f_2^2 \,d\phi -4\pi \int_{\phi_1}^{\phi_2} \rho_2 \sin\phi \, f_2^2 \,d\phi \\
&= -4\pi \int_{\phi_1}^{\phi_2} \int_{\rho_1}^{\rho_2} \p_{\rho}( \rho \sin\phi \, f_2^2) \,d\rho \,d\phi \\
&= -4\int_{D_m}\frac{f_{2}}{\rho} \,\p_\rho f_{2} \,dx - 2\int_{D_m}\frac{f_{2}^2}{\rho^2} \,dx.
\end{split}
\]
Plugging the above expression of $I_{11}$ in \eqref{EI1ZJ}, one has
\be\label{EI1FZJ}
I_1=-\int_{D_m}|\na f_{2}|^2dx-2\int_{D_m} \frac{f_{2}^2}{\rho^2} \,dx - 4\int_{D_m}\frac{f_{2}}{\rho} \,\p_\rho f_{2} \,dx.
\ee
For $I_2$ which can be rewritten as
\[I_2 = 2\pi \int_{\phi_1}^{\phi_2}\int_{\rho_1}^{\rho_2} \frac{\sin\phi}{\rho^2}\, f_2\, \p_{\rho}(\rho^4 \sin\phi\, \O) \,d\rho\, d\phi,\]
we use integration by parts and the fact that $\O=0$ on $\p^{A}D_{m}$ to obtain
\be\label{EI2FZJ} \begin{split}
I_2 &= - 2\pi \int_{\phi_1}^{\phi_2}\int_{\rho_1}^{\rho_2} \p_{\rho}\Big( \frac{\sin\phi}{\rho^2} \, f_2 \Big) \,  \rho^4 \sin\phi\, \O \,d\rho\, d\phi\\
&= 2\int_{D_m}\frac{\sin\phi}{\rho}\, f_2\, \O \,dx - \int_{D_m} \sin\phi\,  (\p_\rho f_{2})\,\O \,dx.
\end{split}\ee
Plugging (\ref{EI1FZJ}) and (\ref{EI2FZJ}) into (\ref{E0ZJ}) yields
\[
-\int_{D_m} |\nabla f_2|^2 \,dx + \int_{D_m} (3 - \cot^2\phi) \frac{f_{2}^2}{\rho^2} \,dx = 2\int_{D_m}\frac{\sin\phi}{\rho}\, f_2\, \O \,dx - \int_{D_m} \sin\phi\,  (\p_\rho f_{2})\,\O \,dx.
\]

As a result,
\[
\int_{D_m} |\nabla f_2|^2 \,dx \leq 3\int_{D_m} \frac{f_{2}^2}{\rho^2} \,dx + 2\int_{D_m}\frac{\sin\phi}{\rho}\, |f_2\, \O | \,dx + \int_{D_m} \sin\phi\, |(\p_\rho f_{2})\,\O| \,dx.
\]
By Cauchy-Schwarz inequality, for any $\e_1>0$ and $\e_2>0$,
\be\label{f_2 energy ineq}
\int_{D_m} |\nabla f_2|^2 \,dx \leq (3+\e_1) \int_{D_m} \frac{f_{2}^2}{\rho^2} \,dx + \e_2 \int_{D_{m}} |\p_{\rho} f_2 |^2 \,dx + \Big( \frac{1}{\e_1} + \frac{1}{4\e_2} \Big) \int_{D_m}  \O^2 \,dx.
\ee
Since $f_2=0$ on $\p^{R}D_{m}$,  then by a similar derivation as that in (\ref{es-omgNA}), we get
	\[\int_{D_m} \frac{f_2^2}{\rho^2} \,dx \leq C_{\a,B} \int_{D_m} \bigg(\frac{\p_\phi f_2}{\rho}\bigg)^2 \,dx.
	\]
Putting the above estimate into (\ref{f_2 energy ineq}) and recalling the estimate $C_{\a,B}\leq \frac{3}{25}$ in (\ref{two P-consts}),  we obtain
\[\int_{D_m} |\nabla f_2|^2 \,dx \leq \frac{3(3+\e_1)}{25} \int_{D_m} \bigg(\frac{\p_{\phi} f_{2}}{\rho}\bigg)^2 \,dx + \e_2 \int_{D_{m}} |\p_{\rho} f_2 |^2 \,dx + \Big( \frac{1}{\e_1} + \frac{1}{4\e_2} \Big) \int_{D_m}  \O^2 \,dx.\]
By choosing $\e_1=2$ and choosing $\e_2=\frac{3(3+\e_1)}{25}=\frac{3}{5}$, we find
\[\int_{D_m} |\nabla f_2|^2 \,dx \leq \frac{3}{5}\int_{D_m} |\nabla f_2|^2 \,dx + \int_{D_m}  \O^2 \,dx.\]
This implies that
\[\int_{D_m} |\nabla f_2(x,t)|^2 \,dx \leq 3\int_{D_m}  \O^2(x,t) \,dx,\]
which proves (\ref{EE1ZJ}).

Next, we are going to prove \eqref{EE2ZJ}.  Multiplying \eqref{EFZJ} by $\frac{f_{2}}{\rho^2}$ and integrating on $D_m$ yields
\be\label{E1ZJ}
\underbrace{\int_{D_m}\frac{f_{2}}{\rho^2}\, \Dl f_{2} \,dx}_{J_1} + 4\int_{D_m}\frac{1}{\rho^3} \,f_{2}\p_\rho f_{2} \,dx + \int_{D_m} \frac{5-\cot^2\phi}{\rho^4} \, f_{2}^2 \,dx = \underbrace{\int_{D_m}\frac{f_{2}}{\rho^6} \,\p_\rho(\rho^4\sin\phi\,\O) \,dx}_{J_2}.
\ee
Using integration by parts,
\begin{align*}
J_1 &= -\int_{D_m} \nabla\Big( \frac{f_2}{\rho^2}\Big) \cdot \nabla f_2 \,dx + \int_{\p D_m} \frac{f_2}{\rho^2}\, \p_{n}f_2 \,dS \\
&= -\int_{D_m}\frac{1}{\rho^2} |\nabla f_{2}|^2 \,dx + 2\int_{D_m}\frac{f_{2}}{\rho^3} \,\p_\rho f_{2} \,dx + \underbrace{\int_{\p D_m}\frac{f_{2}}{\rho^2}\, \p_{n}f_2 \,dS}_{J_{11}}.
\end{align*}
Similar to the computation of $I_{11}$ above, we find
\[
J_{11} = 2\int_{D_m}\frac{f_{2}^2}{\rho^4} \,dx - 4\int_{D_m}\frac{f_{2}}{\rho^3} \,\p_\rho f_{2} \,dx.
\]
So
\be\label{J1ZJ}
J_1 = -\int_{D_m}\frac{1}{\rho^2} |\nabla f_{2}|^2 \,dx + 2\int_{D_m}\frac{f_{2}^2}{\rho^4} \,dx - 2\int_{D_m}\frac{f_{2}}{\rho^3} \,\p_\rho f_{2} \,dx.
\ee
Next, by direct computation,
\[
J_2 = 4\int_{D_m}\frac{\sin\phi}{\rho^3}\, f_2\,  \O \,dx + \int_{D_m}\frac{\sin\phi}{\rho^2}\, f_2\, \p_\rho\O \,dx.
\]
Substituting the above expression for $J_2$ and (\ref{J1ZJ}) for $J_1$ into (\ref{E1ZJ}), one deduces
\[
\begin{split}
\int_{D_m}\frac{1}{\rho^2}|\nabla f_{2}|^2dx &= \int_{D_m} \frac{7-\cot^2\phi}{\rho^4}\,f_2^2\,dx + 2\int_{D_m}\frac{f_{2}}{\rho^3}\, \p_\rho f_{2} \,dx \\
&\quad - 4\int_{D_m}\frac{\sin\phi}{\rho^3}\, f_2\,  \O \,dx - \int_{D_m}\frac{\sin\phi}{\rho^2}\, f_2\, \p_\rho\O \,dx.
\end{split}
\]
Thus,
\be\label{f_2-energy-e1}\begin{split}
\int_{D_m}\frac{1}{\rho^2}|\nabla f_{2}|^2dx &\leq 7 \int_{D_m} \frac{f_2^2}{\rho^4}\,dx + 2\int_{D_m}\frac{1}{\rho^3}\, | f_2\, \p_\rho f_{2} | \,dx \\
&\quad + 4\int_{D_m}\frac{1}{\rho^3}\, |f_2\,  \O | \,dx + \int_{D_m}\frac{1}{\rho^2}\, | f_2\, \p_\rho\O | \,dx.
\end{split}\ee
By Cauchy-Schwarz inequality, for any constants $\e_1$, $\e_2$, $\e_3>0$,  one has
\be\label{f_2-energy-e2}\begin{split}
\int_{D_m}\frac{1}{\rho^2}|\nabla f_{2}|^2dx &\leq \Big(7 + \frac{1}{\e_1} + \e_2 + \e_3\Big) \int_{D_m} \frac{f_2^2}{\rho^4}\,dx + \e_1 \int_{D_m}\frac{1}{\rho^2}\, | \p_\rho f_{2} |^2 \,dx \\
&\quad + \frac{4}{\e_2}\int_{D_m}\frac{ \O^2}{\rho^2} \,dx + \frac{1}{4\e_3}\int_{D_m}| \p_\rho\O |^2 \,dx.
\end{split}\ee

Now since $f_2=\O=0$ on $\p^{R}D_{m}$, then similar to the derivation of (\ref{es-omgNA}), we obtain
\be\label{PCE1ZJ}
\int_{D_m}\frac{f_{2}^2}{\rho^4} \,dx \leq C_{\a,B}\int_{D_m}\frac{1}{\rho^2} \bigg( \frac{\p_\phi f_2}{\rho} \bigg)^2 \,dx,
\ee
\be\label{PCE2ZJ}
\int_{D_m}\frac{\O^2}{\rho^2} \,dx \leq C_{\a,B}\int_{D_m} \bigg(\frac{ \p_\phi\O}{\rho}\bigg)^2 \,dx.
\ee
Plugging (\ref{PCE1ZJ}) and (\ref{PCE2ZJ}) into (\ref{f_2-energy-e2}) and recalling $C_{\a,B}\leq \frac{3}{25}$, one deduces
\[\begin{split}
\int_{D_m}\frac{1}{\rho^2}|\nabla f_{2}|^2dx &\leq \frac{3}{25}\Big(7 + \frac{1}{\e_1} + \e_2 + \e_3\Big) \int_{D_m}\frac{1}{\rho^2} \bigg( \frac{\p_\phi f_2}{\rho} \bigg)^2 \,dx + \e_1 \int_{D_m}\frac{1}{\rho^2}\, | \p_\rho f_{2} |^2 \,dx \\
&\quad + \frac{12}{25\e_2}\int_{D_m} \bigg( \frac{\p_\phi\O}{\rho} \bigg)^2 \,dx + \frac{1}{4\e_3}\int_{D_m}| \p_\rho\O |^2 \,dx.
\end{split}\]
By choosing $\e_1=\f{20}{21}$, $\e_2=\f{1}{20}$ and $\e_3=\f{1}{40}$, we have
\[
\int_{D_m}\frac{1}{\rho^2}|\nabla f_{2}|^2dx \leq \f{39}{40}\int_{D_m}\frac{1}{\rho^2}|\nabla f_{2}|^2dx + 10\int_{D_m}| \nabla \O |^2 \,dx.
\]
This implies
\[\int_{D_m}\frac{1}{\rho^2}|\nabla f_{2}(x,t)|^2dx \leq 400\int_{D_m}| \nabla \O(x,t) |^2 \,dx,\]
completing the proof of (\ref{EE2ZJ}) and Lemma \ref{Lemma, e-trans-v_phi}.
\end{proof}

\subsection{Uniform bounds for $\Vert (K, F, \O)   \Vert_{L^\infty_t L^2_x}$ and $\Vert (\nabla K, \nabla F, \nabla \O) \Vert_{L_{tx}^2}$}
\label{Subsec, KFO-ub}

\quad

In this subsection, we will derive some energy estimates for $K$,  $F$ and $\O$.

\begin{lemma}\label{Lemma, energy id for KFO}
	Let the region $D_{m}$ be as defined in (\ref{app domain-sph}) with $m \geq 2$ and the angle $\alpha \in\big(0, \frac{\pi}{6}\big]$. Let $K$, $F$ and $\O$ be defined as in (\ref{K-F-O}). Then for any $T>0$, the following three energy identities (\ref{energy eq for K})--(\ref{energy eq for O}) hold.
	\be\label{energy eq for K}	\begin{split}
	& \frac12 \int_{D_m} K^2(x,T)\,dx - \f12\int_{D_m} K^2(x,0)\,dx + \int_{0}^{T}\int_{D_m} | \nabla K |^2\,dx\,dt \\
	=\,\, & 3 \int_{0}^{T}\int_{D_m}\frac{K^2}{\rho^2}\,dx\,dt - 2 \int_{0}^{T}\int_{D_m}\frac{K}{\rho}\, \p_{\rho}K \,dx\,dt \\
	\,\, & + \int_{0}^{T}\int_{D_m}\frac{v_\th}{\rho}\, \Big[ \p_{\phi} \Big( \frac{v_\rho}{\rho} \Big)\, \p_{\rho}K - \p_{\rho} \Big( \frac{v_\rho}{\rho} \Big)\, \p_{\phi}K \Big] \,dx\,dt.
	\end{split}	\ee

	\be\label{energy eq for F}\begin{split}
	& \frac12 \int_{D_m} F^2(x,T)\,dx - \frac12 \int_{D_m} F^2(x,0)\,dx  + \int_{0}^{T}\int_{D_m} | \nabla F |^2\,dx\,dt \\
	=\,\, & \int_{0}^{T}\int_{D_m} \frac{1-\cot^2\phi}{\rho^2}\, F^2 \,dx\,dt - 2\int_{0}^{T}\int_{D_m}\frac{\cot\phi}{\rho^2}\, F \p_{\phi}F \,dx\,dt + 2\int_{0}^{T}\int_{D_m} \frac{(\p_\phi K)F}{\rho^2} \,dx\,dt  \\
	\,\, & + \int_{0}^{T}\int_{D_m}\frac{v_\th}{\rho}\, \Big[ \p_{\phi} \Big( \frac{v_\phi}{\rho} \Big)\, \p_{\rho}F -  \p_{\rho} \Big( \frac{v_\phi}{\rho} \Big)\, \p_{\phi}F \Big] \,dx\,dt.
	\end{split}\ee

	\be\label{energy eq for O}\begin{split}
	& \frac12 \int_{D_m} \O^2(x,T)\,dx - \frac12 \int_{D_m} \O^2(x,0)\,dx + \int_{0}^{T}\int_{D_m} | \nabla \O |^2\,dx\,dt \\
	=\,\, & -2\int_{0}^{T}\int_{D_m} \frac{v_\th}{\rho \sin\phi}\, K\O \,dx\,dt - 2\int_{0}^{T}\int_{D_m} \frac{v_{\th} \cos\phi}{\rho \sin^2\phi}\, F\O \,dx\,dt.
	\end{split}\ee
\end{lemma}

\begin{proof}
Firstly, since $ v\in E^{\sigma,s}_{m,T} \cap L_t^2 H_x^{2} \cap L_{tx}^\infty(D_m\times[0,T]) $ and $ \rho $ has the lower bound $ \frac1m $ on $ D_m $, all of $ K $, $ F $ and $ \O $ are in $ L_t^2 H_x^{1}(D_m\times[0,T]) $. Meanwhile, all the integrals in (\ref{energy eq for K})--(\ref{energy eq for O}) are well-defined. In addition, $ K $ is even, and $ F $ and $ \O $ are odd symmetric with respect to the plane $ \{\phi=\frac{\pi}{2}\} $. (\ref{energy eq for K}), (\ref{energy eq for F}) and (\ref{energy eq for O}) can be justified by testing (\ref{eq of K-F-O})$_{1}$,  (\ref{eq of K-F-O})$_{2}$ and (\ref{eq of K-F-O})$_{3}$ by $K$, $F$ and $\O$ respectively. The derivations for these three energy identities are similar, so we will only show details for (\ref{energy eq for F}) which is relatively the most complicated one.

Multiplying (\ref{eq of K-F-O})$_{2}$ by F and integrating on $D_{m}\times [0,T]$ yields
	\be\label{F energy id-1}
	\frac12 \int_{D_m} F^2(x,T)\,dx - \frac12 \int_{D_m} F^2(x,0)\,dx = L_1 - L_2 + L_3 + L_4,
	\ee
where
	\begin{align}
	L_1 & = \int_{0}^{T}\int_{D_m} \bigg[ \bigg(\Delta + \frac{2}{\rho}\p_\rho + \frac{1-\cot^2 \phi}{\rho^2}\bigg) F \bigg] F  \,dx\,dt,  \label{L_1} \\
	L_2 & = \int_{0}^{T}\int_{D_m} (b\cdot \nabla F) F \,dx\,dt,   \label{L_2} \\
	L_3 & = \int_{0}^{T}\int_{D_m} \frac{2}{\rho^2}\, (\p_{\phi}K)F  \,dx\,dt,  \label{L_3} \\
	L_4 & = \int_{0}^{T}\int_{D_m} \bigg( \o \cdot \nabla \Big(\frac{v_\phi}{\rho}\Big) \bigg) F \,dx\,dt.  \label{L_4}
	\end{align}

We will first compute $L_1$.  Using integration by parts,
	\be\label{F energy-lap}
	\int_{0}^{T}\int_{D_m} (\Delta F) F\,dx\,dt = \int_{0}^{T}\int_{\p D_m}(\p_{n}F)F \,dS\,dt - \int_{0}^{T}\int_{D_m} | \nabla F|^2 \,dx\,dt.
	\ee
According to Lemma \ref{Lemma, bdry cond}, $F=0$ on $\p^{A}D_m$ and $\p_{\phi}F = -\cot\phi\, F$ on $\p^{R}D_m$, so	
	\[\begin{split}
	\int_{\p D_m}(\p_{n}F)F \,dS &= \int_{\p^{R} D_m}(\p_{n}F)F \,dS \\
	&= -\int_{R_{1,m}} \bigg( \f{1}{\rho} \p_\phi F\bigg)F \,dS + \int_{R_{2,m}} \bigg( \f{1}{\rho} \p_\phi F\bigg)F \,dS \\
	&= \int_{R_{1,m}} \f{\cot\phi_1}{\rho}\, F^2 \,dS - \int_{R_{2,m}} \f{\cot\phi_2}{\rho}\, F^2 \,dS,
	\end{split}\]
where the definition of the boundaries $ R_{1,m} $ and $ R_{2,m} $ can be found in Figure \ref{Fig,app domain-sph}. Noticing $dS = 2\pi \rho \sin\phi\, d\rho$,  we find
	\[
	\int_{\p D_m}(\p_{n}F)F \,dS = 2\pi \int_{\rho_1}^{\rho_2} \cos\phi_1 \, F^2 \,d\rho - 2\pi \int_{\rho_1}^{\rho_2} \cos\phi_2 \, F^2 \,d\rho.
	\]
Now applying the fundamental theorem of Calculus yields
	\be\label{F energy-lap bdry}\begin{split}
	\int_{\p D_m}(\p_{n}F)F \,dS	&= -2\pi \int_{\rho_1}^{\rho_2}\int_{\phi_1}^{\phi_2} \p_\phi (\cos\phi\, F^2) \,d\phi\,d\rho \nonumber \\
	&= 2\pi \int_{\rho_1}^{\rho_2}\int_{\phi_1}^{\phi_2} \sin\phi\, F^2 \,d\phi\,d\rho - 4\pi \int_{\rho_1}^{\rho_2}\int_{\phi_1}^{\phi_2} \cos\phi\, F \p_\phi F \,d\phi\,d\rho \nonumber \\
	&= \int_{D_m} \f{F^2}{\rho^2} \,dx - 2\int_{D_m} \f{\cot\phi}{\rho^2}\, F\p_{\phi}F \,dx.
	\end{split}\ee
Substituting the above identity into (\ref{L_1}) gives
	\be\label{F energy-lap2}\begin{split}
	\int_{0}^{T}\int_{D_m} (\Delta F) F\,dx\,dt &= - \int_{0}^{T}\int_{D_m} | \nabla F|^2 \,dx\,dt + \int_{0}^{T}\int_{D_m} \f{F^2}{\rho^2} \,dx\,dt \\
	&\quad - 2\int_{0}^{T}\int_{D_m} \f{\cot\phi}{\rho^2}\, F \p_\phi F \,dx\,dt.
	\end{split}\ee	
We continue to deal with the first-order term in (\ref{L_1}).
	\[
	\int_{D_m} \f{2}{\rho}\,(\p_\rho F) F \,dx = 2\pi \int_{\phi_1}^{\phi_2} \int_{\rho_1}^{\rho_2} \rho \sin\phi\, \p_\rho(F^2) \,d\rho\,d\phi.
	\]
Recalling $F=0$ on $\p^{A}D_m$,  so we apply integration by parts to obtain
	\be\label{F energy-lin rho}
	\int_{D_m} \f{2}{\rho}\,(\p_\rho F) F \,dx = - 2\pi  \int_{\phi_1}^{\phi_2} \int_{\rho_1}^{\rho_2} \sin\phi\, F^2 \,d\rho\,d\phi = -\int_{D_m} \f{F^2}{\rho^2} \,dx.
	\ee
Plugging (\ref{F energy-lap2}) and (\ref{F energy-lin rho}) into (\ref{L_1}) shows
	\be\label{L_1 id-final}\begin{split}
	L_1 = - \int_{0}^{T}\int_{D_m} | \nabla F|^2 \,dx\,dt + \int_{0}^{T}\int_{D_m} \frac{1-\cot^2\phi}{\rho^2}\, F^2 \,dx\,dt - 2\int_{0}^{T}\int_{D_m}\frac{\cot\phi}{\rho^2}\, F \p_{\phi}F \,dx\,dt.
	\end{split}\ee

Next, we calculate $L_2$.  By divergence theorem, we have
	\[\begin{split}
	L_2 &= \f12 \int_{0}^{T} \int_{D_m} b \cdot	\nabla (F^2) \,dx\,dt \\
	&= \f12 \int_{0}^{T}\int_{\p D_m} (b\cdot n) F^2 \,dS\,dt - \f12 \int_{0}^{T}\int_{D_m} (\nabla \cdot b) F^2 \,dx\,dt.
	\end{split}\]
Noticing that $b \cdot n = v \cdot n =0$ on $\p D_m$ and $\nabla \cdot b = \nabla \cdot v = 0$ in $D_m$,  so
	\be\label{L_2 id-final}
	L_2=0.  	\ee
Finally,  the term $L_4$ will be treated.  Based on the formula (\ref{vor f-sph}) for $\o$,

	\[\begin{split}
	\o \cdot \nabla \Big(\frac{v_\phi}{\rho}\Big) & = \o_\rho\, \p_\rho \Big( \f{v_\phi}{\rho} \Big) + \o_\phi\, \f{1}{\rho}\p_\phi \Big( \f{v_\phi}{\rho} \Big) \\
	&= \f{1}{\rho \sin\phi}\, \p_\phi (\sin\phi \, v_\th)\, \p_\rho \Big( \f{v_\phi}{\rho} \Big) - \f{1}{\rho^2}\, \p_\rho (\rho v_\th)\, \p_\phi \Big( \f{v_\phi}{\rho} \Big).
	\end{split}\]
Thus,
	\be\label{L_4 id-1}
	L_4 = L_{41} - L_{42},
	\ee
where
	\[\begin{split}
	L_{41} &= \int_{0}^{T}\int_{D_m} \f{1}{\rho \sin\phi}\, \p_\phi (\sin\phi \, v_\th)\, \p_\rho \Big( \f{v_\phi}{\rho} \Big)\, F \,dx\,dt,\\
	&= 2\pi \int_{0}^{T} \int_{\rho_1}^{\rho_2} \int_{\phi_1}^{\phi_2} \rho\, \p_\phi (\sin\phi \, v_\th)\, \p_\rho \Big( \f{v_\phi}{\rho} \Big)\, F \,d\phi\,d\rho\,dt,
	\end{split}\]
and
	\[\begin{split}	
	L_{42} &=  \int_{0}^{T}\int_{D_m} \f{1}{\rho^2}\, \p_\rho (\rho v_\th)\, \p_\phi \Big( \f{v_\phi}{\rho} \Big)\, F \,dx\,dt, \\
	&= 2\pi \int_{0}^{T} \int_{\phi_1}^{\phi_2} \int_{\rho_1}^{\rho_2} \sin\phi\,  \p_\rho (\rho v_\th)\, \p_\phi \Big( \f{v_\phi}{\rho} \Big)\, F \,d\rho\,d\phi\,dt.
	\end{split}\]
For $L_{41}$,  since $v_\phi=0$ on $\p^{R}D_m$, then $\p_\rho v_\phi=0$ on $\p^{R}D_m$, which further implies $\p_\rho \big( \f{v_\phi}{\rho} \big) = 0$ on $\p^{R}D_m$.  This enables one to do integration by parts with vanishing boundary terms to get
	\[\begin{split}
	\int_{\phi_1}^{\phi_2} \rho\, \p_\phi (\sin\phi \, v_\th)\, \p_\rho \Big( \f{v_\phi}{\rho} \Big)\, F \,d\phi &= - \int_{\phi_1}^{\phi_2}  \rho \sin\phi\, v_\th\, \p_\phi \Big[ \p_\rho \Big( \f{v_\phi}{\rho} \Big)\, F \Big] \,d\phi \\
	&= - \int_{\phi_1}^{\phi_2}  \rho \sin\phi\, v_\th\, \Big[ \p_\phi \p_\rho \Big( \f{v_\phi}{\rho} \Big)\, F + \p_\rho \Big( \f{v_\phi}{\rho} \Big)\, \p_\phi F \Big] \,d\phi.
	\end{split}\]
Hence,
	\be\label{L_41 id}\begin{split}
	L_{41} &= - 2\pi \int_{0}^{T} \int_{\rho_1}^{\rho_2} \int_{\phi_1}^{\phi_2}  \rho \sin\phi\, v_\th \Big[ \p_\phi \p_\rho \Big( \f{v_\phi}{\rho} \Big)\, F + \p_\rho \Big( \f{v_\phi}{\rho} \Big)\, \p_\phi F \Big] \,d\phi\,d\rho\,dt \\
	&= - \int_{0}^{T} \int_{D_m} \f{v_\th}{\rho}\, \Big[ \p_\phi \p_\rho \Big( \f{v_\phi}{\rho} \Big)\, F + \p_\rho \Big( \f{v_\phi}{\rho} \Big)\, \p_\phi F \Big] \,dx\,dt.
	\end{split}\ee
For $L_{42}$,  by taking advantage of the fact that $F=0$ on $\p^{A}D_m$, we can again apply integration by parts with vanishing boundary terms to obtain
\[\begin{split}
	\int_{\rho_1}^{\rho_2} \sin\phi\,  \p_\rho (\rho v_\th)\, \p_\phi \Big( \f{v_\phi}{\rho} \Big)\, F \,d\rho &= - \int_{\rho_1}^{\rho_2} \sin\phi\, \rho v_\th\, \p_\rho \Big[ \p_\phi \Big( \f{v_\phi}{\rho} \Big)\, F \Big]  \,d\rho \\
	&= - \int_{\rho_1}^{\rho_2}  \rho \sin\phi\, v_\th\, \Big[ \p_\rho \p_\phi \Big( \f{v_\phi}{\rho} \Big)\, F + \p_\phi \Big( \f{v_\phi}{\rho} \Big)\, \p_\rho F \Big] \,d\rho.
	\end{split}\]
Hence,
	\be\label{L_42 id}\begin{split}
	L_{42} &= - 2\pi \int_{0}^{T} \int_{\phi_1}^{\phi_2}  \int_{\rho_1}^{\rho_2}  \rho \sin\phi\, v_\th\, \Big[ \p_\rho \p_\phi \Big( \f{v_\phi}{\rho} \Big)\, F + \p_\phi \Big( \f{v_\phi}{\rho} \Big)\, \p_\rho F \Big] \,d\rho \,d\phi\,dt \\
	&= - \int_{0}^{T} \int_{D_m} \f{v_\th}{\rho}\, \Big[ \p_\rho \p_\phi \Big( \f{v_\phi}{\rho} \Big)\, F + \p_\phi \Big( \f{v_\phi}{\rho} \Big)\, \p_\rho F \Big] \,dx\,dt.
	\end{split}\ee
By substituting (\ref{L_41 id}) and (\ref{L_42 id}) into (\ref{L_4 id-1}),  we see that the super-critical terms containing $\p_\rho \p_\phi \Big( \f{v_\phi}{\rho} \Big)$ are canceled out and we find
	\be\label{L_4 id-final}
	L_4 = \int_{0}^{T} \int_{D_m} \f{v_\th}{\rho}\, \Big[ \p_\phi \Big( \f{v_\phi}{\rho} \Big)\, \p_\rho F - \p_\rho \Big( \f{v_\phi}{\rho} \Big)\, \p_\phi F \Big] \,dx\,dt.
	\ee

Finally, putting (\ref{L_1 id-final}), (\ref{L_2 id-final}), (\ref{L_3}) and (\ref{L_4 id-final}) into (\ref{F energy id-1}) justifies (\ref{energy eq for F}).
\end{proof}

In the next lemma, we close the energy estimate for $ K $, $ F $ and $ \O $, which is the key result in this paper.

\begin{lemma}\label{Lemma, energy est for KFO}
	Let the region $D_{m}$ be as defined in (\ref{app domain-sph}) with $m \geq 2$ and the angle $\alpha \in\big(0, \frac{\pi}{6}\big]$. Let $ \Gamma $, $K$, $F$ and $\O$ be defined as in (\ref{Gamma-sph}) and (\ref{K-F-O}). Assume
	\be\label{small Gamma}
	\| \Gamma(\cdot, 0)\|_{L^{\infty}(D_m)}\leq \frac{1}{95}.
	\ee
Then for any $T>0$,
	\begin{align}
	& \int_{D_m} (K^2 + F^2 + \O^2)(x,T) \,dx + \frac{1}{10} \int_{0}^{T}\int_{D_m} \big( | \nabla K |^2 + | \nabla F |^2 + | \nabla \O |^2\big) \,dx\,dt \nonumber\\
	\leq\,\, & \int_{D_m} (K^2 + F^2 + \O^2)(x,0) \,dx \label{energy est for KFO} \\
	\leq\,\, & C\| v_0\|^2_{H^2(D_m)}, \label{KFO-v0}
	\end{align}
    where $ C=C(\a) $.
\end{lemma}

\begin{proof}
We add (\ref{energy eq for K}),  (\ref{energy eq for F}) and (\ref{energy eq for O}) together to obtain
	\be\label{energy est-eq 1}\begin{split}
	& \frac12 \int_{D_m} (K^2 + F^2 + \O^2)(x,T) \,dx - \frac12 \int_{D_m} (K^2 + F^2 + \O^2)(x,0) \,dx \\
	&\quad + \int_{0}^{T}\int_{D_m} |\nabla K |^2 + | \nabla F |^2 + | \nabla \O |^2 \,dx\,dt = S_1 + S_2 + S_3,
	\end{split}\ee
where
	\[\begin{split}
	S_1 &= 3 \int_{0}^{T}\int_{D_m}\frac{K^2}{\rho^2}\,dx\,dt - 2\int_{0}^{T}\int_{D_m}\frac{K}{\rho}\, \p_{\rho}K \,dx\,dt,  \\
	S_2 &=  \int_{0}^{T}\int_{D_m} \frac{1-\cot^2\phi}{\rho^2}\, F^2 \,dx\,dt - 2\int_{0}^{T}\int_{D_m}\frac{\cot\phi}{\rho^2}\, 	F \p_{\phi}F \,dx\,dt  + 2\int_{0}^{T}\int_{D_m} \frac{(\p_\phi K)F}{\rho^2} \,dx\,dt
	\end{split}\]
and
	\be\label{S_3}\begin{split}
	S_3 &= \int_{0}^{T}\int_{D_m} \frac{v_\th}{\rho}\, \bigg[  \p_{\phi} \Big( \frac{v_\rho}{\rho} \Big)\,		\p_{\rho}K- \p_{\rho} \Big( \frac{v_\rho}{\rho} \Big)\, \p_{\phi}K + \p_{\phi} \Big( \frac{v_\phi}{\rho} \Big)\, \p_{\rho}F - \p_{\rho} \Big( \frac{v_\phi}{\rho} \Big)\, \p_{\phi}F \\
	& \qquad\qquad - \frac{2}{\sin\phi}\, K\O - \frac{2\cos\phi}{\sin^2\phi}\, F\O  \bigg] \,dx\,dt.  				\end{split}\ee
	
We first estimate $S_1$ and $S_2$. By Cauchy Schwarz inequality,  for any $\e_1, \e_2, \e_3\in (0,1)$,  we have
	\be\label{S_1 ineq}
	S_1 \leq \e_1 \int_{0}^{T}\int_{D_m} ( \p_{\rho}K )^2 \,dx\,dt +  \Big( 3 + \frac{1}{\e_1} \Big)\int_{0}^{T}\int_{D_m}\frac{K^2}{\rho^2}\,dx\,dt,
	\ee
and
	\[\begin{split}
	S_2 &\leq \e_2 \int_{0}^{T}\int_{D_m} \bigg( \frac{\p_{\phi}K}{\rho} \bigg)^2 \,dx\,dt + \e_3 \int_{0}^{T}\int_{D_m} \bigg( \frac{\p_{\phi}F}{\rho} \bigg)^2 \,dx\,dt \\
	&\quad + \Big( 1+\frac{1}{\e_2} \Big) \int_{0}^{T}\int_{D_m} \frac{F^2}{\rho^2}\,dx\,dt + \Big( \frac{1}{\e_3}-1 \Big) \int_{0}^{T}\int_{D_m}\frac{\cot^2\phi}{\rho^2}\, F^2\,dx\,dt.
	\end{split}\]
Since $\a\in \big(0,\frac{\pi}{6}\big]$,  then $\phi\in \big[ \frac{2\pi}{3},\frac{4\pi}{3}\big]$ and $\cot^2\phi\leq \frac13$. Therefore,
	\be\label{S_2 ineq}\begin{split}
	S_2 &\leq \e_2 \int_{0}^{T}\int_{D_m} \bigg( \f{\p_{\phi}K}{\rho} \bigg)^2 \,dx\,dt + \e_3 \int_{0}^{T}\int_{D_m} \bigg( \f{\p_{\phi}F}{\rho} \bigg)^2 \,dx\,dt \\
	&\quad + \Big(\f{2}{3} + \f{1}{\e_2} + \f{1}{3\e_3}\Big) \int_{0}^{T}\int_{D_m}\f{F^2}{\rho^2} \,dx\,dt.
	\end{split}\ee
Adding (\ref{S_1 ineq}) and (\ref{S_2 ineq}) together leads to
	\be\label{S_1+S_2, e-ineq-1}\begin{split}
	S_1+S_2 &\leq  \e_1 \int_{0}^{T}\int_{D_m} ( \p_{\rho}K )^2 \,dx\,dt +  \e_2 \int_{0}^{T}\int_{D_m} \bigg( \f{\p_{\phi}K}{\rho} \bigg)^2 \,dx\,dt + \e_3 \int_{0}^{T}\int_{D_m} \bigg( \f{\p_{\phi}F}{\rho} \bigg)^2 \,dx\,dt  \\
	&\quad + \Big( 3 + \frac{1}{\e_1} \Big) \int_{0}^{T}\int_{D_m}\frac{K^2}{\rho^2}\,dx\,dt + \Big(\f{2}{3} + \f{1}{\e_2} + \f{1}{3\e_3}\Big) \int_{0}^{T}\int_{D_m}\f{F^2}{\rho^2} \,dx\,dt.
	\end{split}\ee

According to Lemma \ref{Lemma, bdry cond},  $K=0$ on $\p^{R}D_{m}$,  so similar to the derivation in (\ref{es-omgNA}), we deduce
	\be\label{K-P energy est}
	\int_{D_m} \frac{K^2}{\rho^2} \,dx \leq C_{\a,B} \int_{D_m} \bigg( \frac{\p_{\phi}K}{\rho} \bigg)^2 \,dx.
	\ee
On the other hand,  since $v_\th$ is odd with respect to $\{\phi=\frac{\pi}{2}\}$, then we can derive from (\ref{vor f-sph}) that $\o_\phi$ is also odd with respect to $\{\phi=\frac{\pi}{2}\}$.   Thus, $F$ is odd with respect to $\{\phi=\frac{\pi}{2}\}$, which implies
\[\int_{\phi_1}^{\phi_2} F \sin\phi \,d\phi=0. \]
Then analogous to the estimate in (\ref{P-ave-energy-est}), we get
	\be\label{F-P energy est}
	\int_{D_m} \frac{F^2}{\rho^2} \,dx \leq C_{\a,A} \int_{D_m} \bigg( \frac{\p_{\phi}F}{\rho} \bigg)^2 \,dx.
	\ee
Putting (\ref{K-P energy est}) and (\ref{F-P energy est}) into (\ref{S_1+S_2, e-ineq-1}), and recalling $C_{\a,A}\leq \frac{2}{19}$ and $C_{\a,B}\leq \frac{3}{25}$, we obtain
\be\label{S_1+S_2, e-ineq-2}\begin{split}
	S_1+S_2 &\leq  \e_1 \int_{0}^{T}\int_{D_m} ( \p_{\rho}K )^2 \,dx\,dt +  \e_2 \int_{0}^{T}\int_{D_m} \bigg( \f{\p_{\phi}K}{\rho} \bigg)^2 \,dx\,dt + \e_3 \int_{0}^{T}\int_{D_m} \bigg( \f{\p_{\phi}F}{\rho} \bigg)^2 \,dx\,dt  \\
	&\quad + \f{3}{25}\, \Big( 3 + \frac{1}{\e_1} \Big) \int_{0}^{T}\int_{D_m} \bigg( \frac{\p_\phi K}{\rho} \bigg)^2 \,dx\,dt +  \f{2}{19}\, \Big(\f{2}{3} + \f{1}{\e_2} + \f{1}{3\e_3}\Big) \int_{0}^{T}\int_{D_m} \bigg( \f{\p_{\phi}F}{\rho} \bigg)^2 \,dx\,dt.
	\end{split}\ee
By choosing $\e_1=\frac{9}{10}$, $\e_2=\f13$ and $\e_3=\f15$,  we conclude
	\be\label{S_1+S_2, e-ineq-3}\begin{split}
	S_1+S_2 &\leq \frac{9}{10} \int_{0}^{T}\int_{D_m} ( \p_{\rho}K )^2 + \bigg( \f{\p_{\phi}K}{\rho} \bigg)^2 + \bigg( \f{\p_{\phi}F}{\rho} \bigg)^2    \,dx\,dt \\
	&\leq \frac{9}{10} \int_{0}^{T}\int_{D_m} | \nabla K |^2 + | \nabla F |^2 \,dx\,dt.
	\end{split}\ee

Next, we estimate $S_3$.  Denote $C^{*}=\frac{1}{95}$,  then the assumption on $\Gamma$ becomes $\| \G(\cdot, 0) \|_{ L^{\infty}(D_m) } \leq C^{*}$.  This enables one to derive from Lemma \ref{Lemma, Gamma-Dm} that
	\[ \| \G \|_{ L^{\infty}(D_m\times (0,T)) } \leq C^{*}.
	\]
Recalling $\G = \rho \sin\phi\, v_\th$ and $\a\in \big( 0,\f{\pi}{6} \big]$,  so $\phi\in \big[ 2\pi/3, 4\pi/3 \big]$ and
	\be\label{assump bdd on v}
	| v_\th | \leq \f{C^*}{\rho \sin\phi} \leq \f{2}{ \sqrt{3} }\, \f{C^*}{\rho}.
	\ee
Combining (\ref{assump bdd on v}) with (\ref{S_3}) yields
	\be\label{S_3 est-1}
	|S_3| \leq \frac{2C^*}{\sqrt{3}}\, (S_{31} + S_{32} + S_{33}),
	\ee
	where
	\begin{eqnarray*}
	S_{31} &=& \int_{0}^{T}\int_{D_m} \f{1}{\rho} \bigg| \f{1}{\rho} \p_{\phi} \Big( \f{v_\rho}{\rho} \Big) \,\p_\rho K\bigg| + \f{1}{\rho} \bigg| \p_{\rho} \Big( \f{v_\rho}{\rho} \Big)\, \f{1}{\rho}\p_\phi K\bigg|  \,dx\,dt,  \\
	S_{32} &=& \int_{0}^{T}\int_{D_m} \f{1}{\rho} \bigg| \f{1}{\rho} \p_{\phi} \Big( \f{v_\phi}{\rho} \Big) \,\p_\rho F\bigg| + \f{1}{\rho} \bigg| \p_{\rho} \Big( \f{v_\phi}{\rho} \Big) \,\f{1}{\rho}\p_\phi F\bigg|  \,dx\,dt,  \\
	S_{33} &=& \f{4}{\sqrt{3}} \int_{0}^{T}\int_{D_m} \frac{1}{\rho^2}\, | K\O | \,dx\,dt  + \f{4}{3} \int_{0}^{T}\int_{D_m} \frac{1}{\rho^2}\, | F\O | \,dx\,dt.
	\end{eqnarray*}
By using Cauchy-Schwarz inequality,
	\begin{align*}
	S_{31} & \leq \int_{0}^{T}\int_{D_m} \bigg| \f{1}{\rho}\nabla \Big( \f{v_\rho}{\rho} \Big) \bigg|\, |\nabla K | \,dx\,dt  \\
	&\leq  \int_{0}^{T} \bigg\| \f{1}{\rho} \nabla \Big( \f{v_\rho}{\rho}(\cdot, t) \Big) \bigg\|_{L^2(D_m)}\, \| \nabla K(\cdot, t) \|_{L^2(D_m)} \,dt.
	\end{align*}
Now applying Lemma \ref{Lemma, e-trans-v_rho} and Cauchy-Schwarz inequality,  we deduce
	\begin{align}\label{S_31 est}
	S_{31} &\leq \sqrt{44} \int_{0}^{T} \| \nabla \O(\cdot, t) \|_{L^2(D_m)}  \| \nabla K(\cdot, t) \|_{L^2(D_m)} \,dt \nonumber \\
	&\leq \e_4 \int_{0}^{T}\int_{D_m} | \nabla K|^2 \,dx\,dt + \frac{11}{\e_4} \int_{0}^{T}\int_{D_m} | \nabla \O|^2 \,dx\,dt,
	\end{align}
where $\e_4$ is any positive number.  In a similar manner by using Cauchy-Schwarz inequality and Lemma \ref{Lemma, e-trans-v_phi},  we have
	\be\label{S_32 est}
	S_{32} \leq \e_5 \int_{0}^{T}\int_{D_m} | \nabla F|^2 \,dx\,dt + \frac{100}{\e_5} \int_{0}^{T}\int_{D_m} | \nabla \O|^2 \,dx\,dt,
	\ee
where $\e_5$ is any positive number.  For $S_{33}$, it directly follows from Cauchy-Schwarz inequality that
	\be\label{S_33 est-1}
	S_{33} \leq \int_{0}^{T}\int_{D_m} \f{K^2}{\rho^2} \,dx\,dt  + \int_{0}^{T}\int_{D_m} \f{F^2}{\rho^2} \,dx\,dt + \f{16}{9}\int_{0}^{T}\int_{D_m} \f{\O^2}{\rho^2} \,dx\,dt.
	\ee
	
Based on (\ref{K-P energy est}), we know
	\[
	\int_{0}^{T}\int_{D_m} \f{K^2}{\rho^2} \,dx\,dt \leq C_{\a,B} \int_{0}^{T}\int_{D_m} \bigg(\f{\p_{\phi} K}{\rho}\bigg)^2 \,dx\,dt \leq \f{3}{25} \int_{0}^{T}\int_{D_m} | \nabla K|^2 \,dx\,dt.
	\]
Similarly,
	\[
	\int_{0}^{T}\int_{D_m} \f{\O^2}{\rho^2} \,dx\,dt \leq \f{3}{25} \int_{0}^{T}\int_{D_m} | \nabla \O|^2 \,dx\,dt.
	\]
On the other hand, according to (\ref{F-P energy est}) and the assumption that $ \a\in \big(0, \frac{\pi}{6}\big]$, we attain
	\[ \int_{0}^{T}\int_{D_m} \f{F^2}{\rho^2} \,dx\,dt \leq C_{\a,A} \int_{0}^{T}\int_{D_m} \bigg(\f{\p_{\phi} F}{\rho}\bigg)^2 \,dx\,dt \leq \f{2}{19} \int_{0}^{T}\int_{D_m} | \nabla F|^2 \,dx\,dt. \]
Plugging the above estimates into (\ref{S_33 est-1}) yields
	\be\label{S_33 est-2}
	S_{33} \leq \f{3}{25} \int_{0}^{T}\int_{D_m} | \nabla K|^2 \,dx\,dt + \f{2}{19} \int_{0}^{T}\int_{D_m} | \nabla F|^2 \,dx\,dt  + \f{16}{75} \int_{0}^{T}\int_{D_m} | \nabla \O|^2 \,dx\,dt.
	\ee
Putting (\ref{S_31 est}), (\ref{S_32 est}) and (\ref{S_33 est-2}) into (\ref{S_3 est-1}) leads to
	\[\begin{split}
	|S_3| &\leq \f{2C^*}{\sqrt{3}} \Bigg[ \Big(\e_4 + \f{3}{25}\Big) \int_{0}^{T}\int_{D_m} | \nabla K|^2 \,dx\,dt + \Big( \e_5 + \f{2}{19} \Big)  \int_{0}^{T}\int_{D_m} | \nabla F|^2 \,dx\,dt  \\
	&\qquad\quad + \Big( \f{11}{\e_4} + \f{100}{\e_5} + \f{16}{75} \Big) \int_{0}^{T}\int_{D_m} | \nabla \O|^2 \,dx\,dt \Bigg].
	\end{split}\]
By choosing $\e_4=\e_5=3$,  we derive from the above inequality that
	\[
	|S_3| \leq \f{2C^*}{\sqrt{3}} \bigg( 4\int_{0}^{T}\int_{D_m} | \nabla K|^2 + |\nabla F|^2 \,dx\,dt + 40 \int_{0}^{T}\int_{D_m} | \nabla \O|^2 \,dx\,dt \bigg).
	\]
Recalling $C^{*}=\f{1}{95}$, so the above estimate implies that
	\be\label{S_3  est-final}
	|S_3| \leq \f{1}{20} \int_{0}^{T}\int_{D_m} | \nabla K|^2 + |\nabla F|^2 \,dx\,dt + \f{1}{2} \int_{0}^{T}\int_{D_m} | \nabla \O|^2 \,dx\,dt.
	\ee

Finally, by plugging (\ref{S_1+S_2, e-ineq-3}) and (\ref{S_3  est-final}) into (\ref{energy est-eq 1}), we conclude that
	\[\begin{split}
	& \frac12 \int_{D_m} (K^2 + F^2 + \O^2)(x,T) \,dx - \frac12 \int_{D_m} (K^2 + F^2 + \O^2)(x,0) \,dx \\
	& + \frac{1}{20} \int_{0}^{T}\int_{D_m} | \nabla K |^2 + | \nabla F |^2 + | \nabla \O |^2 \,dx\,dt \leq 0,
	\end{split}\]
which implies (\ref{energy est for KFO}).

Now it remains to justify (\ref{KFO-v0}), we first use the Poincar\'e inequality in Corollary \ref{Cor, P-sine-bdry} and the fact that $ \o_\rho=0 $ on $ \p^{R} D_m $ to establish
\[ \int_{D_m} K^2(x,0)\,dx = \int_{D_m} \frac{\o_{0,\rho}^2}{\rho^2}\,dx \leq C\int_{D_m} \Big| \frac{1}{\rho}\p_\phi \o_{0,\rho}\Big|^2\,dx \leq C \|v_0\|_{H^2(D_m)}^2, \]
where $ C=C(\a)$. Then the term $ \int_{D_m} \O^2(x,0)\,dx $ can be handled in the same way. In order to treat $ F $, we take advantage of the property that $ \o_\phi $ is odd with respect to $ \{\phi=\frac{\pi}{2}\} $ and then use the Poincar\'e inequality in Corollary \ref{Cor, P-sine-ave} to obtain
\[ \int_{D_m} F^2(x,0)\,dx = \int_{D_m} \frac{\o_{0,\phi}^2}{\rho^2}\,dx \leq C\int_{D_m} \Big| \frac{1}{\rho}\p_\phi \o_{0,\phi}\Big|^2\,dx \leq C \|v_0\|_{H^2(D_m)}^2. \]
Hence, (\ref{KFO-v0}) is verified.
\end{proof}

\subsection{A Uniform bound for $\|v/\rho\|_{L^\infty_t L^6_x}$ }
\label{Subsec, LtwLx6}
\quad

In the previous Section \ref{Subsec, est on v_rho/rho} and Section \ref{Subsec, est on v_phi/rho}, we have used the first two relations in the Biot-Savart law (\ref{Biot-Savart-sph}) to obtain estimates on some norms about $ v_\rho $ and $v_\phi$ in Lemma \ref{Lemma, e-trans-v_rho} and Lemma \ref{Lemma, e-trans-v_phi} via $ \o_\th $. Now we will use the third relation in (\ref{Biot-Savart-sph}) to deduce similar estimates about $v_\th$ via $ \o_\rho $ and $ \o_\phi $.

\begin{lemma}\label{Lemma, ee for vtdr}
	Let the region $ D_m $ be as defined in (\ref{app domain-sph}) with $ m\geq 2 $ and the angle $\alpha \in (0,\frac{\pi}{6}]$. Then for any $ T>0 $ and for a.e $t\in [0,T]$,
	\begin{align}
	\left\|\nabla \Big(\frac{v_\th}{\rho}(\cdot, t) \Big) \right\|_{L^2(D_m)}  \leq 5\left(\|K(\cdot, t)\|_{L^{2}\left(D_{m}\right)}+\|F(\cdot, t)\|_{L^{2}\left(D_{m}\right)}\right), \label{est-vthor} \\
	\left\|\frac{1}{\rho} \nabla\Big(\frac{v_{\th}}{\rho}(\cdot, t)\Big)\right\|_{L^2(D_{m})} \leq 2 \sqrt{3}\left(\|\nabla K(\cdot, t)\|_{L^{2}(D_m)} + \|\nabla F(\cdot, t)\|_{L^{2}(D_m)}\right).
	\label{est-vthor2}
	\end{align}

\end{lemma}

	\begin{proof}
	Since $ v\in E^{\sigma,s}_{m,T} \cap L_t^2 H_x^{2} \cap L_{tx}^\infty(D_m\times[0,T]) $ and $ \rho $ has the lower bound $ \frac1m $ on $ D_m $, we know $ \o_\rho, \o_\phi, K, F\in L_t^2 H_x^{1}(D_m\times[0,T])$. So there exists a set $ S_T\subset[0,T] $ such that $ [0,T]\setminus S_T $ has measure 0 and for any $ t\in S_T $, all of $ \o_\rho, \o_\phi, K, F $ belong to $ H^{1}(D_m) $. Fixing any $ t\in S_T $, it suffices to prove (\ref{est-vthor}) for such $ t $. For convenience of notation, we will drop all the temporal variables in the following proof.

	Recall that the third equation in the Biot-Savart law (\ref{Biot-Savart-sph}) reads
	$$
	\left(\Delta-\frac{1}{\rho^{2} \sin ^{2} \phi}\right) v_{\theta} = -\frac{1}{\rho} \partial_{\rho}(\rho \omega_{\phi} ) + \frac{1}{\rho} \partial_{\phi} \omega_{\rho} \quad \text{in} \quad D_m.
	$$
	Denote $g=\frac{v_{\theta}}{\rho}$. Then it follows from the above equation and the boundary condition for $ v_\th $ in Lemma \ref{Lemma, bdry cond} that
	\be\label{eq for vtdr}
	\begin{cases}
		\left(\Delta+\frac{2}{\rho} \partial_{\rho}+\frac{1-\cot ^{2} \phi}{\rho^{2}}\right) g  =-\frac{1}{\rho^{2}} \partial_{\rho}(\rho^{2} F) + \frac{1}{\rho} \partial_{\phi} K \quad \text{ in } \quad D_{m}, \\
		\partial_{\phi} g=- (\cot \phi) g \quad \text { on } \quad \partial^{R} D_{m}, \quad
		\partial_{\rho} g=-\frac{2}{\rho} g \quad \text { on } \quad \partial^{A} D_{m} .
	\end{cases}
	\ee
	
	Testing (\ref{eq for vtdr}) by $g$ on $D_{m}$, then it follows from the integration by parts and the previous trick of converting boundary integrals into interior integrals that 	
	\be\label{ee for vtdr}
	\begin{aligned}
		& \int_{D_m}|\nabla g|^{2} d x + \int_{D_{m}} \frac{\cot ^{2} \phi}{\rho^{2}} g^{2} \,d x \\
		=&-2 \int_{D_{m}} \frac{1}{\rho} g \partial_{\rho} g \,d x-2 \int_{D_{m}} \frac{\cot \phi}{\rho^{2}} g \partial_{\phi} g \,d x +\int_{D_{m}} \bigg( \frac{K}{\rho} \partial_{\phi} g + \frac{K \cot \phi}{\rho} g - F \partial_{\rho} g \bigg) \,d x.
	\end{aligned}
	\ee	
	By Cauchy-Schwarz inequality, for any $\epsilon_{1}, \epsilon_{2}, \epsilon_{3}>0$,
	$$
	\begin{aligned}
		\text { RHS of (\ref{ee for vtdr})} &\leq \epsilon_{1} \int_{D_{m}}(\partial_\rho g)^{2} \,d x + \frac{1}{\epsilon_{1}} \int_{D_{m}} \frac{g^{2}}{\rho^{2}} \,d x + \epsilon_{2} \int_{D_{m}}\left(\frac{1}{\rho} \partial_{\phi} g\right)^{2} \,d x + \frac{1}{\epsilon_{2}} \int_{D_{m}} \frac{\cot^{2} \phi}{\rho^{2}} g^{2} \,d x\\
		&\quad + \epsilon_{3} \int_{D_{m}} \left[ \left(\frac{1}{\rho} \partial_{\phi} g\right)^{2}+\frac{\cot^{2} \phi}{\rho^{2}} g^{2} + \left(\partial_{\rho} g\right)^{2} \right] d x + \frac{1}{\epsilon_{3}} \int_{D_{m}} \left( \frac12 K^{2} + \frac14 F^{2} \right) d x .
	\end{aligned}
	$$
	Since $ g $ is odd with respect to the plane $ \{\phi=\frac{\pi}{2}\} $, $ \int_{\pi/2-\alpha}^{\pi/2+\alpha} g(\rho,\phi) \,d\phi = 0$ for any $ \rho $. As a consequence, it follows from the weighted Poincar\'e inequality in Corollary \ref{Cor, P-sine-ave} that for any $0<\alpha \leq \frac{\pi}{6}$,
	$$
	\begin{aligned}
		&\int_{D_m} \frac{g^{2}}{\rho^{2}} \,d x \leq \frac{2}{19} \int_{D_{m}}\left(\frac{1}{\rho} \,\partial_{\phi} g\right)^{2} d x \\
		&\int_{D_{m}} \frac{\cot ^{2} \phi}{\rho^{2}} g^{2} \,d x \leq \tan ^{2} \alpha \int_{D_{m}} \frac{g^{2}}{\rho^{2}} d x \leq \frac{2}{57} \int_{D_{m}}\left(\frac{1}{\rho} \partial_{\phi} g\right)^{2} d x.
	\end{aligned}
	$$
	Therefore.
	$$
	\begin{aligned}
		\text { RHS of  (\ref{ee for vtdr}) } & \leq \left(\epsilon_{1}+\epsilon_{3}\right) \int_{D_m}\left(\partial_{\rho} g\right)^{2} d x + \left(\frac{2}{19 \epsilon_{1}} + \epsilon_{2} + \frac{2}{57 \epsilon_{2}} + \frac{59}{57} \epsilon_{3}\right) \int_{D_{m}}\left(\frac{1}{\rho} \partial_{\phi} g\right)^{2} d x \\
		&\quad + \frac{1}{\epsilon_{3}} \int_{D_{m}} \left( \frac12 K^{2} + \frac14 F^{2}\right) \,d x.
	\end{aligned}
	$$
	Choosing $\epsilon_{1}=\frac{3}{5}, \epsilon_{2}=\frac{1}{5}$ and $\epsilon_{3}=\frac{1}{10}$.
	Then
	\[
	\text{RHS of (\ref{ee for vtdr}) } \leq 0.7 \int_{D_{m}}| \nabla g |^{2} d x + 5 \int_{D_{m}} ( K^{2} + F^{2}) \,d x.
	\]
	As a result,
	\[ \int_{D_{m}}|\nabla g |^{2} \,d x \leq \frac{50}{3} \int_{D_{m}} ( K^2 + F^2) \,d x, \]
	which implies (\ref{est-vthor}).
	
	Next, we will verify (\ref{est-vthor2}). Testing (\ref{eq for vtdr}) by $\frac{1}{\rho^{2}} g$ on $D_{m}$, then it follows from the integration by parts and the previous trick of converting boundary integrals into interior integrals that 	
	\[
	\begin{split}
		\int_{D_m} \left | \frac{1}{\rho} \nabla g \right |^{2} \,d x &= \int_{D_{m}} \frac{4-\cot ^{2} \phi}{\rho^{4}} g^{2} d x - 2 \int_{D_{m}} \frac{\cot \phi}{\rho^{4}} g \partial_{\phi} g d x \\
		& \quad + \int_{D_m}\left( \frac{1}{\rho} \partial_{\rho} g - \frac{2 g}{\rho^{2}}\right) \frac{F}{\rho} \,dx + \int_{D_m} \left(\frac{1}{\rho^{2}} \p_{\phi} g + \frac{\cot \phi}{\rho^2} g\right) \frac{K}{\rho} \,d x.
	\end{split}	\]
	Then for any $\epsilon_{1}, \epsilon_{2} \in (0,1)$, we apply Cauchy-Schwarz inequality to obtain
	\be\label{vthe2-1}
	\begin{aligned}
		\int_{D_{m}}\left|\frac{1}{\rho} \nabla g\right|^{2} \,d x & \leq
		4 \int_{D_{m}} \frac{g^{2}}{\rho^{4}} \,d x + \Big(\frac{1}{\epsilon_1}-1\Big) \int_{D_{m}} \frac{\cot^2\phi }{\rho^{4}}g^{2} \,d x + \epsilon_{1} \int_{D_{m}}\left(\frac{1}{\rho^{2}} \partial_{\phi} g\right)^{2} \,d x\\
		&\quad + \epsilon_{2} \int_{D_{m}} \left[ \left(\frac{1}{\rho} \partial_{\rho} g-\frac{2 g}{\rho^{2}}\right)^{2} + \left(\frac{1}{\rho^{2}} \partial_{\phi} g + \frac{\cot \phi}{\rho} g\right)^{2} \right] \,d x + \frac{1}{4 \epsilon_{2}} \int_{D_{m}} \frac{F^2+K^2}{\rho^2} \,d x.
	\end{aligned}
	\ee
	By Cauchy-Schwarz inequality again,
	\be\label{vthe2-2} \begin{split}
		& \epsilon_2 \int_{D_{m}} \left[ \left(\frac{1}{\rho} \partial_{\rho} g-\frac{2 g}{\rho^{2}}\right)^{2} + \left(\frac{1}{\rho^{2}} \partial_{\phi} g + \frac{\cot \phi}{\rho} g\right)^{2} \right] \,d x \\
		\leq\,\, & 2 \epsilon_{2} \int_{D_{m}} \left[ \left(\frac{1}{\rho} \partial_{\rho} g\right)^{2} + \frac{4 g^{2}}{\rho^{4}} + \left(\frac{1}{\rho^{2}} \p_{\phi} g\right)^{2} + \frac{\cot ^{2} \phi}{\rho^4} g^2 \right] \,dx.
	\end{split} \ee
	Plugging (\ref{vthe2-2}) into (\ref{vthe2-1}) yields
	\be\label{vthe2-3} \begin{split}
		\int_{D_{m}}\left|\frac{1}{\rho} \nabla g\right|^{2} \,d x & \leq
		2\epsilon_2  \int_{D_{m}}\left(\frac{1}{\rho} \p_{\rho} g\right)^{2} \,dx + (\epsilon_1+2\epsilon_2)\int_{D_{m}}\left(\frac{1}{\rho^{2}} \partial_{\phi} g\right)^{2} \,d x + (4+8\epsilon_2)\int_{D_m} \frac{g^{2}}{\rho^{4}} \,d x \\
		&\quad + \Big(\frac{1}{\epsilon_1}-1 + 2\epsilon_2\Big) \int_{D_{m}} \frac{\cot^2\phi }{\rho^{4}}g^{2} \,d x +  \frac{1}{4 \epsilon_{2}} \int_{D_{m}} \frac{F^2+K^2}{\rho^2} \,d x.
	\end{split}\ee
	By choosing $ \epsilon_{1}=\frac{1}{5}$  and $\epsilon_{2}=\frac{1}{40}$, and noticing $ \cot^2\phi\leq \frac13 $, we obtain
	\be\label{vthe2-4} \begin{split}
		\int_{D_{m}}\left|\frac{1}{\rho} \nabla g\right|^{2} \,d x & \leq
		\frac{1}{20}  \int_{D_{m}}\left(\frac{1}{\rho} \p_{\rho} g\right)^{2} \,dx + \frac14 \int_{D_{m}}\left(\frac{1}{\rho^{2}} \partial_{\phi} g\right)^{2} \,d x + 6\int_{D_m} \frac{g^{2}}{\rho^{4}} \,d x + 10 \int_{D_{m}} \frac{F^2+K^2}{\rho^2} \,d x.
	\end{split}\ee
	Since $ v_\th $ is odd with respect to $\{\phi=\frac{\pi}{2}\}$, it then follows from the Poincar\'e inequality in Lemma \ref{Lemma, P ave} that
	\be\label{vthe2-5}\begin{split}
	 \int_{D_m} \frac{g^{2}}{\rho^{4}} \,d x  \leq \frac{2}{19} \int_{D_{m}}\left(\frac{1}{\rho^{2}} \partial_{\phi} g\right)^{2} \,d x.
	\end{split}\ee
	Putting (\ref{vthe2-5}) into (\ref{vthe2-4}) leads to
	\[ \int_{D_{m}}\left|\frac{1}{\rho} \nabla g\right|^{2} \,d x \leq \frac{9}{10} \int_{D_{m}}\left|\frac{1}{\rho} \nabla g\right|^{2} \,d x  + 10 \int_{D_{m}} \frac{F^2+K^2}{\rho^2} \,d x, \]
	which implies
	\[\int_{D_{m}}\left|\frac{1}{\rho} \nabla g\right|^{2} \,d x \leq 100 \int_{D_{m}} \frac{F^2+K^2}{\rho^2} \,d x. \]
	Combining this estimate with Poincar\'e inequalities in Lemma \ref{Lemma, P ave} and Lemma \ref{Lemma, P bdry}, we find
	\[ \begin{split}
		\int_{D_{m}}\left|\frac{1}{\rho} \nabla g\right|^{2} \,dx &\leq 100\Bigg[ \frac{2}{19}\int_{D_{m}}\left(\frac{1}{\rho} \partial_{\phi} F\right)^{2} \,dx + \frac{3}{25}\int_{D_{m}}\left(\frac{1}{\rho} \partial_{\phi} K\right)^{2} \,dx \Bigg].
	\end{split} \]
	So
	\[
	\int_{D_{m}}\left|\frac{1}{\rho} \nabla g \right|^{2} \,dx \leq 12\bigg(\int_{D_{m}} |\nabla F|^{2} \,dx + \int_{D_{m}} |\nabla K|^{2} \,dx \bigg),
	\]
	which results in (\ref{est-vthor2}).
	\end{proof}

Before estimating the $ L_t^\infty L_x^6 $ norms of $ v_\rho/\rho $, $ v_\phi/\rho $ and $ v_\th/\rho $, we need a uniform Sobolev embedding on regions $ \{ D_m \}_{m\geq 2} $. The key point here is that the embedding constant $ s_0 $ in (\ref{sobolev emb}) is independent of $ m $.  Since the regions $ \{ D_m\}_{m\geq 2} $ are Lipschitz and their limiting region, as $ m\to\infty $, is also Lipschitz,
the embedding result is essentially known.  But for completeness, we still give a short illustration based on \cite{AF03}.

\begin{lemma}\label{Lemma, sobolev emb}
	Let $D_{m}$ be the region in (\ref{app domain-sph}) with $m \geq 2$ and the angle $\alpha \in\big(0, \frac{\pi}{6}\big]$. Then there exist two constants $ s_0 $ and $ s_1 $, which depend on $ \a $ but are independent of $ m $, such that the following two estimates hold.
	\begin{enumerate}[(a)]
		\item For any $f \in W^{1,2}\left(D_{m}\right)$, $ \|f\|_{L^{6}\left(D_m\right)} \leq s_{0}\| f\|_{H^{1}\left(D_{m}\right)} $.
		
		\item For any $f \in W^{1,2}\left(D_{m}\right)$ such that either $f=0$ on $\partial^{R} D_{m}$ or $\int_{\pi / 2 - \alpha}^{\pi / 2 + \alpha} f(\rho, \phi) \sin \phi \, d \phi=0$ for any $ \rho\in\big(\frac1m, 1\big)$,
		\be\label{sobolev emb}
		\|f\|_{L^{6}\left(D_m\right)} \leq s_{1}\|\nabla f\|_{L^{2}\left(D_{m}\right)}.
		\ee
	\end{enumerate}
\end{lemma}

\begin{proof}
	Recall the {\it cone condition} in Definition 4.6 on Page 82 in \cite{AF03}: a domain $ \Omega $ satisfies the {\it cone condition} if there exists a finite cone $ C $ such that each $ x\in\Omega $ is the vertex of a finite cone $ C_x $ contained in $ \Omega $ and congruent to $ C $.
	
	Based on the above definition, it is readily seen that for any $ m\geq 2 $, $ D_m $ satisfies the cone condition. Moreover, the cone $ C $ in the cone condition for $ D_m $ can be chosen as a uniform one (i.e. independent of $ m $) since all $ D_m $ share the same angle $ \alpha $.
	
	Now we recall Theorem 4.12 (Part I, Case C) on Page 85 in \cite{AF03} which implies that if $ \O\in\bR^3 $ satisfies the cone condition, then $ H^{1}(\O) $ is embedded in $ L^{6}(\O) $, where the embedding constant only depends on the dimensions of the cone $ C $ in the cone condition.
	
	Thanks to this theorem and the fact that the cone $ C $ in the cone condition for $ D_m $ is uniform,
	we can find a constant $ s_0 $, which only depends on $ \a $, such that
	\be\label{se}
	\| f \|_{L^6(D_m)} \leq s_0 \| f \|_{H^1(D_m)}.\ee
	This justifies part (a).
	
	For part (b), due to the extra assumption (i) or (ii) and the restriction $ \a\leq \pi/6 $, we are allowed to apply Poincar\'e inequality in the $ \phi $ direction to conclude
	\begin{align*}
		\int_{D_m} f^2(x) \,dx &= 2\pi \int_{1/m}^{1} \rho^2 \int_{\pi/2-\a}^{\pi/2+\a} f^{2}(\rho,\phi) \sin\phi \,d\rho\,d\phi \\
		&\leq 2\pi \lam_1 \int_{1/m}^{1} \rho^2 \int_{\pi/2-\a}^{\pi/2+\a} (\p_\phi f)^{2}(\rho,\phi) \sin\phi \,d\rho\,d\phi \\
		&\leq \lam_1 \int_{D_m} |\nabla f (x)|^2 \,dx,
	\end{align*}
	where $ \lam_1$ is a constant that only depends on $ \a $. Combining this inequality with (\ref{se}) leads to (\ref{sobolev emb}).
\end{proof}

Now we can take advantage of the above Sobolev embedding to control the $ L_t^\infty L_x^6 $ norms of $ v_\rho/\rho $, $ v_\phi/\rho $ and $ v_\th/\rho $.

\begin{lemma}\label{Lemma, LtwLx6}
	Let the region $ D_m $ be as defined in (\ref{app domain-sph}) with $ m\geq 2 $ and the angle $\alpha \in \big(0,\frac{\pi}{6}\big]$. Then there exists some constant $ C=C(\a)$ such that for any $ T>0 $,
	\be\label{LtwLx6}
	\left\|\frac{|v_{\rho}| + |v_{\phi}| + |v_{\th}|}{\rho}\right\|_{L_{t}^{\infty} L_{x}^{6} (D_m \times[0, T])} \leq C \big\| |K| + |F| + |\O| \big\|_{L_t^\infty L_x^2(D_m\times[0,T])}.
	\ee
\end{lemma}

\begin{proof}
Firstly, it follows from Lemma \ref{Lemma, e-trans-v_rho}, Lemma \ref{Lemma, e-trans-v_phi} and Lemma \ref{Lemma, ee for vtdr} that for a.e. $ t\in[0,T] $,
\[\begin{split}
	& \left\|\nabla \Big(\frac{v_\rho}{\rho}(\cdot, t) \Big) \right\|_{L^2(D_m)} + \left\|\nabla \Big(\frac{v_\phi}{\rho}(\cdot, t) \Big) \right\|_{L^2(D_m)} + \left\|\nabla \Big(\frac{v_\th}{\rho}(\cdot, t) \Big) \right\|_{L^2(D_m)} \\
	\leq\,\, & 5\left( \| \O(\cdot, t) \|_{L^2(D_m)} + \|K(\cdot, t)\|_{L^{2}\left(D_{m}\right)} + \|F(\cdot, t)\|_{L^{2}\left(D_{m}\right)}\right).
\end{split}\]	
Next, due to the property (\ref{mean0NA}) of $ v_\rho $, the boundary condition of $ v_\phi $ on $ \p^{R} D_m $, and the odd-symmetry of $ v_\th $ with respect to $ \{\phi=\frac{\pi}{2}\} $, we can apply part (b) in Lemma \ref{Lemma, sobolev emb} to conclude
\[\begin{split}
& \left\| \frac{v_\rho}{\rho}(\cdot, t) \right\|_{L^6(D_m)} + \left\| \frac{v_\phi}{\rho}(\cdot, t) \right\|_{L^6(D_m)} + \left\| \frac{v_\th}{\rho}(\cdot, t) \right\|_{L^6(D_m)} \\
\leq\,\, & C \left( \left\|\nabla \Big(\frac{v_\rho}{\rho}(\cdot, t) \Big) \right\|_{L^2(D_m)} + \left\|\nabla \Big(\frac{v_\phi}{\rho}(\cdot, t) \Big) \right\|_{L^2(D_m)} + \left\|\nabla \Big(\frac{v_\th}{\rho}(\cdot, t) \Big) \right\|_{L^2(D_m)} \right).
\end{split}\]
Combining the above two estimates leads to (\ref{LtwLx6}).
\end{proof}

\subsection{Uniform bounds for $\Vert v \Vert_{L_{tx}^\infty}$ and $ \Vert \o_\th \Vert_{L_{tx}^\infty}$.}
\label{Subsec, v-ub}
\quad

The goal of this subsection is to obtain uniform bounds on $\Vert v \Vert_{L_{tx}^\infty}$ and $ \Vert \o_\th \Vert_{L_{tx}^\infty}$ which are independent of the time $ T $ and only dependent on $ \a $ and the initial value.

\subsubsection{$L^\i$ boundedness of $v_\theta$}
\quad

We first derive an upper bound for the supremum norm of $ v_\th $.

\begin{proposition}\label{Prop, vth-ub}
	Let the region $ D_m $ be as defined in (\ref{app domain-sph}) with $ m\geq 2 $ and the angle  $\alpha \in \big(0,\frac{\pi}{6}\big]$. Then for any $ T>0 $,
	\be\label{vth-usb}
	\| v_{\th} \|_{L_{tx}^{\infty}(D_{m} \times [0, T])} \leq C C_{*}^{5}\Big(\|v_{0}\|_{L^{2}\left(D_{m}\right)}+ \|v_{0,\th}\|_{L^{\infty}(D_m)} + 1 \Big),
	\ee
	where $ C=C(\a) $ and
	\be\label{vth-Cstar}
	C_{*} = 2 + \left\|\frac{ |v_{\th}| + |v_{\rho}| + |v_{\phi}|}{\rho}\right\|_{L_{t}^{\infty} L_{x}^{6} (D_m \times[0, T])}.
	\ee
\end{proposition}
	
	\begin{proof}
	Fix any $T>0$ and let $\eta:[0, T] \rightarrow[0,1]$ be a smooth function in the time-variable. The specific choice of $\eta$ will be determined later. For any rational number $ q\geq 1 $ in the form of $ (2k-1)/(2l-1) $, where $ k $ and $ l $ are positive integers, denote
	\[ f=v_\th^q. \]
	Based on equation (\ref{vth}) for $ v_\th $, we know $ f $ solves the following problem:
	\be\label{vthq}
	\begin{cases}
		\Delta f-q(q-1) v_{\theta}^{q-2}\left|\nabla v_{\theta}\right|^{2}-\frac{q}{\rho^{2} \sin ^{2} \phi} f - b \cdot \nabla f - \frac{q}{\rho} (v_\rho + \cot \phi\, v_{\phi}) f-\partial_{t} f=0 \, \text { in } \, D_m\times(0,T]; \\
		\partial_{\phi} f = -q \cot\phi\, f  \, \text { on } \, \partial^{R} D_{m} \times(0, T], \quad \partial_{\rho} f = -\frac{q}{\rho} f \,  \text { on }\, \partial^{A} D_{m} \times(0, T];\\
		f(x,0) = v_{0,\theta}^q (x), \quad x\in D_m.
	\end{cases}
	\ee
	
	For any $t \in(0, T]$, we test (\ref{vthq}) by $\eta^{2} f$ on $D_{m} \times[0, t]$.
	By using integration by parts and then converting the boundary integral into the interior integral, we find
	\[
	\begin{split}
		\int_{0}^{t} \eta^2 \int_{D_m} f \Delta f \,dx\,d\tau & = - \int_{0}^{t} \eta^2 \int_{D_m}  |\nabla f|^2 \,dx\,d\tau - 2 q \int_{0}^{t} \eta^{2} \int_{D_m} \frac{1}{\rho} f (\partial_\rho f)  \,d x \,d \tau \\
		& \qquad - 2 q \int_{0}^{t} \eta^2 \int_{D_m} \frac{\cot \phi}{\rho^{2}} f (\partial_{\phi} f) \,d x \,d \tau.
	\end{split}
	\]
	As a result, we obtain	
	\be\label{vthq-ues}
	\begin{aligned}
		& \frac{2 q-1}{q} \int_{0}^{t} \int_{D_m} | (\nabla f) \eta |^{2} \,d x \,d \tau + q \int_{0}^{t} \int_{D_m} \frac{1}{\rho^{2} \sin ^{2} \phi} f^{2} \eta^{2} \,d x \,d \tau + \frac{1}{2} \eta^{2}(t) \int_{D_m} f^{2}(x, t) \,d x  \\
		= & \underbrace{-2 q \int_{0}^{t} \int_{D_m} \frac{1}{\rho} f(\partial_\rho f) \eta^{2} \,d x \,d \tau - 2 q \int_{0}^{t} \int_{D_m} \frac{\cot \phi}{\rho^{2}} f(\partial_{\phi} f) \eta^{2} \,d x \,d \tau}_{R_1} \\
		& -q \int_{0}^{t} \int_{D_m} \frac{v_{\rho}+\cot \phi\, v_{\phi}}{\rho} f^{2} \eta^{2} \,d x \,d \tau + \frac{1}{2} \eta^{2}(0)  \int_{D_m} f^{2}(x, 0)\,d x + \int_{0}^{t} \int_{D_{m}} f^{2} \eta \eta^{\prime} \,d x \,d \tau .
	\end{aligned}
	\ee
	Note when deriving the above equation, we used the fact that $ \int_{D_m} (b\cdot \na f) f \,dx = 0$ due to the incompressibility and the boundary condition of $ b $. Using Cauchy-Schwarz inequality, we find
	\begin{align*}
		|R_1| &\leq 2 q \int_{0}^{t} \int_{D_{m}}\left|\frac{1}{\rho} f \partial_{\rho} f\right| \eta^{2} d x \,d\tau + 2 q \int_{0}^{t} \int_{D_{m}}\left|\frac{\cot \phi}{\rho^{2}} f \partial_{\phi} f\right| \eta^{2} d x \,d\tau \\
	    &\leq \frac{1}{2} \int_{0}^{t} \int_{D_{m}} | (\nabla f) \eta | ^{2} d x \,d \tau + 2 q^{2} \int_{0}^{t} \int_{D_{m}} \left( \frac{f^{2} \eta^{2}}{\rho^{2}} + \frac{\cot ^{2} \phi}{\rho^{2}} f^{2} \eta^{2}\right) d x \,d \tau.
    \end{align*}
	When $\alpha \in\big(0, \frac{\pi}{6}\big]$, $\cot ^{2} \phi \leq \tan ^{2} \alpha \leq \frac{1}{3}$, so
	\[ |R_1| \leq \frac{1}{2} \int_{0}^{t} \int_{D_{m}} | (\nabla f) \eta |^{2} \,d x \,d \tau+\frac{8}{3} q^{2} \int_{0}^{t} \int_{D_m} \frac{1}{\rho^{2}} f^{2} \eta^{2} \,d x \,d \tau.\]
	Combining with (\ref{vthq-ues}) and noticing $ \frac{2q-1}{q}\geq 1 $, we deduce
	\[
	\begin{aligned}
		& \frac{1}{2} \int_{0}^{t} \int_{D_{m}}| (\nabla f) \eta |^{2} \,d x \,d \tau+\frac{1}{2} \eta^{2}(t) \int_{D_{m}} f^{2}(x, t) \,d x \\
		\leq\,\, & \frac{8}{3} q^{2} \int_{0}^{t} \int_{D_m} \frac{1}{\rho^{2}} f^{2} \eta^{2} d x \,d \tau + q \int_{0}^{t} \int_{D_{m}} \frac{|v_{\rho}| + |v_{\phi}|}{\rho} f^{2} \eta^{2} d x \,d \tau \\
		& + \frac{1}{2} \eta^{2}(0) \int_{D_{m}} f^{2}(x, 0) d x + \int_{0}^{t} \int_{D_{m}} f^{2}\left|\eta \eta^{\prime}\right| d x \,d \tau.
	\end{aligned}
	\]
	Taking supremum norm with respect to $ t\in[0,T] $, we obtain
	\be\label{vthqcut-ee}\begin{split}
		& \frac{1}{2} \int_{0}^{T} \int_{D_{m}}| (\nabla f) \eta |^{2} \,d x \,d \tau + \frac{1}{2} \sup_{t\in[0,T]} \int_{D_{m}} f^{2}(x, t)  \eta^{2}(t) \,d x \\
		\leq\,\, & \frac{16}{3} q^{2} \int_{0}^{T} \int_{D_m} \frac{1}{\rho^{2}} f^{2} \eta^{2} d x \,d \tau + 2q \int_{0}^{T} \int_{D_{m}} \frac{|v_{\rho}| + |v_{\phi}|}{\rho} f^{2} \eta^{2} d x \,d \tau \\
		& +  \eta^{2}(0) \int_{D_{m}} f^{2}(x, 0) d x + 2\int_{0}^{T} \int_{D_{m}} f^{2}\left|\eta \eta^{\prime}\right| d x \,d \tau.
	\end{split}\ee

	Since $ v_\th $ is odd with respect to $ \{\phi = \frac{\pi}{2}\} $ and $ q $ is in the form of $ (2k-1)/(2l-1) $, where $ k $ and $ l $ are positive integers, we know $ f $ is also odd with respect to $ \{\phi = \frac{\pi}{2}\} $. Therefore, it follows from part (b) in Lemma \ref{Lemma, sobolev emb} that
	\[ \| f(\cdot, \tau) \|_{L_x^6(D_m)} \leq s_1 \big\| \nabla \big( f(\cdot, \tau) \big) \big\|_{L_x^2(D_m)}, \quad \forall\, \tau \in[0,T], \]
	where $ s_1 $ is some constant that only depends on $ \a $. Hence, it follows from (\ref{vthqcut-ee}) that
	\be\label{vthq-Le}\begin{aligned}
		& \frac{1}{s_1^2} \| f\eta \|_{L_t^2 L_x^6(D_m\times[0,T])}^2 + \| f\eta \|^2_{L_t^\infty L_x^{2}(D_m\times[0,T])} \\
		\leq \,\, & \frac{32}{3}\, q^2\, \Big\| \frac{f\eta}{\rho} \Big\|^2_{L_{tx}^2(D_m\times[0,T])} + 4q\, \bigg\| \frac{|v_\rho| + |v_\phi|}{\rho} f^2\eta^2 \bigg\|_{L^1_{tx}(D_m\times[0,T])} \\
		& + 2 \eta^2(0)\,\| f(\cdot, 0)\|^2_{L^2(D_m)} + 4\| f^2 \eta \eta'\|_{L^1_{tx}(D_m\times[0,T])}.
	\end{aligned}\ee

	Denote $ C_{*} $ as in (\ref{vth-Cstar}) and define $ h $ as
	\be\label{vthq-h}
	h = |v_\th|^q \vee 1  = (|v_\th| \vee 1)^q,
	\ee
	where $ ``\vee” $ means ``max”. Then
	$$
	\begin{aligned}
		\left\|\frac{f \eta}{\rho}\right\|_{L_{tx}^{2}\left(D_{m} \times[0, T]\right)} &=\left\|\frac{v_{\theta}}{\rho} \cdot v_{\theta}^{q-1} \eta\right\|_{ L_{tx}^{2}(D_{m} \times[0, T])} \\
		& \leq \left\|\frac{v_{\theta}}{\rho}\right\|_{L_{t}^{\infty} L_{x}^{6}(D_{m} \times[0, T])}\left\|v_{\theta}^{q-1} \eta\right\|_{L_{t}^{2} L_{x}^{3} (D_{m} \times[0, T])} \leq C_{*}\|h \eta \|_{L_{t}^{2} L_{x}^{3} (D_{m} \times[0, T])}.
	\end{aligned}
	$$
	and
	\[\begin{aligned}
		\bigg\| \frac{|v_\rho| + |v_\phi|}{\rho} f^2\eta^2 \bigg\|_{L^1_{tx}(D_m\times[0,T])} &\leq \bigg \|\frac{ |v_\rho| + |v_\phi|}{\rho}\bigg\|_{L_{t}^{\infty} L_{x}^{6} (D_{m} \times[0, T])} \|f \eta\|_{L_t^2 L_x^{12/5} (D_{m} \times[0, T])}^{2} \\
		&\leq 2C_*  \|f \eta\|_{L_t^2 L_x^{12/5} (D_{m} \times[0, T])}^{2}.
	\end{aligned}\]
	Plugging the above estimates into (\ref{vthq-Le}) yields
	\be\label{vthq-mi0}\begin{split}
	& \frac{1}{s_1^2} \| f\eta \|_{L_t^2 L_x^6(D_m\times[0,T])}^2 + \| f\eta \|^2_{L_t^\infty L_x^{2}(D_m\times[0,T])}  \\
	\leq\,\, & 12C_*^2 q^2 \|h \eta \|_{L_{t}^{2} L_{x}^{3} (D_{m} \times[0, T])}^2 + 8C_* q \|f \eta\|_{L_t^2 L_x^{12/5} (D_{m} \times[0, T])}^{2} \\
	& + 2\eta^2(0)\,\| f(\cdot, 0)\|^2_{L^2(D_m)} + 4\| f^2 \eta \eta'\|_{L^1_{tx}(D_m\times[0,T])}.
	\end{split}\ee
	where $f=v_{\theta}^{q}$ and $h = (|v_\th| \vee 1)^q$. Next, we have two cases to deal with.
	
	Case 1: $T \leq 2$. In this case, we take $\eta \equiv 1$ on $[0, T]$. Putting this $\eta$ into (\ref{vthq-mi0}), we have
	\[\begin{split}
		&\frac{1}{s_1^2} \| f \|_{L_t^2 L_x^6(D_m\times[0,T])}^2 + \|f\|^2_{L_t^\infty L_x^{2}(D_m\times[0,T])} \\
		\leq\,\, & 12C_*^2 q^2 \|h \|_{L_{t}^{2} L_{x}^{3} (D_{m} \times[0, T])}^2 + 8C_* q \|f\|_{L_t^2 L_x^{12/5} (D_{m} \times[0, T])}^{2} + 2\| f(\cdot, 0)\|^2_{L^2(D_m)}.
	\end{split}\]
	Recalling $ h = |f| \vee 1 $, so there exists a constant $ C=C(\a) $ such that
	\be\label{vthq-h-e1}
	\begin{aligned}
		& \|h\|_{L_{t}^{2} L_{x}^{6}(D_{m} \times[0, T])}^{2}+\|h\|_{L_{t}^{\infty} L_{x}^{2}\left(D_{m} \times[0, T]\right)}^{2} \\
		\leq\,\, & C\Big(  C_*^2 q^2 \|h \|_{L_{t}^{2} L_{x}^{3} (D_{m} \times[0, T])}^2 + C_* q \| h \|_{L_t^2 L_x^{12/5} (D_{m} \times[0, T])}^{2} + \| h(\cdot, 0)\|^2_{L^2(D_m)} \Big).
	\end{aligned}
	\ee
	In order to estimate the right-hand side of (\ref{vthq-h-e1}), we interpolate $ L_t^2 L_x^3 $ and $ L_t^2 L_x^{12/5} $ between $ L_t^2 L_x^6 $ and $ L_t^2 L_x^2 $, and then apply the Young's inequality. Consequently, it follows from (\ref{vthq-h-e1}) that
	\[
		\|h\|_{L_{t}^{2} L_{x}^{6}\left(D_{m} \times[0, T]\right)}^{2} + \|h\|_{L_{t}^{\infty} L_{x}^{2}\left(D_{m} \times[0, T]\right)}^{2} \leq  C \Big(C_{*}^{4} q^{4}\|h\|_{L_{t}^{2} L_{x}^{2}\left(D_{m} \times[0, T]\right)}^{2}+\|h(\cdot, 0)\|_{L^{2}(D_m)}^{2} \Big).
	\]
	Again, by applying interpolation to the left-hand side of the above estimate, we obtain
	\be\label{h-est-st}
	\|h\|_{L_{tx}^{10/3} (D_{m} \times[0, T])} \leq C \Big(C_{*}^2 q^{2}\|h\|_{L_{tx}^{2} (D_{m} \times[0, T])}+\|h(\cdot, 0)\|_{L^{2} (D_m)}\Big).
	\ee
	Since $ h = \psi^q $, where $ \psi:= |v_\th| \vee 1 $, then it follows from the above relation that
	\[\begin{split}
	\|\psi\|_{L_{t x}^{10 q / 3}\left(D_{m} \times[0, T]\right)}^{q} & \leq C C_{*}^2 q^{2}\|\psi\|_{L_{t x}^{2 q}(D_{m} \times[0, T]}^{q} + C\|\psi(\cdot, 0)\|_{L^{2q} (D_m)}^q \\
	& \leq C C_{*}^2 q^{2}\|\psi\|_{L_{t x}^{2 q}\left(D_{m} \times[0, T]\right)}^{q} + C\left|D_{m}\right|^{\frac{1}{2}}\|\psi(\cdot, 0)\|_{L^\infty (D_m)}^{q}.
	\end{split}\]
	Hence,
	\be\label{vth-mi1}
	\Big(\|\psi\|_{L_{t x}^{10 q/3} (D_m \times[0, T])} \vee \|\psi(\cdot, 0) \|_{L^{\infty} (D_m)}\Big) \leq (C C_{*}^2)^{\frac1q} q^{\frac2q} \Big( \|\psi\|_{L_{t x}^{2 q}\left(D_{m} \times[0, T]\right)} \vee \|\psi(\cdot, 0)\|_{L^\infty D_m) }\Big).
	\ee
	By choosing $ q=q_k=\big(\frac53\big)^{k} $ for $ k=0,1,2,\cdots $ in (\ref{vth-mi1}), and applying Moser's iteration, we find
	\[
	\Big( \|\psi\|_{L^\infty_{t x} (D_m \times[0, T]) } \vee \|\psi(\cdot, 0)\|_{L^\infty (D_m)} \Big) \leq C C_{*}^{5} \Big( \|\psi\|_{L_{t x}^{2} (D_{m} \times [0, T]} \vee \|\psi(\cdot, 0)\|_{L^\infty (D_m)} \Big).
	\]
	Since $ \psi = |v_\th| \vee 1 $ and $ T\leq 2 $, we deduce that
	\be\label{vth-mi-st}
	\|v_{\th}\|_{L^\infty_{t x} (D_m \times[0, T])} \leq C C_{*}^{5} \big( \|v_{\th}\|_{L_{tx}^2 (D_{m} \times[0, T])} + \|v_{\th}(\cdot, 0)\|_{L^\infty (D_m) } + 1 \big).
	\ee
	Finally, thanks to the energy estimate (\ref{en1s}), the above inequality implies that
	\be\label{vth-usb-sT}
	\|v_{\th}\|_{L^\infty_{t x} (D_m \times[0, T])}  \leq C C_{*}^{5} \big( \|v_{0} \|_{L^{2}(D_m)} + \|v_{0, \theta}\|_{L^\infty(D_m)} + 1 \big), \quad\forall\, 0<T\leq 2.
	\ee
	
	Case 2: $T>2$. In this case, we take $\eta \in C^{\infty}([0, T])$ such that $ 0\leq \eta\leq 1 $ and
	$$
	\eta(t) = \left\{\begin{array}{lll}
		0, & \text { if } & 0 \leq t \leq T-2, \\
		1, & \text { if } & T-1 \leq t \leq T .
	\end{array}\right.
	$$
	Putting this $\eta$ into (\ref{vthq-mi0}), we know
		\[\begin{split}
		&\frac{1}{s_1^2}\|f \eta\|_{L_{t}^{2} L_{x}^{6} (D_m \times[T-2, T])}^{2} + \|f\eta \|_{L_t^\infty L_x^2 (D_m\times[T-2,T])}^{2} \\
		\leq\,\, & 12 C_{*}^2 q^{2}\|h \eta\|_{L_t^2 L_x^{3} (D_{m} \times[T-2, T] )}^2 + 8 C_{*} q\|f \eta\|^2_{L_t^2 L_x^{12 / 5} (D_m \times[T-2, T])} + 4 \|f^{2} \eta \eta' \|_{L_{tx}^{1} (D_m \times[T-2, T])}.
	\end{split}\]	
	Then similar to the derivation of (\ref{h-est-st}), we know there exists some constant $C=C(\alpha)$ such that
	\[
	\|h \eta\|^2_{L_{tx}^{10 / 3}(D_{m} \times[T-2, T])} \leq C\Big(C_{*}^4 q^4 \|h \eta\|_{L_{tx}^2 (D_m \times[T-2, T])}^2 + \|h^{2} \eta  \eta' \|_{L_{tx}^{1} (D_m \times[T-2, T])} \Big).
	\]
	Recalling $ h=\psi^q $, where $ \psi = |v_\th| \vee 1 $, so
	\be\label{vth-mi2}
	\| \psi^q \eta \|^2_{L_{tx}^{10 / 3}(D_{m} \times[T-2, T])} \leq C\Big(C_{*}^4 q^4 \| \psi^q \eta\|_{L_{tx}^2 (D_m \times[T-2, T])}^2 + \|\psi^{2q} \eta  \eta' \|_{L_{tx}^{1} (D_m \times[T-2, T])} \Big).
	\ee
	For $k=0,1,2, \cdots$, we denote $q_{k}=\big(\frac53\big)^{k}$, $T_{k}=T-1-2^{-k}$. Meanwhile, we define $ \eta_k\in C^\infty([0,T]) $ such that $ 0\leq \eta_k \leq 1 $,
	\[
	\eta_{k}(t) = \left\{\begin{array}{lll}
		0, & \text { if } \quad 0\leq t \leq T_{k}. \\
		1, & \text { if } \quad T_{k+1} \leq t \leq T,
	\end{array}\right.
	\]
	and $\sup\limits _{t \in [0,T]} |\eta_{k}'(t)| \leq 2^{k+2}$. 	
	Plugging $q=q_{k}$ and $\eta=\eta_{k}$ into (\ref{vth-mi2}), we find
	\[
	\| \psi^{q_k} \|^2_{L_{tx}^{10 / 3}(D_{m} \times[T_{k+1}, T])} \leq C\Big(C_{*}^4 q_k^4 \| \psi^{q_k}\|_{L_{tx}^2 (D_m \times[T_k, T])}^2 + 2^{k+2} \|\psi^{2q_k} \|_{L_{tx}^{1} (D_m \times[T_k, T])} \Big).
	\]
	Therefore,
	\[
	\|\psi\|^{2q_k}_{L_{t x}^{2 q_{k+1}} (D_m \times [T_{k+1}, T])} \leq C C_{*}^{4} q_k^4 \|\psi\|^{2q_k}_{L_{t x}^{2q_k} (D_m \times [T_k, T])}.
	\]
	Now we can apply Moser's iteration to obtain
	\be\label{psi-sub-lT}
	\|\psi\|_{L^\infty_{tx}(D_m\times[T-1,T])} \leq C C_*^{5}\|\psi\|_{L^2_{tx}(D_m\times[T-2,T])}.
	\ee
	This implies that
	\be\label{vth-mi-lt}
	\|v_{\th}\|_{L^\infty_{t x} (D_m \times[T-1, T])}  \leq CC_*^5 \big(
	\|v_\th\|_{L_{tx}^2(D_m\times[T-2,T])} + \| 1 \|_{L_{tx}^2(D_m\times[T-2,T])} \big).
	\ee	
	Taking advantage of the energy estimate (\ref{en1s}) again, we deduce from (\ref{vth-mi-lt}) that
	\be\label{vth-usb-lT}
	\|v_{\th}\|_{L^\infty_{t x} (D_m \times[T-1, T])}  \leq C C_{*}^{5} \big( \|v_{0} \|_{L^{2}(D_m)} + 1 \big), \quad\forall\, T> 2.
	\ee	

	Finally, by combining (\ref{vth-usb-sT}) in Case 1 and (\ref{vth-usb-lT}) in Case 2 together, we have justified (\ref{vth-usb}).
	\end{proof}

\subsubsection{$L^\i$ boundedness of $\o_\theta$}
\quad	

In this subsection, we will prove the $ L^\infty $ bound of $ \o_\th $ which is needed to establish the $ L^\infty $ bounds of $ v_\rho $ and $ v_\phi $ in the next subsection.

\begin{proposition}\label{Prop, oth-ub}
	Let the region $ D_m $ be as defined in (\ref{app domain-sph}) with $ m\geq 2 $ and  $\alpha \in (0,\frac{\pi}{6}]$. Then for any $ T>0 $,
	\be\label{oth-usb}
	\|\o_{\theta}\|_{L_{t x}^{\infty}\left(D_{m} \times [0, T]\right)} \leq C C_{*}^{10}\left(\left\|v_{0}\right\|_{L^{2}\left(D_{m}\right)} + \|\omega_{0, \theta}\|_{L^{\infty}(D_m)} + 1\right),
	\ee
	where $ C=C(\a) $ and
	\be\label{oth-Cstar}
	\begin{aligned}
		C_{*}=\max \left\{ \left\|\frac{ |v_\rho| + | v_{\phi}|}{\rho}\right\|_{L_{t}^{\infty} L_{x}^{6}\left(D_{m} \times[0, T]\right)}, \left\|v_{\theta}\right\|_{L_{tx}^{\infty} \left(D_{m} \times[0, T]\right)}, \big\| |K|+|F| \big\|_{L_{t}^{\infty} L_{x}^{2}\left(D_{m} \times[0, T]\right)}, 2 \right\}.
	\end{aligned}
	\ee
\end{proposition}
	
\begin{proof}
    Recall $ \o_\th $ satisfies (\ref{omega-th-eq}), i.e.,
	\[
	\left\{\begin{array}{l}
		\Big(\Delta-\frac{1}{\rho^{2} \sin ^{2} \phi}\Big) \omega_{\theta}-b \cdot \nabla \omega_{\theta}+\frac{1}{\rho}\left(v_{\rho}+\cot \phi\, v_{\phi}\right) \omega_{\theta} \\
		\qquad\qquad - \frac{1}{\rho^{2}} \partial_{\phi}\big(v_{\theta}^{2}\big) + \frac{\cot \phi}{\rho} \partial_{\rho}\big(v_{\theta}^{2}\big) - \partial_{t} \omega_{\theta}=0,  \quad \text { in } \quad D_{m} \times(0, T], \\
		\omega_{\theta}=0, \quad\text { on }\quad \partial D_{m} \times(0, T], \\
		\omega_{\theta}(x, 0) = \omega_{0, \theta}(x), \quad x \in D_{m}.
	\end{array} \right.
	\]
	Noticing
	\[ -\frac{1}{\rho^2} \partial_{\phi}\big(v_\theta^2\big) + \frac{\cot \phi}{\rho} \partial_{\rho}\big(v_\theta^2\big) = -\frac{2 v_\theta}{\rho} (\omega_\rho + \cot \phi\, \omega_\phi) = -2 v_\th (K + \cot \phi\, F),
	\]
	so the above equation about $ \o_\th $ can be written as
	\be\label{othe}
	\left\{\begin{array}{l}
		\Big(\Delta-\frac{1}{\rho^{2} \sin ^{2} \phi}\Big) \omega_{\theta}-b \cdot \nabla \omega_{\theta}+\frac{1}{\rho}\left(v_{\rho}+\cot \phi\, v_{\phi}\right) \omega_{\theta} \\
		\qquad\qquad -2 v_\th (K + \cot \phi\, F) - \partial_{t} \omega_{\theta}=0,  \quad \text { in } \quad D_{m} \times(0, T], \\
		\omega_{\theta}=0, \quad\text { on }\quad \partial D_{m} \times(0, T], \\
		\omega_{\theta}(x, 0) = \omega_{0, \theta}(x), \quad x \in D_{m}.
	\end{array} \right.
	\ee

	Let $\eta: [0,T]\longrightarrow [0,1]$ be a smooth function in the time variable. The specific choice of $\eta$ will be determined later. For any rational number $ q\geq 1 $ in the form of $ (2k-1)/(2l-1) $, where $ k $ and $ l $ are positive integers. Denote $f = \o_{\theta}^{q}$. Then for any $t \in(0, T]$, we test (\ref{othe}) by $q \omega_{\th}^{2q-1} \eta^{2}$ on $D_{m} \times(0, t]$ to find
	\[
	\begin{aligned}
		& \frac{2q-1}{q} \int_{0}^{t} \int_{D_m} | (\nabla f) \eta |^{2} \,d x \,d \tau + q \int_{0}^{t} \int_{D_m} \frac{1}{\rho^{2} \sin^{2} \phi} f^{2} \eta^{2} \,d x \,d \tau + \frac{1}{2} \eta^2(t) \int_{D_m} f^{2}(x, t) \,d x  \\
		=\,\, & q \int_{0}^{t} \int_{D_m} \frac{1}{\rho} (v_\rho + \cot \phi\, v_\phi) f^{2} \eta^{2} \,d x \,d \tau - 2 q \int_{0}^{t} \int_{D_m} v_{\theta} (K + \cot \phi\, F) \omega_{\th}^{2q-1} \eta^{2} \,d x \,d \tau \\
		& + \int_{0}^{t} \int_{D_m} f^{2} \eta \eta' \,d x \,d \tau + \frac{1}{2}\eta^2(0) \int_{D_m} f^{2}(x, 0)  \,d x.
	\end{aligned}
	\]
	As a consequence,
	\[
	\begin{aligned}
		& \int_{0}^{t} \int_{D_m}| (\nabla f) \eta|^{2} \,d x \,d \tau + \frac{1}{2} \eta^{2}(t) \int_{D_m} f^{2}(x, t) \,d x \\
		\leq\,\, & q \int_{0}^{t} \int_{D_m} \frac{ |v_\rho| + |v_\phi|}{\rho} f^2 \eta^2 \,d x \,d \tau + 2 q \int_{0}^{t} \int_{D_m} |v_{\theta}| \big( |K|+|F| \big) |\o_{\theta}|^{2q-1} \eta^2 \,d x \,d \tau \\
		& + \int_{0}^{t} \int_{D_m} f^2 \eta |\eta'| \,d x \,d \tau + \frac{1}{2} \eta^{2}(0) \int_{D_m} f^{2}(x,0)  \,d x.
	\end{aligned}
	\]
	Taking supremum with respect to $ t $ on $ [0,T] $, then
	\be\label{othq}
	\begin{aligned}
		& \int_{0}^{T} \int_{D_m}| (\nabla f) \eta|^{2} \,d x \,d \tau + \frac{1}{2} \sup_{t\in[0,T]} \int_{D_m} f^{2}(x, t)  \eta^{2}(t) \,d x \\
		\leq\,\, & 2q \int_{0}^{T} \int_{D_m} \frac{ |v_\rho| + |v_\phi|}{\rho} f^2 \eta^2 \,d x \,d \tau + 4 q \int_{0}^{T} \int_{D_m} |v_{\theta}| \big( |K|+|F| \big) |\o_{\theta}|^{2q-1} \eta^2 \,d x \,d \tau \\
		& + 2\int_{0}^{T} \int_{D_m} f^2 \eta |\eta'| \,d x \,d \tau + \eta^{2}(0) \int_{D_m} f^{2}(x,0)  \,d x.
	\end{aligned}
	\ee
	Since $f=0$ on $\partial D_{m}$, it follows from Lemma \ref{Lemma, sobolev emb} that
	\[
	\| f(\cdot,\tau) \eta(\tau) \|_{L_x^6(D_m)} \leq s_1 \big\| \nabla \big( f(\cdot,\tau) \eta(\tau) \big) \big\|_{L_x^2(D_m)}, \quad \forall\, \tau\in[0,T],
	\]
	where $ s_1=s_1(\a) $. Thus, it follows from (\ref{othq}) that
	\be\label{othq-h}
	\begin{aligned}
	& \frac{1}{s_1^2}\| f\eta \|^2_{L_t^2L_x^6(D_m\times[0,T])} + \frac12 \| f\eta \|^2_{L_t^\infty L_x^2(D_m\times[0,T])} \\
	\leq\,\, & 2q \bigg\| \frac{|v_\rho| + |v_\phi|}{\rho} f^2\eta^2 \bigg\|_{L^1_{tx}(D_m\times[0,T])} + 4q \big\| v_{\theta}(|K| + |F|) \omega_{\th}^{2q-1} \eta^{2}  \big\|_{L_{tx}^1(D_m\times[0,T])} \\
	& + 2\| f^2 \eta\eta' \|_{L_{tx}^1(D_m\times[0,T])} + \eta^2(0) \| f(\cdot,0) \|_{L^2(D_m)}^2.
	\end{aligned}
	\ee
	
	Denote $ C_{*} $ as in (\ref{oth-Cstar}) and define $ h $ as
	\[
	h = |\o_\th|^q \vee 1  = (|\o_\th| \vee 1)^q,
	\]
	where $ ``\vee” $ means ``max”. Then
	\[\begin{aligned}
		\bigg\| \frac{|v_\rho| + |v_\phi|}{\rho} f^2\eta^2 \bigg\|_{L^1_{tx}(D_m\times[0,T])} &\leq \bigg \|\frac{ |v_\rho| + |v_\phi|}{\rho}\bigg\|_{L_{t}^{\infty} L_{x}^{6} (D_{m} \times[0, T])} \|f \eta\|_{L_t^2 L_x^{12/5} (D_{m} \times[0, T])}^{2} \\
		&\leq C_*  \|f \eta\|_{L_t^2 L_x^{12/5} (D_{m} \times[0, T])}^{2}.
	\end{aligned}\]
	and
	\[\begin{aligned}
		&\big\| v_{\theta}(|K| + |F|) \omega_{\th}^{2q-1} \eta^{2}  \big\|_{L_{tx}^1(D_m\times[0,T])} \\
		\leq\,\, &  \|v_{\theta}\|_{L_{tx}^\infty(D_m\times[0,T])} \big\| |K| + |F| \big\|_{L_{t}^{\infty} L_{x}^{2}(D_m\times[0,T])} \|h \eta\|_{L_t^2 L_x^4(D_m\times[0,T])}^{2}  \\
		\leq\,\, & C_{*}^{2} \|h \eta\|_{L_t^2 L_x^4(D_m\times[0,T])}^{2}.
	\end{aligned}\]
	Plugging the above estimates into (\ref{othq-h}) yields
	\be\label{othq-mi0}
	\begin{aligned}
		& \frac{1}{s_1^2}\| f\eta \|^2_{L_t^2L_x^6(D_m\times[0,T])} + \frac12 \| f\eta \|^2_{L_t^\infty L_x^2(D_m\times[0,T])} \\
		\leq\,\, & 2q C_*  \|f \eta\|_{L_t^2 L_x^{12/5} (D_{m} \times[0, T])}^{2} + 4q C_{*}^{2} \|h \eta\|_{L_t^2 L_x^4(D_m\times[0,T])}^{2} \\
		& + 2\| f^2 \eta\eta' \|_{L_{tx}^1(D_m\times[0,T])} + \eta^2(0) \| f(\cdot,0) \|_{L^2(D_m)}^2,
	\end{aligned}
	\ee	
	where $f=\o_{\theta}^{q}$ and $h = (|\o_\th| \vee 1)^q$. Then there are two cases to be dealt with.
	
	Case 1: $T \leq 2$. In this case, we follow the argument for (\ref{vth-mi-st}) in Case 1 in the proof of Proposition \ref{Prop, vth-ub} to obtain
	\[
		\|\o_{\th}\|_{L^\infty_{t x} (D_m \times[0, T])} \leq C C_{*}^{10} \big( \|\o_{\th}\|_{L_{tx}^2 (D_{m} \times[0, T])} + \|\o_{\th}(\cdot, 0)\|_{L^\infty (D_m) } + 1 \big).
	\]
	Actually, the zero boundary condition of $ \o_\th $ makes the argument simpler. Combining with the energy estimate (\ref{en1s}), we find
	\be\label{oth-usb-sT}
	\|\o_{\th}\|_{L^\infty_{t x} (D_m \times[0, T])} \leq C C_{*}^{10} \big( \| v_0 \|_{L^2 (D_m)} + \|\o_{0,\th} \|_{L^\infty (D_m) } + 1 \big).
	\ee
	
	Case 2: $ T>2 $. In this case, we follow the argument for (\ref{vth-mi-lt}) in Case 2 in the proof of Proposition \ref{Prop, vth-ub} to find
	\[
	\|\o_{\th}\|_{L^\infty_{t x} (D_m \times[T-1, T])} \leq C C_{*}^{10} \big( \|\o_{\th}\|_{L_{tx}^2 (D_{m} \times[T-2, T])} + \| 1 \|_{L_{tx}^2 (D_{m} \times[T-2, T])} \big).
	\]
	Then due to the energy estimate (\ref{en1s}) again, we conclude
	\be\label{oth-usb-lT}
	\|\o_{\th}\|_{L^\infty_{t x} (D_m \times[T-1, T])} \leq C C_{*}^{10} \big( \| v_0 \|_{L^2 (D_m)} + 1 \big).
	\ee
	
Finally, by combining (\ref{oth-usb-sT}) in Case 1 and (\ref{oth-usb-lT}) in Case 2 together, (\ref{oth-usb}) is justified.
	
\end{proof}

\subsubsection{$L^\i$ boundedness of $v_\rho$ and $v_\phi$}
\begin{proposition}\label{Prop, vr-vp-ub}
Let the region $ D_m $ be as defined in (\ref{app domain-sph}) with $ m\geq 2 $ and the angle $\alpha \in \big(0,\frac{\pi}{6}\big]$. Then for any $ T>0 $,
\be\label{vr-vp-usb}
\| v_{\rho} \|_{L_{tx}^{\infty}(D_{m} \times [0, T])} + \| v_{\phi} \|_{L_{tx}^{\infty}(D_{m} \times [0, T])} \leq C C_{*}^{3} (\|v_{0}\|_{L^{2}\left(D_{m}\right)}+ 1 ),
\ee
where $ C=C(\a) $ and
\be\label{vr-vp-Cstar}
\begin{aligned}
	C_{*} = \max \left\{ \left\|\frac{ |v_\rho| + | v_{\phi}|}{\rho}\right\|_{L_{t}^{\infty} L_x^6 (D_m \times[0, T])}, \|\o_{\theta}\|_{L_{tx}^{\infty} (D_m \times[0, T])}, 2 \right\}.
\end{aligned}
\ee
\end{proposition}
\begin{proof}
	Fix any $ t\in[0,T] $. The following proof will be derived based on this fixed $ t $ and we will drop the temporal variable within the proof for simplicity.
	
	We first estimate $ \| v_{\rho} \|_{L_{tx}^{\infty}(D_{m} \times [0, T])} $. According to the Biot-Savart law (\ref{v-rho, v-phi, O}) and the boundary conditions in Lemma \ref{Lemma, bdry cond}, $ v_\rho $ satisfies the following equations.
	\[\begin{cases}
	\Big(\Delta + \frac{2}{\rho}\,\p_\rho + \frac{2}{\rho^2} \Big)v_{\rho} = -\frac{1}{\rho\sin\phi}\,\p_{\phi}(\sin\phi\,\o_{\th}), \ \text { in } \ D_{m}; \\
	\p_\phi v_\rho = 0\ \text{ on }\ \p^R D_m, \quad v_\rho = 0\ \text{ on }\ \p^A D_m.
	 \end{cases}\]
 	For any integer $ q\geq 1 $, we denote $ v_\rho^q $ by $ f $. Then $ f $ satisfies the equations below.
 		\[\begin{cases}
 		\Dl f - q(q-1) v_\rho^{q-2} |\na v_\rho|^2 + \f{2}{\rho}\p_\rho f+ \f{2q}{\rho^2} f = -\f{qv_\rho^{q-1}}{\rho\sin\phi}\p_\phi(\sin\phi\, \o_\th), \ \text { in } \ D_{m}; \\
 		\p_\phi f = 0\ \text{ on }\ \p^R D_m, \quad f = 0\ \text{ on }\ \p^A D_m.
 	\end{cases}\]
 	Testing the above problem by $ f $ on $ D_m $ yields
	\begin{equation}\label{eqgNA}
		-\frac{2q-1}{q}\int_{D_m} |\nabla f|^2 \,dx + 2\int_{D_m}\frac{1}{\rho} f\partial_\rho f\,dx
		+ 2q\int_{D_m}\frac{1}{\rho^2} f^2 \,dx = - q \int_{D_m}\frac{v_\rho^{2q-1}}{\rho\sin\phi}\, \p_\phi(\sin\phi \, \o_\th) \,dx.
	\end{equation}
	By converting the integrals into the form of spherical coordinates, and then using integration by parts, we have
	\begin{align*}
		\int_{D_m}\frac{2}{\rho} f\partial_\rho f \,dx &= 2\pi \int_{\frac{\pi}{2}-\alpha}^{\frac{\pi}{2}+\alpha}\int_{\frac1m}^1 \rho\sin\phi\, \partial_\rho (f^2) \,d\rho\,d\phi\\
		& = - 2\pi\int_{\frac{\pi}{2}-\alpha}^{\frac{\pi}{2}+\alpha}\int_{\frac1m}^1 \sin\phi\, f^2 \,d\rho \,d\phi = -\int_{D_m}\frac{ f^2}{\rho^2} \,dx;
	\end{align*}
    and
    \begin{align*}	
		-q\int_{D_{m}}\frac{v_\rho^{2q-1}}{\rho\sin\phi}\p_\phi(\sin\phi \,\o_\th) \,dx &= - 2\pi q\int_{\frac1m}^1\int_{\frac{\pi}{2}-\alpha}^{\frac{\pi}{2}+\alpha} \rho v_\rho^{2q-1}\p_\phi(\sin\phi \,\o_\th) \,d\phi \,d\rho \\
		& = 2\pi q(2q-1)\int_{\frac1m}^1\int_{\frac{\pi}{2}-\alpha}^{\frac{\pi}{2}+\alpha} \rho\sin\phi\, v_\rho^{2q-2} (\p_\phi v_\rho) \o_\th \,d\phi \,d\rho\\
		& =(2q-1)\int_{D_m} \frac{\p_\phi  f}{\rho} v_\rho^{q-1}\o_\th \,dx.
	\end{align*}
	Putting the above estimates into (\ref{eqgNA}) and then multiplying the equation by $ \frac{q}{2q-1} $, one deduces
	\begin{align}\label{esgNA}
		\int_{D_m} |\nabla  f|^2 \,dx = q\int_{D_m}\frac{ f^2}{\rho^2} \,dx
		 - q\int_{D_m} \frac{\p_\phi  f}{\rho} v_\rho^{q-1}\o_\th \,dx:=I_1+I_2.
	\end{align}
	For $I_1$, H\"older's inequality shows that
	\be\label{I1NA}\begin{split}
		|I_1| = q \left\| \frac{v_\rho}{\rho}\, v_\rho^{q-1} \right\|_{L^2(D_m)}^2  &\leq q \left\|\frac{v_\rho}{\rho}\right\|_{L^6(D_m)}^2\left\|v_\rho^{q-1}\right\|_{L^{3}(D_m)}^2 \\
		& \leq q \left\|\frac{v_\rho}{\rho}\right\|_{L^6(D_m)}^2\big\|(|v_\rho|\vee 1)^q\big\|_{L^{3}(D_m)}^2.
	\end{split}\ee
	For $I_2$, applying H\"older inequality and Young's inequality, we have
	\begin{align}\label{I2NA}
		\begin{split}
			|I_2|&\leq q \left\|\frac{\partial_\phi f}{\rho}\right\|_{L^2(D_m)}\left\|v_\rho^{q-1}\right\|_{L^{2}(D_m)}\left\|\omega_\theta\right\|_{L^{\infty}(D_m)}\\&\leq \frac14 \|\nabla  f \|_{L^2(D_m)}^2 + q^2 \left\|(|v_\rho|\vee 1)^q\right\|_{L^{2}(D_m)}^2\left\|\omega_\theta\right\|_{L^{\infty}(D_m)}^2.
		\end{split}
	\end{align}
	Plugging \eqref{I1NA} and \eqref{I2NA} into \eqref{esgNA}, we know
	\be\label{vr-mi-ee}\begin{split}
		\frac34 \|\nabla  f \|_{L^2(D_m)}^2 & \leq q \left\|\frac{v_\rho}{\rho}\right\|_{L^6(D_m)}^2\| h \|_{L^{3}(D_m)}^2 + q^2 \left\| h\right\|_{L^{2}(D_m)}^2\left\|\omega_\theta\right\|_{L^{\infty}(D_m)}^2, \\
		& \leq C_*^2 \Big(q \| h \|_{L^{3}(D_m)}^2 + q^2 \| h \|_{L^{2}(D_m)}^2  \Big),
	\end{split}\ee
	where $ h := (|v_\rho|\vee 1)^q$ and $ C_* $ is as defined in (\ref{vr-vp-Cstar}). Then it follows from Lemma \ref{Lemma, sobolev emb} that there exists some constant $ s_0 $, which only depends on $ \a $, such that
	\[ \| f\|_{L^6(D_m)} \leq s_0 \| f\|_{H^1(D_m)}. \]
	So (\ref{vr-mi-ee}) implies that
	\be\label{vr-mi-ee1}
	 	\| f\|_{L^6(D_m)}^2 \leq CC_*^2 \Big(q \| h \|_{L^{3}(D_m)}^2 + q^2 \| h \|_{L^{2}(D_m)}^2  \Big)+ C\|f\|_{L^2(D_m)}^2.
	 \ee
	In addition, since $ f=v_\rho^q $ and $ h = |f|\vee 1 $, we derive from (\ref{vr-mi-ee1}) that
	\be\label{vr-mi-ee2}
	\| h \|_{L^6(D_m)}^2 \leq C C_*^2 \Big(q \| h \|_{L^{3}(D_m)}^2 + q^2 \| h \|_{L^{2}(D_m)}^2  \Big).
	\ee
	Now we interpolate $ \|h\|_{L^3} $ between $ \|h\|_{L^6} $ and $ \|h\|_{L^2} $ to get
	\[ C C_*^2 q \| h \|_{L^{3}(D_m)}^2 \leq \frac14 \|h\|_{L^6(D_m)}^2 + C^2 C_*^4 q^2 \|h\|_{L^2(D_m)}^2.\]
	Therefore, it follows from (\ref{vr-mi-ee2}) that
	\[
	\| h \|_{L^6(D_m)}^2 \leq C C_*^4 q^2 \| h \|_{L^{2}(D_m)}^2.
	\]
	By writing $ h=\psi^q $, where $ \psi=|v_\rho|\vee 1 $, the above estimate is converted into
		\be\label{vr-mi-ee3}
	\| \psi \|_{L^{6q}(D_m)} \leq (C C_*^4)^{\frac{1}{2q}} q^{\frac{1}{q}} \| \psi \|_{L^{2q}(D_m)}.
	\ee
	Now we choose $ q=q_{k}=3^{k} $ in (\ref{vr-mi-ee3}), where $k=0,1,2,\cdots $, then by iterative estimates, we obtain
	\[ \|\psi\|_{L^\infty(D_m)} \leq C C_*^3 \|\psi\|_{L^2(D_m)}, \]
	where $ C=C(\a) $. This result yields
	\[ \| v_\rho(\cdot, t) \|_{L^\infty(D_m)} \leq CC_*^3 (\|v_\rho(\cdot, t)\|_{L^2(D_m)}+1), \quad \forall\, t\in[0,T]. \]
	Taking advantage of the energy estimate (\ref{en1s}) and taking supremum with respect to $ t $, we conclude
	\be\label{vr-ub-fr} \| v_\rho \|_{L_{tx}^\infty(D_m\times[0,T])} \leq CC_*^3 (\|v_0\|_{L^2(D_m)}+1). \ee
	
	Next, we use the similar method as above to estimate $ \| v_{\phi} \|_{L_{tx}^{\infty}(D_{m} \times [0, T])} $. Based on the Biot-Savart law (\ref{v-rho, v-phi, O}) and the boundary conditions in Lemma \ref{Lemma, bdry cond}, $ v_\phi $ satisfies the following equations.
	\[
	\left\{\begin{array}{l}
		\Big( \Delta+\frac{2}{\rho}\p_\rho + \frac{1-\cot^2\phi}{\rho^2} \Big) v_\phi = \frac{1}{\rho^3}\,\p_\rho(\rho^3 \o_\th),\ \text { in } \ D_{m}; \\
		v_\phi = 0\ \text{ on }\ \p^R D_m, \quad \p_\rho v_\phi = - \frac{1}{\rho} v_\phi \ \text{ on }\ \p^A D_m.
	\end{array}\right.
	\]
	For any integer $ q\geq 1 $, we denote $ v_\phi^q $ by $ g $. Then $ g $ satisfies the equations below.
	\[
	\left\{\begin{array}{l}
		\Dl g - q(q-1) v_\phi^{q-2}|\na v_\phi|^2+\f{2}{\rho}\p_\rho g + \f{q(1-\cot^2\phi)}{\rho^2}g = \f{q}{\rho^3} v_\phi^{q-1} \p_\rho(\rho^3\o_\th),\ \text { in } \ D_{m}; \\
		g = 0\ \text{ on }\ \p^R D_m, \quad \p_\rho g = - \frac{q}{\rho}\, g \ \text{ on }\ \p^A D_m.
	\end{array}\right.
	\]
	Testing this problem by $ g $ on $ D_m $, we obtain
	\be\label{MOSZJ}\begin{split}
	& \underbrace{\int_{D_m}g\Dl g \,dx}_{G_1}-\f{q-1}{q}\int_{D_m}|\na g|^2 \,dx + \int_{D_m}\f{2}{\rho}g\p_\rho g \,dx + q\int_{D_m}\f{1-\cot^2\phi}{\rho^2}g^2 \,dx \\
	=\,\, & q\int_{D_m}\f{v_\phi^{2q-1}}{\rho^3}\p_\rho(\rho^3\o_\th) \,dx.
	\end{split}	\ee
	Using integration by parts and then converting the boundary integral to the interior integral, we see
	\[  G_1 = -\int_{D_m}|\na g|^2 \,dx -2q\int_{D_m} \frac{g}{\rho} \p_\rho g \,dx - q\int_{D_m}\f{g^2}{\rho^2} \,dx. \]
	Substituting this identity into (\ref{MOSZJ}) leads to
	\[ \begin{split}
		& \f{2q-1}{q}\int_{D_m}|\na g|^2 \,dx + q\int_{D_m}\f{\cot^2\phi}{\rho^2}g^2 \,dx \\
		=\,\, & -(2q-2) \int_{D_m} \frac{g}{\rho}\p_\rho g \,dx - q\int_{D_m}\f{v_\phi^{2q-1}}{\rho^3} \p_\rho(\rho^3\o_\th) \,dx.
	\end{split} \]
	This implies 	
	\be\label{ENGZJ}
	\int_{D_m}|\na g|^2dx\leq \underbrace{2(q-1)\bigg|\int_{D_m}\f{g}{\rho}\p_\rho g \,dx\bigg|}_{G_2} + \underbrace{q\bigg|\int_{D_m}\f{v_\phi^{2q-1}}{\rho^3}\p_\rho(\rho^3\o_\th)\,dx\bigg|}_{G_3}.
	\ee
	Moreover, by applying H\"older's inequality, we have
	\be\label{EG2ZJ}
	\begin{split}
		G_2 & \leq 2q\|\nabla g\|_{L^2(D_m)}\left\|\f{v_\phi}{\rho}\right\|_{L^6(D_m)}\|v_\phi^{q-1}\|_{L^3(D_m)}\\
		& \leq \f{1}{4}\|\nabla g\|^2_{L^2(D_m)} + 4q^2\left\|\f{v_\phi}{\rho}\right\|_{L^6(D_m)}^2\|(|v_\phi|\vee 1)^{q}\|^2_{L^3(D_m)} \\
		& \leq \f{1}{4}\|\nabla g\|^2_{L^2(D_m)} + 4C_*^2 q^2 \|h_1\|^2_{L^3(D_m)},
	\end{split}
	\ee
	where $ C_* $ is as defined in (\ref{vr-vp-Cstar}) and $ h_1:= (|v_\phi|\vee 1)^q$.
	In order to estimate $ G_3 $, we first use spherical coordinates and integration by parts to find
	\[ \begin{split}
		G_3 & = 2\pi q\bigg| \int_{\f{\pi}{2}-\alpha}^{\f{\pi}{2}+\a} \int_{\frac1m}^{1} \f{v_\phi^{2q-1}}\rho\p_\rho (\rho^3\o_\th)\sin\phi \,d\rho\,d\phi \bigg| \\
		& = 2\pi q\bigg| \int_{\f{\pi}{2}-\alpha}^{\f{\pi}{2}+\a}\int_{\frac1m}^{1} \bigg( -\frac{v_\phi^{2q-1}}{\rho^2} + \frac{2q-1}{\rho} v_\phi^{2q-2}\p_\rho v_\phi \bigg)\rho^3 \o_\th \sin\phi \,d\phi \,d\rho\bigg| \\
		& \leq  q\int_{D_m}|\o_\th|\left|\f{v_\phi}{\rho}\right||v_\phi|^{2q-2} \,dx + (2q-1)\int_{D_m}|\o_\th||\p_\rho g||v_\phi|^{q-1}\,dx.
	\end{split} \]
	Using H\"older's inequality,
	\[\begin{split}
		G_3 &\leq q \|\o_\th\|_{L^6(D_m)}\Big\|\frac{v_\phi}{\rho}\Big\|_{L^6(D_m)}\|h_1\|^2_{L^3(D_m)} + 2q \|\o_\th\|_{L^\infty(D_m)} \|\nabla g\|_{L^2(D_m)} \|h_1\|_{L^2(D_m)} \\
		& \leq C_*^2 q \|h_1\|^2_{L^3(D_m)} + 2C_* q \|\nabla g\|_{L^2(D_m)} \|h_1\|_{L^2(D_m)}.
	\end{split}\]
	Applying Cauchy-Schwarz inequality,
	\be\label{EG3ZJ}
	G_{3}\leq \f{1}{4}\|\na g\|_{L^2(D_m)}^2 + 4C_*^2 q^2 \|h_1\|^2_{L^2(D_m)} + C_*^2 q \|h_1\|^2_{L^3(D_m)}.
	\ee
	Substituting \eqref{EG2ZJ} and \eqref{EG3ZJ} in \eqref{ENGZJ}, one finds
	\be\label{vp-mi-ee} \|\na g\|_{L^2(D_m)}^2 \leq C C_*^2 q^2\big( \|h_1\|^2_{L^3(D_m)} + \|h_1\|^2_{L^2(D_m)}  \big),
	\ee
	where $ C $ is a numerical constant. Since $ g=v_\phi^q=0 $ on $ \p^{R} D_m $, it follows from Lemma \ref{Lemma, sobolev emb} that there exists some constant $ s_1 $, which only depends on $ \a $, such that
	\[ \| g\|_{L^6(D_m)} \leq s_1 \|\na g\|_{L^2(D_m)}. \]
	Moreover, noticing $ h_1= |g|\vee 1 $, so it follows from the above embedding and (\ref{vp-mi-ee}) that
	\[ \|h_1\|_{L^6(D_m)} \leq C C_*^2 q^2\big( \|h_1\|^2_{L^3(D_m)} + \|h_1\|^2_{L^2(D_m)}  \big). \]
	This estimate is a parallel result to (\ref{vr-mi-ee2}), so the remaining proof is similar to that for (\ref{vr-ub-fr}). Thus, we obtain
	\be\label{vp-ub-fr} \| v_\phi \|_{L_{tx}^\infty(D_m\times[0,T])} \leq CC_*^3 (\|v_0\|_{L^2(D_m)}+1).
	\ee
	
	The combination of (\ref{vr-ub-fr}) and (\ref{vp-ub-fr}) justifies (\ref{vr-vp-usb}).
\end{proof}

By tracing the constants in Lemma \ref{Lemma, energy est for KFO}, Lemma \ref{Lemma, LtwLx6}, Propositions \ref{Prop, vth-ub}, \ref{Prop, oth-ub} and \ref{Prop, vr-vp-ub}, we can obtain the following corollary. The key is that both $ \| v_0 \|_{H^2(D_m)} $ and $ \|\o_{0,\th}\|_{L^\infty(D_m)} $ are controlled by $ \|v_0\|_{C^2(\ol{D_m})} $.

\begin{corollary}\label{Cor, ubi}
	Let the region $ D_m $ be as defined in (\ref{app domain-sph}) with $ m\geq 2 $ and the angle $\alpha \in \big(0,\frac{\pi}{6}\big]$. Assume (\ref{small Gamma}), that is $\| \Gamma(\cdot, 0)\|_{L^{\infty}(D_m)}\leq \frac{1}{95}$. Then for any $ T>0 $,
	\be\label{ubi}
	\| v\|_{L_{tx}^{\infty}(D_m\times[0,T])} + 	\| \o_{\th}\|_{L_{tx}^{\infty}(D_m\times[0,T])} \leq C_0^{*},
	\ee
	where $ C_0^{*} $ is a constant which only depends on $ \a $ and $ \|v_0\|_{C^2(\ol{D_m})} $.
\end{corollary}

\section{Uniform bounds for $ \|v\|_{L_t^2 H_x^2} $ and $ \|v\|_{H_t^1 L_x^2} $ on $ D_m\times[0,T] $}
\label{Sec, hreg-ub}

The basic setup of this section is the same as that in the beginning of Section \ref{Sec, inf-ub}. More precisely, for any fixed $ m\geq 2 $ and $ T>0 $, we consider the initial data $ v_0$ which lies in the admissible class $ \mathscr{A}_{m} $ with the even-odd-odd symmetry. For such initial data, we denote by $ v $ the solution in Corollary \ref{Cor, globle soln in ad} so that $ v\in E^{\sigma,s}_{m,T} \cap H_t^1 L_x^2\cap L_t^2 H_x^{2} \cap L_{tx}^\infty(D_m\times[0,T]) $. Moreover, we restrict the range of $ \a $ within $ \big(0, \frac{\pi}{6}\big] $ and require $ \|\Gamma(\cdot,0)\|_{L^\infty(D_m)}\leq \frac{1}{95} $. Then by taking advantage of the results in Section \ref{Sec, inf-ub}, in particular Lemma \ref{Lemma, energy est for KFO} and Corollary \ref{Cor, ubi}, we will obtain uniform bounds, which are independent of $ T $ and dependent on $ m $ only via $ \|v_0\|_{C^2(\ol{D_m})} $, for  $\|v\|_{L_{t}^{2} H_{x}^{2}(D_{m} \times[0, T])} $ and $\|v\|_{H_t^1 L_x^2(D_{m} \times[0, T])} $.  The strategy is as follows:

\begin{itemize}
	\item Step 1: Based on the uniform boundedness of $\|\nabla K\|_{L_{t x}^{2}}$ and $\|\nabla F\|_{L_{t x}^{2}}$ on $ D_m\times [0,T] $, we will derive a uniform bound for $\|v_\theta\|_{L_{t}^{2} H_{x}^{2}(D_m\times[0, T])}$. Then the uniform bound of $\left\|\partial_{t} v_{\theta}\right\|_{L_{t x}^{2}\left(D_{m}\times[0, T]\right)}$ can be obtained via the equation of $v_{\theta}$.
	
	\item Step 2: Thanks to the Biot-Savart law and the uniform boundedness of $\|\nabla \Omega\|_{L_{t x}^{2}}$, we manage to derive uniform bounds for $\|v_{\rho}\|_{L_{t}^{2} H_{x}^{2}\left(D_{m} \times[0, T]\right)}$ and $\|v_{\phi}\|_{L_{t}^{2} H_{x}^{2}\left(D_{m} \times[0, T]\right)}$.
	
	\item Step 3: The uniform boundedness of $\left\|\partial_{t} \omega_{\theta}\right\|_{L_{t x}^{2}}$ can be verified by studying the equation of $\omega_{\th}$.
	
	\item Step 4: By taking advantage of the Biot-Savart law again and also utilizing the uniform boundedness of $\left\|\partial_{t} \omega_{\theta}\right\|_{L_{t x}^{2}}$, we are able to justify both $\|\partial_{t} v_{\rho}\|_{L_{t x}^{2}}$ and $\|\partial_{t} v_{\phi}\|_{L_{tx}^{2}}$ are uniformly bounded.
\end{itemize}

We first summarize some pertinent results from earlier sections with minor extensions which will be needed in the later proof.

\begin{proposition}\label{Prop, summary}
Let the region $D_{m}$ be as defined in (\ref{app domain-sph}) with $m \geq 2$ and the angle $\alpha \in \big(0, \frac{\pi}{6} \big]$. Assume $\|\Gamma(\cdot, 0)\|_{L^{\infty}\left(D_{m}\right)} \leq \frac{1}{95}$. Then there exists a constant $C$, which only depends on $\alpha$ and $\|v_{0}\|_{C^{2}(\ol{D_m})}$ such that for any $T>0$,
\begin{align}
    & \Vert \big(|K|+|F|+|\Omega| \big) \Vert_{L_{t}^{\infty} L_{x}^{2}} + \Vert \big(|\nabla K| + |\nabla F| + |\nabla \Omega|\big) \Vert_{L_{t x}^{2}} \leq C,  \text{ (cf. Lemma \ref{Lemma, energy est for KFO}})  \label{s1e} \\
	& \|v\|_{L_{t x}^{\infty}} + \Big\Vert\frac{1}{\rho} \nabla v \Big\Vert_{L_{t}^{\infty} L^{2}_x} + \Big\|\frac{1}{\rho^{2}} \nabla v \Big\|_{L_{t x}^{2}} + \Big\|\frac{1}{\rho^{3}} v\Big\|_{L_{t x}^{2}} \leq C, \label{s2e}\\
	& \left\|\omega_{\theta}\right\|_{L_{t x}^{\infty}} + \Big\|\frac{1}{\rho} \nabla \omega_{\theta}\Big\|_{L_{t x}^2} + \Big\|\frac{1}{\rho^2} \omega_{\theta}\Big\|_{L_{t x}^2} \leq C, \label{s3e}
\end{align}
where all the above space-time norms are taken on $D_{m} \times[0, T]$.
\end{proposition}
\begin{proof}
	Only (\ref{s2e}) and (\ref{s3e}) are required to be verified. We start with the estimate (\ref{s2e}). Firstly, the uniform boundedness of $ v $ is due to Corollary \ref{Cor, ubi}. Then from Lemmas \ref{Lemma, e-trans-v_rho}, \ref{Lemma, e-trans-v_phi} and \ref{Lemma, ee for vtdr}, we have
	\begin{align}
	& \Big\|\nabla \Big( \frac{v_\rho}{\rho} \Big)\Big\|_{L_t^\infty L_x^2} + \Big\|\nabla \Big( \frac{v_\phi}{\rho} \Big)\Big\|_{L_t^\infty L_x^2} + \Big\|\nabla \Big( \frac{v_\th}{\rho} \Big)\Big\|_{L_t^\infty L_x^2} \leq 5\Vert \big(|K|+|F|+|\Omega| \big) \Vert_{L_{t}^{\infty} L_{x}^{2}}, \label{s2-i2}  \\
	& \Big\| \frac{1}{\rho} \nabla \Big( \frac{v_\rho}{\rho} \Big)\Big\|_{L_{t x}^{2}} + \frac{1}{\rho} \Big\|\nabla \Big( \frac{v_\phi}{\rho} \Big)\Big\|_{L_{t x}^{2}} + \Big\| \frac{1}{\rho} \nabla \Big( \frac{v_\th}{\rho} \Big)\Big\|_{L_{t x}^{2}} \leq 40\Vert \big(|\nabla K| + |\nabla F| + |\nabla \Omega|\big) \Vert_{L_{t x}^{2}}.  \label{s2-22}
	\end{align}
	Since $ \int_{\frac{\pi}{2}-\a}^{\frac{\pi}{2}+\a}  v_\rho(\rho,\phi) \sin\phi \,d\phi=0 $ for any $ \rho\in \big(\frac1m, 1\big) $, it follows from the Poincar\'e inequality in Lemma \ref{Lemma, P ave} that
	\[ \Big\| \frac{1}{\rho^2}v_\rho(\cdot, t) \Big\|_{L^2_x} \leq \sqrt{\frac{2}{19}} \Big\| \frac{1}{\rho^2}\p_\phi v_\rho (\cdot, t) \Big\|_{L^2_x} \leq \frac13  \Big\| \frac{1}{\rho^2}\p_\phi v_\rho (\cdot, t) \Big\|_{L^2_x} , \quad \forall\, t\in[0,T].\]
	In addition, since
	\[ \nabla \Big( \frac{v_\rho}{\rho}\Big) = \frac{1}{\rho}\nabla v_\rho - \bigg(\frac{1}{\rho^2}v_\rho\bigg) e_\rho, \]
	we know that
	\be\label{s2-rho}
	\Big\| \frac{1}{\rho}\nabla v_\rho  \Big\|_{L_t^\infty L_x^2} \leq \frac{4}{3} \Big\| \nabla \Big( \frac{v_\rho}{\rho}\Big) \Big\|_{L_t^\infty L_x^2}. \ee
	Similarly,
	\be\label{s2-phi-th}
	\Big\| \frac{1}{\rho}\nabla v_\phi  \Big\|_{L_t^\infty L_x^2} 	\leq \frac{4}{3}\Big\| \nabla \Big( \frac{v_\phi}{\rho}\Big) \Big\|_{L_t^\infty L_x^2}
	\quad \text{and} \quad
	\Big\| \frac{1}{\rho}\nabla v_\th \Big\|_{L_t^\infty L_x^2} \leq \frac{4}{3}
	\Big\| \nabla \Big( \frac{v_\th}{\rho}\Big) \Big\|_{L_t^\infty L_x^2}.
	\ee
	Plugging (\ref{s2-rho}) and (\ref{s2-phi-th}) into (\ref{s2-i2}), and then using (\ref{s1e}), we obtain $ \big\Vert\frac{1}{\rho} \nabla v \big\Vert_{L_{t}^{\infty} L^{2}_x} \leq C $. By an analogous argument, we can take advantage of  (\ref{s2-22}) to  show $ \big\Vert\frac{1}{\rho^2} \nabla v \big\Vert_{L_{tx}^2}\leq C$, which further implies $ \big\Vert\frac{1}{\rho^3} v \big\Vert_{L_{t}^{\infty} L^{2}_x} \leq C$. Thus, (\ref{s2e}) is justified.
	
	We next investigate the estimate (\ref{s3e}). Firstly, the uniform boundedness of $ \o_\th $ is due to Corollary \ref{Cor, ubi}. Then by direct computation, we find
	\be\label{s3-gO}
	\nabla \O = \nabla \bigg( \frac{\o_\th}{\rho\sin\phi} \bigg) = \frac{1}{\rho\sin\phi}\nabla \o_\th - \frac{\o_\th}{\rho^2\sin\phi} (e_\rho + \cot\phi\, e_\phi),
	\ee
	Since $ 0\leq \cot\phi\leq \cot\a \leq 1/\sqrt{3} $, then for any $ t\in[0,T] $,
	\be\label{s3-ml1} \begin{split}
		\bigg\| \frac{\o_\th(\cdot, t)}{\rho^2\sin\phi} (e_\rho + \cot\phi\, e_\phi) \bigg\|^2_{L_x^2(D_m)} & \leq \frac{4}{3} \bigg\| \frac{\o_\th(\cdot, t)}{\rho^2\sin\phi}  \bigg\|^2_{L_x^2(D_m)}.
	\end{split} \ee
	Thanks to the restriction that $ \a\in \big(0, \frac{\pi}{6}\big] $, we know $ \sqrt{3}/2 \leq \sin\phi \leq 1 $ and therefore,
	\[ \begin{split}
		\bigg\| \frac{\o_\th(\cdot, t)}{\rho^2\sin\phi}  \bigg\|^2_{L_x^2(D_m)} = 2\pi \int_{\frac1m}^{1}\int_{\frac{\pi}{2}-\a}^{\frac{\pi}{2}+\a} \frac{\o_\th^2}{\rho^2\sin\phi}\,d\phi\,d\rho
		\leq  2\pi\,\frac{2}{\sqrt{3}} \int_{\frac1m}^{1}\int_{\frac{\pi}{2}-\a}^{\frac{\pi}{2}+\a} \frac{\o_\th^2}{\rho^2}\,d\phi\,d\rho.
	\end{split}\]
	Since $ \o_\th $ vanishes on the boundary of $ D_m $, we apply the Poincar\'e inequality in Lemma \ref{Lemma, P bdry} to the right-hand side of the above inequality to obtain
	\be\label{s3-ml2}\begin{split}
		\bigg\| \frac{\o_\th(\cdot, t)}{\rho^2\sin\phi}  \bigg\|^2_{L_x^2(D_m)} &\leq 2\pi\,\frac{2}{\sqrt{3}}\,\frac{3}{25} \int_{\frac1m}^{1}\int_{\frac{\pi}{2}-\a}^{\frac{\pi}{2}+\a} \frac{1}{\rho^2}(\p_\phi \o_\th)^2\,d\phi\,d\rho  \\
		& \leq \frac{6}{25\sqrt{3}} \bigg\| \frac{1}{\rho\sin\phi}\nabla \o_\th (\cdot, t) \bigg\|^2_{L_x^2(D_m)}.
	\end{split}\ee
	The combination of (\ref{s3-ml1}) and (\ref{s3-ml2}) yields
	\be\label{s3-ml3}
	\bigg\| \frac{\o_\th(\cdot, t)}{\rho^2\sin\phi} (e_\rho + \cot\phi\, e_\phi) \bigg\|_{L_x^2(D_m)}  \leq   \frac12 \bigg\| \frac{1}{\rho\sin\phi}\nabla \o_\th (\cdot, t) \bigg\|_{L_x^2(D_m)}.
	\ee
	Based on (\ref{s3-ml3}), it then follows from (\ref{s3-gO}) that
	\[ \bigg\| \frac{1}{\rho\sin\phi}\nabla \o_\th \bigg\|_{L_{tx}^2} \leq 2 \|\nabla \O\|_{L_{tx}^2}. \]
	Hence, we conclude $\big\|\frac{1}{\rho} \nabla \omega_{\theta}\big\|_{L_{t x}^{2}}\leq C $, which further implies that $ \big\|\frac{1}{\rho^2} \omega_{\theta}\big\|_{L_{t x}^{2}}\leq C $. Thus, (\ref{s3e}) is established.
	
\end{proof}
	
Based on Proposition \ref{Prop, summary}, we will prove the main result of this section shown as below.
	
\begin{proposition}\label{Prop, high-reg}
Let the region $D_{m}$ be defined as in (\ref{app domain-sph}) with $m \geq 2$ and the angle $\alpha \in\big(0, \frac{\pi}{6}\big]$. Assume $\|\Gamma(\cdot, 0)\|_{L^{\infty}(D_m)} \leq \frac{1}{95}$. Then there exists a constant $C$, which only depends on $\alpha$ and $\|v_0\|_{C^2(\ol{D_m})}$, such that for any $T>0$,
	\be\label{high-reg}
	\| \nabla^{2} v \|_{ L_{tx}^2 (D_m \times [0, T]) } + \| \p_{t} v \|_{ L_{tx}^2 (D_m \times [0, T]) } \leq C.
	\ee
\end{proposition}

\begin{proof}
	In the proof, $ C $ denotes constants which are independent of  $ T $, but may be dependent on $ \a $ and $\|v_0\|_{C^2(\ol{D_m})}$. On the other hand, unless stated otherwise, all the norms in this proof are taken on the space-time domain $ D_m\times[0,T] $.
	
	Step 1: Uniform bounds on $\| \nabla^{2} v_\th \|_{L_{tx}^2}$ and $\| \p_{t} v_\th \|_{L_{tx}^2}$.
	
	Firstly, since $ \nabla v_\th = (\p_\rho v_\th)e_\rho + \Big(\frac1\rho \p_\phi v_\th\Big)e_\phi $, it then follows from the formula (\ref{nabla vec-sph}) that under the basis (\ref{m-basis in sph}),
	\be\label{H-vth}
	\nabla^2 v_\th = \begin{pmatrix}
		\p_\rho^2 v_\th & \frac1\rho \p_\phi \p_\rho v_\th - \frac{1}{\rho^2}\p_\phi v_\th & 0 \vspace{0.08in}\\
		\frac1\rho \p_\rho \p_\phi v_\th - \frac{1}{\rho^2}\p_\phi v_\th & \frac{1}{\rho^2}\p_\phi^2 v_\th + \frac1\rho \p_\rho v_\th & 0 \vspace{0.08in} \\
		0 & 0 & \frac1\rho \p_\rho v_\th + \frac{\cot\phi}{\rho^2}\p_\phi v_\th
	\end{pmatrix}.\ee
	Thanks to (\ref{s2e}) in Proposition \ref{Prop, summary}, we infer from (\ref{H-vth}) that
	\be\label{h-vth-e1}
	\| \nabla^2 v_\th \|_{L_{tx}^2} \leq 2\bigg( \big\| \p_\rho^2 v_\th \big\|_{L_{tx}^2} + \Big\|  \frac{1}{\rho^2}\p_\phi^2 v_\th \Big\|_{L_{tx}^2}  + \Big\| \frac1\rho \p_\rho \p_\phi v_\th \Big\|_{L_{tx}^2} \bigg) + C.
	\ee	
	Recall
	\[ \nabla K = \nabla \Big( \frac{\o_\rho}{\rho} \Big) = K_1 e_\rho + K_2 e_\phi,
	\]
	where
	\[ \left\{\begin{split}
		K_1 & = \frac{1}{\rho^2} \p_\rho \p_\phi v_\th - \frac{2}{\rho^3} \p_\phi v_\th + \frac{\cot \phi}{\rho^2}\p_\rho v_\th - \frac{2\cot\phi}{\rho^3}v_\th, \vspace{0.05in} \\
		K_2 & = \frac{1}{\rho^3} \p_\phi^2 v_\th + \frac{\cot\phi}{\rho^3} \p_\phi v_\th - \frac{1}{\rho^3\sin^2\phi} v_\th.
	\end{split}\right.  \]
	Equivalently,
	\[ \left\{\begin{split}
		\frac{1}{\rho^2} \p_\rho \p_\phi v_\th &= K_1 + \frac{2}{\rho^3} \p_\phi v_\th - \frac{\cot \phi}{\rho^2}\p_\rho v_\th + \frac{2\cot\phi}{\rho^3}v_\th, \\
		\frac{1}{\rho^3} \p_\phi^2 v_\th &= K_2 - \frac{\cot\phi}{\rho^3} \p_\phi v_\th + \frac{1}{\rho^3\sin^2\phi} v_\th.
	\end{split}\right. \]
	Based on the above expressions, we obtain from (\ref{s1e}) and (\ref{s2e}) that
	\be\label{p2d-vth-e1}
	\Big\| \frac{1}{\rho^2} \p_\rho \p_\phi v_\th \Big\|_{L_{tx}^2} + \Big\|  \frac{1}{\rho^3}\p_\phi^2 v_\th \Big\|_{L_{tx}^2} \leq C.
	\ee
	In a similar way, by analyzing $ \nabla F = \na \Big( \frac{\o_\phi}{\rho} \Big) = -\na \Big( \frac{1}{\rho}\p_\rho v_\th + \frac{1}{\rho^2}v_\th \Big) $, we find
	\be\label{p2d-vth-e2}
	\Big\|  \frac{1}{\rho}\p_\rho^2 v_\th \Big\|_{L_{tx}^2} \leq C.
	\ee	
	Combining (\ref{h-vth-e1}), (\ref{p2d-vth-e1}) and (\ref{p2d-vth-e2}) together leads to
	\be\label{h-vth-ub}
	\|\nabla^2 v_\th \|_{L_{tx}^2} \leq C.
	\ee	
	Next, we rewrite the equation (\ref{vth}) for $ v_\th $ as
	\[  \p_t v_\theta = \bigg(\Delta-\frac{1}{\rho^2 \sin^2\phi}\bigg)v_\theta-b\cdot\nabla v_\theta-\frac{1}{\rho}\big(v_\rho + \cot\phi\, v_\phi\big) v_\theta. \]
	Then we deduce from (\ref{h-vth-ub}) and Proposition \ref{Prop, summary} that
	\be\label{td-vth-ub}
	\| \p_t v_\th \|_{L_{tx}^2} \leq C.
	\ee
	
	Step 2: Uniform bounds on $\| \nabla^{2} v_\rho \|_{L_{tx}^2}$ and $\| \nabla^{2} v_\phi \|_{L_{tx}^2}$.
	
	According to the Biot-Savart law (\ref{Biot-Savart-sph}), we know
	\[ \bigg(\Delta + \frac{2}{\rho}\,\p_\rho + \frac{2}{\rho^2} \bigg)v_{\rho} = - \frac{1}{\rho\sin\phi}\,\p_{\phi}(\sin\phi\,\o_{\th}). \]
	By writing $ \Delta v_\rho $ into spherical coordinates, it holds that
	\be\label{mp-vrho-bsl}
	\p_\rho^2 v_\rho + \frac{1}{\rho^2} \p_\phi^2 v_\rho = R_1,
	\ee
	where
	\[ R_1 = - \frac4\rho \p_\rho v_\rho - \frac{\cot\phi}{\rho^2} \p_\phi v_\rho - \frac{2}{\rho^2} v_\rho - \frac{1}{\rho\sin\phi}\,\p_{\phi}(\sin\phi\,\o_{\th}).\]
	It is readily seen that $ \| R_1 \|_{L_{tx}^2}\leq C $ due to Proposition \ref{Prop, summary}. Then we take $ L^2(D_m\times [0,T]) $ norm on both sides of (\ref{mp-vrho-bsl}) to obtain, after rearranging terms,
	\be\label{l2-mph-vrho}
	\int_0^T \int_{D_m} \Big[ \big( \p_\rho^2 v_\rho \big)^2 + \frac{1}{\rho^4} \big( \p_\phi^2 v_\rho \big)^2 \Big] \,dx\,dt + 2 \underbrace{\int_0^T \int_{D_m} \frac{1}{\rho^2} \, \p_\rho^2 v_\rho\, \p_\phi^2 v_\rho \,dx\,dt }_{I_1} \leq C.
	\ee
	By using spherical coordinates,
	\[ I_1 = 2\pi \int_0^T \int_{\frac1m}^1 \int_{\frac{\pi}{2}-\a}^{\frac{\pi}{2}+\a} \sin\phi\, \p_\rho^2 v_\rho\, \p_\phi^2 v_\rho \,d\phi\,d\rho. \]
	Recalling the boundary conditions in Lemma \ref{Lemma, bdry cond} for $ v_\rho $: $ \p_\phi v_\rho = 0 $ on $ \p^{R} D_m $ and $ v_\rho = 0 $ on $ \p^{A}D_m $, we further deduce that $\p_\phi v_\rho = 0$ on $\p^R D_m \cup \p^A D_m$.
	Consequently, one can use integration by parts to find
	\be\label{mix-2d-vrho}\begin{split}
		I_1 = & - 2\pi \int_0^T \int_{\frac1m}^1 \int_{\frac{\pi}{2}-\a}^{\frac{\pi}{2}+\a} \cos\phi\, \p_\rho^2 v_\rho\, \p_\phi v_\rho \,d\phi\,d\rho \\
		& -2\pi \int_0^T \int_{\frac1m}^1 \int_{\frac{\pi}{2}-\a}^{\frac{\pi}{2}+\a} \sin\phi\, \p_\phi\p_\rho^2 v_\rho\, \p_\phi v_\rho \,d\phi\,d\rho \\
		:= & I_{11} + I_{12}.
	\end{split}\ee
	For $ I_{11} $, we change back to Euclidean coordinates to get
	\[I_{11} = - \int_0^T \int_{D_m} \frac{\cot\phi}{\rho^2}\, \p_\rho^2 v_\rho\, \p_\phi v_\rho \,dx\,dt. \]
	For $ I_{12} $, we apply integration by parts again to infer that
	\[ \begin{split}
		I_{12} &= -2\pi \int_0^T \int_{\frac{\pi}{2}-\a}^{\frac{\pi}{2}+\a} \sin\phi \int_{\frac1m}^1 \p_\rho(\p_\phi\p_\rho v_\rho)\, \p_\phi v_\rho \,d\rho \,d\phi\,dt \\
		&= 2\pi \int_0^T \int_{\frac{\pi}{2}-\a}^{\frac{\pi}{2}+\a} \sin\phi \int_{\frac1m}^1 (\p_\phi\p_\rho v_\rho)^2 \,d\rho\,d\phi\,dt \\
		& = \int_0^T \int_{D_m} \frac{1}{\rho^2} (\p_\phi \p_\rho v_\rho)^2 \,dx\,dt.
	\end{split}\]
	Plugging the above expressions for $ I_{11} $ and $ I_{12} $ into (\ref{mix-2d-vrho}), it then follows from (\ref{l2-mph-vrho}) that
	\[\begin{split}
		 	&\quad \int_0^T \int_{D_m} \Big[ \big( \p_\rho^2 v_\rho \big)^2 + \frac{1}{\rho^4} \big( \p_\phi^2 v_\rho \big)^2 \Big] \,dx\,dt + 2 \int_0^T \int_{D_m} \frac{1}{\rho^2} (\p_\phi \p_\rho v_\rho)^2 \,dx\,dt \\
		 	& \leq C + 2 \int_0^T \int_{D_m} \frac{\cot\phi}{\rho^2}\, \p_\rho^2 v_\rho\, \p_\phi v_\rho \,dx\,dt.		
	\end{split}\]
	Using Cauchy-Schwarz inequality, the above estimate further implies that
	\be\label{mh-vrho-e1}\begin{split}
		&\quad \frac12 \int_0^T \int_{D_m} \Big[ \big( \p_\rho^2 v_\rho \big)^2 + \frac{1}{\rho^4} \big( \p_\phi^2 v_\rho \big)^2 \Big] \,dx\,dt + 2 \int_0^T \int_{D_m} \frac{1}{\rho^2} (\p_\phi \p_\rho v_\rho)^2 \,dx\,dt \\
		& \leq C + 2 \int_0^T \int_{D_m} \frac{\cot^2\phi}{\rho^4}\, (\p_\phi v_\rho)^2 \,dx\,dt.		
	\end{split}\ee
	Thanks to Proposition \ref{Prop, summary} and the fact that $ 0\leq \cot\phi\leq 1/\sqrt{3} $,
	\[ \int_0^T \int_{D_m} \frac{\cot^2\phi}{\rho^4} (\p_\phi v_\rho)^2 \,dx\,dt \leq \frac13 \Big\| \frac1\rho \nabla v_\rho \Big\|_{L_{tx}^2} \leq C. \]
	Putting this estimate into (\ref{mh-vrho-e1}) and using Proposition \ref{Prop, summary} again, we conclude
	\be\label{h-vrho-ub}
	\|\nabla^2 v_\rho \|_{L_{tx}^2} \leq C.
	\ee	
	
	In order to estimate $ 	\|\nabla^2 v_\phi \|_{L_{tx}^2} $, we make use of the incompressible condition $ \nabla \cdot v = 0 $, which can be written as
	\be\label{div-d-vphi} \frac1\rho \p_\phi v_\phi = -\p_\rho v_\rho - \frac2\rho v_\rho -\frac{\cot\phi}{\rho} v_\phi. \ee
	By taking derivatives $ \p_\rho $ and $ \frac1\rho \p_\phi $ of (\ref{div-d-vphi}), it follows from (\ref{h-vrho-ub}) and Proposition \ref{Prop, summary} that
	\be\label{p2d-vphi-e1}
	\Big\| \frac{1}{\rho} \p_\rho \p_\phi v_\phi \Big\|_{L_{tx}^2} + \Big\|  \frac{1}{\rho^2}\p_\phi^2 v_\phi \Big\|_{L_{tx}^2} \leq C.
	\ee
	Then according to the Biot-Savart law (\ref{Biot-Savart-sph}) again,
	\[\bigg(\Delta-\frac{1}{\rho^2\sin^2\phi}\bigg) v_{\phi} + \frac{2}{\rho^2}\p_{\phi}v_{\rho} = \frac{1}{\rho}\,\p_{\rho}(\rho \o_{\th}). \]
	Writing $ \Delta v_\phi $ in spherical coordinates, we know
	\[ \p_\rho^2 v_\phi = -\frac{1}{\rho^2} \p_\phi^2 v_\phi - \frac2\rho \p_\rho v_\phi - \frac{\cot\phi}{\rho^2} \p_\phi v_\phi + \frac{1}{\rho^2 \sin^2\phi} v_\phi - \frac{2}{\rho^2} \p_\phi v_\rho + \frac1\rho \p_\rho (\rho \o_\th). \]
	Based on (\ref{p2d-vphi-e1}) and Proposition \ref{Prop, summary}, we infer from the above expression that
	\be\label{p2d-vphi-e2} \| \p_\rho^2 v_\phi \|_{L_{tx}^2}\leq C. \ee
	Combining (\ref{p2d-vphi-e1}), (\ref{p2d-vphi-e2}) with Proposition \ref{Prop, summary} yields
	\be\label{h-vphi-ub}
	\|\nabla^2 v_\phi \|_{L_{tx}^2} \leq C.
	\ee	
	
	Step 3: Uniform bound on $\| \p_{t} \o_\th \|_{L_{tx}^2}$.
	
	By rearranging the equation (\ref{omega-th-eq}) for $ \o_\th $, we have
	\be\label{oth-eqm}
	\left\{\begin{array}{lll}
		\Delta\o_{\th} - \p_{t}\o_{\th} = R_2, & \text {in} & D_{m} \times(0, T]; \\
		\o_{\th} = 0,  & \text {on} & \partial D_{m} \times(0, T]; \\
		\o_{\th}(x, 0) = \omega_{0,\th}(x),  &  \text{in} &  D_{m},
	\end{array}\right.
	\ee
	where
	\[ R_2 =  \frac{1}{\rho^2\sin^2\phi} \o_\th + b\cdot\nabla \o_{\th} - \frac{1}{\rho}(v_{\rho} + \cot\phi\,v_{\phi})\o_{\th} + \frac{1}{\rho^2}\p_{\phi}(v_{\th}^2) - \frac{\cot\phi}{\rho} \p_{\rho}(v_{\th}^2).\]
	Firstly, we deduce from the above expression and Proposition \ref{Prop, summary} that $ \| R_2 \|_{L_{tx}^2}\leq C $. Then according to (\ref{oth-eqm}),
	\[ \int_0^T \int_{D_m} (\Delta \o_\th - \p_t \o_\th)^2 \,dx\,dt = \int_0^T \int_{D_m} R_2^2 \,dx\,dt \leq C.\]
	Equivalently,
	\be\label{dt-oth-e1}
	\int_0^T \int_{D_m} (\Delta \o_\th)^2 + (\p_t \o_\th)^2 \,dx\,dt \leq C + 2 \int_0^T \int_{D_m} (\Delta \o_\th)(\p_t \o_\th) \,dx\,dt.
	\ee
	Since $ \o_\th = 0 $ on $ \p D_m $, which implies $ \p_t \o_\th = 0$ on $ \p D_m $, we can apply integration by parts to obtain
	\[\begin{split}
	\text{RHS of (\ref{dt-oth-e1})} &= C - 2\int_0^T \int_{D_m} (\nabla \o_\th)\cdot \p_t(\nabla \o_\th) \,dx\,dt \\
	&= C - \int_{D_m} \int_0^T \p_{t}\big( | \nabla \o_\th |^2 \big) \,dt\,dx.
	\end{split}\]
	By fundamental theorem of Calculus, the above relation further yields
	\[
	\text{RHS of (\ref{dt-oth-e1})} \leq  C + \int_{D_m} | \nabla \o_{0,\th}|^2 \,dx \leq C + C\|v_0\|_{H^2(D_m)}^2 \leq C.
	\]
	Then we infer from (\ref{dt-oth-e1}) that
	\be\label{td-oth-ub}
	\| \p_t \o_\th \|_{L_{tx}^2} \leq C.
	\ee
	
	Step 4: Uniform bounds on $\| \p_{t} v_\rho \|_{L_{tx}^2}$ and $\| \p_{t} v_\phi \|_{L_{tx}^2}$.
	
	We first estimate $ \| \p_t v_\rho \|_{L_{tx}^2} $. Recall by the Biot-Savart law, $ v_\rho $ solves the following problem:	
	\[\left\{\begin{array}{ll}
		\Big( \Delta + \frac2\rho \p_\rho + \frac{2}{\rho^2} \Big) v_\rho = - \frac{1}{\rho \sin\phi}\, \p_\phi(\sin\phi\, \o_\th) \quad \text{ in } \quad D_m,  \vspace{0.05in}\\
		\p_\phi v_\rho = 0 \quad \text{on}\quad \p^{R}D_m, \quad v_\rho = 0 \quad\text{on}\quad \p^{A}D_m.
	\end{array}\right.\]
	By taking derivative with respect to $ t $, we find $ \p_t v_\rho $ satisfies the equations below:
	\[\left\{\begin{array}{ll}
		\Big( \Delta + \frac2\rho \p_\rho + \frac{2}{\rho^2} \Big) (\p_t v_\rho) = - \frac{1}{\rho \sin\phi}\, \p_\phi(\sin\phi\, \p_t\o_\th) \quad \text{ in } \quad D_m,  \vspace{0.05in}\\
		\p_\phi (\p_t v_\rho) = 0 \quad \text{on}\quad \p^{R}D_m, \quad \p_t v_\rho = 0 \quad\text{on}\quad \p^{A}D_m.
	\end{array}\right.\]
	We emphasize that the above equation can be made rigorously by firstly considering the finite difference in time or the Steklov average of $ v_\rho $ and $ \o_\th $ instead of $ \p_t v_\rho $ and $ \p_t \o_\th $, and then taking the limit.
	Meanwhile, we have
	\be\label{ave0-vrho-dt}
	\int_{\frac{\pi}{2}-\a}^{\frac{\pi}{2}+\a} (\p_t v_\rho) \sin\phi \,d\phi = \p_t \bigg( \int_{\frac{\pi}{2}-\a}^{\frac{\pi}{2}+\a} v_\rho \sin\phi \,d\phi \bigg) = 0.
	\ee
	Then we can argue in an analogous way as that for the proof of Lemma \ref{Lemma, e-trans-v_rho} to deduce
	\[ \| \nabla \p_t v_\rho \|_{L_{tx}^2} \leq C \| \p_t \o_\th \|_{L_{tx}^2}\leq C, \]
	where the last inequality is due to (\ref{td-oth-ub}). Now, by taking advantage of (\ref{ave0-vrho-dt}), it follows from the Poincar\'e inequality in Lemma \ref{Lemma, P ave} that
	\be\label{td-vrho-ub}
	\|\p_t v_\rho\|_{L_{tx}^2} \leq C \|\nabla \p_t v_\rho\|_{L_{tx}^2} \leq C.
	\ee
	Similar to the above argument, we are also able to justify
	\be\label{td-vphi-ub}
	\|\p_t v_\phi\|_{L_{tx}^2} \leq C.
	\ee
	
After establishing the previous Step 1-4, the desired estimate (\ref{high-reg}) will be readily proved. Firstly, the uniform bound $ \|\p_t v\|_{L_{tx}^2}\leq C $ follows from (\ref{td-vth-ub}), (\ref{td-vrho-ub}) and (\ref{td-vphi-ub}). Secondly, since
\[ \| \nabla^2 v \|_{L_{tx}^2} \leq C \bigg( \|\nabla^2 v_\rho\|_{L_{tx}^2} +  \|\nabla^2 v_\phi\|_{L_{tx}^2} + \|\nabla^2 v_\th\|_{L_{tx}^2} + \Big\| \frac1\rho \nabla v \Big\|_{L_{tx}^2} + \Big\| \frac{1}{\rho^2} v \Big\|_{L_{tx}^2} \bigg),\]
the uniform bound $ \|\nabla^2 v\|_{L_{tx}^2}\leq C$ follows from (\ref{h-vth-ub}), (\ref{h-vrho-ub}), (\ref{h-vphi-ub}) and Proposition \ref{Prop, summary}.
\end{proof}

\section{Completion of the proof of Theorem \ref{Thm, cyl}: existence and uniqueness of strong solutions}

\label{Sec, exist-uniq}

In this section, we will establish the main result of this paper by utilizing the uniform bounds derived in the previous Section \ref{Sec, inf-ub} and Section \ref{Sec, hreg-ub}.

\begin{proof}[Proof of Theorem {\ref{Thm, cyl}}]
We first show the existence of a strong solution $ (v,P) $ which has the even-odd-odd symmetry and satisfies (\ref{ss-ub}) and (\ref{v-energy-est}). Pick any $ v_0 $ in the admissible class $ \mathscr{A} $ that satisfies the properties (i) and (ii) in Theorem \ref{Thm, cyl}. By Definition \ref{Def, admissible sets}, there exists a sequence $ \{v^{(m)}_0\}_{m\geq 2} $ such that $ v^{(m)}_0\in\mathscr{A}_m $ and
\be\label{conv-init}
\lim_{m\to\infty} \| v_0 - v^{(m)}_0 \|_{C^2(\ol{D_m})} = 0.
\ee
Since $ v_0 $ has the even-odd-odd symmetry due to property (i), we can modify $ v^{(m)}_0 $ so that it enjoys the same symmetry as well. In fact, by setting
\[ \t{v}^{(m)}_0 = \t{v}^{(m)}_{0,\rho} e_\rho + \t{v}^{(m)}_{0,\phi} e_\phi + \t{v}^{(m)}_{0,\th} e_\th, \]
where
\[\left\{\begin{split}
	\t{v}^{(m)}_{0,\rho}(\rho,\phi) = \big[ v^{(m)}_{0,\rho}(\rho,\phi) + v^{(m)}_{0,\rho}(\rho,\pi-\phi)\big]/2, \\
	\t{v}^{(m)}_{0,\phi}(\rho,\phi) = \big[ v^{(m)}_{0,\phi}(\rho,\phi) - v^{(m)}_{0,\phi}(\rho,\pi-\phi)\big]/2, \\
	\t{v}^{(m)}_{0,\th}(\rho,\phi) = \big[ v^{(m)}_{0,\th}(\rho,\phi) - v^{(m)}_{0,\th}(\rho,\pi-\phi)\big]/2.
\end{split}\right.\]
then one can directly check that $\t{v}^{(m)}_0$ possesses the even-odd-odd symmetry and $ \t{v}^{(m)}_0\in\mathscr{A}_m $. In addition, the convergence (\ref{conv-init}) is still valid by replacing $ v^{(m)}_0 $ with $\t{v}^{(m)}_0$. For ease of notation, we still denote $\t{v}^{(m)}_0$ to be $ v^{(m)}_0 $. On the other hand, due to the convergence (\ref{conv-init}) and the fact that $ \| r v_{0,\th} \|_{L^\infty(D)} \leq \frac{1}{100} $ due to property (ii), there exists some $ m_0 $ such that for any $ m\geq m_0 $,
\begin{align}
	&\|v^{(m)}_0\|_{C^2(\ol{D_m})} \leq \|v_0\|_{C^2(\ol{D})} + 1,	\label{ub-init-C2} \\
	& \| r v^{(m)}_{0,\th} \|_{L^\infty(D_m)} \leq \frac{1}{95}. \label{ub-init-Gamma}
\end{align}
In the following, we will only consider those $ v^{(m)}_0 $ for $ m\geq m_0 $.

Now we fix any time $ T>0 $. According to Corollary \ref{Cor, globle soln in ad}, for each $ m $, there exists a strong solution $ (v^{m},P^{(m)}) $ of (\ref{asns-sph}) on $ D_m\times[0,T] $ with the initial data $ v^{(m)}_0 $ and the NHL boundary condition \ref{NHL slip bdry for Dm-sph}. In addition, $ v^{(m)} $ is bounded and has the even-odd-odd symmetry. On the other hand, we can assume
\be\label{ave-P-0}
\int_{D_m} P^{(m)}(x,t)\,dx = 0, \quad\forall\, t\in[0,T].
\ee
Actually, if we define $ \wt{P}^{(m)}(x,t) = P^{(m)}(x,t) - \frac{1}{|D_m|}\int_{D_m}P^{(m)}(x,t)\,dx $, then $ \wt{P}^{(m)} $ satisfies (\ref{ave-P-0}) and $ (v^{m}, \wt{P}^{(m)}) $ is also a strong solution.

Next, according to Proposition \ref{Prop, summary} and Proposition \ref{Prop, high-reg}, there exists some constant $ C_m $, only depending on $ \a $ and $ \|v^{(m)}_0\|_{C^2(\ol{D_m})} $ such that
\be\label{ubdd-vm}
\|v^{(m)}\|_{L_{tx}^\infty(D_m\times[0,T])} + \|v^{(m)}\|_{H_t^1 L_x^2(D_m\times[0,T])} + \|v^{(m)}\|_{L_t^2 H_x^2(D_m\times[0,T])}  \leq C_m.
\ee
Meanwhile, since $ (v^{m},P^{m}) $ is a strong solution of (\ref{asns-sph}), then
\[
\left\{\begin{array}{l}
	\p_{\rho}P = \big(\Delta + \frac{2}{\rho}\,\p_\rho + \frac{2}{\rho^2} \big)v_{\rho}-b\cdot\nabla v_{\rho}+\frac{1}{\rho}(v_{\phi}^2+v_{\th}^2) -\p_{t}v_{\rho},  \vspace{0.08in}\\
	\frac{1}{\rho}\p_{\phi}P = \big(\Delta-\frac{1}{\rho^2\sin^2\phi}\big)v_{\phi}-b\cdot\nabla v_{\phi}+\frac{2}{\rho^2}\p_{\phi}v_{\rho} -\frac{1}{\rho}v_{\rho}v_{\phi}+\frac{\cot\phi}{\rho}v_{\th}^2 - \p_{t}v_{\phi}.
\end{array}\right.\]
By applying Proposition \ref{Prop, summary} and \ref{Prop, high-reg} again, we find that
\[ \|\na P^{(m)}\|_{L_{tx}^2(D_m\times[0,T])} \leq C_m. \]
Thanks to (\ref{ave-P-0}), the above estimate further implies that
\be\label{ub-P}
\|P^{(m)}\|_{L_t^2 H_x^1(D_m\times[0,T])} \leq C_m.
\ee
Now taking advantage of the uniform bound (\ref{ub-init-C2}), we infer from (\ref{ubdd-vm}) and (\ref{ub-P}) that
\be\label{ubdd-vm-Pm}
\|v^{(m)}\|_{L_{tx}^\infty(D_m\times[0,T])} + \|v^{(m)}\|_{H_t^1 L_x^2(D_m\times[0,T])} + \|v^{(m)}\|_{L_t^2 H_x^2(D_m\times[0,T])} + \|P^{(m)}\|_{L_t^2 H_x^1(D_m\times[0,T])} \leq C,
\ee
where $ C $ is a constant that only depends on $ \a $ and $ \|v_0\|_{C^2(\ol{D})} $. Meanwhile, recalling Proposition \ref{Prop, mod l-h energy}, the following energy inequality for $ v^{(m)} $ holds:
\be\label{vm-es}
\int_{D_m} |v^{(m)}(x, T)|^2 \,dx + \frac23 \int^T_0 \int_{D_m} |\na v^{(m)}(x, t)|^2 \,dx\,dt \le \int_{D_m} |v^{(m)}_0(x)|^2 \,dx.
\ee

Thanks to the uniform bound (\ref{ubdd-vm-Pm}) and the fact that $ D_m $ is increasing to $ D $ with respect to containment, we can extract a subsequence, still denoted as $ \big\{ (v^{(m)}, P^{(m)}) \big\} $, and a vector field $ v\in L_{tx}^\infty\cap H_t^1 L_x^2\cap L_t^2 H_x^2(D\times[0,T])$ and a pressure term $ P\in L_t^2 H_x^1(D\times[0,T])$ such that
\begin{align}
	& \mbox{$v^{(m)}\to v$ pointwisely on $ D\times[0,T] $,} \label{conv-p}\\
	& \mbox{$v^{(m)}\to v$ weakly in $ H_t^1 L_x^2\cap L_t^2 H_x^2(D\times[0,T]) $, $ P^{(m)}\to P $ weakly in $ L_t^2 H_x^1(D\times[0,T]) $,} \label{conv-n} \\
	& \mbox{$v^{(m)}\to v$, $ \na v^{(m)}\to \na v $, $ \na\times v^{(m)}\to \na\times v $ weakly in $ L_{tx}^2(\p D\times[0,T]) $,} \label{conv-st} \\
	& \mbox{$ v^{(m)}(\cdot,t)\to v(\cdot, 0) $, $ v^{(m)}(\cdot, T)\to v(\cdot, T) $ weakly in $ L^2(D) $.} \label{conv-tt}
\end{align}
Based on (\ref{conv-p}), (\ref{conv-n}), we infer from (\ref{ubdd-vm-Pm}) that
\be\label{ubdd-v-P}
\|v\|_{L_{tx}^\infty(D\times[0,T])} + \|v\|_{H_t^1 L_x^2(D\times[0,T])} + \|v\|_{L_t^2 H_x^2(D\times[0,T])} + \|P\|_{L_t^2 H_x^1(D\times[0,T])} \leq C,
\ee
where $ C $ is a constant that only depends on $ \a $ and $ \|v_0\|_{C^2(\ol{D})} $. Meanwhile, we can also deduce from (\ref{conv-p}) that $ v $ enjoys the even-odd-odd symmetry.

Since $ (v^{(m)}, P^{(m)}) $ satisfies (\ref{asns-sph}) in $ L_{tx}^2 $ sense on $ D_m\times[0,T] $ with initial data $ v_0 $ and the NHL boundary condition (\ref{NHL slip bdry for Dm-sph}), by changing back to Euclidean coordinates, we know that $ (v^{(m)},P^{(m)}) $ solves (\ref{nse}) in $ L_{tx}^2 $ sense on $ D_m\times[0,T] $ with initial data $ v_0 $ and the NHL boundary condition (\ref{NHL slip bdry for Dm}). More precisely,
\begin{align}
	& \Delta v^{(m)} -  (v^{(m)}\cdot \nabla) v^{(m)} - \nabla P^{(m)} -\partial_t v^{(m)} = 0, \quad  \text{in }\, L^2\big(D_m\times(0,T]\big),  \label{mpi-eq} \\
	& \na\cdot v^{(m)} = 0, \quad  \text{in }\, L^2\big(D_m\times[0,T]\big),  \label{mpi-divfree} \\
	& v^{(m)}(\cdot, 0) = v^{(m)}_0, \quad  \text{in }\, L^2(D_m), \label{mpi-initial}  \\
	& v^{(m)}\cdot n = 0, \quad (\na\times v^{(m)})\times n = 0, \quad  \text{on }\, L^2\big(\p D_m\times[0,T]\big).  \label{mpi-bdry}
\end{align}
According to (\ref{conv-p}) and (\ref{conv-n}), we have
\[ \big[\Delta v^{(m)} -  (v^{(m)}\cdot \nabla) v^{(m)} - \nabla P^{(m)} -\partial_t v^{(m)}\big] \,\to\, \big[\Delta v -  (v\cdot \nabla) v - \nabla P -\partial_t v\big] \]
weakly in $ L_{tx}^2(D\times[0,T]) $. It then follows from (\ref{mpi-eq}) that
\be\label{pi-eq}
\Delta v -  (v\cdot \nabla) v - \nabla P -\partial_t v = 0, \quad  \text{in}\quad L^2\big(D\times(0,T]\big).
\ee
Similarly, by using the convergence (\ref{conv-p})--(\ref{conv-tt}) and (\ref{conv-init}), we can deduce from the identities (\ref{mpi-divfree})--(\ref{mpi-bdry}) that
\be\label{pi-aux}\left\{\begin{split}
	& \na\cdot v = 0, \quad  \text{in}\quad L^2\big(D\times[0,T]\big),   \\
	& v(\cdot, 0) = v_0, \quad  \text{in}\quad L^2(D),  \\
	& v\cdot n = 0, \quad (\na\times v)\times n = 0, \quad  \text{on}\quad L^2\big(\p D\times(0,T]\big).
\end{split}\right.\ee
Hence, the combination of (\ref{pi-eq}) and (\ref{pi-aux}) shows that $ (v,P) $ is a strong solution of (\ref{nse}) (or equivalently (\ref{eqasns})) on $ D\times[0,T] $ with the initial data $ v_0 $ and the NHL boundary condition \ref{NHL slip bdry}. Moreover, it follows from  (\ref{vm-es}), (\ref{conv-n}) and (\ref{conv-tt}) that
\[ 	\int_{D} |v(x, T)|^2 \,dx + \frac23 \int^T_0 \int_{D} |\na v(x, t)|^2 \,dx\,dt \leq \liminf_{m\to\infty} \int_{D_m} |v^{(m)}_0(x)|^2\,dx  = \int_{D} |v_0(x)|^2 \,dx, \]
where the last equality is due to (\ref{conv-init}). Thus, we indeed find a strong solution $ (v,P) $ which has the even-odd-odd symmetry and satisfies (\ref{ss-ub}) and (\ref{v-energy-est}).

It remains to verify the uniqueness of the strong solution $ v $. Suppose  $ (\t{v}, \t{P}) $ is another strong solution, with even-odd-odd symmetry, of (\ref{nse}) with the same initial data $ v_0 $ and the NHL boundary condition (\ref{NHL slip bdry}) on $ D\times[0,T] $.
We will prove that $ \t{v}$ coincides with $ v $. Let $ f = v-\t{v} $ and $ g= P-\t{P} $. Then $ f $ satisfies
		\be\label{diff-ss} \left\{\, \begin{aligned}
				\Delta f - (f\cdot \nabla) v - (\t v \cdot\nabla)f -  \nabla g - \p_{t} f = 0 \quad\text{in}\quad & D\times (0,T], \\
				\nabla \cdot f = 0  \quad \text{in} \quad & D\times (0,T], \\
				f\cdot n = 0,\quad (\nabla\times f) \times n = 0 \quad\text{on} \quad & \p D\times (0,T],\\
				f(\cdot, 0) = 0 \quad\text{in} \quad & D.
			\end{aligned} \right.\ee
		Since both $ v $ and $ \t{v} $ are strong solutions, $ f $ belongs to the space $\mathscr{S}_T$ of test functions  defined in (\ref{ts}). For any $ 0 < T_1 <T $, we test the first equation in (\ref{diff-ss}) by $ f $ on $ D\times[0,T_1] $ to find
		\[\begin{split}
				&\quad \int_0^{T_1}\int_{D} (\Delta f)\cdot f \,dx\,dt - \int_0^{T_1}\int_{D} [(f\cdot \nabla) v]\cdot f \,dx\,dt \\
				&= \int_0^{T_1}\int_{D} [(\t{v} \cdot\nabla)f]\cdot f \,dx\,dt +  \int_0^{T_1}\int_{D} (\nabla g)\cdot f\,dx\,dt +  \int_0^{T_1}\int_{D} (\p_t f)\cdot f\,dx\,dt.
			\end{split} \]
		Thanks to the boundary condition and the incompressibility condition of $ \t{v} $ and $ f $, we know $ \int_0^{T_1}\int_{D} [(\t v \cdot\nabla)f]\cdot f \,dx\,dt = \int_0^{T_1}\int_{D} (\nabla g)\cdot f\,dx\,dt = 0 $, so
		\be\label{eq1-diff-ss} \int_0^{T_1}\int_{D} (\Delta f)\cdot f \,dx\,dt - \int_0^{T_1}\int_{D} [(f\cdot \nabla) v]\cdot f \,dx\,dt = \frac12 \int_{D} f^2(x,T_1)\,dx.  \ee
		Then by the similar computation as that in Section \ref{Subsec, weak-form-soln}, we know
		\[ \int_0^{T_1}\int_{D} (\Delta f)\cdot f \,dx\,dt = -  \int_0^{T_1}\int_{D} |\nabla \times f|^2 \,dx\,dt.  \]
		On the other hand, by definition,
		\[
			\int_0^{T_1}\int_{D} [(f\cdot \nabla) v]\cdot f \,dx\,dt = \sum_{i,j=1}^3 \int_0^{T_1}\int_{D} f_j (\p_{x_j} v_i) f_i \,dx\,dt.
		\]
		Then using integration by parts and taking advantage of the boundary condition and the incompressibility condition of $ f $, we infer that
		\[ \int_0^{T_1}\int_{D} [(f\cdot \nabla) v]\cdot f \,dx\,dt = -\sum_{i,j=1}^3 \int_0^{T_1}\int_{D} f_j v_i (\p_{x_j} f_i ) \,dx\,dt. \]
		Plugging the above results into (\ref{eq1-diff-ss}) yields
		\be\label{es1-diff-ss}
		\frac12 \int_{D} |f(x,T_1)|^2\,dx +  \int_0^{T_1}\int_{D} |\nabla \times f|^2 \,dx\,dt =  \sum_{i,j=1}^3 \int_0^{T_1}\int_{D} f_j v_i (\p_{x_j} f_i ) \,dx\,dt.
		\ee
		Since $ v\in L_{tx}^{\infty} $ and $ f\in L_t^2 H_x^1 $ on $ D\times[0,T] $, we deduce from (\ref{es1-diff-ss}) that
		\be\label{es2-diff-ss}\begin{split}
				&\frac12 \int_{D} |f(x,T_1)|^2\,dx +  \int_0^{T_1}\int_{D} |\nabla \times f|^2 \,dx\,dt \\ \leq\,\, & C\int_0^{T_1}\int_{D} |f(x,t)|^2\,dx\,dt + \frac16 \int_0^{T_1}\int_{D} |\na f|^2 \,dx\,dt,
			\end{split}\ee
		where $ C $ only depends on $ \| v \|_{L_{tx}^\infty(D\times[0,T]) } $. Since both $ v $ and $ \t v $ own the even-odd-odd symmetry, then so does $ f $. Therefore, we are able to take advantage of the estimate in Remark \ref{Re, curl-grad-D} to find
		\[ \int_0^{T_1}\int_{D} |\na f|^2 \,dx\,dt \leq 3 \int_0^{T_1}\int_{D} |\nabla \times f|^2 \,dx\,dt.\]
		Putting this estimate into (\ref{es2-diff-ss}) yields
		\be\label{es3-diff-ss}
		\int_{D} |f(x,T_1)|^2\,dx \leq 2C\int_0^{T_1}\int_{D} |f(x,t)|^2\,dx\,dt, \quad \forall\, 0<T_1\leq T.
		\ee
		Finally, since both $ v $ and $ \t v $ has the same initial data, $ f(x,0)=0 $ on $ D $. As a result, it follows from (\ref{es3-diff-ss}) and Gr\"onwall's inequality that $ f=0 $ on $ D\times[0,T] $. So $ \t v = v $ on $ D\times[0,T] $, which justifies the uniqueness of the strong solution $ v $. This completes the proof of Theorem \ref{Thm, cyl}.
	
\end{proof}

\section{Blowup solutions with finite energy on special cusp domains}
\label{Sec, blowupsoln}

As a byproduct of studying the NHL boundary condition \eqref{NHL slip bdry}, we will construct a class of blowup solutions to the ASNS \eqref{eqasns} with finite energy on some cusp domains $D_*$. This type of domains was considered in \cite{Zha22} to establish the global existence of bounded solutions to \eqref{eqasns} with finite energy for any smooth initial data under the Navier slip boundary condition as below:
\be
\label{nvslbc}
v \cdot n = 0,   \qquad   (\mathbb{S}(v) n)_{tan} = 0, \quad \text{on} \quad \p D_*.
\ee
Here, $n$ is the unit outward normal of the smooth part of $\p D_*$,  $\mathbb{S}(v)=\frac{1}{2} \big[\na v + (\na v)^T \big]$ is the strain tensor and $(\mathbb{S}(v) n)_{tan}$ stands for the tangential component of the vector $\mathbb{S}(v) n$.
Now we give a precise description of the domain $ D_* $.
\begin{definition}
	\label{Def, Ds}
Let $\beta \in (1, \infty)$ be any number. Define the domain $ D_* $ as follows (also see Figure \ref{Fig, cusp-cyl}).
\be
\lab{dodm}
\al
D_*& := \bigcup^\infty_{m=1}  D_m, \quad \text{with} \qquad  D_m  := \bigcup^m _{j=1}  S_j,  \\
S_j  &:=   \{  (r, x_3)  \, | \,  2^{-j} \le r < 2^{-(j-1)}, \,    0< x_3 < 2^{-\beta (j-1)} \}.
\eal
\ee
\end{definition}

\begin{figure}[!ht]
	\centering
	\includegraphics[scale=0.2]{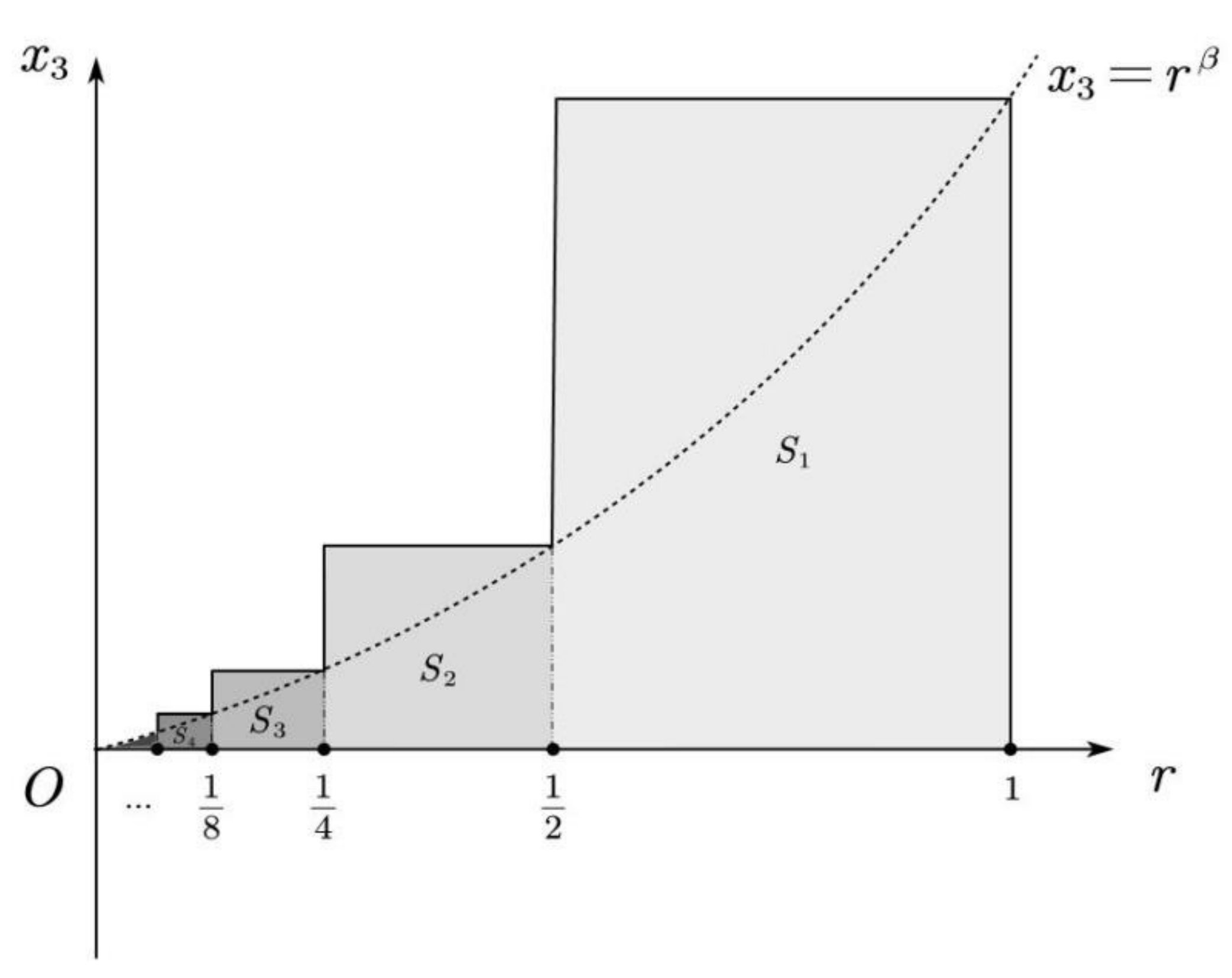}
	\caption{Domain $D_*$ in cylindrical coordinates}
	\label{Fig, cusp-cyl}
\end{figure}

In \cite{Zha22}, one of us chose the parameter $ \beta $ in Definition  \ref{Def, Ds} to lie in $ (1,1.1) $ and proved that no finite-time blowup occurs under the Navier slip boundary condition \eqref{nvslbc}. Now our observation is that when the domain $ D_* $ is sufficiently thin (say when $ \beta>2 $), then a mildly singular forcing term in the standard regularity class can generate infinite speed for the fluid under the NHL boundary condition \eqref{NHL slip bdry}.

Note that the boundary
$\p D_*$ can be written as the union of horizontal and vertical parts, which are denoted by $\p^H D_*$ and $\p^V D_*$ respectively.  Namely,
\be
\lab{d=dhdv}
\p D_* = \p^H D_* \cup \p^V D_*.
\ee
From  \eqref{NHL slip bdry}, one sees that the NHL boundary condition can be expressed explicitly as
\be
\lab{nvbc2}
\bali
v_3=0, \quad \o_r = \o_\th=0, \quad \text{on} \quad \p^H D_*,\\
v_r=0, \quad \o_3=\o_\th=0, \quad \text{on} \quad \p^V D_*.
\eali
\ee

\begin{proposition}
\label{prblows}
Let $D_*$ be the cusp domain in Definition \ref{Def, Ds} with $\beta>2$. Let $\eta=\eta(t)$ be a smooth function of time $t \ge 0$ such that
\[ \eta(t) \left\{\begin{array}{lll}
=0 & \mbox{for} & t\in[0,1], \\
=1 & \mbox{for} & t\geq 2.
\end{array}\right. \]
Then $v := \frac{\eta(t)}{r} e_\th$ is an unbounded solution of the forced axially symmetric Navier-Stokes equation:
\be
\label{nsF}
\Delta v -v \nabla v -\nabla P -\p_t v = - \frac{\eta'(t)}{r} e_\th \quad \text{on} \quad D_* \times [0, \infty),
\ee which satisfies the NHL boundary condition \eqref{NHL slip bdry}. Moreover, for any $ T>0 $, $v$ is in the energy space with respect to the space time domain $D_* \times [0, T]$ and the forcing term is in the standard regularity class $L^\infty_tL^{1.5+}_x \big( D_* \times [0, T] \big)$.
\end{proposition}

\begin{proof}
After setting $v_r=0$ and $v_3=0$, \eqref{nsF} is reduced to:
\be
\begin{aligned}
	\label{pswasns}
	\begin{cases}
		\frac{(v_{\theta})^2}{r}-\partial_r
		P=0,\\
		\big   (\Delta-\frac{1}{r^2}  \big
		)v_{\theta}-
		\partial_t v_{\theta} = -\frac{\eta'(t)}{r},\\
		\p_{x_3} P=0.
	\end{cases}
\end{aligned}
\ee
Then by choosing $ P=-\frac{\eta^2(t)}{2r^2} $, we see that $ v_\th = \frac{\eta(t)}{r} $ solves \eqref{pswasns}. As a result, $v := \frac{\eta(t)}{r} e_\th$ is a solution of \eqref{nsF}.
Meanwhile, it is readily seen that $ \nabla\times v = 0 $, so the NHL boundary condition \eqref{nvbc2} is satisfied.

Thanks to the condition $\beta>2$ in the definition of $D_*$, one can easily deduce
\[
\int_{D_*} |v|^2(x, t) \,dx \le \eta^2(t) \int_{D_*} \frac{1}{r^2} \,dx
\le \eta^2(t)\, 2 \pi \int^{1}_0 \int^{2^\beta r^\beta}_0 \frac1r \,d x_3\, dr<\infty,
\]
\[
\int^T_0 \int_{D_*} |\na v|^2(x, t) \,dx \le \sup_{t\geq 0} \eta^2(t) \int_{D_*} \frac{1}{r^4} \,dx \,dt
\le \sup_{t\geq 0} \eta^2(t)\, 2 \pi \int^{1}_0 \int^{2^\beta  r^\beta}_0 \frac{1}{r^3} \,dx_3 \,dr < \infty.
\]
Therefore, $v$ is in the energy space with respect to $D_* \times [0, T]$. It is also clear that the forcing term $-\frac{\eta'(t)}{r} e_\th$ is in $L^\infty_t L^2_x \subset
L^\infty_t L^{1.5+}_x$ which is the standard regularity class. This proves the proposition.
\end{proof}

Finally, we recall the remarkable paper \cite{ABC22} in which nonuniqueness is established for the Navier-Stokes equations with a supercritical forcing term in $\mathbb{R}^3$. In contrast, in the aforementioned Proposition \ref{prblows}, the forcing term,  with a scaling factor of $-1$, is subcritical, but the domains are special.  It confirms the intuition that if the channel of a fluid is very thin, arbitrarily high speed in the classical sense can be attained under a mildly singular force which is physically reasonable in view that Newtonian gravity and Coulomb force have scaling factor $-2$.

\appendix

\section{Derivations of equations in spherical coordinates}
\label{Sec, deri of eqs}

The main purpose of this appendix is to give a short derivation for the equations of the key quantities $ K $, $ F $ and $ \O $ (see (\ref{K-F-O})) which are necessary in the proof of the boundedness of the velocity.

But at first, we will give an alternative derivation of the equations for the velocity and the vorticity of the Navier-Stokes system in the spherical coordinate system by using the tensor notation which seems succinct and accessible. Moreover, the equations of the velocity $ v $ and the vorticity $ \o $ may be slightly different from the classical ones since we will rewrite them using the divergence free condition to fit our purpose.

\subsection{Velocity equation (\ref{asns-sph})}
\label{Subsec, deri of eq of v}

\quad

We will derive (\ref{asns-sph}) from the results obtained for the cylindrical coordinates.  First, it follows from (\ref{cyl-to-sph}) that
\be\label{sph-to-cyl}
\left\{\begin{array}{l}
e_r = \sin\phi\,e_{\rho}+\cos\phi\,e_\phi,\\
e_3 = \cos\phi\,e_{\rho}-\sin\phi\,e_\phi,
\end{array}\right.
\qquad
\left\{\begin{array}{l}
v_r = \sin\phi\,v_{\rho} + \cos\phi\,v_\phi,\\
v_3 = \cos\phi\,v_{\rho}-\sin\phi\,v_\phi,
\end{array}\right.\ee
where the basis $ (e_r, e_\th, e_3) $ and $ (e_\rho, e_\phi, e_\th) $ are defined as in (\ref{ert3}) and (\ref{erpt}) respectively. In addition, due to relation (\ref{coord-sph to cyl}), we know
\be\label{deri-sph to cyl}
\begin{cases}
\p_{r} = \sin\phi\,\p_{\rho}+\frac{\cos\phi}{\rho}\,\p_{\phi},\vspace{0.05in}\\
\p_{x_3} = \cos\phi\,\p_{\rho}-\frac{\sin\phi}{\rho}\,\p_{\phi}.
\end{cases}\ee
Consequently,
\be\label{div free-sph}\begin{split}
\text{div}\,v = \Big(\p_{r}+\frac{1}{r}\Big) v_{r} + \p_{x_3}v_{3} & = \Big(\p_{\rho}+\frac{2}{\rho}\Big)v_{\rho}+\frac{1}{\rho}(\p_{\phi}+\cot \phi)v_{\phi}\\
&= \frac{1}{\rho^2}\p_{\rho}(\rho^2 v_{\rho})+\frac{1}{\rho\sin\phi}\p_{\phi}(\sin\phi\,v_{\phi}),
\end{split}\ee
and
\be\label{nabla op-sph}
\nabla =\p_{r}\otimes e_{r}+\frac{1}{r}\p_{\th}\otimes e_{\th}+\p_{x_3}\otimes e_{3} =\p_{\rho}\otimes e_{\rho}+\frac{1}{\rho}\p_{\phi}\otimes e_{\phi}+\frac{1}{\rho\sin\phi}\p_{\th}\otimes e_{\th}. \ee
Furthermore,
\be\label{lapla op-sph}\begin{split}
\Delta &=\frac{1}{r}\p_{r}(r\p_{r})+\frac{1}{r^2}\p_{\th}^2+\p_{x_3}^2\\
&=\frac{1}{\rho^2}\p_{\rho}(\rho^2\p_{\rho})+\frac{1}{\rho^2\sin\phi}\p_{\phi}(\sin\phi\,\p_{\phi})+\frac{1}{\rho^2\sin^2\phi}\p_{\th}^2.  \\
&=\p_{\rho}^2+\frac{2}{\rho}\,\p_{\rho}+\frac{1}{\rho^2}\,\p_{\phi}^2+\frac{\cot\phi}{\rho^2}\,\p_{\phi}+\frac{1}{\rho^2\sin^2\phi}\,\p_{\th}^2.
\end{split}\ee

Noticing
\be\label{deri of basis-sph}
\left\{\begin{array}{l}
\p_{\phi} e_{\rho} = e_{\phi}\\
\p_{\phi} e_{\phi} = -e_{\rho}\\
\p_{\phi} e_{\th} = 0
\end{array}\right.
\quad\text{and}\quad
\left\{\begin{array}{l}
\p_{\th} e_{\rho} = \sin\phi\,e_{\th}\\
\p_{\th} e_{\phi}=\cos\phi\,e_{\th}\\
\p_{\th} e_{\th}=-\sin\phi\,e_{\rho}-\cos\phi\,e_{\phi},
\end{array}\right.\ee
Under tensor notations and doing vector calculus under the spherical coordinates, one finds
\begin{align*}
\nabla v &=(\p_{\rho} v)\otimes e_{\rho}+\frac{1}{\rho}\p_{\phi}v\otimes e_{\phi}+\frac{1}{\rho\sin\phi}\p_{\th}v\otimes e_{\th} \\
&= (\p_\rho v_\rho\, e_\rho + \p_\rho v_\phi\, e_\phi + \p_\rho v_\th\, e_\th)\otimes e_{\rho} + \frac{1}{\rho}\big[ (\p_\phi v_\rho-v_\phi)\, e_\rho + (v_\rho+\p_\phi v_\phi)\,  e_\phi + \p_\phi v_\th \,e_\th \big]\otimes e_\phi \\
&\quad + \frac{1}{\rho}\big[ -v_\th \, e_\rho - v_\th \cot\phi\, e_\phi + (v_\rho + v_\phi \cot\phi) \, e_\th \big] \otimes e_\th.
\end{align*}
It is convenient to denote $e_\rho \otimes e_\rho$, $e_\rho \otimes e_\phi$, $e_\rho \otimes e_\th$, $\cdots$, $e_\th \otimes e_\th$ by the nine single-entry matrices in the standard basis for $3\times 3$ matrices:
\be\label{m-basis in sph}
\begin{pmatrix} \vec{e}_{\rho}\otimes\vec{e}_{\rho}  & \vec{e}_{\rho}\otimes\vec{e}_{\phi}  & \vec{e}_{\rho}\otimes\vec{e}_{\th} \vspace{0.05in}\\
\vec{e}_{\phi}\otimes\vec{e}_{\rho} & \vec{e}_{\phi}\otimes\vec{e}_{\phi} & \vec{e}_{\phi}\otimes\vec{e}_{\th} \vspace{0.05in} \\
\vec{e}_{\th}\otimes\vec{e}_{\rho} & \vec{e}_{\th}\otimes\vec{e}_{\phi} & \vec{e}_{\th}\otimes\vec{e}_{\th}\end{pmatrix}.
\ee
Under this basis, $\nabla v$ is given by the following $3\times 3$ matrix:
\be\label{nabla vec-sph}
\nabla v =
\begin{pmatrix} \p_{\rho}v_{\rho} & \frac{1}{\rho}(\p_{\phi}v_{\rho}-v_{\phi}) & -\frac{1}{\rho}\,v_{\th} \vspace{0.05in}\\
\p_{\rho}v_{\phi} & \frac{1}{\rho}(\p_{\phi}v_{\phi}+v_{\rho}) & -\frac{\cot \phi}{\rho}\, v_{\th} \vspace{0.05in} \\
\p_{\rho}v_{\th} & \frac{1}{\rho}\p_{\phi}v_{\th}  & \frac{1}{\rho}(v_{\rho}+\cot\phi\,v_{\phi}) \end{pmatrix}.\ee
As a result,  the coordinate of $(v\cdot \nabla)v$ under the basis $\{e_{\rho},\,e_{\phi},\,e_{\th}\}$ is given by
\begin{align}
(v\cdot \nabla)v &=(\nabla v)v \notag\\
&= \begin{pmatrix} \p_{\rho}v_{\rho} & \frac{1}{\rho}(\p_{\phi}v_{\rho}-v_{\phi}) & -\frac{1}{\rho}\,v_{\th} \vspace{0.05in}\\
\p_{\rho}v_{\phi} & \frac{1}{\rho}(\p_{\phi}v_{\phi}+v_{\rho}) & -\frac{\cot \phi}{\rho}\, v_{\th} \vspace{0.05in}\\
\p_{\rho}v_{\th} & \frac{1}{\rho}\p_{\phi}v_{\th}  & \frac{1}{\rho}(v_{\rho}+\cot\phi\,v_{\phi}) \end{pmatrix}
\begin{pmatrix} v_{\rho} \vspace{0.05in}\\ v_{\phi}  \vspace{0.05in}\\ v_{\th}\end{pmatrix}.  \label{nonlin prod}
\end{align}
In other words,
\be\label{vor stretch-v}\begin{split}
(v\cdot \nabla)v &=
\Big[ \Big(v_{\rho}\p_{\rho}+\frac{1}{\rho}\,v_{\phi}\p_{\phi}\Big)v_{\rho}-\frac{1}{\rho}\,(v_{\phi}^2+v_{\th}^2) \Big]e_{\rho}
+\bigg( \Big[ v_{\rho}\Big(\p_{\rho}+\frac{1}{\rho}\Big)+\frac{1}{\rho}\,v_{\phi}\p_{\phi}\Big]v_{\phi}-\frac{\cot\phi}{\rho}\,v_{\th}^2 \bigg)e_{\phi}\\
&\quad +\Big[ v_{\rho}\Big(\p_{\rho}+\frac{1}{\rho}\Big)v_\th + \frac{1}{\rho}\,v_{\phi}(\p_{\phi}+\cot\phi)v_{\th} \Big] e_{\th}.
\end{split}\ee
Moreover,
\be\label{lapla of vec-sph}\begin{split}
\Delta v &= \Big( \p_{\rho}^2 + \frac{2}{\rho}\,\p_{\rho}+\frac{1}{\rho^2}\,\p_{\phi}^2 + \frac{\cot\phi}{\rho^2}\,\p_{\phi}+\frac{1}{\rho^2\sin^2\phi}\,\p_{\th}^2 \Big)(v_{\rho} e_{\rho} + v_{\phi} e_{\phi}+v_{\th} e_{\th})  \\
&= \Big[\Big(\Delta-\frac{2}{\rho^2}\Big)v_{\rho}-\frac{2}{\rho^2}(\p_{\phi}+\cot\phi)v_{\phi}\Big] e_{\rho}+\Big[\Big(\Delta-\frac{1}{\rho^2\sin^2\phi}\Big)v_{\phi}+\frac{2}{\rho^2}\p_{\phi}v_{\rho}\Big]\, e_{\phi}\\
&\quad + \Big(\Delta-\frac{1}{\rho^2\sin^2\phi}\Big)v_{\th}\, e_{\th}.\end{split}\ee
If $v$ is divergence free, that is $\text{div}\, v=0$, then it follows from (\ref{div free-sph}) that
\[(\p_\phi +\cot\phi)v_\phi = \frac{1}{\sin\phi}\,\p_{\phi}(\sin\phi \, v_\phi) = -\frac{1}{\rho}\,\p_\rho(\rho^2 v_\rho)=-2v_\rho - \rho \p_\rho v_\rho.\]
Putting this relation into (\ref{lapla of vec-sph}) yields
\be\label{lapla of vec-div free-sph}
\Delta v = \Big( \Delta +\frac{2}{\rho}\,\p_\rho + \frac{2}{\rho^2} \Big)v_\rho e_\rho + \Big[\Big(\Delta-\frac{1}{\rho^2\sin^2\phi}\Big)v_{\phi}+\frac{2}{\rho^2}\p_{\phi}v_{\rho}\Big] e_{\phi} + \Big(\Delta-\frac{1}{\rho^2\sin^2\phi}\Big)v_{\th} e_{\th}. \ee

Recalling that $b=v_\rho e_\rho + v_\phi e_\phi$, then the combination of (\ref{vor stretch-v}), (\ref{lapla of vec-div free-sph}) and (\ref{div free-sph}) yields (\ref{asns-sph}) which is the equivalent expression of  (\ref{eqasns}) or (\ref{nse}) under the spherical coordinates.


\subsection{Vorticity field (\ref{vor f-sph}) and  vorticity equation (\ref{asns vor-sph})}
\label{Subsec, deri of eq of omega}

\quad

Recall the vorticity $\o=\nabla\times v$. In cylindrical coordinates,
$\o=\o_r e_r + \o_\th e_\th +\o_3 e_3$, where
\be\label{vor f-cyl}
 \o_{r}=-\p_{x_3}v_{\th}, \quad \o_{\th}=\p_{x_3}v_r-\p_{r}v_3,  \quad \o_3=\p_{r}v_{\th}+\frac{1}{r}v_{\th}. \ee
In spherical coordinates,
$\o=\o_\rho e_\rho + \o_\phi e_\phi + \o_\th e_\th$. According to the relation (\ref{cyl-to-sph}),
$\o_\rho = \sin\phi\,\o_{r}+\cos\phi\,\o_3$. Then the combination of (\ref{vor f-cyl}) and the relation (\ref{deri-sph to cyl}) yields
\[\begin{split}
\o_\rho &= - \sin\phi \, \p_{x_3}v_\th + \cos\phi\, \Big( \p_{r}+\frac{1}{r} \Big)v_\th \\
&= - \sin\phi\, \Big( \cos\phi\,\p_{\rho}-\frac{\sin\phi}{\rho}\,\p_{\phi} \Big) v_\th + \cos\phi\, \Big(  \sin\phi\,\p_{\rho}+\frac{\cos\phi}{\rho}\,\p_{\phi} + \frac{1}{\rho\sin\phi} \Big)v_\th \\
&= \frac{1}{\rho}(\p_{\phi}+\cot\phi)v_{\th}.
\end{split}\]
In a similar way, we can compute $\o_\phi$ and $\o_\th$ in the spherical coordinates to verify (\ref{vor f-sph}).

Next, we will justify the vorticity equations (\ref{asns vor-sph}) in spherical coordinates.  First, we recall that the vorticity equation in the Cartesian coordinates is
\be\label{vor eq}\begin{cases}
 \Delta \o-(v\cdot \nabla)\o+(\o\cdot \nabla)v-\p_{t}\o=0,\\
\text{div}\, \o=0.
\end{cases}\ee
Since $\text{div}\, \o=0$, it follows from (\ref{lapla of vec-div free-sph}) that
\be\label{laplace of vor-sph}
\Delta \o = \Big( \Delta +\frac{2}{\rho}\,\p_\rho + \frac{2}{\rho^2} \Big)\o_\rho e_\rho + \Big[\Big(\Delta-\frac{1}{\rho^2\sin^2\phi}\Big)\o_{\phi}+\frac{2}{\rho^2}\p_{\phi}\o_{\rho}\Big] e_{\phi} + \Big(\Delta-\frac{1}{\rho^2\sin^2\phi}\Big)\o_{\th} e_{\th}. \ee

Then analogous to (\ref{nabla vec-sph}),  the coordinate of $(\o\cdot \nabla)v$ under the basis $\{e_{\rho},\,e_{\phi},\,e_{\th}\}$ is given by
\begin{align}
(\o\cdot \nabla)v &=(\nabla v)\o \notag\\
&=
\begin{pmatrix} \p_{\rho}v_{\rho} & \frac{1}{\rho}(\p_{\phi}v_{\rho}-v_{\phi}) & -\frac{1}{\rho}\,v_{\th} \vspace{0.05in} \\
\p_{\rho}v_{\phi} & \frac{1}{\rho}(\p_{\phi}v_{\phi}+v_{\rho}) & -\frac{\cot \phi}{\rho}\, v_{\th} \vspace{0.05in}\\
\p_{\rho}v_{\th} & \frac{1}{\rho}\p_{\phi}v_{\th}  & \frac{1}{\rho}(v_{\rho}+\cot\phi\,v_{\phi}) \end{pmatrix}
\begin{pmatrix} \o_{\rho} \vspace{0.05in} \\ \o_{\phi}  \vspace{0.05in} \\ \o_{\th}\end{pmatrix}.  \label{nonlin prod,  vorticity}
\end{align}
In other words,
\begin{align*}
(\o\cdot \nabla)v &= \Big[ (\p_{\rho}v_{\rho})\o_{\rho}+\frac{1}{\rho}\,(\p_{\phi}v_{\rho}-v_{\phi})\o_{\phi}-\frac{1}{\rho}\,v_{\th}\o_{\th} \Big]\,e_{\rho}\\
&\quad +\Big[ (\p_{\rho}v_{\phi})\o_{\rho}+\frac{1}{\rho}\,(\p_{\phi}v_{\phi}+v_{\rho})\o_{\phi}-\frac{\cot\phi}{\rho}\,v_{\th}\o_{\th} \Big]\,e_{\phi}\\
&\quad +\Big[ (\p_{\rho}v_{\th})\o_{\rho}+\frac{1}{\rho}\,(\p_{\phi}v_{\th})\o_{\phi}+\frac{1}{\rho}\,(v_{\rho}+\cot\phi\,v_{\phi})\o_{\th} \Big]\, e_{\th}.
\end{align*}
By switching $\o$ and $v$,
\begin{align*}
(v\cdot \nabla)\o &= \Big[ (\p_{\rho}\o_{\rho})v_{\rho}+\frac{1}{\rho}\,(\p_{\phi}\o_{\rho}-\o_{\phi})v_{\phi}-\frac{1}{\rho}\,\o_{\th}v_{\th} \Big]\,e_{\rho}\\
&\quad +\Big[ (\p_{\rho}\o_{\phi})v_{\rho}+\frac{1}{\rho}\,(\p_{\phi}\o_{\phi}+\o_{\rho})v_{\phi}-\frac{\cot\phi}{\rho}\,\o_{\th}v_{\th} \Big]\,e_{\phi}\\
&\quad +\Big[ (\p_{\rho}\o_{\th})v_{\rho}+\frac{1}{\rho}\,(\p_{\phi}\o_{\th})v_{\phi}+\frac{1}{\rho}\,(\o_{\rho}+\cot\phi\,\o_{\phi})v_{\th} \Big]\, e_{\th}.
\end{align*}
Consequently,
\begin{align*}
&\quad -(v\cdot\nabla)\o + (\o\cdot\nabla)v\\
&= \Big[-\Big(v_{\rho}\p_{\rho}+\frac{1}{\rho}v_{\phi}\p_{\phi}\Big)\o_{\rho}
+\Big(\o_{\rho}\p_{\rho}+\frac{1}{\rho}\o_{\phi}\p_{\phi}\Big)v_{\rho}\Big] e_{\rho} \\
&\quad + \Big[ -\Big(v_{\rho}\p_{\rho}+\frac{1}{\rho}v_{\phi}\p_{\phi}\Big)\o_{\phi}+\Big(\o_{\rho}\p_{\rho}+\frac{1}{\rho}\o_{\phi}\p_{\phi}\Big)v_{\phi} +\frac{1}{\rho}\Big(v_{\rho}\o_{\phi}-\o_{\rho}v_{\phi}\Big) \Big]e_{\phi}\\
&\quad + \Big[  -\Big(v_{\rho}\p_{\rho}+\frac{1}{\rho}v_{\phi}\p_{\phi}\Big)\o_{\th}+\Big(\o_{\rho}\p_{\rho}+\frac{1}{\rho}\o_{\phi}\p_{\phi}\Big)v_{\th}+\frac{1}{\rho}(v_{\rho}\o_{\th}-\o_{\rho}v_{\th})+\frac{\cot\phi}{\rho}(v_{\phi}\o_{\th}-\o_{\phi}v_{\th})  \Big]e_{\th}.
\end{align*}
By taking advantage of the formulas for $\o_{\rho}$ and $\o_{\phi}$ in (\ref{vor f-sph}),  we are able to discover some cancellation and therefore simplify the above $e_\th$ component.  Actually,
\[\begin{aligned}
& (\o_{\rho}\p_{\rho}+\frac{1}{\rho}\o_{\phi}\p_{\phi})v_{\th}+\frac{1}{\rho}(v_{\rho}\o_{\th}-\o_{\rho}v_{\th})+\frac{\cot\phi}{\rho}(v_{\phi}\o_{\th}-\o_{\phi}v_{\th}) \\
=& \frac{1}{\rho}(v_{\rho}+\cot\phi\,v_{\phi})\o_{\th} + (\o_{\rho}\p_{\rho}+\frac{1}{\rho}\o_{\phi}\p_{\phi})v_{\th} - \frac{1}{\rho}(\o_\rho+\cot\phi\, \o_\phi)v_\th \\
=& \frac{1}{\rho}(v_{\rho}+\cot\phi\,v_{\phi})\o_{\th}-\frac{1}{\rho^2}\p_{\phi}(v_{\th}^2)+\frac{\cot\phi}{\rho}\p_{\rho}(v_{\th}^2).
\end{aligned}\]
Meanwhile, recall $b=v_\rho e_\rho + v_\phi e_\phi$, so
\be\label{vortex stretch-sph}
\begin{split}
&\quad -(v\cdot\nabla)\o + (\o\cdot\nabla)v\\
&= \Big[-b\cdot \nabla \o_{\rho}
+ \o\cdot \nabla v_{\rho}\Big] e_{\rho} + \Big[ - b\cdot\nabla \o_{\phi}+ \o\cdot\nabla v_{\phi} +\frac{1}{\rho}\Big(v_{\rho}\o_{\phi}-\o_{\rho}v_{\phi}\Big) \Big]e_{\phi}\\
&\quad + \Big[ - b\cdot\nabla \o_{\th}+ \frac{1}{\rho}(v_{\rho}+\cot\phi\,v_{\phi})\o_{\th}-\frac{1}{\rho^2}\p_{\phi}(v_{\th}^2)+\frac{\cot\phi}{\rho}\p_{\rho}(v_{\th}^2)  \Big]e_{\th}.
\end{split}\ee

Putting the above formulas (\ref{laplace of vor-sph}) and (\ref{vortex stretch-sph})  into (\ref{vor eq}) leads to the following (\ref{app-vor-eq}) which is exactly the same as (\ref{asns vor-sph}).
\be\label{app-vor-eq}\begin{cases}
\big(\Delta + \frac{2}{\rho}\,\p_{\rho} +\frac{2}{\rho^2}\big)\o_{\rho}-b\cdot\nabla \o_{\rho} + \o\cdot\nabla v_{\rho} -\p_{t}\o_{\rho}=0,  \vspace{0.1in}\\
\big(\Delta-\frac{1}{\rho^2\sin^2\phi}\big)\o_{\phi}-b\cdot\nabla \o_{\phi}+\frac{2}{\rho^2}\p_{\phi}\o_{\rho}+\o\cdot\nabla v_{\phi} +\frac{1}{\rho}(v_{\rho}\o_{\phi}-\o_{\rho}v_{\phi})-\p_{t}\o_{\phi}=0,   \vspace{0.1in}\\
\big(\Delta-\frac{1}{\rho^2\sin^2\phi}\big)\o_{\th}-b\cdot\nabla \o_{\th}+\frac{1}{\rho}(v_{\rho}+\cot\phi\,v_{\phi})\o_{\th}-\frac{1}{\rho^2}\p_{\phi}(v_{\th}^2)+\frac{\cot\phi}{\rho}\p_{\rho}(v_{\th}^2)-\p_{t}\o_{\th}=0,  \vspace{0.1in}\\
\frac{1}{\rho^2}\p_{\rho}(\rho^2 \o_{\rho})+\frac{1}{\rho\sin\phi}\p_{\phi}(\sin\phi\,\o_{\phi})=0.
\end{cases}\ee

\subsection{System (\ref{eq of K-F-O}) of $K$, $F$ and $\O$}
\label{Subsec, deri of KFO}

\quad

Recall the definition (\ref{K-F-O}) for $K$, $F$ and $\Omega$: $K=\frac{\o_{\rho}}{\rho}$, $F=\frac{\o_{\phi}}{\rho}$ and $\O=\frac{\o_{\th}}{\rho\sin\phi}$. In other words,
\be\label{K-F-O to vor}
\o_\rho = \rho K,\quad \o_\phi = \rho F,\quad \o_\th = \rho\sin\phi\,\Omega.\ee
Let's first deal with $K$.  Based on the first equation in (\ref{app-vor-eq}), we have
\[0 = \Big(\Delta + \frac{2}{\rho}\,\p_{\rho} +\frac{2}{\rho^2}\Big)\o_{\rho}-  \Big(v_{\rho}\p_{\rho}+\frac{1}{\rho}v_{\phi}\p_{\phi}\Big) \o_{\rho} + \Big(\o_{\rho}\p_{\rho}+\frac{1}{\rho}\o_{\phi}\p_{\phi}\Big) v_{\rho} -\p_{t}\o_{\rho}. \]
Putting the relation (\ref{K-F-O to vor}) into this equation yields
\be\label{rhoK}\begin{split}
0 &= \Big(\Delta + \frac{2}{\rho}\,\p_{\rho} +\frac{2}{\rho^2}\Big)(\rho K) -  \Big(v_{\rho}\p_{\rho}+\frac{1}{\rho}v_{\phi}\p_{\phi}\Big) (\rho K) + \Big(\rho K\p_{\rho}+F\p_{\phi}\Big) v_{\rho} -\p_{t}(\rho K). \
\end{split}\ee
Noticing
\[
 \Delta(\rho K)  =  \rho\Delta K + 2\p_\rho K + \frac{2}{\rho}\,K,
 \]
 \[ - \Big(v_{\rho}\p_{\rho}+\frac{1}{\rho}v_{\phi}\p_{\phi}\Big) (\rho K)+v_\rho K = -\rho v_\rho\p_\rho K - v_\phi \p_\phi K = -\rho b\cdot\nabla K,
\]
and
\[\begin{split}
 \Big(\rho K \p_{\rho}+F\p_{\phi}\Big) v_{\rho} - K v_\rho & = K(\rho\p_\rho v_\rho -v_\rho) + F\p_\phi v_\rho\\
 &= \rho \o_\rho \p_\rho\Big(\frac{v_\rho}{\rho}\Big) + \rho \o_\phi \frac{1}{\rho}\p_\phi\Big(\frac{v_\rho}{\rho}\Big) = \rho \o\cdot\nabla \Big(\frac{v_\rho}{\rho}\Big).
\end{split}\]
Putting all these identities into (\ref{rhoK}) yields
\[0= \Big( \rho\Delta + 4\p_\rho + \frac{6}{\rho} \Big)K - \rho b\cdot\nabla K +  \rho \o\cdot\nabla \Big(\frac{v_\rho}{\rho}\Big) -\p_{t}( \rho K).  \]
Dividing by $\rho$ leads to
\[0= \Big(\Delta +\frac{4}{\rho}\p_\rho +\frac{6}{\rho^2} \Big) K - b\cdot\nabla K + \o\cdot\nabla \Big(\frac{v_\rho}{\rho}\Big)-\p_{t}K. \]
This verifies the $K$ equation in (\ref{eq of K-F-O}).

Now we will continue to discuss the case for $F$.  Based on the second equation in (\ref{app-vor-eq}), we know
\[\begin{split}
0 &=\Big(\Delta-\frac{1}{\rho^2\sin^2\phi}\Big)\o_{\phi}- \Big(v_{\rho}\p_{\rho}+\frac{1}{\rho}v_{\phi}\p_{\phi}\Big) \o_{\phi}+\frac{2}{\rho^2}\p_{\phi}\o_{\rho}\\
&\quad + \Big(\o_{\rho}\p_{\rho}+\frac{1}{\rho}\o_{\phi}\p_{\phi}\Big) v_{\phi} +\frac{1}{\rho}(v_{\rho}\o_{\phi}-\o_{\rho}v_{\phi})-\p_{t}\o_{\phi}.
\end{split}\]
Putting the relation (\ref{K-F-O to vor}) into this equation yields
\be\label{rhoF}\begin{split}
0 &=\Big(\Delta-\frac{1}{\rho^2\sin^2\phi}\Big)(\rho F) - \Big(v_{\rho}\p_{\rho}+\frac{1}{\rho}v_{\phi}\p_{\phi}\Big) (\rho F)+\frac{2}{\rho^2}\p_{\phi}(\rho K)\\
&\quad + \Big(\rho K \p_{\rho}+F\p_{\phi}\Big) v_{\phi} + (v_{\rho} F - K v_{\phi})-\p_{t}(\rho F).
\end{split}\ee
Noticing
\[\Delta(\rho F)  =  \rho\Delta F + 2\p_\rho F + \frac{2}{\rho}\,F, \]
\[
- \Big(v_{\rho}\p_{\rho}+\frac{1}{\rho}v_{\phi}\p_{\phi}\Big) (\rho F)+v_\rho F
= -\rho v_\rho\p_\rho F - v_\phi \p_\rho F = -\rho b\cdot\nabla F,
\]
and
\[\begin{split}
\Big(\rho K \p_{\rho}+F\p_{\phi}\Big) v_{\phi} - K v_\phi &= K(\rho\p_\rho v_\phi-v_\phi) + F\p_\phi v_\phi\\
&= \rho \o_\rho \p_\rho\Big(\frac{v_\phi}{\rho}\Big) + \rho \o_\phi \frac{1}{\rho}\p_\phi\Big(\frac{v_\phi}{\rho}\Big) = \rho \o\cdot\nabla \Big(\frac{v_\phi}{\rho}\Big).
\end{split}\]
Putting all these identities into (\ref{rhoF}) yields
\[0= \Big( \rho\Delta + 2\p_\rho + \frac{2}{\rho} - \frac{1}{\rho\sin^2\phi} \Big)F-\rho b\cdot\nabla F + \frac{2}{\rho^2}\p_{\phi}(\rho K) +  \rho \o\cdot\nabla \Big(\frac{v_\phi}{\rho}\Big) - \p_{t}(\rho F).  \]
Dividing by $\rho$ leads to
\[0= \Big( \Delta + \frac{2}{\rho}\,\p_\rho + \frac{1-\cot^2\phi}{\rho^2} \Big)F- b\cdot\nabla F + \frac{2}{\rho^2}\p_{\phi}K +  \o\cdot\nabla \Big(\frac{v_\phi}{\rho}\Big) - \p_{t}F.  \]
This verifies the $F$ equation in (\ref{eq of K-F-O}).

Finally, the equation for $\O$ in spherical coordinates will be deduced.  Rather than deriving its equation directly,  it is helpful to take advantage of the result in the cylindrical coordinates case. In fact, it has already been known that $\O$ satisfies the following equation (see \cite{UY}).
\be\label{eq of O-cyl}
\Delta \O-b\cdot\nabla\O+\frac{2}{r}\,\p_{r}\O- \frac{2}{r^2}\,v_\th \o_r -\p_{t}\O=0.\ee
Based on (\ref{deri-sph to cyl}), we have
\[\frac{2}{r}\,\p_r \O = \frac{2}{\rho \sin\phi}\, \Big( \sin\phi\,\p_{\rho}+\frac{\cos\phi}{\rho}\,\p_{\phi} \Big)\O = \frac{2}{\rho}\, \Big( \p_{\rho}+\frac{\cot\phi}{\rho}\,\p_{\phi} \Big)\O.\]
Applying the relation (\ref{sph-to-cyl}) to $\o$,
\[\frac{2}{r^2}\,v_\th \o_r = \frac{2 v_\th}{\rho^2 \sin^2\phi}\, ( \sin\phi\,\o_{\rho} + \cos\phi\,\o_\phi ) =  \frac{2 v_\th}{\rho \sin\phi}(K + \cot\phi\, F).  \]
Putting the above two identities into (\ref{eq of O-cyl}) gives
\[\Big(\Delta +\frac{2}{\rho}\p_\rho +\frac{2\cot\phi}{\rho^2}\p_\phi \Big) \O-b\cdot\nabla\O-\frac{2v_{\th}}{\rho\sin\phi}\,(K+\cot\phi\, F)-\p_{t}\O=0.\]
This verifies the $\Omega$ equation in (\ref{eq of K-F-O}).

The last equation (\ref{eq of K-F-O})$_{4}$ can be derived immediately from the divergence free condition (\ref{asns vor-sph})$_4$ of $ \o $.

\subsection{Integration Identity for strong solutions of (\ref{nse}) under the NHL boundary condition}
\label{Subsec, weak-form-soln}

\quad

The purpose of this subsection is to justify the integration identity (\ref{test by sf}). Assume $ v\in\mathscr{S}_T $ (see (\ref{ts})) and $ P\in L_t^2 H_x^1(D\times[0,T]) $ such that $ (v,P) $ satisfies the following equations:
\be\label{NS2} \left\{\, \begin{aligned}
	\Delta v - (v\cdot \nabla) v - \nabla P - \p_{t} v = 0 \quad\text{in}\quad & D\times (0,T], \\
	\nabla \cdot v = 0  \quad \text{in} \quad & D\times (0,T], \\
	v\cdot n = 0,\quad \o\times n = 0 \quad\text{on} \quad & \p D\times (0,T],\\
	v(\cdot, 0) = v_0(\cdot) \quad\text{in} \quad & D.
\end{aligned} \right.\ee
Then for any vector field $ f\in \mathscr{S}_{T} $, we will prove the following integration identity:
\be\label{test by sf 2}\begin{split}
	& \int_{D} v(x,T) \cdot f(x,T) \,dx + \int_{0}^{T} \int_{D} (\nabla\times v) \cdot (\nabla \times f) \,dx\,dt \\
	= & \int_{D} v_0(x) \cdot f(x,0) \,dx - \int_{0}^{T}\int_{D} [(v\cdot \nabla) v] \cdot f \,dx\,dt + \int_{0}^{T}\int_{D} v\cdot (\p_t f) \,dx\,dt.
\end{split}\ee

\begin{proof}
Testing the first equation in (\ref{NS2}) by $ f $, we find
\be\label{wif1}\begin{split}
	\underbrace{\int_0^T \int_{D}  f \cdot (\Delta v) \,dx\,dt  }_{I_1} 	&= \underbrace{\int_{0}^{T}\int_{D} [(v\cdot \nabla) v] \cdot f \,dx\,dt}_{I_2} + \underbrace{\int_0^T \int_{D}  f \cdot \nabla P  \,dx\,dt  }_{I_3} \\
	& \qquad + \underbrace{\int_0^T \int_{D} f\cdot (\p_t v) \,dx\,dt  }_{I_4}.
\end{split}\ee

We first compute $ I_1 $. Since $ \nabla \cdot v=0 $, it holds that
\[\begin{split}
	\int_{D} f\cdot \Delta v \,dx &= \sum_{i,j=1}^{3} \int_{D} f_i \p_j^2 v_i \,dx = \sum_{i,j=1}^{3} \int_{D} f_i \p_j (\p_j v_i - \p_i v_j) \,dx + \sum_{i,j=1}^{3} \int_{D} f_i \p_i(\p_j v_j) \,dx \\
	&= \sum_{i,j=1}^{3} \int_{D} f_i \p_j (\p_j v_i - \p_i v_j) \,dx.
\end{split}\]
Then using integration by parts,
\[
\sum_{i,j=1}^{3} \int_{D} f_i \p_j (\p_j v_i - \p_i v_j) \,dx = \sum_{i,j=1}^3 \int_{\p D} f_i (\p_j v_i - \p_i v_j) n_j \,dS - \sum_{i,j=1}^3 \int_{D} (\p_j f_i) (\p_j v_i - \p_i v_j) \,dx.
\]
Since $ \o\times n = 0 $ on $ \p D $, then $ \sum\limits_{j=1}^3 (\p_j v_i - \p_i v_j)n_j=0 $ for any fixed $ i $. Therefore, the above surface integral on the boundary $ \p D $ vanishes and the equation reduces to
\[
\sum_{i,j=1}^{3} \int_{D} f_i \p_j (\p_j v_i - \p_i v_j) \,dx = - \sum_{i,j=1}^3 \int_{D} (\p_j f_i) (\p_j v_i - \p_i v_j) \,dx.
\]
Denote $ J_1 = \sum\limits_{i,j=1}^3 \int_{D} (\p_j f_i) (\p_j v_i - \p_i v_j) \,dx $. Then we can split $ J_1 $ to be $ J_1 = J_{11}+ J_{12} $, where
\[ J_{11} = \sum_{i,j=1}^3 \int_{D} (\p_j f_i - \p_i f_j) (\p_j v_i - \p_i v_j) \,dx, \quad J_{12} = \sum_{i,j=1}^3 \int_{D} (\p_i f_j) (\p_j v_i - \p_i v_j) \,dx. \]
Noticing that $J_{11} = 2\int_{D} (\nabla\times f) \cdot (\nabla\times v) \,dx$ and $J_{12} = -J_1$, we obtain $J_1 = \int_{D} (\nabla\times f) \cdot (\nabla\times v) \,dx$. As a result,
\[ I_1 = - \int_0^T \int_D  (\nabla\times f) \cdot (\nabla\times v) \,dx. \]

Next, we compute the RHS of (\ref{wif1}). For $ I_2 $, it is kept unchanged. For $ I_3 $, it follows from integration by parts and the property of the set $ \mathscr{S}_T $ that
\[ I_3 = \int_0^T \int_{\p D} (f\cdot n) P \,dS - \int_0^T \int_D (\nabla \cdot f) P \,dx = 0. \]
For $ I_4 $, using integration by parts in the temporal variable yields
\[ I_4 =  \int_{D} v(x,T) \cdot f(x,T) \,dx - \int_{D} v_0(x) \cdot f(x,0) \,dx -  \int_{0}^{T}\int_{D} v\cdot (\p_t f) \,dx\,dt.\]

Plugging the above computations of $ I_1$-$I_4 $ into (\ref{wif1}) leads to (\ref{test by sf 2}).

\end{proof}


\section*{Acknowledgments} We wish to thank  Professors Hongjie Dong and Zhen Lei for helpful discussions. Z. Li is supported by Natural Science Foundation of Jiangsu Province (No. BK20200803) and National Natural Science Foundation of China (No. 12001285). X. Pan is supported by National Natural Science Foundation of China (No. 11801268 and No. 12031006). Q. S. Zhang is grateful to the support of the Simons Foundation through grant No. 710364. N. Zhao is supported by Shanghai Sailing Program (No. 22YF1412300) and the Shanghai University of Finance and Economics startup fund.


\bigskip

{\small
\bibliographystyle{plain}
\bibliography{Ref-NS}
}

\end{document}